\RequirePackage{rotating}
\documentclass[a4paper,11pt]{amsart}
\usepackage[utf8]{inputenc}

\usepackage[margin=1in]{geometry}
\usepackage{mathtools}
\usepackage{amsfonts}
\usepackage{mathrsfs}
\usepackage{amsthm}
\usepackage{amssymb}
\usepackage{bm}
\usepackage{tikz-cd}
\usepackage{tikz}
\usetikzlibrary{positioning}
\usepackage[colorlinks,linkcolor=blue,citecolor=blue,urlcolor=blue]{hyperref}
\usepackage[shortlabels]{enumitem}
\usepackage{rotating}

\usepackage[textsize=tiny]{todonotes}
\usepackage[style=alphabetic,doi=false,isbn=false,url=false,maxnames=10,maxalphanames=10,sorting=nyt]{biblatex}
\renewbibmacro*{journal+issuetitle}{%
  \usebibmacro{journal}%
  \iffieldundef{volume}
    { % no volume -> only print year if present (with one leading space)
      \iffieldundef{year}
        {}
        {\setunit{\addspace}\printtext[parens]{\printfield{year}}}%
    }
    { % have volume -> print volume, then optionally " (year)"
      \setunit{\addspace}\printfield{volume}%
      \iffieldundef{year}
        {}
        {\setunit{\addspace}\printtext[parens]{\printfield{year}}}%
    }%
  \setunit*{\addcomma\space}%
  \printfield{number}%
  \newunit}

\DeclareFieldFormat[article, thesis, incollection, inbook]{title}{\mkbibemph{#1}}

\DeclareFieldFormat[article]{journaltitle}{#1}

\DeclareFieldFormat[article]{volume}{\textbf{#1}}

\DeclareFieldFormat[article]{number}{no.~#1}
 
\DeclareFieldFormat[incollection, inbook]{booktitle}{#1}

\renewbibmacro*{in:}{}

\DeclareFieldFormat{pages}{#1}

\newcommand{\Z}{\mathbb{Z}}
\newcommand{\pvd}{\mathrm{pvd}}
\newcommand{\per}{\mathrm{per}}
\newcommand{\ind}{\mathrm{ind}}
\newcommand{\C}{\mathcal{C}}
\newcommand{\D}{\mathcal{D}}
\newcommand{\I}{\mathcal{I}}
\newcommand{\Q}{\mathcal{Q}}
\newcommand{\R}{\mathcal{R}}
\newcommand{\M}{\mathcal{M}}
\newcommand{\Hom}{\mathrm{Hom}}
\newcommand{\RHom}{\bm{\mathrm{R}}\mathrm{Hom}}
\newcommand{\Ltensor}{\bm{\mathrm{L}}}

\newcommand{\rad}{\mathrm{rad}}
\newcommand{\Aut}{\mathrm{Aut}}
\newcommand{\Ext}{\mathrm{Ext}}

\newcommand{\res}{\mathrm{res}}
\newcommand{\id}{\mathrm{id}}
\newcommand{\im}{\imath}
\newcommand{\jm}{\jmath}

\DeclareMathOperator{\modcat}{\mathrm{mod}}
\DeclareMathOperator{\add}{\mathrm{add}}

\newtheorem{thm}{Theorem}[section]
\newtheorem{lemma}[thm]{Lemma}
\newtheorem{prop}[thm]{Proposition}
\newtheorem{cor}[thm]{Corollary}

\theoremstyle{definition}
\newtheorem{defn}[thm]{Definition}
\newtheorem{remark}[thm]{Remark}

\newtheorem{example}[thm]{Example}

\title[A categorification of combinatorial Auslander--Reiten quivers]{A categorification of combinatorial\\
Auslander--Reiten quivers}
\author{Ricardo Canesin}
\address{Université Paris Cité and Sorbonne Université, CNRS, IMJ-PRG, F-75013 Paris, France}
\email{ricardo.canesin@imj-prg.fr}

\keywords{Combinatorial AR quivers, Q-data, reduced words, commutation classes, Ginzburg dg algebras, preprojective algebras} 
\subjclass[2020]{05E10, 18G80, 20F55}

\begin{document}

\begin{abstract}
	We provide a categorification of Oh and Suh's combinatorial Auslander--Reiten quivers in the simply laced case. We work within the perfectly valued derived category $\pvd(\Pi_Q)$ of the 2-dimensional Ginzburg dg algebra of a Dynkin quiver $Q$. For any commutation class $[\bm{i}]$ of reduced words in the corresponding Weyl group, we define a subcategory $\C([\bm{i}])$ of $\pvd(\Pi_Q)$ whose objects are obtained by applying a sequence of spherical twist functors to the simple objects. We describe the $\Hom$-order for $\C([\bm{i}])$ in terms of $[\bm{i}]$, generalizing a result of Bédard. Furthermore, when $[\bm{i}]$ is a commutation class for the longest element, we construct a category $\mathcal{D}([\bm{i}])$ generalizing the bounded derived category of $Q$. It is realized as a certain subquotient of $\pvd(\Pi_Q)$. We demonstrate the existence of particular distinguished triangles in $\mathrm{pvd}(\Pi_Q)$ with corners in $\D([\bm{i}])$, which allows us to extend the classical mesh-additivity to arbitrary commutation classes. Additionally, we define an analog of the Euler form and prove that its symmetrization yields the corresponding Cartan--Killing form. For commutation classes $[\bm{i}]$ arising from Q-data --- a generalization of Dynkin quivers with a height function introduced by Fujita and Oh --- we establish the existence of a partial Serre functor on $\D([\bm{i}])$. Lastly, we apply our results to reinterpret a formula by Fujita and Oh for the inverse of the quantum Cartan matrix.
\end{abstract}

\maketitle

\setcounter{tocdepth}{1}
\tableofcontents

\section{Introduction}

Let $\mathfrak{g}$ be a complex finite-dimensional simple Lie algebra, and let $U_q(L\mathfrak{g})$ denote the associated quantum loop algebra, with $q$ an indeterminate. When $\mathfrak{g}$ is of simply laced type (i.e., type $\mathsf{ADE}$), interesting connections have emerged between the representation theory of $U_q(L\mathfrak{g})$ and that of a Dynkin quiver $Q$ of the same type as $\mathfrak{g}$. A seminal work in this direction is \cite{HernandezLeclerc}, where Hernandez and Leclerc construct a rigid monoidal subcategory $\mathscr{C}_{\Z}$ of the category of finite-dimensional representations of $U_q(L\mathfrak{g})$ and study its quantum Grothendieck ring $\mathscr{K}_t(\mathscr{C}_{\Z})$, a deformation of the classical Grothendieck ring. Their main result is an isomorphism between the derived Hall algebra \cite{Toen} of the derived category of $Q$ over a finite field $F$ and the specialization of $\mathscr{K}_t(\mathscr{C}_{\Z})$ at $t = \sqrt{|F|}$.

A key element in Hernandez and Leclerc’s proof is their computation of the inverse $\widetilde{C}(q)$ of the quantum Cartan matrix of $\mathfrak{g}$, using the Auslander--Reiten (AR) quiver of the bounded derived category $\D^b(\modcat KQ)$ of $Q$ over a field $K$. Their formulas were generalized to arbitrary Dynkin types by Fujita and Oh in \cite{FujitaOh} through their concept of a \emph{Q-datum}, a generalization of a Dynkin quiver with a height function. For each Q-datum $\Q$, they use Weyl group combinatorics to define a \emph{twisted AR quiver} $\Gamma_{\Q}$ and a \emph{generalized twisted Coxeter element} $\tau_{\Q}$. In the simply laced case, these definitions specialize to the AR quiver of the path algebra $KQ$ and its AR translation. We refer the reader to Section \ref{subsection:Q-data and twisted AR quivers} for examples of twisted AR quivers. This framework not only describes $\widetilde{C}(q)$ for nonsimply laced types but also finds significant applications in the representation theory of $U_q(L\mathfrak{g})$ (see, e.g., \cite{FujitaOh}, \cite{FujitaHernandezOhOya22}, \cite{FujitaHernandezOhOya23}, \cite{KashiwaraKimOhPark}).

It is important to note that Fujita and Oh's combinatorics build upon a theory developed earlier by Oh and Suh in \cite{OhSuh19a} and \cite{OhSuh19b}. Let us briefly recall the context of their work. Given a Dynkin diagram $\Delta$ of type $\mathsf{ADE}$, let $\mathsf{W}$ denote its corresponding Weyl group. The longest element of $\mathsf{W}$ is denoted by $w_0$, and one can consider reduced words $\bm{i} = (i_1,\dots,i_N) \in \Delta_0^N$ for $w_0$. We say that two reduced words $\bm{i}$ and $\bm{j}$ are \emph{commutation-equivalent} if one can be transformed into the other via a sequence of short braid relations involving commuting simple reflections of $\mathsf{W}$. This gives rise to a commutation relation whose equivalence classes are called \emph{commutation classes}. If $Q$ is an orientation of $\Delta$, then the set of all reduced words for $w_0$ that are source sequences for $Q$ forms a commutation class, denoted by $[Q]$. It induces a partial order $\preceq_{[Q]}$ on the set of positive roots $\mathsf{R}^+$ of the root system associated with $\Delta$ (see Section \ref{section:commutation classes and combinatorial AR quivers}). Bédard showed in \cite{Bedard} that there is an isomorphism between the AR quiver of the path algebra $KQ$ and the Hasse quiver of the partial order $\preceq_{[Q]}$, sending an indecomposable representation to its dimension vector. On the other hand, one can similarly define a partial order $\preceq_{[\bm{i}]}$ on $\mathsf{R}^+$ for any commutation class $[\bm{i}]$. This led Oh and Suh in \cite{OhSuh19a} to study the Hasse quiver of this partial order and to explore its applications to representation theory. They called it the \emph{combinatorial AR quiver} $\Upsilon_{[\bm{i}]}$ associated with $[\bm{i}]$. We highlight that the twisted AR quivers defined by Fujita and Oh are specific examples of combinatorial AR quivers associated with certain commutation classes introduced in \cite{OhSuh19b}.

If $Q$ is a Dynkin quiver, then $\modcat KQ$ can be studied through the combinatorics of $\Upsilon_{[Q]}$. This raises a natural question: for any commutation class $[\bm{i}]$, is there a ``category of representations''  that can be described by $\Upsilon_{[\bm{i}]}$? Additionally, one might also ask for an analog of the derived category of $Q$. In this paper, we provide candidates for such categories and investigate their properties. In the case of Q-data, we hope that our constructions could shed new light on the connections with the representation theory of quantum loop algebras mentioned earlier.

To explain our constructions in more detail, let $\Pi_Q$ be the (complete) 2-dimensional Ginzburg dg algebra associated with an orientation $Q$ of $\Delta$ (see Section \ref{subsection:Calabi--Yau completion}). We work within its perfectly valued derived category $\pvd(\Pi_Q)$, also known as the finite-dimensional derived category. For each vertex $i \in \Delta_0$, the corresponding simple $\Pi_Q$-module $S_i$ is $2$-spherical as an object in $\pvd(\Pi_Q)$; hence, it yields the spherical twist functor $T_i: \pvd(\Pi_Q) \longrightarrow \pvd(\Pi_Q)$ (see \cite{SeidelThomas}). It is well known that $\pvd(\Pi_Q)$ and its spherical twist functors categorify the root system of $\Delta$ and the action of its Weyl group (see Section \ref{subsection:categorification of the root system}).

Given a reduced word $\bm{i} = (i_1,\dots,i_N) \in \Delta_0^N$ for $w_0$, we define the \emph{category of representations} $\C([\bm{i}])$ associated with the commutation class $[\bm{i}]$ to be the additive subcategory of $\pvd(\Pi_Q)$ generated by the objects
\[
T_{i_1}\dotsb T_{i_{k-1}}(S_{i_k})
\]
for $1 \leq k \leq N$. When $\bm{i}$ is a source sequence for an orientation $Q'$ of $\Delta$, this construction recovers the category $\modcat KQ'$. More generally, we remark that the indecomposable objects above can be seen as the consecutive quotients of a particular filtration of the preprojective algebra of type $\Delta$, introduced by Amiot--Iyama--Reiten--Todorov \cite{AmiotIyamaReitenTodorov}.

Our first main result demonstrates how to recover the combinatorial AR quiver from the category of representations $\C([\bm{i}])$. It can be seen as a generalization of Bédard's result.

\begin{thm}[Theorem \ref{thm:combinatorial AR quiver is obtained from Gabriel quiver}]
Let $[\bm{i}]$ be a commutation class of reduced words for $w_0$. The combinatorial Auslander--Reiten quiver $\Upsilon_{[\bm{i}]}$ is isomorphic to the quiver obtained from the Gabriel quiver of $\C([\bm{i}])$ by removing all arrows parallel to paths of length at least two.
\end{thm}

For general commutation classes, $\C([\bm{i}])$ is no longer an abelian category. Nevertheless, we construct a category $\D([\bm{i}])$ equipped with specific equivalences that resemble $\D^b(\modcat KQ)$ and the classical reflection functors of Bernstein--Gelfand--Ponomarev \cite{BernsteinGelfandPonomarev}. This category is obtained as a quotient of the full subcategory of $\pvd(\Pi_Q)$ generated by the objects of $\C([\bm{i}])$ and their shifts. See Section \ref{section:objects} for a precise definition. We call it the \emph{c-derived category of} $[\bm{i}]$.

Although $\D([\bm{i}])$ has no natural structure of triangulated category, we show the existence of some particular distinguished triangles in $\pvd(\Pi_Q)$ with corners in $\D([\bm{i}])$ (Theorem \ref{thm:mesh gives rise to triangles}). They should be thought of as analogs to the AR triangles in $\D^b(\modcat KQ)$ \cite{Happel}. As a result (Corollary \ref{cor:generalized g-additive property}), we provide an alternative proof of the $\mathfrak{g}$-additive property from \cite[Theorem 3.41]{FujitaOh}, extending it to arbitrary commutation classes.

The c-derived category is endowed with a suspension functor $\Sigma$, allowing us to define extension groups between objects $M,N \in \D([\bm{i}])$ as
\[
\Ext^n_{[\bm{i}]}(M,N) = \Hom_{\D([\bm{i}])}(M,\Sigma^nN).
\]
If $M,N \in \C([\bm{i}])$, these functors vanish unless $n=0$ or $n=1$, implying that $\C([\bm{i}])$ is ``hereditary'' in a similar manner to the fact that $KQ$ is a hereditary algebra. Using these groups, we construct an Euler form on $\D([\bm{i}])$, and we prove that its symmetrization induces the symmetric bilinear form on the root lattice of $\Delta$ (Theorem \ref{thm:symmetrization of Euler form is Cartan form}). 

Let us now consider the case of a commutation class $[\Q]$ given by a Q-datum $\Q = (\Delta,\sigma,\xi)$, in the sense of \cite{FujitaOh}. We categorify the generalized twisted Coxeter element $\tau_{\Q}$ of \cite[Section 3.6]{FujitaOh} and view it as an autoequivalence of $\D([\Q])$. If $r$ is the order of the automorphism $\sigma$ of $\Delta$, let $S_{\Q}: \D([\Q]) \longrightarrow \D([\Q])$ be the composition of $\tau_{\Q}^r$ with the suspension functor $\Sigma$. It behaves almost like a Serre functor \cite{BondalKapranov} and reveals a remarkable symmetry of $\D([\Q])$.
\begin{thm}[Theorem \ref{thm:Serre duality}]
Let $\Q = (\Delta, \sigma, \xi)$ be a Q-datum and take an indecomposable object $M \in \D([\Q])$ of residue $\im \in \Delta_0$. If the orbit of $\im$ under the action of $\sigma$ has $r$ elements, then there is an isomorphism
\[\begin{tikzcd}
	{\Hom_{\D([\Q])}(M, N)} & {D\Hom_{\D([\Q])}(N,S_{\Q}M)}
	\arrow["\sim", from=1-1, to=1-2]
\end{tikzcd}\]
natural in the variable $N \in \D([\Q])$, where $D$ denotes the duality functor for $K$-vector spaces.
\end{thm}

As an application, we reinterpret a formula from \cite{FujitaOh} for the inverse $\widetilde{C}(q)$ of the quantum Cartan matrix of $\mathfrak{g}$ in terms of the Euler form $\langle-,-\rangle_{\Q}$ on the c-derived category associated with a Q-datum $\Q$ of the same type.

To state this result, let us introduce some notation. For $u \in \Z$, denote by $\widetilde{c}_{ij}(u)$ the coefficient of $q^u$ in the expansion of the $(i,j)$-entry of $\widetilde{C}(q)$ as a formal Laurent series in the indeterminate $q$. Let $I$ be the Dynkin diagram of $\mathfrak{g}$, and $D = \mathrm{diag}(d_i \mid i \in I)$ the minimal left symmetrizer of its Cartan matrix. Let $\widehat{I} \subset I \times \Z$ be the associated \emph{folded repetition quiver} (\cite[Section 3.8]{FujitaOh}). Given a Q-datum $\Q$ for $\mathfrak{g}$, there is a bijection $H_{\Q}$ from $\widehat{I}$ to the set of isomorphism classes of indecomposable objects of the c-derived category $\D([\Q])$. See Section \ref{subsection:applications} for more details.

\begin{thm}[Proposition \ref{prop:computing the inverse of the quantum cartan matrix}]
Let $(i,p), (j,s) \in \widehat{I}$ be such that $p - s + d_i \geq 0$. If we have $\max\{d_i,d_j\} = r$, then
\[
\widetilde{c}_{ij}(p-s+d_i) = \left\langle H_{\Q}(j,s), \bigoplus_{k=0}^{\lceil d_j/d_i\rceil-1}\tau_{\Q}^k(H_{\Q}(i,p))\right\rangle_{\!\!\Q}
\]
for any Q-datum $\Q = (\Delta, \sigma, \xi)$ for $\mathfrak{g}$.
\end{thm}

When $\mathfrak{g}$ is simply laced, this result recovers \cite[Corollary 3.6]{Fujita}, which comes from a formula of Hernandez--Leclerc in \cite{HernandezLeclerc}.

\subsection*{Organization} This paper is organized as follows. In Section \ref{section:Q-data combinatorics}, we recall the main definitions and results of \cite{OhSuh19a}, \cite{OhSuh19b} and \cite{FujitaOh}. In Section \ref{section:the categories}, the 2-dimensional Ginzburg dg algebra $\Pi_Q$, the category $\pvd(\Pi_Q)$, and its spherical twists are introduced, and we explain how they categorify the root system of the same type. The category of representations $\C([\bm{i}])$ and the c-derived category $\D([\bm{i}])$, together with an auxiliary repetition category $\R([\bm{i}])$, are constructed in Section \ref{section:objects}. In Section \ref{section:morphisms}, we describe the Gabriel quiver of these categories in terms of combinatorial AR quivers and combinatorial repetition quivers. In Section \ref{section:meshes and distinguished triangles}, we prove the existence of certain distinguished triangles in $\pvd(\Pi_Q)$ with corners in $\R([\bm{i}])$ and deduce the generalized $\mathfrak{g}$-additive property. In Section \ref{section:extension groups and the Euler form}, we introduce projective and injective objects, extension groups, and the Euler form. The case of commutation classes coming from Q-data is treated in Section \ref{section:the case of Q-data}, where we categorify the generalized twisted Coxeter element $\tau_{\Q}$, prove the partial Serre duality property, and reinterpret some results of \cite{FujitaOh} on inverse quantum Cartan matrices. Finally, in Appendix \ref{appendix:examples}, we provide a description of the indecomposable objects of $\C([\Q])$ for a Q-datum $\Q$ in all nonsimply laced Dynkin types.

\subsection*{Acknowledgments} I thank my Ph.D. advisor, Bernhard Keller, for his invaluable guidance throughout this project. His ideas and expertise were essential to its development, and I am also grateful for his support during the writing of this article. I thank Geoffrey Janssens for interesting discussions on Q-data and quantum affine algebras, which led to the results concerning the Euler form and to the definition of the c-derived category. I am also grateful to Hugh Thomas for drawing my attention to the connection between this work and \cite{AmiotIyamaReitenTodorov}. I thank Se-jin Oh for pointing out that Corollary \ref{cor:generalized g-additive property} was already implicit in some of his works (see Remark \ref{rem:implicit in the literature}). Last but not least, I thank the anonymous referee for their valuable remarks during the review of this paper.

\section{Combinatorial AR quivers and Q-data combinatorics}\label{section:Q-data combinatorics}

The main goal of this section is to recall the definition of combinatorial Auslander--Reiten (AR) quivers, Q-data, and their combinatorics, following \cite{OhSuh19a}, \cite{OhSuh19b} and \cite{FujitaOh}. We also prove some lemmas needed for our main results. Along the way, we introduce most of the language and notation used in the subsequent sections.

\subsection{Simply laced root systems}\label{subsection:basic notation for root systems}
From now on, $\Delta$ denotes a simply laced Dynkin diagram
(i.e., of type $\mathsf{ADE}$) with vertex set $\Delta_0$.
For $i,j \in \Delta_0$, we write $i \sim j$ if the vertices $i$
and $j$ are adjacent in $\Delta$.

Let $\mathsf{g}$ be the complex finite-dimensional simple Lie
algebra associated with $\Delta$, and denote by $\mathsf{h}$ a
Cartan subalgebra. Denote by $\mathsf{R} \subset \mathsf{h}^*$ 
the associated root system, where $\mathsf{h}^*$ is the dual 
of $\mathsf{h}$. Let $\{\alpha_i\}_{i \in \Delta_0}$ and 
$\{\varpi_i\}_{i \in \Delta_0}$ denote a choice of simple roots and
the corresponding fundamental weights, respectively, which live 
inside $\mathsf{h}^*$. With this choice, denote by $\mathsf{R}^+$ 
the set of positive roots and define $N = |\mathsf{R}^+|$. 
The simple roots form a basis for the root lattice $\mathsf{Q}$, while the 
fundamental weights form a basis for the weight lattice $\mathsf{P}$, 
both contained in $\mathsf{h}^*$. Notice that $\mathsf{Q} \subset \mathsf{P}$. 
We have a symmetric bilinear form 
$(-,-): \mathsf{P} \times \mathsf{P} \longrightarrow \mathbb{Q}$ 
defined by $(\varpi_i,\alpha_j) = \delta_{ij}$.

Denote by $\mathsf{W} \subset \Aut(\mathsf{P})$ the Weyl group 
associated with the root system $\mathsf{R}$. It is generated by 
the simple reflections 
$s_i: \mathsf{P} \longrightarrow \mathsf{P}$ ($i \in \Delta_0$) 
defined by $s_i(\lambda) = \lambda - (\lambda,\alpha_i)\alpha_i$ 
for $\lambda \in \mathsf{P}$. Such reflections are involutions 
subject to the \emph{commutation relation} $s_is_j = s_js_i$, 
if $i \not\sim j$, and the \emph{braid relation} $s_is_js_i = s_js_is_j$, 
if $i \sim j$. In this way, the pair 
$(\mathsf{W}, \{s_i\}_{i \in \Delta_0})$ forms a finite Coxeter 
system. Let $w_0$ be the longest element of $\mathsf{W}$, which has 
length $N$. It induces an involution $i \mapsto i^*$ of $\Delta_0$ 
determined by $w_0(\alpha_i) = -\alpha_{i^*}$.

\subsection{Commutation classes and combinatorial AR 
quivers}\label{section:commutation classes and combinatorial AR quivers}

A sequence $\bm{i} = (i_1,\dots,i_t)$ of vertices in $\Delta$ 
is a \emph{reduced word for} 
$w \in \mathsf{W}$ if $w$ has length $t$ and we have $w = s_{i_1}\dotsb s_{i_t}$. If $\bm{i}$ can be obtained from another reduced word $\bm{i'}$ 
by the direct application of a commutation relation, we say that 
$\bm{i}$ is obtained from $\bm{i'}$ by a \emph{commutation move}. Two reduced words for $w \in \mathsf{W}$ are \emph{commutation equivalent} if one can be obtained from the other by a sequence of commutation moves. This defines an equivalence relation on the set of reduced words for $w$. The equivalence class containing a reduced word $\bm{i}$ is called a \emph{commutation class} and is denoted by $[\bm{i}]$.

Let $\bm{i} = (i_1,\dots,i_t)$ be a reduced word for $w \in \mathsf{W}$. 
Define $\beta^{\bm{i}}_k = s_{i_1}\dotsb s_{i_{k-1}}(\alpha_{i_k})$ 
for $1 \leq k \leq t$. They are distinct positive 
roots and form a subset $\mathsf{R}^+(w) \subseteq \mathsf{R}^+$ which depends only on $w$ (\cite[p. 158, Corollaire 2]{Bourbaki}). In particular, if $w = w_0$ is the longest element, we have $\mathsf{R}^+(w_0) = \mathsf{R}^+$. There is a total order $<_{\bm{i}}$ on  $\mathsf{R}^+(w)$ 
defined by $\beta^{\bm{i}}_k <_{\bm{i}} \beta^{\bm{i}}_l$ 
if $k < l$. For a commutation class $[\bm{i}]$, we define a 
partial order $\preceq_{[\bm{i}]}$ on $\mathsf{R}^+(w)$ given by 
$\alpha \preceq_{[\bm{i}]} \beta$ if $\alpha = \beta$ or $\alpha <_{\bm{j}} \beta$ 
for all $\bm{j} \in [\bm{i}]$.

For a commutation class $[\bm{i}]$ and a positive root 
$\alpha \in \mathsf{R}^+(w)$, we define the \emph{residue of $\alpha$ 
with respect to $[\bm{i}]$}, denoted by $\res^{[\bm{i}]}(\alpha)$, 
to be $i_k \in \Delta_0$ if we have $\bm{i} = (i_1,\dots,i_t)$ and $\alpha = \beta^{\bm{i}}_k$. As the notation suggests, this is 
well defined and independent of the choice of representative for the 
commutation class.

When $w = w_0$, we define the notion of injective and projective positive roots with respect to $[\bm{i}]$. Write $\bm{i} = (i_1,\dots,i_N)$. The \emph{injective positive root associated with} $i \in \Delta_0$ is $\beta^{\bm{i}}_k$ where $1 \leq k \leq N$ is the first index such that $i_k = i$. Dually, the \emph{projective positive root associated with} $i \in \Delta_0$ is $\beta^{\bm{i}}_l$ where $1 \leq l \leq N$ is the last index such that $i_l = i^*$. Both definitions do not depend on the choice of representative for $[\bm{i}]$.

Following \cite{OhSuh19a} (which builds on \cite{Bedard}), we define the 
\emph{combinatorial AR quiver} $\Upsilon_{[\bm{i}]}$ 
associated with a commutation class $[\bm{i}]$ of reduced words for 
$w \in \mathsf{W}$ as follows. Its vertex set $(\Upsilon_{[\bm{i}]})_0$ is the set $\mathsf{R}^+(w)$ of positive roots associated with $w$. After having chosen a representative 
$\bm{i} = (i_1,\dots,i_t)$ for $[\bm{i}]$, we define an arrow in 
$\Upsilon_{[\bm{i}]}$ from $\beta^{\bm{i}}_k$ to $\beta^{\bm{i}}_l$ 
if $1 \leq l < k \leq t$, $i_k \sim i_l$ and there is 
no index $l < j < k$ such that $i_j = i_k$ or $i_j = i_l$. This 
definition is also independent of the choice of the representative 
for the commutation class.

\begin{thm}[{\cite[Theorem 2.22]{OhSuh19a}}]\label{thm:combinatorial AR quiver realizes partial order}
Let $[\bm{i}]$ be a commutation class of reduced words for $w \in \mathsf{W}$. 
The combinatorial AR quiver $\Upsilon_{[\bm{i}]}$ is 
isomorphic to the Hasse quiver of the partial order 
$\preceq_{[\bm{i}]}$. More precisely, for 
$\alpha, \beta \in \mathsf{R}^+(w)$, we have 
$\alpha \preceq_{[\bm{i}]} \beta$ if and only if there is a path 
from $\beta$ to $\alpha$ in $\Upsilon_{[\bm{i}]}$.
\end{thm}

Let $X \subseteq (\Upsilon_{[\bm{i}]})_0$. A total ordering $(x_1,x_2,\dots,x_l)$ of $X$ is a \emph{compatible reading of} $X$ if we have $k \leq k'$ whenever there is an oriented path from $x_{k'}$ to $x_k$ in 
$\Upsilon_{[\bm{i}]}$. Given such a compatible reading, we 
define the element $w[X] \in \mathsf{W}$ to be
\[
w[X] = s_{i_1}s_{i_2}\dotsb s_{i_l},
\]
where $i_k = \res^{[\bm{i}]}(x_k)$ for $1 \leq k \leq l$. By adapting the proof of \cite[Lemma 3.18]{FujitaOh}, one can show that $w[X]$ is independent 
of the choice of compatible reading.

\begin{thm}[{\cite[Theorem 2.22]{OhSuh19a}}]\label{thm:compatible reading gives reduced word}
Let $[\bm{i}]$ be a commutation class of reduced words for $w \in \mathsf{W}$. A sequence 
$(i_1,\dots,i_t)$ of vertices in $\Delta_0$ is in $[\bm{i}]$ if and only if there is a compatible 
reading $(x_1,\dots,x_t)$ of $(\Upsilon_{[\bm{i}]})_0$ such that 
$i_k = \res^{[\bm{i}]}(x_k)$ for all $1 \leq k \leq t$.
\end{thm}

\subsection{Combinatorial reflection functors and repetition quivers}\label{subsection:combinatorial reflection functors} In this subsection, we only work with reduced words for the longest element $w_0 \in \mathsf{W}$. We say that $i \in \Delta_0$ is a \emph{source} (resp. \emph{sink}) of a commutation class $[\bm{i}]$ of reduced words for $w_0$ if there is a reduced word $\bm{j} \in [\bm{i}]$ starting with $i$ (resp. ending with $i^*$). By \cite[Proposition 4.4]{OhSuh19a}, $i$ is a source (resp. sink) of $[\bm{i}]$ if and only if $\alpha_i$ is a sink (resp. source) of $\Upsilon_{[\bm{i}]}$.

If $\bm{j} = (j_1,\dots,j_N)$ is a reduced word for $w_0$, then so is the sequence $r_{j_1}\bm{j} = (j_2,\dots,j_N,j_1^*)$. If $i \in \Delta_0$ is a source of a commutation class $[\bm{i}]$, we define $r_i[\bm{i}] = [r_i\bm{j}]$ where $\bm{j} \in [\bm{i}]$ is a reduced word starting with $i$. This operation on commutation classes is well defined and is called a \emph{(combinatorial) reflection functor}. It is injective and its image is the set of commutation classes for which $i$ is a sink. We denote its inverse by $r_i^{-1}$. A sequence of vertices $(i_1,\dots,i_k)$ of $\Delta$ is a \emph{source sequence for $[\bm{i}]$} if $i_l$ is a source of $r_{i_{l-1}}r_{i_{l-2}}\dotsb r_{i_1}[\bm{i}]$ for all $1 \leq l \leq k$. Similarly, a sequence of vertices $(i_1,\dots,i_k)$ of $\Delta$ is a \emph{sink sequence for $[\bm{i}]$} if $i_l$ is a sink of $r_{i_{l-1}}^{-1}r_{i_{l-2}}^{-1}\dotsb r_{i_1}^{-1}[\bm{i}]$ for all $1 \leq l \leq k$. If two commutation classes can be obtained from one another by applying a sequence of reflection functors (or their inverses), we say that they are \emph{in the same $r$-cluster point}.

For a reduced word $\bm{i} = (i_1,\dots,i_N)$ for $w_0$, denote by $\widehat{\bm{i}} = (i_k)_{k \in \Z}$ the infinite sequence extending $\bm{i}$ and satisfying $i_{k+N} = i_k^*$ for $k \in \Z$. Define $\widehat{\beta}^{\bm{i}}_{s+tN} = ((-1)^t\beta^{\bm{i}}_s,-t) \in \mathsf{R} \times \Z$ for $1 \leq s \leq N$ and $t \in \Z$. Since the permutation $i \mapsto i^*$ is induced by $w_0$, observe that the first coordinate of $\widehat{\beta}^{\bm{i}}_k$ is
\[
\begin{cases}
	s_{i_1}s_{i_2}\dotsb s_{i_{k-1}}(\alpha_{i_k}) &\textrm{if } k \geq 1,\\
	-s_{i_0}s_{i_{-1}}\dotsb s_{i_{k+1}}(\alpha_{i_k}) &\textrm{if } k \leq 0.
\end{cases}
\]
Let $\widehat{\mathsf{R}}$ be the subset of $\mathsf{R} \times \Z$ whose elements are the roots $\widehat{\beta}^{\bm{i}}_k$ for $k \in \Z$. Alternatively, since $\mathsf{R}^+(w_0) = \mathsf{R}^+$, $\widehat{\mathsf{R}}$ is the subset of pairs $(\alpha,k) \in \mathsf{R} \times \Z$ such that $(-1)^k\alpha \in \mathsf{R}^+$.

Let $[\bm{i}]$ be a commutation class of reduced words for $w_0$. For any root $\alpha \in \mathsf{R}$, we define its residue with respect to $[\bm{i}]$, denoted by $\res^{[\bm{i}]}(\alpha)$, to be $i_k \in \Delta_0$ if we have $\widehat{\bm{i}} = (i_l)_{l \in \Z}$ and $\alpha$ is the first coordinate of $\widehat{\beta}^{\bm{i}}_k$. It is well defined and extends the previous definition of residue. We remark that $\res^{[\bm{i}]}(-\alpha) = \res^{[\bm{i}]}(\alpha)^*$. For $(\alpha,k) \in \widehat{\mathsf{R}}$, we define its residue to be $\res^{[\bm{i}]}(\alpha)$.

We can enlarge the combinatorial AR quiver $\Upsilon_{[\bm{i}]}$ to the \emph{(combinatorial) repetition quiver} $\widehat{\Upsilon}_{[\bm{i}]}$ as follows. Its vertex set $(\widehat{\Upsilon}_{[\bm{i}]})_0$ is the set $\widehat{\mathsf{R}}$ defined above. Given a representative $\bm{i} \in [\bm{i}]$, we construct the infinite sequence $\widehat{\bm{i}} = (i_k)_{k \in \Z}$. We define an arrow in $\widehat{\Upsilon}_{[\bm{i}]}$ from $\widehat{\beta}^{\bm{i}}_k$ to $\widehat{\beta}^{\bm{i}}_l$ if $l < k$, $i_k \sim i_l$ and there is no index $l < j < k$ such that $i_j = i_k$ or $i_j = i_l$. As before, $\widehat{\Upsilon}_{[\bm{i}]}$ depends only on the commutation class. By construction, we can identify $\Upsilon_{[\bm{i}]}$ with the full subquiver of $\widehat{\Upsilon}_{[\bm{i}]}$ with vertex set $\mathsf{R}^+ \times \{0\} \subset \widehat{\mathsf{R}}$. In particular, we see the injective (resp. projective) positive roots as \emph{injective} (resp. \emph{projective}) \emph{vertices} of $\widehat{\Upsilon}_{[\bm{i}]}$. Similarly, as we did for $\Upsilon_{[\bm{i}]}$, we can define what is a compatible reading for any finite subset $X$ of $(\widehat{\Upsilon}_{[\bm{i}]})_0$ and use it to define an element $w[X] \in \mathsf{W}$.

If $\widehat{\beta}^{\bm{i}}_k$ and $\widehat{\beta}^{\bm{i}}_l$ have the same residue and $l < k$, then there is a path in $\widehat{\Upsilon}_{[\bm{i}]}$ from $\widehat{\beta}^{\bm{i}}_k$ to $\widehat{\beta}^{\bm{i}}_l$ when $\Delta \neq \mathsf{A}_1$. We refer to this as the \emph{segment property of $\widehat{\Upsilon}_{[\bm{i}]}$}. We also point out that $\Upsilon_{[\bm{i}]}$ is a \emph{convex subquiver} of $\widehat{\Upsilon}_{[\bm{i}]}$, that is, any directed path in $\widehat{\Upsilon}_{[\bm{i}]}$ between vertices of $\Upsilon_{[\bm{i}]}$ is entirely contained in $\Upsilon_{[\bm{i}]}$. This follows from the fact that, for an arrow $(\alpha,k) \to (\alpha',k')$ in $\widehat{\Upsilon}_{[\bm{i}]}$, we must have $k' = k$ or $k' = k+1$.

\begin{example}
Suppose $\Delta = \mathsf{A}_3$ and enumerate its vertices as $\Delta_0 = \{1,2,3\}$, where $2$ is the central vertex. An example of reduced word for $w_0$ is $\bm{i} = (3,2,1,2,3,2)$. In this case, the repetition quiver $\widehat{\Upsilon}_{[\bm{i}]}$ is the following:
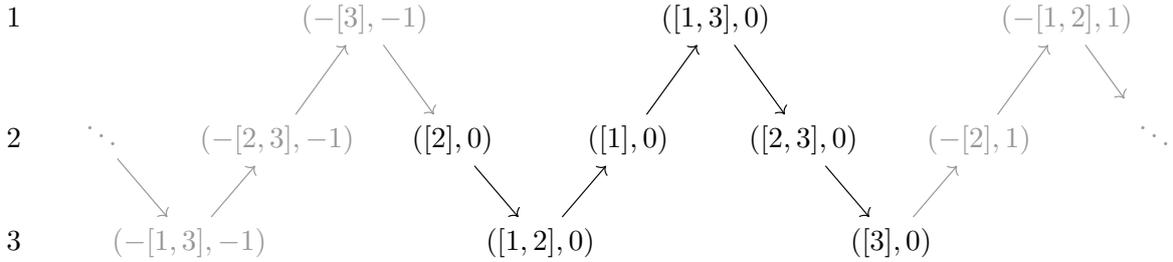
\[\begin{tikzcd}[column sep={3em,between origins}]
	\phantom{0} & \phantom{0} & \phantom{0} & \phantom{0} & \phantom{0} & \phantom{0} & \phantom{0} & \phantom{0} & \phantom{0} & \phantom{0} & \phantom{0} & \phantom{0} & \phantom{0} & \phantom{0}\\[-3.5em]
	1 &&&&  \textcolor{rgb,255:red,150;green,150;blue,150}{(-[3],-1)} &&&& {([1,3],0)} &&&& \textcolor{rgb,255:red,150;green,150;blue,150}{(-[1,2],1)} \\
	2 & \textcolor{rgb,255:red,150;green,150;blue,150}{\ddots} && \textcolor{rgb,255:red,150;green,150;blue,150}{(-[2,3],-1)} && {([2],0)} && {([1],0)} && {([2,3],0)} && \textcolor{rgb,255:red,150;green,150;blue,150}{(-[2],1)} && \textcolor{rgb,255:red,150;green,150;blue,150}{\ddots} \\
	3 && \textcolor{rgb,255:red,150;green,150;blue,150}{(-[1,3],-1)} &&&& {([1,2],0)} &&&& {([3],0)}
	\arrow[color={rgb,255:red,150;green,150;blue,150}, from=2-5, to=3-6]
	\arrow[from=2-9, to=3-10]
	\arrow[color={rgb,255:red,150;green,150;blue,150}, from=2-13, to=3-14]
	\arrow[color={rgb,255:red,150;green,150;blue,150}, from=3-2, to=4-3]
	\arrow[color={rgb,255:red,150;green,150;blue,150}, from=3-4, to=2-5]
	\arrow[from=3-6, to=4-7]
	\arrow[from=3-8, to=2-9]
	\arrow[from=3-10, to=4-11]
	\arrow[color={rgb,255:red,150;green,150;blue,150}, from=3-12, to=2-13]
	\arrow[color={rgb,255:red,150;green,150;blue,150}, from=4-3, to=3-4]
	\arrow[from=4-7, to=3-8]
	\arrow[color={rgb,255:red,150;green,150;blue,150}, from=4-11, to=3-12]
\end{tikzcd}\]
Here, we denote $[i,j] = \alpha_i + \alpha_{i+1} + \dotsb + \alpha_j$ for $1 \leq i \leq j \leq 3$. In each row, all vertices have the same residue, which is indicated in the column on the left. The subquiver highlighted in a darker color is $\Upsilon_{[\bm{i}]}$. The injective vertices associated with $1$, $2$, and $3$ are $([1,3],0)$, $([2,3],0)$ and $([3],0)$, respectively. The corresponding projective vertices are $([1,2],0)$, $([2],0)$ and $([1,3],0)$.
\end{example}

If $\bm{i}$ is a reduced word for $w_0$ starting with $i \in \Delta_0$, notice that the infinite sequence $\widehat{r_i\bm{i}}$ is just a shift of $\widehat{\bm{i}}$. It is thus immediate from the definition that we have an isomorphism of quivers $\widehat{\Upsilon}_{[\bm{i}]} \longrightarrow \widehat{\Upsilon}_{[r_i\bm{i}]}$ that preserves residues. It sends the vertex $\widehat{\beta}^{\bm{i}}_l$ to the vertex $\widehat{\beta}^{r_i\bm{i}}_{l-1}$ or, equivalently, the vertex $(\alpha,k) \in \widehat{\mathsf{R}}$ to the vertex
\[
\begin{cases}
	(s_i(\alpha),k) &\textrm{if } \alpha \neq \pm\alpha_i,\\
	(s_i(\alpha),k+1) &\textrm{if } \alpha = \pm\alpha_i.
\end{cases}
\]
Consequently, if two commutation classes $[\bm{i}]$ and $[\bm{j}]$ are in the same $r$-cluster point, a sequence of reflection functors connecting them defines an isomorphism $\widehat{\Upsilon}_{[\bm{i}]} \longrightarrow \widehat{\Upsilon}_{[\bm{j}]}$ preserving residues.

\begin{lemma}\label{lemma:projective goes to projective iff injective goes to injective}
For two commutation classes $[\bm{i}]$ and $[\bm{j}]$ in the same $r$-cluster point, let $\varphi: \widehat{\Upsilon}_{[\bm{i}]} \longrightarrow \widehat{\Upsilon}_{[\bm{j}]}$ be an isomorphism as above. Let $x_{[\bm{i}]}, x^*_{[\bm{i}]} \in (\widehat{\Upsilon}_{[\bm{i}]})_0$ be the injective and projective vertices associated with $i \in \Delta_0$. Then $\varphi(x_{[\bm{i}]})$ is injective if and only if $\varphi(x^*_{[\bm{i}]})$ is projective.
\end{lemma}

\begin{proof}
We may assume $\Delta \neq \mathsf{A}_1$. Let $[\bm{i}]$ be an arbitrary commutation class of reduced words for $w_0$ and take $i \in \Delta_0$. Denote by $X_{[\bm{i}]} \subset \widehat{\mathsf{R}}$ the set of all pairs with residue $i$ with respect to $[\bm{i}]$. Let $x_{[\bm{i}]} \in X_{[\bm{i}]}$ be the injective vertex associated with $i$. By the segment property of $\widehat{\Upsilon}_{[\bm{i}]}$, there is a unique bijective map $d_{[\bm{i}]}: X_{[\bm{i}]} \longrightarrow \Z$ such that $d_{[\bm{i}]}(x_{[\bm{i}]}) = 0$ and $d_{[\bm{i}]}(v_1) < d_{[\bm{i}]}(v_2)$ if and only if there is a nontrivial path from $v_1$ to $v_2$ in $\widehat{\Upsilon}_{[\bm{i}]}$. If $j \in \Delta_0$ is a source of $[\bm{i}]$, let $\varphi_j: \widehat{\Upsilon}_{[\bm{i}]} \longrightarrow \widehat{\Upsilon}_{r_j[\bm{i}]}$ be the induced map. Since it preserves residues and is an isomorphism of quivers, we have $\varphi_j(X_{[\bm{i}]}) = X_{r_j[\bm{i}]}$ and the difference $d_{r_j[\bm{i}]} \circ \varphi_j - d_{[\bm{i}]}$ is always constant and equal to $m = d_{r_j[\bm{i}]}(\varphi_j(x^{[\bm{i}]}))$. A simple case-by-case analysis shows that $m = \delta_{ij}$, that is, Kronecker's delta.

Analogously, if $X^*_{[\bm{i}]} \subset \widehat{\mathsf{R}}$ is the set of all pairs with residue $i^*$ with respect to $[\bm{i}]$ and $x^*_{[\bm{i}]} \in X^*_{[\bm{i}]}$ denotes the projective vertex associated with $i$, there is a unique bijective map $d^*_{[\bm{i}]}: X^*_{[\bm{i}]} \longrightarrow \Z$ such that $d^*_{[\bm{i}]}(x^*_{[\bm{i}]}) = 0$ and $d^*_{[\bm{i}]}(v_1) < d^*_{[\bm{i}]}(v_2)$ if and only if there is a nontrivial path from $v_1$ to $v_2$ in $\widehat{\Upsilon}_{[\bm{i}]}$. If $j \in \Delta_0$ is a source of $[\bm{i}]$, one can show as before that $\varphi_j(X^*_{[\bm{i}]}) = X^*_{r_j[\bm{i}]}$ and the difference $d^*_{r_j[\bm{i}]} \circ \varphi_j - d^*_{[\bm{i}]}$ is a constant function equal to $\delta_{ij}$.

We can now prove the lemma. Applying the argument above multiple times, we can show that $d_{[\bm{j}]} \circ \varphi - d_{[\bm{i}]}$ and $d^*_{[\bm{j}]} \circ \varphi - d^*_{[\bm{i}]}$ are constant and equal to same value. But we have
\[
\varphi(x_{[\bm{i}]}) \textrm{ is injective } \iff d_{[\bm{j}]}(\varphi(x_{[\bm{i}]})) = 0 \iff d_{[\bm{j}]} \circ \varphi - d_{[\bm{i}]} \equiv 0
\]
and, similarly, $\varphi(x^*_{[\bm{i}]})$ is projective if and only if $d^*_{[\bm{j}]} \circ \varphi - d^*_{[\bm{i}]}$ is identically zero. This finishes the proof.
\end{proof}

\subsection{Meshes of repetition quivers}\label{subsection:meshes} Throughout this subsection, we assume that $\Delta \neq \mathsf{A}_1$. Let $\bm{i}$ be a reduced word for the longest element $w_0$ and write $\widehat{\bm{i}} = (i_l)_{l \in \Z}$. Given $k \in \Z$, let $k^+ \in \Z$ be the smallest integer such that $k < k^+$ and $i_k = i_{k^+}$. We define the \emph{mesh of $\widehat{\Upsilon}_{[\bm{i}]}$ at the vertex $\widehat{\beta}^{\bm{i}}_k$} to be the smallest convex subquiver $\M = \M_{[\bm{i}]}(\widehat{\beta}^{\bm{i}}_k)$ of $\widehat{\Upsilon}_{[\bm{i}]}$ containing the vertices $\widehat{\beta}^{\bm{i}}_k$ and $\widehat{\beta}^{\bm{i}}_{k^+}$. Equivalently, $\M$ is the full subquiver of $\widehat{\Upsilon}_{[\bm{i}]}$ whose set of vertices $\M_0$ consists of those vertices that appear in a path from $\widehat{\beta}^{\bm{i}}_{k^+}$ to $\widehat{\beta}^{\bm{i}}_k$. If $x = \widehat{\beta}^{\bm{i}}_k$, we will denote $\widehat{\beta}^{\bm{i}}_{k^+}$ by $s_{[\bm{i}]}(x)$, which is the unique source of $\M$. We define the \emph{set of abutters $V_{[\bm{i}]}(x)$ of $x$} to be the set of all vertices $y \in \M_0$ such that $\res^{[\bm{i}]}(y) \sim \res^{[\bm{i}]}(x)$.

\begin{lemma}\label{lemma:mesh is contained in combinatorial AR quiver}
Let $\bm{i}$ be a reduced word for $w_0$. In the notation above, we have $k^+ < k+N$ for any $k \in \Z$.
\end{lemma}

\begin{proof}
Write $\widehat{\bm{i}} = (i_l)_{l \in \Z}$ and let $\bm{j} = (i_{k+1},i_{k+2},\dots,i_{k+N})$, which is a reduced word for $w_0$. If $k^+ \geq k+N$, then we have $i_l \neq i_k$ for all $k < l < k+N$. Thus, for $1 \leq l < N$, if we write $\beta^{\bm{j}}_l$ as a linear combination of positive simple roots, the coefficient of $\alpha_{i_k}$ is zero. This implies that there is at most one positive root involving $\alpha_{i_k}$, which contradicts the fact that $\Delta \neq \mathsf{A}_1$.
\end{proof}

\begin{lemma}\label{lemma:mesh gives a reduced word ready to apply braid move}
Let $[\bm{i}]$ be a commutation class of reduced words for $w_0$ and take a mesh $\M$ of $\widehat{\Upsilon}_{[\bm{i}]}$. If $(x_1,\dots,x_t)$ is a compatible reading of $\M_0$, denote $i_k = \res^{[\bm{i}]}(x_k)$ for $1 \leq k \leq t$. If $t > 3$, then the sequence $(i_1,\dots,i_{t-1},i_t,i_{t-1})$ is a reduced word for $s_{i_1}\dotsb s_{i_{t-1}}s_{i_t}s_{i_{t-1}}$.
\end{lemma}

\begin{proof}
By Lemma \ref{lemma:mesh is contained in combinatorial AR quiver}, we may suppose that $\M \subseteq \Upsilon_{[\bm{i}]}$ after reflecting at sources/sinks of $[\bm{i}]$. Since $\M$ is convex, Theorem \ref{thm:compatible reading gives reduced word} implies that $(i_1,\dots,i_t)$ is a reduced word for $w[\M_0] = s_{i_1}\dotsb s_{i_t}$. Notice that, since $\M$ is a mesh, we have $i_1 = i_t$ (we will denote this common value by $i$), $i_l \neq i$ for $1 < l < t$, and $i_2,i_{t-1} \sim i$.

To prove the lemma, we need to show that $s_{i_1}\dotsb s_{i_t}(\alpha_{i_{t-1}})$ is a positive root (see \cite[Lemma 1.6]{Humphreys}). Since $i_{t-1} \sim i_t = i$, we have
\[
s_{i_{t-1}}s_{i_t}(\alpha_{i_{t-1}}) = s_{i_{t-1}}(\alpha_{i_{t-1}} + \alpha_{i_t}) = \alpha_{i_t} = \alpha_i.
\]
Since $i_l \neq i$ for $1 < l \leq t-2$, if we write $s_{i_2}s_{i_3}\dotsb s_{i_{t-2}}(\alpha_i)$ as a linear combination of the positive simple roots, then the coefficient of $\alpha_i$ must be $1$. Thus, $s_{i_2}s_{i_3}\dotsb s_{i_{t-2}}(\alpha_i)$ is a positive root. We are finished if we prove that $s_{i_2}s_{i_3}\dotsb s_{i_{t-2}}(\alpha_i) \neq \alpha_i$. Suppose by contradiction this is not the case. Then \cite[Theorem 1.7]{Humphreys} implies
\[
s_is_{i_2}\dotsb s_{i_{t-2}} = s_{i_2}\dotsb s_{i_{t-2}}s_i.
\]
By the previous paragraph, the word on the left is reduced, hence so is the one on the right. By Matsumoto's theorem \cite{Matusmoto}, we can obtain the second word from the first by applying commutation and braid relations. However, one can check that this is impossible because $s_i$ appears only at the beginning of the first word and $i_2 \sim i$ (observe that $s_{i_2}$ appears in the word as $t-2 \geq 2$ by hypothesis).
\end{proof}

\subsection{Folding of Dynkin diagrams}\label{subsection:folding}

Let $\mathfrak{g}$ be a complex finite-dimensional simple Lie 
algebra that is not necessarily simply laced. It is well known that its Dynkin diagram can be obtained by ``folding'' the Dynkin diagram $\Delta$ of a simple 
Lie algebra $\mathsf{g}$ of simply laced type along an automorphism $\sigma$. In Table \ref{table:folding} 
(taken from \cite{FujitaOh}), we display which simply laced type 
and automorphism are needed to obtain each Lie algebra $\mathfrak{g}$. See also Figure \ref{figure:automorphisms} for the definition of the automorphisms $\vee$ and $\widetilde{\vee}$.

\begin{table}
\centering
 { \arraycolsep=1.6pt\def\arraystretch{1.5}
\begin{tabular}{|c|c|c|c|c|c|}
\hline
$\mathfrak{g}$ & $\mathsf{g}$ & $\sigma$ & $r$ & $h^{\vee}$ & $N$ \\
\hline
\hline
$\mathsf{A}_{n}$ & $\mathsf{A}_{n}$ & $\id$ & $1$ & $n+1$ & $n(n+1)/2$ \\
$\mathsf{D}_{n}$ & $\mathsf{D}_{n}$ & $\id$ & $1$ & $2n-2$ & $n(n-1)$ \\
$\mathsf{E}_{6,7,8}$ & $\mathsf{E}_{6,7,8}$ & $\id$ & $1$ & $12, 18, 30$& $36, 63, 120$ \\
\hline
$\mathsf{B}_{n}$ & $\mathsf{A}_{2n-1}$ & $\vee$ & $2$ & $2n-1$ & $n(2n-1)$\\
$\mathsf{C}_{n}$ & $\mathsf{D}_{n+1}$ & $\vee$ & $2$ & $n+1$ & $n(n+1)$ \\
$\mathsf{F}_{4}$ & $\mathsf{E}_{6}$ & $\vee$ & $2$ & $9$ & $36$ \\
\hline
$\mathsf{G}_{2}$ & $\mathsf{D}_{4}$ & $\widetilde{\vee}, \widetilde{\vee}^{2}$ & $3$ & $4$ & $12$ \\
\hline
\end{tabular}
  }\\[1.5ex]
    \caption{The Dynkin diagram of $\mathfrak{g}$ is obtained by 
	folding that of $\mathsf{g}$ along the automorphism $\sigma$. The last three columns give the order of $\sigma$, the dual Coxeter number of $\mathfrak{g}$, and the number of positive roots of $\mathsf{g}$.}
    \label{table:folding}
\end{table}

\begin{figure}
	\centering
	\begin{tikzpicture}[scale=0.8]

		%A_{2n-1}
	
		\node (A) at (0.5,-0.5) {($\mathsf{A}_{2n-1}, \vee$)};
		
		\node[circle, draw, inner sep=1.5pt] (A1) at (2,0) {};
		\node[above] at (2,0.1) {$1$};
		\node[circle, draw, inner sep=1.5pt] (A2) at (3.5,0) {};
		\node[above] at (3.5,0.1) {$2$};
		\node[circle, draw, inner sep=1.5pt] (Anm2) at (5,0) {};
		\node[above] at (5,0.1) {$n-2$};
		\node[circle, draw, inner sep=1.5pt] (Anm1) at (6.5,0) {};
		\node[above] at (6.5,0.1) {$n-1$};
		\node[circle, draw, inner sep=1.5pt] (An) at (8,-0.5) {};
		\node[above] at (8,-0.4) {$n$};
		\node[circle, draw, inner sep=1.5pt] (Anp1) at (6.5,-1) {};
		\node[below] at (6.5,-1.1) {$n+1$};
		\node[circle, draw, inner sep=1.5pt] (Anp2) at (5,-1) {};
		\node[below] at (5,-1.1) {$n+2$};
		\node[circle, draw, inner sep=1.5pt] (A2nm2) at (3.5,-1) {};
		\node[below] at (3.5,-1.1) {$2n-2$};
		\node[circle, draw, inner sep=1.5pt] (A2nm1) at (2,-1) {};
		\node[below] at (2,-1.1) {$2n-1$};
	
		\draw (A1) -- (A2);
		\draw[dotted] (A2) -- (Anm2);
		\draw (Anm2) -- (Anm1);
		\draw (Anm1) -- (An);
		\draw (An) -- (Anp1);
		\draw (Anp1) -- (Anp2);
		\draw[dotted] (Anp2) -- (A2nm2);
		\draw (A2nm2) -- (A2nm1);
		
		\draw[<->, red] (A1) -- (A2nm1);
		\draw[<->, red] (A2) -- (A2nm2);
		\draw[<->, red] (Anm2) -- (Anp2);
		\draw[<->, red] (Anm1) -- (Anp1);
		
		%B_n
		\node (B) at (0.5,-2.25) {$\mathsf{B}_n$};
		
		\node[circle, draw, inner sep=1.5pt] (B1) at (2,-2.25) {};
		\node[below] at (2,-2.35) {$1$};
		\node[circle, draw, inner sep=1.5pt] (B2) at (3.5,-2.25) {};
		\node[below] at (3.5,-2.35) {$2$};
		\node[circle, draw, inner sep=1.5pt] (Bnm2) at (5,-2.25) {};
		\node[below] at (5,-2.35) {$n-2$};
		\node[circle, draw, inner sep=1.5pt] (Bnm1) at (6.5,-2.25) {};
		\node[below] at (6.5,-2.35) {$n-1$};
		\node (Bvoid) at (7.25,-2.25) {};
		\node[circle, draw, inner sep=1.5pt] (Bn) at (8,-2.25) {};
		\node[below] at (8,-2.35) {$n$};
	
		\draw (B1) -- (B2);
		\draw[dotted] (B2) -- (Bnm2);
		\draw (Bnm2) -- (Bnm1);
		\draw[double, double distance=1.5pt] (Bvoid.west) -- (Bn);
		\draw[->, double, double distance=1.5pt] (Bnm1) -- (Bvoid.east);
	
		%Connecting A and B
	
		\draw[->, dashed, blue] (A1) edge[bend left=70] (B1);
		\draw[->, dashed, blue] (A2) edge[bend left=70] (B2);
		\draw[->, dashed, blue] (Anm2) edge[bend left=70] (Bnm2);
		\draw[->, dashed, blue] (Anm1) edge[bend left=70] (Bnm1);
	
		\draw[->, dashed, blue] (A2nm1) edge[bend right=70] (B1);
		\draw[->, dashed, blue] (A2nm2) edge[bend right=70] (B2);
		\draw[->, dashed, blue] (Anp2) edge[bend right=70] (Bnm2);
		\draw[->, dashed, blue] (Anp1) edge[bend right=70] (Bnm1);
	
		\draw[->, dashed, blue] (An) -- (Bn);
	
		\draw[->] (A) -- (B);
	
		%D_{n+1}
	
		\node (D) at (10,-0.5) {($\mathsf{D}_{n+1}, \vee$)};
		
		\node[circle, draw, inner sep=1.5pt] (D1) at (11.5,-0.5) {};
		\node[above] at (11.5,-0.4) {$1$};
		\node[circle, draw, inner sep=1.5pt] (D2) at (13,-0.5) {};
		\node[above] at (13,-0.4) {$2$};
		\node[circle, draw, inner sep=1.5pt] (Dnm2) at (14.5,-0.5) {};
		\node[above] at (14.5,-0.4) {$n-2$};
		\node[circle, draw, inner sep=1.5pt] (Dnm1) at (16,-0.5) {};
		\node[above] at (16,-0.4) {$n-1$};
		\node[circle, draw, inner sep=1.5pt] (Dn) at (17.5,0) {};
		\node[right] at (17.6,0) {$n$};
		\node[circle, draw, inner sep=1.5pt] (Dnp1) at (17.5,-1) {};
		\node[right] at (17.6,-1) {$n+1$};
	
		\draw (D1) -- (D2);
		\draw[dotted] (D2) -- (Dnm2);
		\draw (Dnm2) -- (Dnm1);
		\draw (Dnm1) -- (Dn);
		\draw (Dnm1) -- (Dnp1);
		
		\draw[<->, red] (Dn) -- (Dnp1);
		
		%C_n
		\node (C) at (10,-2.25) {$\mathsf{C}_n$};
		
		\node[circle, draw, inner sep=1.5pt] (C1) at (11.5,-2.25) {};
		\node[below] at (11.5,-2.35) {$1$};
		\node[circle, draw, inner sep=1.5pt] (C2) at (13,-2.25) {};
		\node[below] at (13,-2.35) {$2$};
		\node[circle, draw, inner sep=1.5pt] (Cnm2) at (14.5,-2.25) {};
		\node[below] at (14.5,-2.35) {$n-2$};
		\node[circle, draw, inner sep=1.5pt] (Cnm1) at (16,-2.25) {};
		\node[below] at (16,-2.35) {$n-1$};
		\node (Cvoid) at (16.75,-2.25) {};
		\node[circle, draw, inner sep=1.5pt] (Cn) at (17.5,-2.25) {};
		\node[below] at (17.5,-2.35) {$n$};
	
		\draw (C1) -- (C2);
		\draw[dotted] (C2) -- (Cnm2);
		\draw (Cnm2) -- (Cnm1);
		\draw[double, double distance=1.5pt] (Cnm1) -- (Cvoid.east);
		\draw[->, double, double distance=1.5pt] (Cn) -- (Cvoid.west);
	
		%Connecting D and C
	
		\draw[->, dashed, blue] (D1) -- (C1);
		\draw[->, dashed, blue] (D2) -- (C2);
		\draw[->, dashed, blue] (Dnm2) -- (Cnm2);
		\draw[->, dashed, blue] (Dnm1) -- (Cnm1);
	
		\draw[->, dashed, blue] (Dn) edge[bend right=45] (Cn);
		\draw[->, dashed, blue] (Dnp1) edge[bend left=45] (Cn);
	
		\draw[->] (D) -- (C);
	
		%E_6
	
		\node (E) at (0.5,-4.5) {($\mathsf{E}_6, \vee$)};
		
		\node[circle, draw, inner sep=1.5pt] (E1) at (2,-4) {};
		\node[above] at (2,-3.9) {$1$};
		\node[circle, draw, inner sep=1.5pt] (E2) at (3.5,-4) {};
		\node[above] at (3.5,-3.9) {$2$};
		\node[circle, draw, inner sep=1.5pt] (E3) at (5,-4.5) {};
		\node[above] at (5,-4.4) {$3$};
		\node[circle, draw, inner sep=1.5pt] (E4) at (3.5,-5) {};
		\node[below] at (3.5,-5.1) {$4$};
		\node[circle, draw, inner sep=1.5pt] (E5) at (2,-5) {};
		\node[below] at (2,-5.1) {$5$};
		\node[circle, draw, inner sep=1.5pt] (E6) at (6.5,-4.5) {};
		\node[above] at (6.5,-4.4) {$6$};
	
		\draw (E1) -- (E2);
		\draw (E2) -- (E3);
		\draw (E3) -- (E4);
		\draw (E4) -- (E5);
		\draw (E3) -- (E6);
		
		\draw[<->, red] (E1) -- (E5);
		\draw[<->, red] (E2) -- (E4);
		
		%F_4
		\node (F) at (0.5,-6.25) {$\mathsf{F}_4$};
		
		\node[circle, draw, inner sep=1.5pt] (F1) at (2,-6.25) {};
		\node[below] at (2,-6.35) {$1$};
		\node[circle, draw, inner sep=1.5pt] (F2) at (3.5,-6.25) {};
		\node[below] at (3.5,-6.35) {$2$};
		\node (Fvoid) at (4.25,-6.25) {};
		\node[circle, draw, inner sep=1.5pt] (F3) at (5,-6.25) {};
		\node[below] at (5,-6.35) {$3$};
		\node[circle, draw, inner sep=1.5pt] (F4) at (6.5,-6.25) {};
		\node[below] at (6.5,-6.35) {$4$};
	
		\draw (F1) -- (F2);
		\draw[double, double distance=1.5pt] (Fvoid.west) -- (F3);
		\draw[->, double, double distance=1.5pt] (F2) -- (Fvoid.east);
		\draw (F3) -- (F4);
	
		%Connecting E and F
	
		\draw[->, dashed, blue] (E1) edge[bend right=45] (F1);
		\draw[->, dashed, blue] (E2) edge[bend left=45] (F2);
		\draw[->, dashed, blue] (E5) edge[bend left=45] (F1);
		\draw[->, dashed, blue] (E4) edge[bend right=45] (F2);
	
		\draw[->, dashed, blue] (E3) -- (F3);
		\draw[->, dashed, blue] (E6) -- (F4);
	
		\draw[->] (E) -- (F);
	
		%D_4
	
		\node (Dtil) at (10,-4.5) {($\mathsf{D}_4, \widetilde{\vee}$)};
		
		\node[circle, draw, inner sep=1.5pt] (Dtil1) at (11.5,-4) {};
		\node[above] at (11.5,-3.9) {$1$};
		\node[circle, draw, inner sep=1.5pt] (Dtil2) at (13,-4.5) {};
		\node[above] at (13,-4.4) {$2$};
		\node[circle, draw, inner sep=1.5pt] (Dtil3) at (11.5,-4.5) {};
		\node[left] at (11.55,-4.5) {$3$};
		\node[circle, draw, inner sep=1.5pt] (Dtil4) at (11.5,-5) {};
		\node[below] at (11.5,-5.1) {$4$};
	
		\draw (Dtil1) -- (Dtil2);
		\draw (Dtil2) -- (Dtil3);
		\draw (Dtil2) -- (Dtil4);
		
		\draw[->, red] (Dtil3) -- (Dtil4);
		\draw[->, red] (Dtil4) edge[bend left=90] (Dtil1);
		\draw[->, red] (Dtil1) -- (Dtil3);
		
		%G_2
		\node (G) at (10,-6.25) {$\mathsf{G}_2$};
		
		\node[circle, draw, inner sep=1.5pt] (G1) at (11.5,-6.25) {};
		\node[below] at (11.5,-6.35) {$1$};
		\node (Gvoid) at (12.25,-6.25) {};
		\node[circle, draw, inner sep=1.5pt] (G2) at (13,-6.25) {};
		\node[below] at (13,-6.35) {$2$};
	
		\tikzset{Rightarrow/.style={double, double distance=2pt},
		triple/.style={-,preaction={draw,Rightarrow}}}
		\draw[triple] (Gvoid.west) -- (G2);
		
		\tikzset{Rightarrow/.style={double, double distance=2pt,->},
		triple/.style={-,preaction={draw,Rightarrow}}}
		\draw[->, triple] (G1) -- (Gvoid.east);
	
		%Connecting D_4 and G
	
		\draw[->, dashed, blue] (Dtil2) -- (G2);
	
		\draw[->, dashed, blue] (Dtil4) edge[bend right=45] (G1);
		\draw[->, dashed, blue] (Dtil3) edge[bend left=45] (G1);
		\draw[->, dashed, blue] (Dtil1) edge[bend left=45] (G1);
	
		\draw[->] (Dtil) -- (G);
	
	\end{tikzpicture}
	\caption{The action of the automorphisms $\vee$ and $\widetilde{\vee}$ on the nonfixed vertices is shown in red. After identifying the orbits via the blue arrows, we get the nonsimply laced Dynkin diagrams.}
	\label{figure:automorphisms}
\end{figure}

\begin{remark}
	The automorphism $\sigma$ induces an automorphism of the Lie algebra $\mathsf{g}$ and gives rise to the invariant subalgebra $\mathsf{g}^{\sigma}$, which is again simple. We warn the reader that $\mathsf{g}^{\sigma}$ is not always isomorphic to $\mathfrak{g}$, but is instead Langlands dual to it. Hence, our folding procedure is dual to the classical folding of Lie algebras. For example, the invariant subalgebra of the special linear Lie algebra $\mathfrak{sl}_{2n}(\mathbb{C})$ of type $\mathsf{A}_{2n-1}$ under the automorphism induced by $\vee$ is the symplectic Lie algebra $\mathfrak{sp}_{2n}(\mathbb{C})$ of type $\mathsf{C}_n$ (and not $\mathsf{B}_n$).
\end{remark}

\begin{remark}\label{rmk:notation from FO}
Whenever we deal with $\mathsf{g}$ and $\mathfrak{g}$ at the same 
time, we will adopt the convention of \cite{FujitaOh} to notate 
the indices of their Dynkin diagrams. More specifically, we will 
mainly use the symbols $\im$ and $\jm$ (and their variations) for 
the vertices of $\Delta$ and save the more common notation $i$ and 
$j$ for the vertices of the Dynkin diagram of $\mathfrak{g}$. 
If $\mathfrak{g}$ does not appear at all (as in most of 
Sections \ref{section:the categories}--\ref{section:extension groups and the Euler form}), we shall use the usual notation. 
\end{remark}

From now on, $\sigma$ will always denote the automorphism of 
$\Delta$ that corresponds to $\mathfrak{g}$. We set 
$r \in \{1,2,3\}$ to be its order. Let $I$ be the set of 
$\sigma$-orbits of $\Delta_0$, which is the set of vertices of the 
Dynkin diagram of $\mathfrak{g}$. The quotient map 
$\Delta_0 \longrightarrow I$ is denoted by 
$\im \mapsto \overline{\im}$. Define $n = |I|$. For each 
orbit $i \in I$, denote its cardinality by $d_i$, which is either 
$1$ or $r$. We will write $i \sim j$ for $i,j \in I$ if there are 
$\im \in i$ and $\jm \in j$ such that $\im \sim \jm$. We denote by 
$h^{\vee}$ the dual Coxeter number of $\mathfrak{g}$, which 
satisfies the relation $nrh^{\vee} = 2N = |\mathsf{R}|$.

The automorphism $\sigma$ can be used to turn $\Delta$ into a weighted graph and determine a distance function 
$d_{\Delta}^{\sigma}$ on the set $\Delta_0$ in the following way. 
We first set the weight of an edge between vertices 
$\im,\jm \in \Delta_0$ to be 
$\min(d_{\overline{\im}},d_{\overline{\jm}})$. Then, for general 
$\im,\jm \in \Delta_0$, $d_{\Delta}^{\sigma}(\im,\jm)$ is the sum of the weights of the edges in the unique simple path between $\im$ and $\jm$.

In addition, $\sigma$ naturally defines an automorphism of the 
weight lattice $\mathsf{P}$ if we set 
$\sigma(\varpi_{\im}) = \varpi_{\sigma(\im)}$ for 
$\im \in \Delta_0$. We have 
$\sigma(\alpha_{\im}) = \alpha_{\sigma(\im)}$ and 
$\sigma s_{\im}\sigma^{-1} = s_{\sigma(\im)}$ for 
$\im \in \Delta_0$. In particular, we can consider the coset 
$\mathsf{W}\sigma$ of $\mathsf{W}$ in $\Aut(\mathsf{P})$.

\subsection{Q-data and twisted AR quivers}\label{subsection:Q-data and twisted AR quivers}

We recall the definition of a Q-datum, introduced in \cite{FujitaOh}.

\begin{defn}
A \emph{Q-datum} for $\mathfrak{g}$ is a triple 
$\Q = (\Delta, \sigma, \xi)$ where $(\Delta, \sigma)$ is the pair 
corresponding to $\mathfrak{g}$ as in Table \ref{table:folding} and 
$\xi: \Delta_0 \longrightarrow \Z$ is a \emph{height function on} 
$(\Delta, \sigma)$, that is, a function satisfying the following 
two conditions (where we write $\xi_{\im} = \xi(\im)$):
\begin{enumerate}[(1)]
    \item For any $\im, \jm \in \Delta_0$ such that $\im \sim \jm$ 
	and $d_{\overline{\im}} = d_{\overline{\jm}}$, we have 
	$|\xi_{\im} - \xi_{\jm}| = d_{\overline{\im}} = d_{\overline{\jm}}$.
    \item For any $i,j \in I$ with $i \sim j$, $d_i = 1$ and $d_j = r$, there is a unique $\jm \in j$ such that 
	$|\xi_{\im} - \xi_{\jm}| = 1$ and 
	$\xi_{\sigma^k(\jm)} = \xi_{\jm} - 2k$ for all $1 \leq k < r$, 
	where $\im$ is the unique element of $i$.
\end{enumerate}
\end{defn}

As suggested by \cite{FujitaOh}, one should think that these are 
local conditions at a ``nonbranching point'' and at a ``branching 
point,'' respectively.

\begin{remark}
In the simply laced case, when $\sigma = \id$, a height function is 
just a function $\xi: \Delta_0 \longrightarrow \Z$ satisfying 
$|\xi_{\im} - \xi_{\jm}| = 1$ for 
$\im,\jm \in \Delta_0$ such that $\im \sim \jm$. Up to the addition of a constant, this is the same
as choosing an orientation for the graph $\Delta$: for 
$\im \sim \jm$, we say that there is an arrow from $\im$ to $\jm$ 
if $\xi_{\im} > \xi_{\jm}$. Hence, in this case, a Q-datum is 
essentially the same as a quiver whose underlying graph is $\Delta$.
\end{remark}

Let $\Q = (\Delta, \sigma, \xi)$ be a Q-datum. A vertex $\im \in \Delta_0$ is 
a \emph{source of $\Q$} if we have $\xi_{\im} > \xi_{\jm}$ for 
all $\jm \in \Delta_0$ such that $\im \sim \jm$. In this case, the 
function $s_{\im}\xi: \Delta_0 \longrightarrow \Z$ given by 
$(s_{\im}\xi)_{\jm} = \xi_{\jm}$ for $\jm \neq \im$ and 
$(s_{\im}\xi)_{\im} = \xi_{\im} - 2d_{\overline{\im}}$ is again a 
height function, and we can define the \emph{reflected} Q-datum 
$s_{\im}\Q = (\Delta, \sigma, s_{\im}\xi)$. A sequence of vertices 
$(\im_1, \dots, \im_l)$ of $\Delta$ is a \emph{source sequence for} 
$\Q$ if $\im_k$ is a source of 
$s_{\im_{k-1}}s_{\im_{k-2}}\dotsb s_{\im_1}\Q$ for all 
$1 \leq k \leq l$.

We define a \emph{sink of $\Q$} to be a vertex 
$\im \in \Delta_0$ satisfying 
$\xi_{\im} + 2d_{\overline{\im}} < \xi_{\jm} + 2d_{\overline{\jm}}$ 
for all $\jm \in \Delta_0$ such that $\im \sim \jm$. In this case, 
the function $s_{\im}^{-1}\xi: \Delta_0 \longrightarrow \Z$ given 
by $(s_{\im}^{-1}\xi)_{\jm} = \xi_{\jm}$ for $\jm \neq \im$ and 
$(s_{\im}^{-1}\xi)_{\im} = \xi_{\im} + 2d_{\overline{\im}}$ is again a 
height function, and we can also define the \emph{reflected} 
Q-datum $s_{\im}^{-1}\Q = (\Delta, \sigma, s_{\im}^{-1}\xi)$. 
The definition of a \emph{sink sequence} is analogous to that of a 
source sequence.

\begin{remark}
Observe that, if $\im \in \Delta_0$ is a source of $\Q$, then 
$\im$ is a sink of $s_{\im}\Q$ and we have 
$s_{\im}^{-1}s_{\im}\Q = \Q$. Similarly, if $\im \in \Delta_0$ is 
a sink of $\Q$, then $\im$ is a source of $s_{\im}^{-1}\Q$ and we 
have $s_{\im}s_{\im}^{-1}\Q = \Q$.
\end{remark}

\begin{example}\label{example:some different Q-data}
Let $\mathfrak{g}$ be of type $\mathsf{B}_3$, so that 
$(\Delta,\sigma) = (\mathsf{A}_5, \vee)$, and consider the 
following height functions:
\begin{center}
\begin{tikzpicture}

    %\Q^1

    \node at (0.2,0) {$\Q^1 =$};
    
    \node[circle, draw, inner sep=1pt] (a1) at (1,0) {};
    \node[above] at (1,0.2) {6};
    \node[circle, draw, inner sep=1pt] (a2) at (2,0) {};
    \node[above] at (2,0.2) {4};
    \node[circle, draw, inner sep=1pt] (a3) at (3,0) {};
    \node[above] at (3,0.2) {5};
    \node[circle, draw, inner sep=1pt] (a4) at (4,0) {};
    \node[above] at (4,0.2) {6};
    \node[circle, draw, inner sep=1pt] (a5) at (5,0) {};
    \node[above] at (5,0.2) {8};

    \draw (a1) -- (a2);
    \draw (a2) -- (a3);
    \draw (a3) -- (a4);
    \draw (a4) -- (a5);

    %\Q^2

    \node at (0.2,-1) {$\Q^2 =$};
    
    \node[circle, draw, inner sep=1pt] (b1) at (1,-1) {};
    \node[above] at (1,-0.8) {2};
    \node[circle, draw, inner sep=1pt] (b2) at (2,-1) {};
    \node[above] at (2,-0.8) {4};
    \node[circle, draw, inner sep=1pt] (b3) at (3,-1) {};
    \node[above] at (3,-0.8) {5};
    \node[circle, draw, inner sep=1pt] (b4) at (4,-1) {};
    \node[above] at (4,-0.8) {6};
    \node[circle, draw, inner sep=1pt] (b5) at (5,-1) {};
    \node[above] at (5,-0.8) {8};

    \draw (b1) -- (b2);
    \draw (b2) -- (b3);
    \draw (b3) -- (b4);
    \draw (b4) -- (b5);

    %\Q^3

    \node at (0.2,-2) {$\Q^3 =$};
    
    \node[circle, draw, inner sep=1pt] (c1) at (1,-2) {};
    \node[above] at (1,-1.8) {2};
    \node[circle, draw, inner sep=1pt] (c2) at (2,-2) {};
    \node[above] at (2,-1.8) {4};
    \node[circle, draw, inner sep=1pt] (c3) at (3,-2) {};
    \node[above] at (3,-1.8) {7};
    \node[circle, draw, inner sep=1pt] (c4) at (4,-2) {};
    \node[above] at (4,-1.8) {6};
    \node[circle, draw, inner sep=1pt] (c5) at (5,-2) {};
    \node[above] at (5,-1.8) {8};

    \draw (c1) -- (c2);
    \draw (c2) -- (c3);
    \draw (c3) -- (c4);
    \draw (c4) -- (c5);
    
\end{tikzpicture}
\end{center}
Here we put the value of the height functions above the vertices 
of $\mathsf{A}_5$. They give examples of Q-data 
$\Q^1$, $\Q^2$, and $\Q^3$. Notice, for example, that the first 
vertex is a source of $\Q^1$ and that $s_1\Q^1 = \Q^2$. Similarly, 
the middle vertex is a source of $\Q^3$ and $s_3\Q^3 = \Q^2$. In particular, 
the middle vertex is a sink of $\Q^2$, which may not be clear at first glance.
\end{example}

Let $\xi$ be a height function on $(\Delta, \sigma)$. The 
\emph{(twisted) repetition quiver} associated with $(\Delta, \sigma)$ is the 
quiver $\widehat{\Delta}^{\sigma}$ whose set of vertices 
$\widehat{\Delta}^{\sigma}_0$ and set of arrows 
$\widehat{\Delta}^{\sigma}_1$ are defined as follows:
\begin{align*}
    \widehat{\Delta}^{\sigma}_0 &= \{(\im,p) \in \Delta_0 \times \Z \mid p - \xi_{\im} \in 2d_{\overline{\im}}\Z\},\\
    \widehat{\Delta}^{\sigma}_1 &= \{(\im,p) \longrightarrow (\jm,s) \mid (\im,p),(\jm,s) \in \widehat{\Delta}^{\sigma}_0, \jm \sim \im, s-p = \min(d_{\overline{\im}},d_{\overline{\jm}})\}.
\end{align*}
We observe that if there is a path from $(\im,p)$ to $(\jm,s)$ in 
$\widehat{\Delta}^{\sigma}$, then 
$s-p \geq d_{\Delta}^{\sigma}(\im,\jm)$. We say that $\im$ is the 
\emph{residue} of the vertex $(\im,p)$, while $p$ is its 
\emph{height}. We denote by 
$\pi: \widehat{\Delta}^{\sigma}_0 \longrightarrow \Delta_0$ the 
projection onto the first coordinate.

\begin{remark}
As explained in \cite{FujitaOh}, the repetition quiver depends only 
on the \emph{$\sigma$-parity function} congruent to $\xi$. 
Therefore, if we vary $\xi$, the only change incurred to 
$\widehat{\Delta}^{\sigma}$ is possibly a uniform shift of the 
height of its vertices by some integer. For this reason, we will 
omit the height function when referring to 
$\widehat{\Delta}^{\sigma}$. Whenever we work with a Q-datum 
$\Q = (\Delta, \sigma, \xi)$, it will be implicitly assumed that $\Q$ is such that $\xi$ defines $\widehat{\Delta}^{\sigma}$. If we want to highlight this fact, we will say that $\Q$ has the same 
\emph{parity} as $\widehat{\Delta}^{\sigma}$.
\end{remark}

\begin{example}
Here we have some examples of the repetition quiver 
$\widehat{\Delta}^{\sigma}$.

\begin{enumerate}[(1)]

\item Type $\mathsf{B}_3$, where $(\Delta,\sigma) = 
(\mathsf{A}_5,\vee)$:
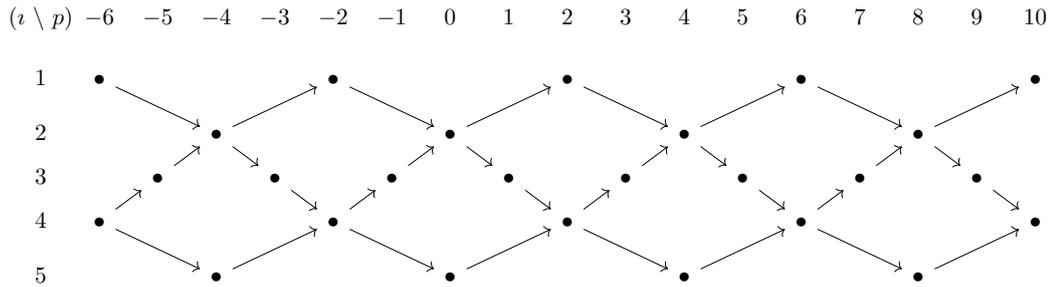
\begin{center}
\begin{tikzpicture}
\node[scale=.8] (a) at (0,0){
\begin{tikzcd}[column sep={2.5em,between origins},row sep={1em}]
	{(\im \ \backslash \ p)} & {-6} & {-5} & {-4} & {-3} & {-2} & {-1} & 0 & 1 & 2 & 3 & 4 & 5 & 6 & 7 & 8 & 9 & 10 \\
	1 & \bullet &&&& \bullet &&&& \bullet &&&& \bullet &&&& \bullet \\
	2 &&& \bullet &&&& \bullet &&&& \bullet &&&& \bullet \\[-0.5em]
	3 && \bullet && \bullet && \bullet && \bullet && \bullet && \bullet && \bullet && \bullet \\[-0.5em]
	4 & \bullet &&&& \bullet &&&& \bullet &&&& \bullet &&&& \bullet \\
	5 &&& \bullet &&&& \bullet &&& {} & \bullet &&&& \bullet
	\arrow[from=2-2, to=3-4]
	\arrow[from=2-6, to=3-8]
	\arrow[from=2-10, to=3-12]
	\arrow[from=2-14, to=3-16]
	\arrow[from=3-4, to=2-6]
	\arrow[from=3-4, to=4-5]
	\arrow[from=3-8, to=2-10]
	\arrow[from=3-8, to=4-9]
	\arrow[from=3-12, to=2-14]
	\arrow[from=3-12, to=4-13]
	\arrow[from=3-16, to=2-18]
	\arrow[from=3-16, to=4-17]
	\arrow[from=4-3, to=3-4]
	\arrow[from=4-5, to=5-6]
	\arrow[from=4-7, to=3-8]
	\arrow[from=4-9, to=5-10]
	\arrow[from=4-11, to=3-12]
	\arrow[from=4-13, to=5-14]
	\arrow[from=4-15, to=3-16]
	\arrow[from=4-17, to=5-18]
	\arrow[from=5-2, to=4-3]
	\arrow[from=5-2, to=6-4]
	\arrow[from=5-6, to=4-7]
	\arrow[from=5-6, to=6-8]
	\arrow[from=5-10, to=4-11]
	\arrow[from=5-10, to=6-12]
	\arrow[from=5-14, to=4-15]
	\arrow[from=5-14, to=6-16]
	\arrow[from=6-4, to=5-6]
	\arrow[from=6-8, to=5-10]
	\arrow[from=6-12, to=5-14]
	\arrow[from=6-16, to=5-18]
\end{tikzcd}};
\end{tikzpicture}
\end{center}

\item Type $\mathsf{C}_4$, where 
$(\Delta,\sigma) = (\mathsf{D}_5,\vee)$:
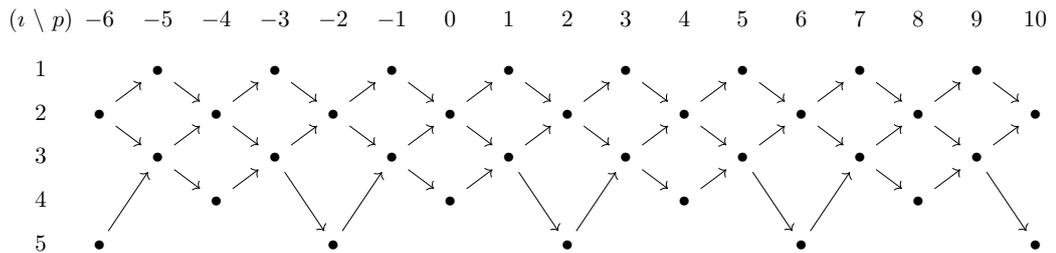
\begin{center}
\begin{tikzpicture}
\node[scale=.8] (a) at (0,0){
\begin{tikzcd}[column sep={2.5em,between origins},row sep={0.5em}]
	{(\im\ \backslash \ p)} & {-6} & {-5} & {-4} & {-3} & {-2} & {-1} & 0 & 1 & 2 & 3 & 4 & 5 & 6 & 7 & 8 & 9 & 10 \\
	1 && \bullet && \bullet && \bullet && \bullet && \bullet && \bullet && \bullet && \bullet \\
	2 & \bullet && \bullet && \bullet && \bullet && \bullet && \bullet && \bullet && \bullet && \bullet \\
	3 && \bullet && \bullet && \bullet && \bullet && \bullet && \bullet && \bullet && \bullet \\
	4 &&& \bullet &&&& \bullet &&&& \bullet &&&& \bullet \\
	5 & \bullet &&&& \bullet &&&& \bullet &&&& \bullet &&&& \bullet
	\arrow[from=2-3, to=3-4]
	\arrow[from=2-5, to=3-6]
	\arrow[from=2-7, to=3-8]
	\arrow[from=2-9, to=3-10]
	\arrow[from=2-11, to=3-12]
	\arrow[from=2-13, to=3-14]
	\arrow[from=2-15, to=3-16]
	\arrow[from=2-17, to=3-18]
	\arrow[from=3-2, to=2-3]
	\arrow[from=3-2, to=4-3]
	\arrow[from=3-4, to=2-5]
	\arrow[from=3-4, to=4-5]
	\arrow[from=3-6, to=2-7]
	\arrow[from=3-6, to=4-7]
	\arrow[from=3-8, to=2-9]
	\arrow[from=3-8, to=4-9]
	\arrow[from=3-10, to=2-11]
	\arrow[from=3-10, to=4-11]
	\arrow[from=3-12, to=2-13]
	\arrow[from=3-12, to=4-13]
	\arrow[from=3-14, to=2-15]
	\arrow[from=3-14, to=4-15]
	\arrow[from=3-16, to=2-17]
	\arrow[from=3-16, to=4-17]
	\arrow[from=4-3, to=3-4]
	\arrow[from=4-3, to=5-4]
	\arrow[from=4-5, to=3-6]
	\arrow[from=4-5, to=6-6]
	\arrow[from=4-7, to=3-8]
	\arrow[from=4-7, to=5-8]
	\arrow[from=4-9, to=3-10]
	\arrow[from=4-9, to=6-10]
	\arrow[from=4-11, to=3-12]
	\arrow[from=4-11, to=5-12]
	\arrow[from=4-13, to=3-14]
	\arrow[from=4-13, to=6-14]
	\arrow[from=4-15, to=3-16]
	\arrow[from=4-15, to=5-16]
	\arrow[from=4-17, to=3-18]
	\arrow[from=4-17, to=6-18]
	\arrow[from=5-4, to=4-5]
	\arrow[from=5-8, to=4-9]
	\arrow[from=5-12, to=4-13]
	\arrow[from=5-16, to=4-17]
	\arrow[from=6-2, to=4-3]
	\arrow[from=6-6, to=4-7]
	\arrow[from=6-10, to=4-11]
	\arrow[from=6-14, to=4-15]
\end{tikzcd}};
\end{tikzpicture}
\end{center}

\item Type $\mathsf{G}_2$, where 
$(\Delta,\sigma) = (\mathsf{D}_4,\widetilde{\vee})$:
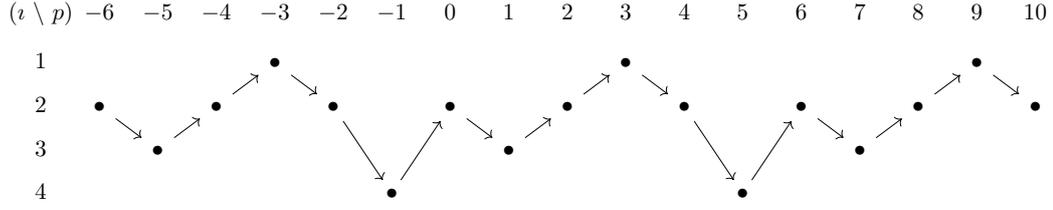
\begin{center}
\begin{tikzpicture}
\node[scale=.8] (a) at (0,0){
\begin{tikzcd}[column sep={2.5em,between origins},row sep={0.5em}]
	{(\im \ \backslash \ p)} & {-6} & {-5} & {-4} & {-3} & {-2} & {-1} & 0 & 1 & 2 & 3 & 4 & 5 & 6 & 7 & 8 & 9 & 10 \\
	1 &&&& \bullet &&&&&& \bullet &&&&&& \bullet \\
	2 & \bullet && \bullet && \bullet && \bullet && \bullet && \bullet && \bullet && \bullet && \bullet \\
	3 && \bullet &&&&&& \bullet &&&&&& \bullet \\
	4 &&&&&& \bullet &&&&&& \bullet
	\arrow[from=2-5, to=3-6]
	\arrow[from=2-11, to=3-12]
	\arrow[from=2-17, to=3-18]
	\arrow[from=3-2, to=4-3]
	\arrow[from=3-4, to=2-5]
	\arrow[from=3-6, to=5-7]
	\arrow[from=3-8, to=4-9]
	\arrow[from=3-10, to=2-11]
	\arrow[from=3-12, to=5-13]
	\arrow[from=3-14, to=4-15]
	\arrow[from=3-16, to=2-17]
	\arrow[from=4-3, to=3-4]
	\arrow[from=4-9, to=3-10]
	\arrow[from=4-15, to=3-16]
	\arrow[from=5-7, to=3-8]
	\arrow[from=5-13, to=3-14]
\end{tikzcd}};
\end{tikzpicture}
\end{center}
\end{enumerate}
\end{example}

The \emph{twisted AR quiver} $\Gamma_{\Q}$ of a 
Q-datum $\Q = (\Delta, \sigma, \xi)$ is the full subquiver of 
$\widehat{\Delta}^{\sigma}$ whose vertices are given by the set
\[
(\Gamma_{\Q})_0 = \{(\im,p) \in \widehat{\Delta}^{\sigma}_0 \mid \xi_{\im^*} - rh^{\vee} < p \leq \xi_{\im}\}.
\]
One can show that it is a convex subquiver of $\widehat{\Delta}^{\sigma}$.

Let $\im \in \Delta_0$. The 
\emph{injective vertex associated with $\im$} is the vertex 
$(\im,p) \in (\Gamma_{\Q})_0$ with the largest height. We 
have $p = \xi_{\im}$, in this case. Similarly, the \emph{projective 
vertex associated with $\im$} is the vertex 
$(\im^*,p) \in (\Gamma_{\Q})_0$ with the smallest height. Its 
height is given by the lemma below.

\begin{lemma}\label{lemma:height of projective vertex}
Let $\Q = (\Delta, \sigma, \xi)$ be a Q-datum and take 
$\im \in \Delta_0$. The projective vertex associated with $\im$ in 
$\Gamma_{\Q}$ has height given by 
$\xi_{\im} - rh^{\vee} + 2d_{\overline{\im}}$.
\end{lemma}

\begin{proof}
We need to show that $\xi_{\im} - rh^{\vee}$ is congruent to 
$\xi_{\im^*}$ modulo $2d_{\overline{\im}}$ in order to prove that 
$(\im^*,\xi_{\im} - rh^{\vee} + 2d_{\overline{\im}})$ is a vertex 
of $\Gamma_{\Q}$. We will consider two different cases.

First, assume that $h^{\vee}$ is even. In this case, since 
$d_{\overline{\im}}$ divides $r$, we are reduced to the 
verification that $\xi_{\im}$ is congruent to $\xi_{\im^*}$ modulo 
$2d_{\overline{\im}}$. If $r = 1$, then $d_{\overline{\im}} = 1$ 
and the congruence can be easily checked by inspecting all types. 
If $r > 1$, then $\im^* = \im$ by \cite[Remark 3.1]{FujitaOh} and 
the congruence is trivial.

Now, assume that $h^{\vee}$ is odd. We have again two cases: 
$r = 1$ or $r = 2$. In the first case, we must have 
$\Delta = \mathsf{A}_n$ with $n$ even. It is not hard to see that 
we then have $\xi_{\im} - \xi_{\im^*}$ odd, so that 
$(\xi_{\im} - rh^{\vee}) - \xi_{\im^*}$ is even, as desired. In 
the second case, the permutation $\jm \mapsto \jm^*$ coincides 
with $\sigma$ (again by \cite[Remark 3.1]{FujitaOh}). Hence, by the statement (2) in \cite[Lemma 3.9]{FujitaOh}, we have 
$\xi_{\im} - \xi_{\im^*}$ congruent to $2$ modulo $4$, and so 
$(\xi_{\im} - rh^{\vee}) - \xi_{\im^*}$ is divisible by $4 = 2r$, 
finishing the proof.
\end{proof}

\begin{example}\label{example:first example of twisted AR quiver}
Take $(\Delta,\sigma) = (\mathsf{A}_5,\vee)$, which is of type 
$\mathsf{B}_3$, and consider $\Q$ to be the second Q-datum from 
Example \ref{example:some different Q-data}. Its twisted 
AR quiver is the following:
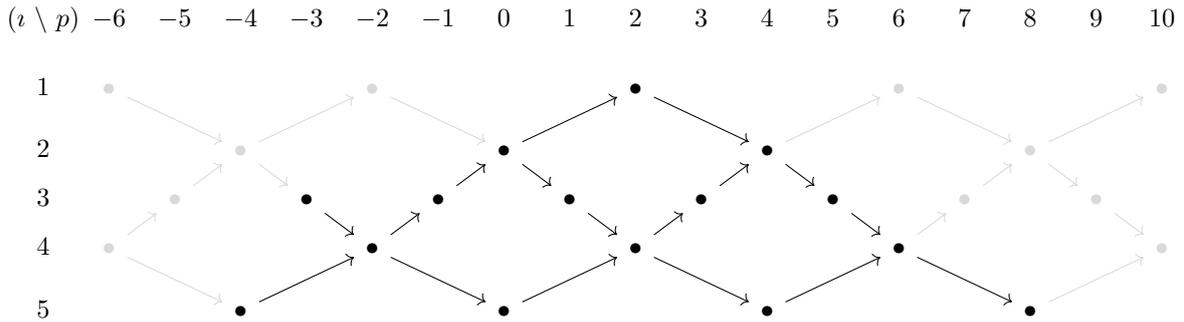
\begin{center}
\begin{tikzpicture}
\node[scale=.9] (a) at (0,0){
\begin{tikzcd}[column sep={2.5em,between origins},row sep={1em}]
	{(\im \ \backslash \ p)} & {-6} & {-5} & {-4} & {-3} & {-2} & {-1} & 0 & 1 & 2 & 3 & 4 & 5 & 6 & 7 & 8 & 9 & 10 \\
	1 & \textcolor{rgb,255:red,217;green,217;blue,217}{\bullet} &&&& \textcolor{rgb,255:red,217;green,217;blue,217}{\bullet} &&&& \bullet &&&& \textcolor{rgb,255:red,217;green,217;blue,217}{\bullet} &&&& \textcolor{rgb,255:red,217;green,217;blue,217}{\bullet} \\
	2 &&& \textcolor{rgb,255:red,217;green,217;blue,217}{\bullet} &&&& \bullet &&&& \bullet &&&& \textcolor{rgb,255:red,217;green,217;blue,217}{\bullet} \\[-0.5em]
	3 && \textcolor{rgb,255:red,217;green,217;blue,217}{\bullet} && \bullet && \bullet && \bullet && \bullet && \bullet && \textcolor{rgb,255:red,217;green,217;blue,217}{\bullet} && \textcolor{rgb,255:red,217;green,217;blue,217}{\bullet} \\[-0.5em]
	4 & \textcolor{rgb,255:red,217;green,217;blue,217}{\bullet} &&&& \bullet &&&& \bullet &&&& \bullet &&&& \textcolor{rgb,255:red,217;green,217;blue,217}{\bullet} \\
	5 &&& \bullet &&&& \bullet &&& {} & \bullet &&&& \bullet
	\arrow[color={rgb,255:red,217;green,217;blue,217}, from=2-2, to=3-4]
	\arrow[color={rgb,255:red,217;green,217;blue,217}, from=2-6, to=3-8]
	\arrow[from=2-10, to=3-12]
	\arrow[color={rgb,255:red,217;green,217;blue,217}, from=2-14, to=3-16]
	\arrow[color={rgb,255:red,217;green,217;blue,217}, from=3-4, to=2-6]
	\arrow[color={rgb,255:red,217;green,217;blue,217}, from=3-4, to=4-5]
	\arrow[from=3-8, to=2-10]
	\arrow[from=3-8, to=4-9]
	\arrow[color={rgb,255:red,217;green,217;blue,217}, from=3-12, to=2-14]
	\arrow[from=3-12, to=4-13]
	\arrow[color={rgb,255:red,217;green,217;blue,217}, from=3-16, to=2-18]
	\arrow[color={rgb,255:red,217;green,217;blue,217}, from=3-16, to=4-17]
	\arrow[color={rgb,255:red,217;green,217;blue,217}, from=4-3, to=3-4]
	\arrow[from=4-5, to=5-6]
	\arrow[from=4-7, to=3-8]
	\arrow[from=4-9, to=5-10]
	\arrow[from=4-11, to=3-12]
	\arrow[from=4-13, to=5-14]
	\arrow[color={rgb,255:red,217;green,217;blue,217}, from=4-15, to=3-16]
	\arrow[color={rgb,255:red,217;green,217;blue,217}, from=4-17, to=5-18]
	\arrow[color={rgb,255:red,217;green,217;blue,217}, from=5-2, to=4-3]
	\arrow[color={rgb,255:red,217;green,217;blue,217}, from=5-2, to=6-4]
	\arrow[from=5-6, to=4-7]
	\arrow[from=5-6, to=6-8]
	\arrow[from=5-10, to=4-11]
	\arrow[from=5-10, to=6-12]
	\arrow[color={rgb,255:red,217;green,217;blue,217}, from=5-14, to=4-15]
	\arrow[from=5-14, to=6-16]
	\arrow[from=6-4, to=5-6]
	\arrow[from=6-8, to=5-10]
	\arrow[from=6-12, to=5-14]
	\arrow[color={rgb,255:red,217;green,217;blue,217}, from=6-16, to=5-18]
\end{tikzcd}};
\end{tikzpicture}
\end{center}
The gray vertices are the vertices of $\widehat{\Delta}^{\sigma}$ 
outside of $\Gamma_{\Q}$. The injective vertices are 
$(1,2)$, $(2,4)$, $(3,5)$, $(4,6)$, and $(5,8)$, while the 
projective vertices are $(5,-4)$, $(4,-2)$, $(3,-3)$, $(2,0)$, 
and $(1,2)$. Notice that $\im \in \Delta_0$ is a source of $\Q$ 
if and only if the corresponding injective vertex is a sink of 
$\Gamma_{\Q}$. Hence, for example, $5$ is the unique source of $\Q$ 
because the corresponding injective vertex $(5,8)$ is the unique 
sink of $\Gamma_{\Q}$. Similarly, $\im \in \Delta_0$ is a sink of 
$\Q$ if and only if the corresponding projective vertex is a 
source of $\Gamma_{\Q}$ (by Lemma \ref{lemma:height of projective 
vertex}). Therefore, $1$ and $3$ are the sinks of $\Q$, which 
correspond to the projective vertices $(5,-4)$ and $(3,-3)$, 
respectively.
\end{example}

For a Q-datum $\Q$, we define $[\Q]$ to be the set of all reduced 
words for the longest element $w_0 \in \mathsf{W}$ that form a source sequence 
for $\Q$. The following result shows that we can see the associated 
twisted AR quiver as a combinatorial AR 
quiver by means of $[\Q]$.

\begin{thm}[{\cite{OhSuh19b}, \cite[Theorem 3.24]{FujitaOh}}]\label{thm:isomorphism between twisted and combinatorial AR quivers}
For a Q-datum $\Q$, the set $[\Q]$ is a commutation class of 
reduced words for the longest element $w_0 \in \mathsf{W}$. 
Moreover, there is a unique isomorphism of quivers 
$\Gamma_{\Q} \longrightarrow \Upsilon_{[\Q]}$ that intertwines 
$\pi$ and $\res^{[\Q]}$.
\end{thm}

\begin{remark}
If $\im \in \Delta_0$, then $\im$ is a source of a Q-datum $\Q$ if and only if it is a source of the commutation class $[\Q]$. In this case, one can check that $[s_{\im}\Q] = r_{\im}[\Q]$. A similar statement holds for sinks. We also remark that the isomorphism of quivers above preserves injective/projective vertices.
\end{remark}

\begin{example}\label{example:twisted AR quiver with positive roots}
Let $\Q$ be the second Q-datum of type $\mathsf{B}_3$ from 
Example \ref{example:some different Q-data}. The quiver 
$\Gamma_{\Q}$ is displayed in Example \ref{example:first example 
of twisted AR quiver}. Using a convenient compatible reading of 
$\Gamma_{\Q} \cong \Upsilon_{[\Q]}$, we see that
\[
\bm{i} = (5,4,3,2,5,3,1,4,3,2,5,3,4,3,5)
\]
is a reduced word for the longest element $w_0$ and forms a 
source sequence for $\Q$. It allows us to 
associate a positive root in $\mathsf{R}^+$ 
(of type $\mathsf{A}_5$ and not $\mathsf{B}_3$) to every vertex 
of $\Gamma_{\Q}$. The result is the following:
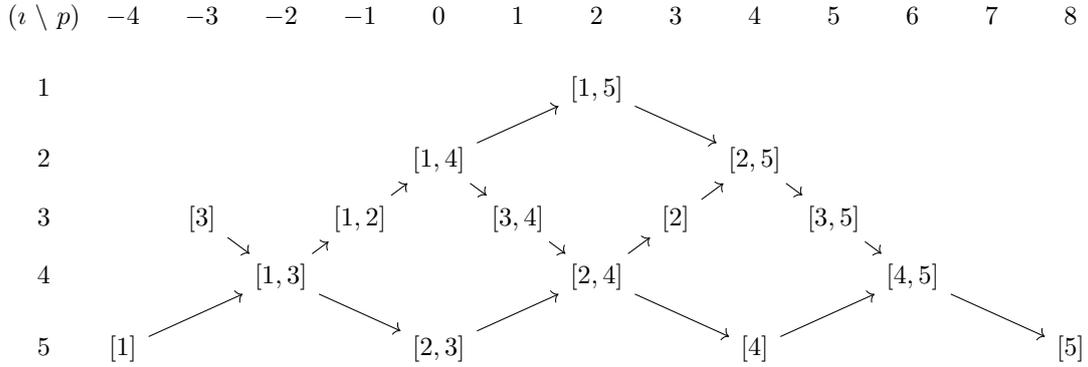
\begin{center}
\begin{tikzpicture}
\node[scale=.9] (a) at (0,0){
\begin{tikzcd}[column sep={3em,between origins},row sep={1em}]
	{(\im \ \backslash \ p)} & {-4} & {-3} & {-2} & {-1} & 0 & 1 & 2 & 3 & 4 & 5 & 6 & 7 & 8 \\
	1 &&&&&&& {[1,5]} \\
	2 &&&&& {[1,4]} &&&& {[2,5]} \\[-0.5em]
	3 && {[3]} && {[1,2]} && {[3,4]} && {[2]} && {[3,5]} \\[-0.5em]
	4 &&& {[1,3]} &&&& {[2,4]} &&&& {[4,5]} \\
	5 & {[1]} &&&& {[2,3]} &&&& {[4]} &&&& {[5]}
	\arrow[from=2-8, to=3-10]
	\arrow[from=3-6, to=2-8]
	\arrow[from=3-6, to=4-7]
	\arrow[from=3-10, to=4-11]
	\arrow[from=4-3, to=5-4]
	\arrow[from=4-5, to=3-6]
	\arrow[from=4-7, to=5-8]
	\arrow[from=4-9, to=3-10]
	\arrow[from=4-11, to=5-12]
	\arrow[from=5-4, to=4-5]
	\arrow[from=5-4, to=6-6]
	\arrow[from=5-8, to=4-9]
	\arrow[from=5-8, to=6-10]
	\arrow[from=5-12, to=6-14]
	\arrow[from=6-2, to=5-4]
	\arrow[from=6-6, to=5-8]
	\arrow[from=6-10, to=5-12]
\end{tikzcd}};
\end{tikzpicture}
\end{center}
Here $[\im,\jm]$ denotes the sum of simple roots  
$\alpha_{\im} + \alpha_{\im+1} \dotsb + \alpha_{\jm}$ for 
$1 \leq \im \leq \jm \leq 5$. To find the root associated with a 
vertex $x \in (\Gamma_{\Q})_0$, one should read $\bm{i}$ from the 
beginning until the position that corresponds to $x$ 
(according to the compatible reading chosen before). For 
example, the vertex $(3,3)$ corresponds to the sixth residue 
in $\bm{i}$; thus, the picture implies that we should have
\[
\alpha_2 = s_5s_4s_3s_2s_5(\alpha_3),
\]
which the reader can easily check to be true.
\end{example}

\subsection{Generalized twisted Coxeter elements}\label{subsection:twisted 
Coxeter elements}

We recall the construction of the generalized twisted Coxeter elements, introduced 
in \cite[Section 3.6]{FujitaOh}. Let $\Q = (\Delta, \sigma, \xi)$ 
be a Q-datum. For $i \in I$, let $i^{\circ} \in i$ be the  vertex satisfying
\[
\xi_{i^{\circ}} = \max\{\xi_{\im} \mid \im \in i\}
\]
(which is unique by \cite[Lemma 3.9]{FujitaOh}). Set 
$I_{\Q}^{\circ} = \{i^{\circ} \in \Delta_0 \mid i \in I\}$. 
Define the set
\[
X_{\Q}^{\circ} = \{(\im,\xi_{\im}) \in 
\widehat{\Delta}^{\sigma}_0 \mid \im \in I_{\Q}^{\circ}\}.
\]
We define $\tau_{\Q}^{\circ} = w[X_{\Q}^{\circ}]\sigma 
\in \mathsf{W}\sigma$. From $\xi$, we can define 
$\xi^{\circ}: \Delta_0 \longrightarrow \Z$ by 
$\xi^{\circ}_{\sigma^k(\im)} = \xi_{\im} - 2k$ for 
$\im \in I_{\Q}^{\circ}$ and $0 \leq k < d_{\overline{\im}}$. 
This is a height function on $(\Delta, \sigma)$, and so, we have a 
Q-datum $\Q^{\circ} = (\Delta, \sigma, \xi^{\circ})$. Define
\[
X_{\Q}' = \{(\sigma(i^{\circ}),p) \in 
\widehat{\Delta}^{\sigma}_0 \mid i \in I, 
\xi_{\sigma(i^{\circ})} < p \leq \xi_{i^{\circ}} - 2\},
\]
which allows us to construct $w[X_{\Q}']$.

\begin{defn}
With the notation above, the \emph{generalized twisted Coxeter element} (or 
the \emph{generalized $\sigma$-Coxeter element}) 
$\tau_{\Q} \in \mathsf{W}\sigma$ associated with $\Q$ is defined 
by
\[
\tau_{\Q} = w[X_{\Q}']^{-1} \cdot \tau_{\Q}^{\circ} \cdot w[X_{\Q}'].
\]
Note that $\tau_{\Q} = \tau_{\Q}^{\circ}$ when $\Q = \Q^{\circ}$.
\end{defn}

\begin{example}
Let $\Q$ be the Q-datum of type $\mathsf{B}_3$ from Example 
\ref{example:first example of twisted AR quiver}. The sets 
$X_{\Q}^{\circ}$ and $X_{\Q}'$ defined above are displayed in the colors red and green, respectively, in the picture below.
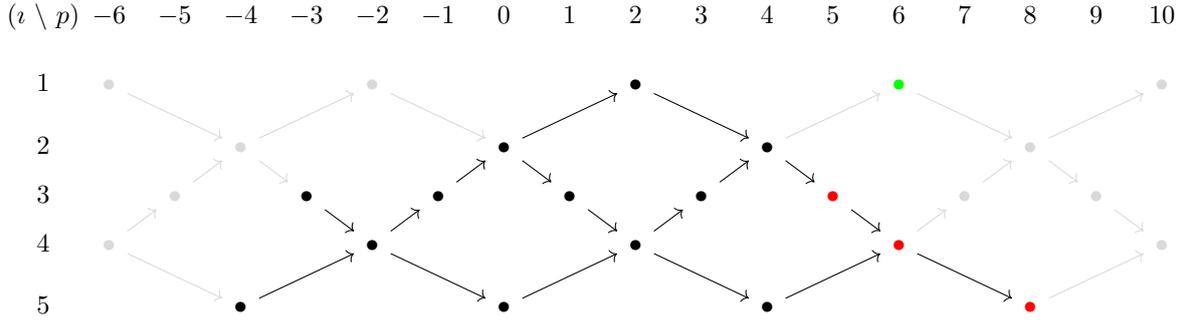
\begin{center}
\begin{tikzpicture}
\node[scale=0.9] (a) at (0,0){
\begin{tikzcd}[column sep={2.5em,between origins},row sep={1em}]
	{(\im \ \backslash \ p)} & {-6} & {-5} & {-4} & {-3} & {-2} & {-1} & 0 & 1 & 2 & 3 & 4 & 5 & 6 & 7 & 8 & 9 & 10 \\
	1 & \textcolor{rgb,255:red,217;green,217;blue,217}{\bullet} &&&& \textcolor{rgb,255:red,217;green,217;blue,217}{\bullet} &&&& \bullet &&&& \textcolor{rgb,255:red,0;green,255;blue,0}{\bullet} &&&& \textcolor{rgb,255:red,217;green,217;blue,217}{\bullet} \\
	2 &&& \textcolor{rgb,255:red,217;green,217;blue,217}{\bullet} &&&& \bullet &&&& \bullet &&&& \textcolor{rgb,255:red,217;green,217;blue,217}{\bullet} \\[-0.5em]
	3 && \textcolor{rgb,255:red,217;green,217;blue,217}{\bullet} && \bullet && \bullet && \bullet && \bullet && \textcolor{rgb,255:red,255;green,0;blue,0}{\bullet} && \textcolor{rgb,255:red,217;green,217;blue,217}{\bullet} && \textcolor{rgb,255:red,217;green,217;blue,217}{\bullet} \\[-0.5em]
	4 & \textcolor{rgb,255:red,217;green,217;blue,217}{\bullet} &&&& \bullet &&&& \bullet &&&& \textcolor{rgb,255:red,255;green,0;blue,0}{\bullet} &&&& \textcolor{rgb,255:red,217;green,217;blue,217}{\bullet} \\
	5 &&& \bullet &&&& \bullet &&& {} & \bullet &&&& \textcolor{rgb,255:red,255;green,0;blue,0}{\bullet}
	\arrow[color={rgb,255:red,217;green,217;blue,217}, from=2-2, to=3-4]
	\arrow[color={rgb,255:red,217;green,217;blue,217}, from=2-6, to=3-8]
	\arrow[from=2-10, to=3-12]
	\arrow[color={rgb,255:red,217;green,217;blue,217}, from=2-14, to=3-16]
	\arrow[color={rgb,255:red,217;green,217;blue,217}, from=3-4, to=2-6]
	\arrow[color={rgb,255:red,217;green,217;blue,217}, from=3-4, to=4-5]
	\arrow[from=3-8, to=2-10]
	\arrow[from=3-8, to=4-9]
	\arrow[color={rgb,255:red,217;green,217;blue,217}, from=3-12, to=2-14]
	\arrow[from=3-12, to=4-13]
	\arrow[color={rgb,255:red,217;green,217;blue,217}, from=3-16, to=2-18]
	\arrow[color={rgb,255:red,217;green,217;blue,217}, from=3-16, to=4-17]
	\arrow[color={rgb,255:red,217;green,217;blue,217}, from=4-3, to=3-4]
	\arrow[from=4-5, to=5-6]
	\arrow[from=4-7, to=3-8]
	\arrow[from=4-9, to=5-10]
	\arrow[from=4-11, to=3-12]
	\arrow[from=4-13, to=5-14]
	\arrow[color={rgb,255:red,217;green,217;blue,217}, from=4-15, to=3-16]
	\arrow[color={rgb,255:red,217;green,217;blue,217}, from=4-17, to=5-18]
	\arrow[color={rgb,255:red,217;green,217;blue,217}, from=5-2, to=4-3]
	\arrow[color={rgb,255:red,217;green,217;blue,217}, from=5-2, to=6-4]
	\arrow[from=5-6, to=4-7]
	\arrow[from=5-6, to=6-8]
	\arrow[from=5-10, to=4-11]
	\arrow[from=5-10, to=6-12]
	\arrow[color={rgb,255:red,217;green,217;blue,217}, from=5-14, to=4-15]
	\arrow[from=5-14, to=6-16]
	\arrow[from=6-4, to=5-6]
	\arrow[from=6-8, to=5-10]
	\arrow[from=6-12, to=5-14]
	\arrow[color={rgb,255:red,217;green,217;blue,217}, from=6-16, to=5-18]
\end{tikzcd}};
\end{tikzpicture}
\end{center}
Therefore, we have $\tau_{Q}^{\circ} = s_5s_4s_3\sigma$ and 
$w[X_{\Q}'] = s_1$, which yields 
$\tau_{\Q} = s_1s_5s_4s_3\sigma s_1$.
\end{example}

\begin{prop}[{\cite[Proposition 3.34]{FujitaOh}}]\label{prop:properties of Coxeter element}
Let $\Q = (\Delta, \sigma, \xi)$ be a Q-datum.
\begin{enumerate}[(1)]
    \item If $\im \in \Delta_0$ is a source of $\Q$, then 
	$s_{\im}^{-1}\tau_{\Q}s_{\im} = \tau_{s_{\im}\Q}$.
    \item The order of $\tau_{\Q}$ is $rh^{\vee}$.
    \item If $\sigma \neq \id$, then 
	$\tau_{\Q}^{rh^{\vee}/2} = -1$.
\end{enumerate}
\end{prop}

\subsection{The bijection \texorpdfstring{$\phi_{\Q}$}{} and the \texorpdfstring{$\mathfrak{g}$}{}-additive property}\label{subsection:phiQ and g-additive} Following \cite[Section 3.7]{FujitaOh}, we define a map $\phi_{\Q}: \widehat{\Delta}^{\sigma}_0 \longrightarrow \widehat{\mathsf{R}}$ for a Q-datum $\Q = (\Delta,\sigma,\xi)$. First, for $\im \in \Delta_0$, we set $\phi_{\Q}(\im,\xi_{\im}) = (\gamma_{\im}^{\Q},0)$, where
\[
\gamma^{\Q}_{\im} = 
(1 - \tau_{\Q}^{d_{\overline{\im}}})\varpi_{\im} \in \mathsf{R}^+.
\]
Now, if $\phi_{\Q}(\im,p) = (\alpha,k)$, then
\[
\phi_{\Q}(\im,p\pm 2d_{\overline{\im}}) = \begin{cases}
(\tau_{\Q}^{\mp d_{\overline{\im}}}(\alpha),k) &\textrm{if }(-1)^k\tau_{\Q}^{\mp d_{\overline{\im}}}(\alpha) \in \mathsf{R}^+,\\
(\tau_{\Q}^{\mp d_{\overline{\im}}}(\alpha),k\pm 1) &\textrm{if }(-1)^{k\pm 1}\tau_{\Q}^{\mp d_{\overline{\im}}}(\alpha) \in \mathsf{R}^+.
\end{cases}
\]
This recursively defines $\phi_{\Q}$. Composing $\phi_{\Q}$ with the projection onto the first coordinate, we get the function $\psi_{\Q}: \widehat{\Delta}^{\sigma}_0 \longrightarrow \mathsf{R}$ given by
\[
\psi_{\Q}(\im,p) = \tau_{\Q}^{(\xi_{\im} - p)/2}(\gamma^{\Q}_\im)
\]
for $(\im,p) \in \widehat{\Delta}^{\sigma}$.

\begin{thm}\label{thm:coordinate 
	of a positive root}
Let $\Q = (\Delta,\sigma,\xi)$ be a Q-datum. The map 
$\phi_{\Q}: \widehat{\Delta}^{\sigma}_0 \longrightarrow \widehat{\mathsf{R}}$ is a bijection. Moreover, we can upgrade it to an isomorphism of quivers $\widehat{\Delta}^{\sigma} \longrightarrow \widehat{\Upsilon}_{[\Q]}$ preserving residues, which restricts to the isomorphism $\Gamma_{\Q} \longrightarrow \Upsilon_{[\Q]}$ in Theorem \ref{thm:isomorphism between twisted and combinatorial AR quivers}.
\end{thm}

\begin{proof}
The first statement follows from \cite[Theorem 3.35]{FujitaOh}. This theorem also states that, if we identify $\Upsilon_{[\Q]}$ as a subquiver of $\widehat{\Upsilon}_{[\Q]}$, then the restriction of $\phi_{\Q}$ to $(\Gamma_{\Q})_0$ is the underlying bijection of the isomorphism $\Gamma_{\Q} \longrightarrow \Upsilon_{[\Q]}$ in Theorem \ref{thm:isomorphism between twisted and combinatorial AR quivers}. The proof in \cite{FujitaOh} shows that, if $\im \in \Delta_0$ is a source of $\Q$, then the following triangle commutes:
\[\begin{tikzcd}
	{\widehat{\Delta}^{\sigma}_0} & {(\widehat{\Upsilon}_{[\Q]})_0} \\
	& {(\widehat{\Upsilon}_{[s_{\im}\Q]})_0}
	\arrow["{\phi_{\Q}}", from=1-1, to=1-2]
	\arrow["{\phi_{s_{\im}\Q}}"', from=1-1, to=2-2]
	\arrow[from=1-2, to=2-2]
\end{tikzcd}\]
where the vertical map is the underlying map of the isomorphism $\widehat{\Upsilon}_{[\Q]} \longrightarrow \widehat{\Upsilon}_{[s_{\im}\Q]}$ from the end of Section \ref{subsection:combinatorial reflection functors}. 
Now, notice that a Q-datum $\Q'$ can be obtained from $\Q$ by a sequence of reflections precisely if it has the same type and parity as $\Q$. Moreover, if we see $\Upsilon_{[\Q']}$ as a subquiver of $\widehat{\Upsilon}_{[\Q]}$ via the isomorphism $\widehat{\Upsilon}_{[\Q']} \longrightarrow \widehat{\Upsilon}_{[\Q]}$ coming from a sequence of reflections from $\Q'$ to $\Q$, then $\widehat{\Upsilon}_{[\Q]}$ is the union of the subquivers $\Upsilon_{[\Q']}$ as we vary over all such $\Q'$. Similarly, $\widehat{\Delta}^{\sigma}$ is the union of the subquivers $\Gamma_{\Q'}$ as we vary $\Q'$. Therefore, we conclude the proof by repeatedly applying the commutative triangle above.
\end{proof}

\begin{lemma}\label{lemma:projective in terms of tau and the 
	fundamental weight}
Let $\Q = (\Delta, \sigma, \xi)$ be a Q-datum and take 
$\im \in \Delta_0$. If $x_{\im} \in (\Gamma_{\Q})_0$ is the 
projective vertex associated with $\im$, then
\[
\psi_{\Q}(x_{\im}) = 
(1-\tau_{\Q}^{-d_{\overline{\im}}})\varpi_{\im}.
\]
\end{lemma}

\begin{proof}
If $\sigma = \id$, we can view $\Q$ as a quiver whose underlying graph is $\Delta$. Note that $\tau_{\Q} = s_{\im_1}\dotsb s_{\im_n}$ where $(\im_1,\dots,\im_n)$ is a source sequence of $\Q$ where each vertex appears exactly once. One can then compute that $(1-\tau_{\Q}^{-1})\varpi_{\im}$ is the sum of all simple roots $\alpha_{\jm}$ such that there is an oriented path from $\im$ to $\jm$ in $\Q$. In particular, $(1-\tau_{\Q}^{-1})\varpi_{\im}$ gives the dimension vector of the indecomposable projective representation of $\Q$ corresponding to the vertex $\im$. The result then follows from \cite[Theorem 2.5]{FujitaOh} and the discussion in \cite[Section 2.3]{FujitaOh}.

If $\sigma \neq \id$, notice that $rh^{\vee}$ and 
$\xi_{\im^*} - \xi_{\im}$ are even. By Lemma \ref{lemma:height of 
projective vertex}, we have 
$x_{\im} = (\im^*, \xi_{\im} - rh^{\vee} + 2d_{\overline{\im}})$. 
Hence,
\[
\psi_{\Q}(x_{\im}) = \tau_{\Q}^{(\xi_{\im^*} - 
(\xi_{\im} - rh^{\vee} + 
2d_{\overline{\im}}))/2}(\gamma^{\Q}_{\im^*}) = 
\tau_{\Q}^{rh^{\vee}/2 - 
d_{\overline{\im}}}(\tau_{\Q}^{(\xi_{\im^*} - 
\xi_{\im})/2}(\gamma^{\Q}_{\im^*})) = \tau_{\Q}^{rh^{\vee}/2 - 
d_{\overline{\im}}}(\gamma^{\Q}_{\im}),
\]
where the last equality follows from \cite[Lemma 3.38]{FujitaOh}. 
We conclude using the item (3) of Proposition \ref{prop:properties of 
Coxeter element}:
\[
\tau_{\Q}^{rh^{\vee}/2 - d_{\overline{\im}}}(\gamma^{\Q}_{\im}) = 
-\tau_{\Q}^{-d_{\overline{\im}}}(\gamma^{\Q}_{\im}) = 
-\tau_{\Q}^{-d_{\overline{\im}}}(1 - \tau_{\Q}^{d_{\overline{\im}}})\varpi_{\im} 
= (1-\tau_{\Q}^{-d_{\overline{\im}}})\varpi_{\im},
\]
as desired.
\end{proof}

To finish this subsection, we state the \emph{$\mathfrak{g}$-additive property} of \cite[Theorem 3.41]{FujitaOh}.

\begin{thm}\label{thm:g-additive property}
Let $\Q = (\Delta,\sigma,\xi)$ be a Q-datum. For $(\im,p) \in \widehat{\Delta}^{\sigma}_0$, we have
\[
\psi_{\Q}(\im,p) + \psi_{\Q}(\im,p-2d_{\overline{\im}}) = \sum_{(\jm,s) \in V(\im,p)} \psi_{\Q}(\jm,s)
\]
where $V(\im,p) \subset \widehat{\Delta}^{\sigma}_0$ is the set of vertices $(\jm,s)$ such that $\jm \sim \im$ and $p - 2d_{\overline{\im}} < s < p$.
\end{thm}

\begin{proof}
This result is essentially Proposition 8.32 of \cite{OhSuh19b} rewritten with our notation. One can also deduce the formula from \cite[Theorem 3.41]{FujitaOh}, where they give a different (but equivalent) expression for the sum above.
\end{proof}

\begin{remark}
If we identify $\widehat{\Delta}^{\sigma}$ with $\widehat{\Upsilon}_{[\Q]}$ via Theorem \ref{thm:coordinate of a positive root}, then we have $(\im,p-2d_{\overline{\im}}) = s_{[\Q]}(\im,p)$ and the set $V(\im,p)$ above is the set of abutters of $(\im,p)$ in $\widehat{\Upsilon}_{[\Q]}$ (see Section \ref{subsection:meshes}).
\end{remark}

\section{The ambient category}\label{section:the categories}

Let $Q$ be an orientation of the simply laced Dynkin diagram $\Delta$. We will now define and give some properties of the complete 2-dimensional Ginzburg dg algebra $\Pi_Q$ associated with $Q$. Our main goal will be to explain how the perfectly valued derived category $\pvd(\Pi_Q)$ and its spherical twists categorify the root system of $\Delta$ and the action of its Weyl group. This category will serve as an ``ambient category'' in which we will define the categories $\C([\bm{i}])$ and $\R([\bm{i}])$ associated with a commutation class $[\bm{i}]$ of reduced words (see Section \ref{section:objects}).

\begin{remark}
We will assume that the reader is familiar with the theory of differential graded (=dg) algebras and their derived categories (see, e.g., \cite{Keller94} and \cite{KellerYang}).
\end{remark}

\begin{remark}
We use the cohomological notation when dealing with complexes. All (dg) modules are \emph{left} (dg) modules. In particular, 
representations of a quiver naturally correspond to modules over its path algebra.
\end{remark}

\subsection{The 2-dimensional Ginzburg dg algebra}
\label{subsection:Calabi--Yau completion} Let $Q$ be an orientation of 
the simply laced Dynkin diagram $\Delta$. We denote by $Q_0$ its vertex set and by $Q_1$ its arrow set. Let $Q^*$ be 
the quiver obtained from $Q$ by adding an arrow 
$\alpha^*: j \to i$ for every arrow $\alpha: i \to j$ in $Q$, 
and a loop $t_i: i \to i$ for every vertex $i \in Q_0$. We regard 
$Q^*$ as a graded quiver where the arrows $\alpha$ and $\alpha^*$ 
(for $\alpha \in Q_1$) are of degree $0$ and the loops $t_i$ 
(for $i \in Q_0$) are of degree $-1$.

Let $K$ be an 
algebraically closed field. The \emph{(complete) 2-dimensional Ginzburg dg algebra $\Pi_Q$ associated with $Q$} is 
defined as the complete differential graded path algebra $\widehat{KQ^*}$ whose differential is continuous and
determined by the equations $d(\alpha) = d(\alpha^*) = 0$ for $\alpha \in Q_1$ and
\[
d(t_i) = \sum_{\alpha \in Q_1}e_i(\alpha\alpha^* - \alpha^*\alpha)e_i,
\]
where $e_i$ denotes the idempotent associated with the vertex 
$i \in Q_0$. Here, we complete the path algebra $KQ^*$ in the category of graded vector spaces with respect to the ideal generated by the arrows of $Q^*$. Hence, the $n$-th component of the graded algebra $\Pi_Q$ consists of elements of the form $\Sigma_p \lambda_p p$ where $p$ runs over all paths in $Q^*$ of degree $n$. Notice that there are canonical homomorphisms of algebras $KQ \hookrightarrow \Pi_Q$ and $\Pi_Q \twoheadrightarrow KQ$ that compose to the identity on the path algebra $KQ$.

The dg algebra $\Pi_Q$ does not depend, up to isomorphism, on the 
orientation $Q$, only on $\Delta$. Indeed, if $Q'$ is another 
orientation of $\Delta$, then we have an isomorphism 
$\Pi_Q \longrightarrow \Pi_{Q'}$ of dg algebras which sends 
$e_i$ to $e_i$ and $t_i$ to $t_i$ (for $i \in Q_0$), $\alpha$ to 
$\beta$ and $\alpha^*$ to $\beta^*$ (if we have arrows 
$\alpha : i \to j$ in $Q$ and $\beta: i \to j$ in $Q'$), and 
$\alpha$ to $-\beta^*$ and $\alpha^*$ to $\beta$ (if we have 
arrows $\alpha : i \to j$ in $Q$ and $\beta: j \to i$ in $Q'$). 
In particular, we can canonically identify $KQ'$ as a subalgebra and as a quotient
of $\Pi_{Q}$, and we shall do so without mention from now on.

Any automorphism $\sigma$ of $\Delta$ (as in Section 
\ref{section:Q-data combinatorics}) induces an automorphism 
of $\Pi_Q$ which we now define. First, define a new orientation 
$\sigma Q$ of $\Delta$ as follows: for $i,j \in \Delta_0$, there 
is an arrow $\sigma(i) \to \sigma(j)$ in $\sigma Q$ if and only 
if there is an arrow $i \to j$ in $Q$. It is not hard to see that 
there is an isomorphism $\Pi_{Q} \longrightarrow \Pi_{\sigma Q}$ 
sending each $e_i$ to $e_{\sigma(i)}$, each $t_i$ to 
$t_{\sigma(i)}$, each $\alpha$ to $\sigma(\alpha)$, and each 
$\alpha^*$ to $\sigma(\alpha)^*$. Composing with the isomorphism 
$\Pi_{\sigma Q} \longrightarrow \Pi_{Q}$ of the previous paragraph, 
we get the desired automorphism of $\Pi_{Q}$. If we choose $Q$ 
such that $\sigma Q = Q$ (which we can always do), this last isomorphism is simply the identity.

\subsection{The perfectly valued category and its spherical twists} The \emph{perfectly valued derived category} $\pvd(\Pi_Q)$ of $\Pi_Q$ is defined as the full 
subcategory of the derived category of $\Pi_{Q}$ whose objects 
are the \emph{perfectly valued} dg modules, that is, dg modules 
that have finite-dimensional total cohomology. By \cite{Keller11} (see also \cite{Keller18} and \cite{Yeung}), 
it is a $2$-Calabi--Yau triangulated category, which implies the 
existence of a natural isomorphism
\[
\Hom_{\pvd(\Pi_Q)}(M,N) \cong D\Hom_{\pvd(\Pi_Q)}(N,\Sigma^2 M)
\]
for $M,N \in \pvd(\Pi_Q)$, where $D$ denotes the $K$-duality 
functor and $\Sigma$ is the suspension functor from the canonical 
triangulated structure.

For a vertex $i \in \Delta_0$, let $S_i$ be the corresponding 
simple $\Pi_Q$-module, viewed as an object in $\pvd(\Pi_{Q})$. One can easily compute (using, e.g., \cite[Lemma 2.15]{KellerYang}) that these objects satisfy
\[
\Ext^{k}_{\Pi_Q}(S_i,S_i) \cong \begin{cases}
K &\textrm{if } k=0,2,\\
0 &\textrm{otherwise.}
\end{cases}
\]
Together with the 2-Calabi--Yau property, this implies that these simple 
objects are \emph{2-spherical} in the sense of \cite{SeidelThomas}. Hence, by \cite{SeidelThomas} (see also \cite{HocheneggerKalckPloog}), for each $i \in \Delta_0$, the 
spherical object $S_i$ yields a \emph{spherical twist} 
$T_i : \pvd(\Pi_{Q}) \to \pvd(\Pi_{Q})$. It is a triangulated autoequivalence, 
and for every $X \in \pvd(\Pi_Q)$, there is a distinguished 
triangle
\begin{equation}\tag{$*$}\label{eq:triangle for spherical twist}
    \begin{tikzcd}
    {\RHom_{\Pi_Q}(S_i,X) \otimes_K S_i} & {X} & {T_i(X)} & \Sigma{\RHom_{\Pi_Q}(S_i,X) \otimes_K S_i}.
	\arrow[from=1-1, to=1-2]
	\arrow[from=1-2, to=1-3]
	\arrow[from=1-3, to=1-4]
    \end{tikzcd}
\end{equation}
Here we can identify $\RHom_{\Pi_Q}(S_i,X)$ with the quasi-isomorphic complex with zero differential whose entries are the cohomologies $\Ext^k_{\Pi_Q}(S_i,X) \cong \Hom_{\pvd(\Pi_Q)}(S_i,\Sigma^kX)$. Under this identification, the first map above is the evident evaluation map. 

We list some important properties of these functors below. In particular, by items (2) and (3), they satisfy the same relations that define the generalized braid group associated with $\Delta$.

\begin{prop}\label{prop:properties of spherical twists}
Let $i,j \in \Delta_0$. With the definitions above, the following 
statements hold.
\begin{enumerate}[(1)]
    \item We have $T_i(S_i) \cong \Sigma^{-1}S_i$.
    \item If $i \not\sim j$, then $T_iT_j \cong T_jT_i$.
    \item If $i \sim j$, then $T_iT_jT_i \cong T_jT_iT_j$.
    \item Let $F$ be a triangulated autoequivalence of 
	$\pvd(\Pi_Q)$. If $F(S_i) \cong S_j$, then 
	$FT_iF^{-1}(X) \cong T_j(X)$ for all $X \in \pvd(\Pi_Q)$.
\end{enumerate}
\end{prop}

\begin{proof}
Statements (1) and (4) are Lemma 3.8 and Corollary 3.7 of 
\cite{Opper}, respectively. The other two statements follow from 
\cite[Propositions 2.12 and 2.13]{SeidelThomas} and 
\cite[Lemma 2.15]{KellerYang}.
\end{proof}

\subsection{Preprojective algebras and reflection functors}\label{subsection:alternative descriptions of spherical twists} To make explicit calculations, we will use some alternative descriptions of $\pvd(\Pi_{Q})$ and its spherical twists in terms of preprojective algebras.

The \emph{double quiver} $\overline{Q}$ of $Q$ is the 
subquiver of $Q^*$ obtained by removing the loops $t_i$ for 
$i \in Q_0$. The \emph{preprojective algebra} is the quotient of 
the path algebra $K\overline{Q}$ by the 
\emph{preprojective relations}
\[
\sum_{\alpha \in Q_1}e_i(\alpha\alpha^* - \alpha^*\alpha)e_i = 0
\]
for $i \in \Delta_0$. Since $Q$ is a Dynkin quiver, this is a finite-dimensional algebra  (see \cite{Ringel} for alternative constructions that make this fact more evident). We can extend 
the map $\alpha \mapsto \alpha^*$ to an involution on the set $\overline{Q}_1$ of 
arrows of $\overline{Q}$ if we impose $(\alpha^*)^* = \alpha$ for 
$\alpha \in Q_1$. With this definition, the preprojective relation at $i \in \Delta_0$ becomes
\[
\sum_{\substack{\alpha \in \overline{Q}_1\\\alpha:j\to i}}\epsilon(\alpha)\alpha\alpha^* = 0,
\]
where $\epsilon(\alpha) = 1$ if $\alpha \in Q_1$, and 
$\epsilon(\alpha) = -1$ otherwise.

Since $\Pi_Q$ is a connective dg algebra (i.e., its cohomology is 
concentrated in nonpositive degrees), its derived category has a 
canonical t-structure whose left (resp. right) aisle consists of 
the dg modules whose cohomology is concentrated in nonpositive 
(resp. nonnegative) degrees. Its heart can be identified with 
the module category of $H^0(\Pi_Q)$ via restriction along the canonical map 
$\Pi_Q \longrightarrow H^0(\Pi_Q)$ (see \cite[Section 5.1]{KellerYang}). One can check that $H^0(\Pi_Q)$ is isomorphic to the preprojective 
algebra $\Lambda_Q$. Thus, if we restrict the canonical 
t-structure to $\pvd(\Pi_Q)$, we deduce that its heart is 
equivalent to $\modcat \Lambda_Q$, the category of finite-dimensional $\Lambda_Q$-modules. For this reason, we shall 
often see an object of $\pvd(\Pi_Q)$ concentrated in degree $0$ 
as a finite-dimensional $\Lambda_Q$-module or, equivalently, as a 
representation of the double quiver $\overline{Q}$ satisfying the 
preprojective relations. We also obtain that $\pvd(\Pi_Q)$ is the 
thick hull of the simple objects $S_i$ for $i \in \Delta_0$.

Let $\widetilde{\Delta}$ denote the affine Dynkin diagram 
corresponding to $\Delta$. Choose an orientation $\widetilde{Q}$ 
for it extending $Q$ and similarly define the dg algebra 
$\Pi_{\widetilde{Q}}$. Let $e_0 \in \Pi_{\widetilde{Q}}$ be the idempotent corresponding to the extended vertex. The canonical quotient map $\Pi_{\widetilde{Q}} \to \Pi_{Q}$ that 
sends $e_0$ to zero gives rise to a restriction functor
\[
\pvd(\Pi_{Q}) \longrightarrow \pvd(\Pi_{\widetilde{Q}}).
\]
Computing the functors $\Ext^n$ between the simple objects $S_i$ for $i \in \Delta_0$, we see that they are the same whether we compute them in $\pvd(\Pi_Q)$ or in $\pvd(\Pi_{\widetilde{Q}})$. Since $\pvd(\Pi_Q)$ is the thick hull of its simple objects, we deduce by a dévissage argument that the restriction functor above is fully faithful. In particular, we can see $\pvd(\Pi_{Q})$ as a full  triangulated subcategory of $\pvd(\Pi_{\widetilde{Q}})$.

Since $\widetilde{\Delta}$ is not a Dynkin 
diagram, $\Pi_{\widetilde{Q}}$ is a stalk dg algebra whose $0$-th cohomology is isomorphic to $\Lambda_{\widetilde{Q}}$, the preprojective algebra associated with 
$\widetilde{Q}$ completed at the ideal generated by the arrows. This follows from the noncomplete case, which is done, for example, in \cite[Section 5.2]{KellerWang}. In particular, $\Pi_{\widetilde{Q}}$ and 
$\Lambda_{\widetilde{Q}}$ are quasi-isomorphic dg algebras and have triangle-equivalent derived categories. Therefore, we can identify $\pvd(\Pi_{\widetilde{Q}})$ with $\D^b(\modcat \Lambda_{\widetilde{Q}})$, the full subcategory of the derived category of $\Lambda_{\widetilde{Q}}$ whose objects have finite-dimensional total cohomology. The simple object $S_i$ for $i \in \Delta_0$ is also 2-spherical in $\D^b(\modcat \Lambda_{\widetilde{Q}})$, and the inverse of the induced spherical twist is naturally isomorphic to the functor $I_i \otimes^{\Ltensor}_{\Lambda_{\widetilde{Q}}}-$, where $I_i = \Lambda_{\widetilde{Q}}(1-e_i)\Lambda_{\widetilde{Q}}$ is the kernel of the quotient map $\Lambda_{\widetilde{Q}} \longrightarrow S_i$ (see \cite[Theorem 6.14]{IyamaReiten} and \cite[Section III]{BuanIyamaReitenScott}). Consequently, the induced spherical twist is given by $\RHom_{\Lambda_{\widetilde{Q}}}(I_i,-)$. Viewing $\pvd(\Pi_Q)$ as a full triangulated subcategory of $\D^b(\modcat \Lambda_{\widetilde{Q}})$, this gives an alternative way for computing $T_i$ and $T_i^{-1}$.

\begin{lemma}\label{lemma:tensor product is product}
Let $(i_1,\dots,i_t)$ be a reduced word for $w \in \mathsf{W}$. Then $I_{i_1} \otimes^{\Ltensor}_{\Lambda_{\widetilde{Q}}} I_{i_2} \otimes^{\Ltensor}_{\Lambda_{\widetilde{Q}}} \dotsb \otimes^{\Ltensor}_{\Lambda_{\widetilde{Q}}} I_{i_t}$ is isomorphic to the product of ideals $I_{i_1}I_{i_2}\dotsb I_{i_t}$.
\end{lemma}

\begin{proof}
This result appears in the proof of \cite[Theorem III.1.9]{BuanIyamaReitenScott}.
\end{proof}

\begin{remark}
	In the noncompleted case, \cite{MizunoYang} characterizes the inverse spherical twist $T_i^{-1}$ as a derived tensor product with a dg ideal of $\Pi_Q$ defined without referring to the extended preprojective algebra $\Lambda_{\widetilde{Q}}$.
\end{remark}

We now recall the notion of reflection functors introduced in \cite{BaumannKamnitzer} and independently in \cite{Bolten} (see also \cite[Section 5]{BaumannKamnitzerTingley}). Let $M$ be a 
$\Lambda_Q$-module. We see $M$ as a representation 
of the double quiver $\overline{Q}$, that is, the data of a 
finite-dimensional $K$-vector space $M(j)$ for each 
$j \in Q_0$ and a linear map 
$M_{\alpha}: M(j) \to M(j')$ for every arrow $\alpha: j \to j'$ in 
$\overline{Q}$. For a fixed vertex $i \in \Delta_0$, consider the diagram
\[\begin{tikzcd}[column sep=5em]
	{\displaystyle\bigoplus\limits_{\substack{\alpha \in \overline{Q}_1\\\alpha:j\to i}}M(j)} & {M(i)} & {\displaystyle\bigoplus\limits_{\substack{\alpha \in \overline{Q}_1\\\alpha:j\to i}}M(j).}
	\arrow["{(\epsilon(\alpha)M_{\alpha})}", from=1-1, to=1-2]
	\arrow["{(M_{\alpha^*})}", from=1-2, to=1-3]
\end{tikzcd}\]
To simplify the notation, we will denote it as
\[\begin{tikzcd}
	{\widetilde{M}(i)} & {M(i)} & {\widetilde{M}(i).}
	\arrow["{M_{\mathrm{in}(i)}}", from=1-1, to=1-2]
	\arrow["{M_{\mathrm{out}(i)}}", from=1-2, to=1-3]
\end{tikzcd}\]
By the preprojective relations, we have 
$M_{\mathrm{in}(i)}M_{\mathrm{out}(i)} = 0$, so the image of 
$M_{\mathrm{out}(i)}$ is contained in $\ker M_{\mathrm{in}(i)}$. 
Therefore, we get a new representation $\Sigma_i(M)$ from $M$ if 
we replace the data of the previous diagram with
\[\begin{tikzcd}[column sep=5em]
	{\widetilde{M}(i)} & {\ker M_{\mathrm{in}(i)}} & {\widetilde{M}(i).}
	\arrow["{M_{\mathrm{out}(i)}M_{\mathrm{in}(i)}}", from=1-1, to=1-2]
	\arrow[hook, from=1-2, to=1-3]
\end{tikzcd}\]
Observe that the preprojective relations are still satisfied, 
hence $\Sigma_i(M)$ is a $\Lambda_Q$-module. One can 
easily extend $\Sigma_i$ to a functor on $\modcat\Lambda_Q$, the \emph{reflection functor} of \cite{BaumannKamnitzer}.

\begin{prop}\label{prop:computing the spherical twist}
Let $M \in \pvd(\Pi_Q)$ be an object concentrated in degree $0$ and view it as a $\Lambda_Q$-module. For $i \in \Delta_0$, $T_i(M)$ is again concentrated in degree $0$ if and only if the map $M_{\mathrm{in}(i)}$ defined above is surjective. In this case, we have $T_i(M) \cong \Sigma_i(M)$.
\end{prop}

\begin{proof}
Let $M \in \pvd(\Pi_Q)$ be an object concentrated in degree $0$. Using the identifications above, we have $T_i(M) \cong \RHom_{\Lambda_{\widetilde{Q}}}(I_i,M)$. By \cite[Proposition III.1.4]{BuanIyamaReitenScott}, the projective dimension of $I_i$ is at most $1$, thus we deduce that $T_i(M)$ is concentrated in degree $0$ if and only if $\Ext^1_{\Lambda_{\widetilde{Q}}}(I_i,M) = 0$. By \cite[Example 5.6(i)]{BaumannKamnitzerTingley}, this happens precisely when the map $M_{\mathrm{in}(i)}$ is surjective. In this case, we conclude that $T_i(M) \cong \Hom_{\Lambda_{\widetilde{Q}}}(I_i,M) \cong \Sigma_i(M)$, where the last isomorphism follows from \cite[Proposition 5.1]{BaumannKamnitzerTingley}.
\end{proof}

\begin{remark}\label{rem:spherical twist gives reflection functor on the classical case}
Let $Q'$ be another orientation of $\Delta$ and suppose that $i \in Q'_0$ is a source. Denote the reflected quiver $s_iQ'$ by $Q''$. Let $F_i^-: \modcat KQ' \longrightarrow \modcat KQ''$ be the classical reflection functor of Bernstein--Gelfand--Ponomarev \cite{BernsteinGelfandPonomarev}. We claim that we have a commutative diagram of functors (up to natural isomorphism):
\[\begin{tikzcd}
	{\D^b(\modcat KQ')} & {\pvd(\Pi_{\widetilde{Q}}) \cong \D^b(\modcat\Lambda_{\widetilde{Q}})} \\
	{\D^b(\modcat KQ'')} & {\pvd(\Pi_{\widetilde{Q}}) \cong \D^b(\modcat\Lambda_{\widetilde{Q}})}
	\arrow[from=1-1, to=1-2]
	\arrow["\Ltensor F_i^-"', from=1-1, to=2-1]
	\arrow["I_i \otimes^{\Ltensor}_{\Lambda_{\widetilde{Q}}}-", shift left=10, from=1-2, to=2-2]
	\arrow[from=2-1, to=2-2]
\end{tikzcd}\]
where the functor on the left is the left derived functor of $F_i^-$ and the horizontal functors are induced by the natural quotient maps $\Pi_{\widetilde{Q}} \longrightarrow KQ'$ and $\Pi_{\widetilde{Q}} \longrightarrow KQ''$ which factor through $\Pi_Q$. To prove this, notice that all functors above can be seen as left derived functors (as the restriction functors are exact). Therefore, by \cite[Lemma 6.4]{Keller94}, it is enough to show that $I_i \otimes^{\Ltensor}_{\Lambda_{\widetilde{Q}}} KQ' \cong \Ltensor F_i^-(KQ')$ in the derived category of $\Lambda_{\widetilde{Q}}$-$KQ'$-bimodules. This isomorphism can be verified using a dual version of Proposition \ref{prop:computing the spherical twist} after we realize that the reflection functors of \cite{BaumannKamnitzer} restrict to those of \cite{BernsteinGelfandPonomarev} (see \cite[Proposition 7.1]{BaumannKamnitzer}). Alternatively, see \cite[Corollary 2.12]{AmiotIyamaReitenTodorov}.

As a consequence, we obtain the following commutative diagram of functors:
\[\begin{tikzcd}
	{\D^b(\modcat KQ')} & {\pvd(\Pi_Q)} \\
	{\D^b(\modcat KQ'')} & {\pvd(\Pi_Q)}
	\arrow[from=1-1, to=1-2]
	\arrow["\Ltensor F_i^-"', from=1-1, to=2-1]
	\arrow["T_i^{-1}", from=1-2, to=2-2]
	\arrow[from=2-1, to=2-2]
\end{tikzcd}\]
Since $\Ltensor F_i^-$ is an equivalence with quasi-inverse given by the right derived functor of the other reflection functor $F_i^+: \modcat KQ'' \longrightarrow \modcat KQ'$ of \cite{BernsteinGelfandPonomarev}, we also obtain a similar diagram involving $\bm{\mathrm{R}}F_i^+$ and $T_i$.
\end{remark}

\subsection{A categorification of the root system}\label{subsection:categorification of the root system} Let 
$K_0(\pvd(\Pi_Q))$ be the Grothendieck group of $\pvd(\Pi_Q)$. 
One can show that it is isomorphic to the Grothendieck group of 
the canonical heart $\modcat\Lambda_Q$; hence, $K_0(\pvd(\Pi_Q))$ is a 
finitely generated free abelian group with a canonical basis 
given by the classes $[S_i]$ ($i \in \Delta_0$) of the simple 
objects. In particular, we can identify $K_0(\pvd(\Pi_Q))$ with 
the root lattice $\mathsf{Q}$ associated with $\Delta$ in such a 
way that the class $[S_i]$ corresponds to the simple root 
$\alpha_i$.

We can also categorify the weight lattice $\mathsf{P}$ of $\Delta$. 
Let $\per(\Pi_Q)$ be the \emph{perfect derived category of} 
$\Pi_Q$, that is, the subcategory of compact objects of the 
derived category of $\Pi_Q$. Equivalently, $\per(\Pi_Q)$ is 
given by the thick hull of $\Pi_Q$ viewed as an object in its derived category. Its Grothendieck group $K_0(\per(\Pi_Q))$ is generated by the 
classes $[P_i]$ for $i \in \Delta_0$, where 
$P_i = \Pi_Qe_i$ (this follows, e.g., from \cite[Lemma 2.14]{Plamondon}). These classes satisfy some orthogonality relations against the classes $[S_j] \in K_0(\pvd(\Pi_Q))$ under the Euler form (see below). Thus, they are also linearly independent and form a basis of $K_0(\per(\Pi_Q))$. We identify 
$K_0(\per(\Pi_Q))$ with $\mathsf{P}$ by 
sending $[P_i]$ to the fundamental weight $\varpi_i$. 

By \cite[Theorem 2.19]{KellerYang}, we have $\pvd(\Pi_Q) \subseteq \per(\Pi_Q)$. This inclusion induces a map
\[
\mu: K_0(\pvd(\Pi_Q)) \longrightarrow K_0(\per(\Pi_Q))
\]
between the corresponding Grothendieck groups. The next result 
shows that $\mu$ is injective and that the inclusion 
$K_0(\pvd(\Pi_Q)) \subseteq K_0(\per(\Pi_Q))$ agrees with the 
inclusion $\mathsf{Q} \subseteq \mathsf{P}$ under the previous 
identifications. 

\begin{prop}
With the notation above, the matrix of the map $\mu$ in the 
canonical bases is the Cartan matrix of $\Delta$. In particular, 
$\mu$ is injective.
\end{prop}

\begin{proof}
The following argument is based on \cite[Section 2.14]{KellerYang}. Take $i \in \Delta_0$ and let $\pi: P_i \longrightarrow S_i$ be the canonical projection. The kernel of $\pi$ is the closed dg submodule of $P_i$ generated by all nontrivial paths starting at $i$. Hence, we have
\[
\ker\pi = \sum\limits_{\substack{\alpha \in Q^*_1\\\alpha:i\to j}}P_j\alpha = \bigoplus\limits_{\substack{\alpha \in Q^*_1\\\alpha:i\to j}}P_j\alpha.
\]
We remark that the direct sum decomposition above only holds in the category of graded $\Pi_Q$-modules because the summand $P_it_i$ is not stable under the differential. To fix this, we consider the following filtration of $\ker\pi$ by dg submodules:
\[
0 \subset \bigoplus\limits_{\substack{\alpha \in \overline{Q}_1\\\alpha:i\to j}}P_j\alpha \subset \ker\pi.  
\]
Notice that the second term above is isomorphic as a dg module to
\[
\bigoplus\limits_{\substack{\alpha \in \overline{Q}_1\\\alpha:i\to j}}\Sigma^{-|\alpha|}P_j = \bigoplus\limits_{\substack{\alpha \in \overline{Q}_1\\\alpha:i\to j}}P_j, 
\]
while the quotient of the last two terms is isomorphic to $\Sigma^{-|t_i|}P_i = \Sigma P_i$. Since there is precisely one arrow in $\overline{Q}$ starting at $i$ for each $j \in \Delta_0$ with $j \sim i$, we get the following equality in $K_0(\per(\Pi_Q))$:
\[
[\ker\pi] = \sum_{j \sim i}[P_j] - [P_i].
\]
Now, since $S_i$ is isomorphic in $\per(\Pi_Q)$ to the cone of the canonical inclusion $\ker\pi \longrightarrow P_i$, we deduce that
\[
\mu([S_i]) = [P_i] - [\ker\pi] = 2[P_i] - \sum_{j \sim i}[P_j],
\]
proving that the matrix representing $\mu$ is indeed the Cartan 
matrix of $\Delta$. The last statement is true because the 
determinant of the Cartan matrix is nonzero.
\end{proof}

For $i,j \in \Delta_0$, it is easy to see that $\RHom_{\Pi_Q}(P_i,S_j)$ is concentrated in degree $0$ and its zeroth cohomology has dimension given by Kronecker's delta $\delta_{ij}$. We deduce that $\RHom_{\Pi_Q}(P,M)$ has finite-dimensional total cohomology for any $P \in \per(\Pi_Q)$ and $M \in \pvd(\Pi_Q)$. Therefore, we can define the \emph{Euler form} as the pairing
\[
(-,-): K_0(\per(\Pi_Q)) \times K_0(\pvd(\Pi_Q)) \longrightarrow \Z
\]
given by
\[
([P],[M]) = 
\sum_{k \in \Z}(-1)^k\dim_KH^k(\RHom_{\Pi_Q}(P,M)) = 
\sum_{k \in \Z}(-1)^k\dim_K\Ext_{\Pi_Q}^k(P,M)
\]
for $P \in \per(\Pi_Q)$ and $M \in \pvd(\Pi_Q)$. Since 
$([P_i],[S_j]) = \delta_{ij}$ for all $i,j \in \Delta_0$, the Euler 
form agrees with the symmetric bilinear form defined in Section 
\ref{subsection:basic notation for root systems}. Moreover, using 
the distinguished triangle (\ref{eq:triangle for spherical twist}) 
in the definition of the spherical twist, we have
\[
[T_i(X)] = [X] - ([S_i],[X])[S_i]
\]
for all $i \in \Delta_0$ and $X \in \pvd(\Pi_Q)$. Thus, the 
spherical twists categorify the simple reflections $s_i$ of 
Section \ref{subsection:basic notation for root systems}.

If $\sigma$ is an automorphism of $\Delta$, we saw in Section 
\ref{subsection:Calabi--Yau completion} that it induces an 
automorphism of $\Pi_Q$. Restricting along its inverse induces an autoequivalence $\sigma: \per(\Pi_Q) \longrightarrow \per(\Pi_Q)$. In this setting, notice that the action of $\sigma$ in 
$K_0(\per(\Pi_Q))$ agrees with its action on the weight lattice 
$\mathsf{P}$ as defined at the end of Section 
\ref{subsection:folding}.

\section{Categories associated with commutation classes}\label{section:objects}

In this section, we introduce the category of representations $\C([\bm{i}])$ for a commutation class $[\bm{i}]$ of reduced words for $w \in \mathsf{W}$. When $w = w_0$, we also construct what we call its repetition category $\R([\bm{i}])$ and its c-derived category $\D([\bm{i}])$.

We will freely use the notation introduced in Sections 
\ref{section:Q-data combinatorics} and \ref{section:the categories}.

\begin{remark}
We fix an algebraically closed field $K$ and an orientation $Q^{\circ}$ of $\Delta$ to 
work with the dg algebra $\Pi_{Q^{\circ}}$ and its perfectly valued derived category. For simplicity, we will write $Q$ for $Q^{\circ}$ in general. 
However, when dealing with other quivers, we will prefer to denote 
them by $Q$, and the full notation $Q^{\circ}$ will be used to 
avoid ambiguity.
\end{remark}

\subsection{The category of representations}\label{section:indecomposable objects from commutation classes} Let 
$\bm{i} = (i_1, i_2, \dots, i_t)$ be a reduced word for $w \in \mathsf{W}$. For 
$1 \leq k \leq t$, define
\[
M^{\bm{i}}_k = T_{i_1}T_{i_2}\dotsb T_{i_{k-1}}(S_{i_k}) 
\in \pvd(\Pi_{Q}).
\]
It is an indecomposable object whose class in 
$K_0(\pvd(\Pi_{Q}))$ corresponds to the positive root 
$\beta^{\bm{i}}_k \in \textsf{R}^+(w)$. For the commutation class $[\bm{i}]$, define
\[
\ind([\bm{i}]) = \{M^{\bm{i}}_k \in 
\pvd(\Pi_{Q}) \mid 1 \leq k \leq t\},
\]
the set of \emph{indecomposable objects associated with} 
$[\bm{i}]$. Let
\[
\C({[\bm{i}]}) = \add \ind([\bm{i}]) \subseteq \pvd(\Pi_{Q})
\]
be the full additive subcategory generated by these objects 
and closed under isomorphisms. We call it the \emph{category of representations of} $[\bm{i}]$. As the notation suggests, we 
have the following lemma.

\begin{lemma}\label{lemma:indQ well defined}
The set $\ind([\bm{i}])$ does not depend (up to isomorphism of its elements) on 
the choice of the representative for the commutation class 
$[\bm{i}]$.    
\end{lemma}

\begin{proof}
Take $\bm{j} \in [\bm{i}]$. We may suppose that $\bm{j}$ is 
obtained from $\bm{i}$ by applying a commutation move that exchanges $i_l$ and $i_{l+1}$ for some $1 \leq l < t$ such that 
$i_l \not\sim i_{l+1}$. Since $T_{i_l}$ and $T_{i_{l+1}}$ commute (Proposition \ref{prop:properties of spherical twists}), 
we have
\[
M^{\bm{i}}_k \cong M^{\bm{j}}_k
\]
for $1 \leq k \leq t$ with $k \neq l,l+1$. Furthermore, we have 
$M^{\bm{i}}_l \cong M^{\bm{j}}_{l+1}$ and 
$M^{\bm{i}}_{l+1} \cong M^{\bm{j}}_l$, because 
$i_l \not\sim i_{l+1}$ and so $T_{i_{l+1}}(S_{i_l}) \cong S_{i_l}$ 
and $T_{i_l}(S_{i_{l+1}}) \cong S_{i_{l+1}}$ (by Proposition \ref{prop:computing the spherical twist}). Hence, the sets 
$\ind([\bm{i}])$ and $\ind([\bm{j}])$ coincide (up to isomorphism 
of its elements).
\end{proof}

By the lemma above, we deduce that 
$M^{\bm{i}}_k \cong M^{\bm{j}}_l$ if $\bm{i}$ and $\bm{j}$ are 
commutation-equivalent and $\beta^{\bm{i}}_k = \beta^{\bm{j}}_l$. 
In this way, given a commutation class $[\bm{i}]$ and a positive 
root $\alpha \in \textsf{R}^+(w)$, it is well defined (up to 
isomorphism) to set $M^{[\bm{i}]}_{\alpha}$ as $M^{\bm{i}}_k$ where 
$k$ satisfies $\alpha = \beta^{\bm{i}}_k$. We will adopt this 
notation from now on.

\begin{thm}\label{thm:indecomposables are concentrated in zero}
If $\bm{i} = (i_1,\dots,i_t)$ is any reduced word for $w \in \mathsf{W}$, 
then $M^{\bm{i}}_k$ has cohomology concentrated in degree $0$ 
for all $1 \leq k \leq t$. 
\end{thm}

\begin{proof}
Let $\bm{j} = (j_1,\dots,j_s)$ be a reduced word for some $w' \in \mathsf{W}$. Let us first show that
$M_k' = T_{j_1}^{-1}\dotsb T_{j_{k-1}}^{-1}(S_{j_k})$ 
is concentrated in degree $0$ for any $1 \leq k \leq s$. Using the characterization of the 
spherical twists and the identifications given in Section \ref{subsection:alternative descriptions of spherical 
twists} (due to \cite{IyamaReiten} and \cite{BuanIyamaReitenScott}), 
we have an isomorphism
\[
M_k' \cong I_{j_1} \otimes^{\Ltensor}_{\Lambda_{\widetilde{Q}}} \dotsb \otimes^{\Ltensor}_{\Lambda_{\widetilde{Q}}} I_{j_{k-1}}  \otimes^{\Ltensor}_{\Lambda_{\widetilde{Q}}} S_{j_k}
\]
in the derived category of the completed preprojective algebra $\Lambda_{\widetilde{Q}}$. Applying the functor $I_{j_1} \otimes^{\Ltensor}_{\Lambda_{\widetilde{Q}}} \dotsb \otimes^{\Ltensor}_{\Lambda_{\widetilde{Q}}} I_{j_{k-1}}  \otimes^{\Ltensor}_{\Lambda_{\widetilde{Q}}}-$ to the distinguished triangle induced by the exact sequence
\[\begin{tikzcd}
	0 & {I_{j_k}} & \Lambda_{\widetilde{Q}} & {S_{j_k}} & 0,
	\arrow[from=1-1, to=1-2]
	\arrow[from=1-2, to=1-3]
	\arrow[from=1-3, to=1-4]
	\arrow[from=1-4, to=1-5]
\end{tikzcd}\]
we get a distinguished triangle of the form
\[\begin{tikzcd}
    {I_{j_1} \otimes^{\Ltensor}_{\Lambda_{\widetilde{Q}}} \dotsb \otimes^{\Ltensor}_{\Lambda_{\widetilde{Q}}} I_{j_{k}}} & {I_{j_1} \otimes^{\Ltensor}_{\Lambda_{\widetilde{Q}}} \dotsb \otimes^{\Ltensor}_{\Lambda_{\widetilde{Q}}} I_{j_{k-1}}} & {M_k'} & \Sigma{I_{j_1} \otimes^{\Ltensor}_{\Lambda_{\widetilde{Q}}} \dotsb \otimes^{\Ltensor}_{\Lambda_{\widetilde{Q}}} I_{j_{k-1}}}.
	\arrow[from=1-1, to=1-2]
	\arrow[from=1-2, to=1-3]
	\arrow[from=1-3, to=1-4]
\end{tikzcd}\]
By Lemma \ref{lemma:tensor product is product}, we deduce that
\[
M_k' \cong \frac{I_{j_1}I_{j_2}\dotsb I_{j_{k-1}}}{I_{j_1}I_{j_2}\dotsb I_{j_k}},
\]
hence $M_k'$ is concentrated in degree $0$.

Now, let us prove the theorem. Since $\bm{i}$ can be extended to a reduced word for the longest element $w_0$ (this follows, e.g., from \cite[p. 158, Corollaire 3]{Bourbaki}), we can assume that $t = N$ and $w = w_0$. By the previous paragraph, we will finish the proof once we show that
\[
M^{\bm{i}}_k \cong T_{i_N^*}^{-1}T_{i_{N-1}^*}^{-1}\dotsb T_{i_{k+1}^*}^{-1}(S_{i_k^*}).
\]
Indeed, by applying $T_{i_{k-1}}^{-1}\dotsb T_{i_2}^{-1}T_{i_1}^{-1}$ on both sides, this is equivalent to
\[
S_{i_k} \cong T_{i_{k-1}}^{-1}\dotsb T_{i_2}^{-1}T_{i_1}^{-1}T_{i_N^*}^{-1}T_{i_{N-1}^*}^{-1}\dotsb T_{i_{k+1}^*}^{-1}(S_{i_k^*}).
\]
Since $(i_{k-1},\dots,i_2,i_1,i_N^*,i_{N-1}^*,\dots,i_k^*)$ is a reduced word for $w_0$, the previous paragraph implies that the object on the right is concentrated in degree $0$. One can check that its class in $K_0(\pvd(\Pi_Q))$ is $[S_{i_k}]$, so we obtain the isomorphism above because $S_{i_k}$ is a simple object of the canonical heart of $\pvd(\Pi_Q)$.
\end{proof}

\begin{remark}
When $w = w_0$, the proof above shows that $\ind([\bm{i}])$ coincides with the set of \emph{layers} of the preprojective algebra $\Lambda_Q$ associated with the reduced word $(i_N^*,i_{N-1}^*,\dots,i_1^*)$ in the sense of \cite{AmiotIyamaReitenTodorov}. This alternative description of the layers using the spherical twist functors has already been shown in \cite[Proposition 2.2]{AmiotIyamaReitenTodorov}.
\end{remark}

\begin{remark}\label{rmk:the longest twist acts as twisted shift}
	Let $\bm{i} = (i_1,\dots,i_N)$ be a reduced word for $w_0$. Consider the composition of functors $T = T_{i_1^*}\dotsb T_{i_N^*}$. The end of the proof above shows that $T(S_{i_1}) \cong \Sigma^{-1}S_{i_1^*}$. But notice that $T$ does not depend on the choice of reduced word by Proposition \ref{prop:properties of spherical twists}. Thus, more generally, $T(S_j) \cong \Sigma^{-1}S_{j^*}$ for all $j \in \Delta_0$. As a consequence, if $\bm{j}$ is the reduced word $(i_1^*,\dots,i_N^*)$, then the objects of $\C([\bm{j}])$ can be obtained from the objects of $\C([\bm{i}])$ by applying $\Sigma \circ T$.
\end{remark}

\begin{cor}\label{cor:simples are in C([i])}
If $[\bm{i}]$ is a commutation class of reduced words for the longest element $w_0$, then $\C([\bm{i}])$ contains $S_i$ for all $i \in \Delta_0$.
\end{cor}

\begin{proof}
Since $M^{[\bm{i}]}_{\alpha_i}$ is concentrated in degree $0$ and has the same class in $K_0(\pvd(\Pi_Q))$ as the simple $S_i$, we must have $M^{[\bm{i}]}_{\alpha_i} \cong S_i$.
\end{proof}

The next result shows that we can see $\C([\bm{i}])$ as a 
generalization of the category of representations of a Dynkin 
quiver. See also \cite[Theorem 3.3]{AmiotIyamaReitenTodorov}.

\begin{prop}\label{prop:concentrated in zero for adapted word}
If $\bm{i} = (i_1,\dots,i_N)$ is a reduced word for $w_0$ that is a source sequence for some orientation $Q$ of $\Delta$, then 
$M^{\bm{i}}_k$ is the image of an indecomposable $KQ$-module via 
the restriction functor
\[
\modcat KQ \longrightarrow \pvd(\Pi_{Q^{\circ}})
\]
for all $1 \leq k \leq N$. In particular, $\C([\bm{i}])$ and 
$\modcat KQ$ are equivalent as $K$-linear categories.
\end{prop}

\begin{proof}
Observe that the restriction functor above is fully faithful. Assuming the first statement above and considering the map induced 
on Grothendieck groups, this restriction must send any 
indecomposable $KQ$-module to an object isomorphic to an 
indecomposable object in $\C([\bm{i}])$. This proves the 
equivalence of categories claimed above.

Let us focus on the first part. Take $Q' =
s_{i_{k-1}}s_{i_{k-2}}\dotsb s_{i_1}Q$,
which is well defined by the hypothesis on $Q$. Observe that the 
sequence $(i_{k-1}, i_{k-2}, \dots, i_1)$ is a sink sequence for 
$Q'$. Recall that we have the restriction functor
\[
\modcat KQ''\longrightarrow \pvd(\Pi_{Q^{\circ}})
\]
induced by the quotient map $\Pi_{Q^{\circ}} \longrightarrow KQ''$ for every Dynkin quiver $Q''$ of type $\Delta$. We will prove by 
backward induction on $1 \leq l \leq k$ that
\[
N_l = T_{i_l}T_{i_{l+1}}\dotsb T_{i_{k-1}}(S_{i_k})
\]
is the image of an indecomposable $KQ''$-module via the functor 
above, where 
$Q''$ is given by $s_{i_l}^{-1}s_{i_{l+1}}^{-1}\dotsb s_{i_{k-1}}^{-1}Q'$. 
This will prove the lemma since $N_1 = M^{\bm{i}}_k$ and 
$Q = s_{i_1}^{-1}\dotsb s_{i_{k-1}}^{-1}Q'$.

One important observation for the induction to work is the 
following: if $l > 1$, we cannot have $N_l \cong S_{i_{l-1}}$. 
Indeed, the classes of $N_{l-1}$ and $N_l$ in $K_0(\pvd(\Pi_{Q^{\circ}}))$ correspond to the roots $\beta = s_{i_{l-1}}s_{i_l}\dotsb s_{i_{k-1}}(\alpha_{i_k})$ and $s_{i_{l-1}}(\beta)$, respectively. Note that $\beta$ is a positive root since $(i_{l-1}, i_l, \dots, i_k)$ represents a reduced word for some element in $\mathsf{W}$. However, if we had $N_l \cong S_{i_{l-1}}$, we would get 
$\alpha_{i_{l-1}} = s_{i_{l-1}}(\beta)$ and hence 
$\beta = -\alpha_{i_{l-1}}$, a contradiction.

Our claim is immediate for $l = k$. Take $1 < l \leq k$ and suppose 
that $N_l$ is an indecomposable $KQ''$-module for 
$Q'' = s_{i_l}^{-1}\dotsb s_{i_{k-1}}^{-1}Q'$. Note that the vertex 
$i_{l-1}$ is a sink for $Q''$. Seeing $N_l$ as a representation of 
the quiver $Q''$, the natural map from the direct sum of the vector 
spaces at the neighbors of $i_{l-1}$ to the vector space at 
$i_{l-1}$ must be surjective, otherwise $N_l$ would have a direct 
summand isomorphic to $S_{i_{l-1}}$, contradicting the observation 
from the previous paragraph. Hence, by Proposition 
\ref{prop:computing the spherical twist}, if we see $N_l$ as a $\Lambda_{Q^{\circ}}$-module, we have 
$N_{l-1} = T_{i_{l-1}}(N_l) \cong \Sigma_{i_{l-1}}(N_l)$. 
Since $N_l$ is a $KQ''$-module and $i_l$ is a sink of 
$Q''$, the reflection functor 
$\Sigma_{i_{l-1}}$ acts on $N_l$ as the classical reflection 
functor of Bernstein--Gelfand--Ponomarev \cite{BernsteinGelfandPonomarev}. Therefore, we conclude 
that $N_{l-1}$ is an indecomposable representation of 
$s_{i_{l-1}}^{-1}Q''$, finishing the induction and the proof.
\end{proof}

\begin{prop}\label{prop:equivalence under reflection, general case}
Let $\bm{i} = (i_1,\dots,i_N)$ be a reduced word for $w_0$. For a source $i \in \Delta_0$ of $[\bm{i}]$, the equivalence $T_i^{-1}: \pvd(\Pi_{Q}) \longrightarrow \pvd(\Pi_{Q})$ restricts to an equivalence from the full additive subcategory of $\C([\bm{i}])$ generated by $\ind([\bm{i}]) \setminus \{S_i\}$ to the full additive subcategory of $\C(r_i[\bm{i}])$ generated by $\ind(r_i[\bm{i}]) \setminus \{S_i\}$.
\end{prop}

\begin{proof}
We can suppose $i_1 = i$. Thus, $r_i\bm{i} = (i_2, \dotsb, i_N, i_1^*)$ is in $r_i[\bm{i}]$. It is immediate that $T_i^{-1}M^{\bm{i}}_k \cong M^{r_i\bm{i}}_{k-1}$ for $1 < k \leq N$, so we get the equivalence once we note that $M^{\bm{i}}_1 \cong S_i \cong M^{r_i\bm{i}}_N$.
\end{proof}

We conclude this subsection with two concrete examples involving commutation classes associated with Q-data. More examples can be found in Appendix \ref{appendix:examples}. By Theorem \ref{thm:indecomposables are concentrated in zero}, the indecomposable objects of $\C([\bm{i}])$ are objects of $\pvd(\Pi_Q)$ whose cohomology is concentrated in degree $0$. As discussed in Section \ref{subsection:alternative descriptions of spherical twists}, one can thus view these objects as $\Lambda_{Q}$-modules or, equivalently, as representations of the double quiver $\overline{Q}$ satisfying the preprojective relations. We will represent these objects in the examples in this way.

We will display representations graphically, following the usual practice. For example, the diagram
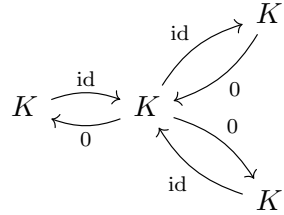
\[\begin{tikzcd}
	&& K \\
	K & K \\
	&& K
	\arrow["0", bend left=20, from=1-3, to=2-2]
	\arrow["{\mathrm{id}}", bend left=20, from=2-1, to=2-2]
	\arrow["{\mathrm{id}}", bend left=20, from=2-2, to=1-3]
	\arrow["0", bend left=20, from=2-2, to=2-1]
	\arrow["0", bend left=20, from=2-2, to=3-3]
	\arrow["{\mathrm{id}}", bend left=20, from=3-3, to=2-2]
\end{tikzcd}\]
depicts a representation of the double quiver in type $\mathsf{D}_4$. Most of our representations will have vector spaces of dimension at most one on the vertices. In this case, the linear maps between the vertices will always be zero or the identity. To simplify the notation, we will omit the arrow for zero maps and the label for identity maps. For instance, we would denote the representation above as
\begin{center}
\begin{tikzpicture}

\node at (0,0) {$K$};
\node at (0.3,0) {$K$};
\node at (0.7,0.2) {$K$};
\node at (0.7,-0.2) {$K$};
\draw[->] (0,0.25) -- (0.3,0.25);
\draw[->] (0.35,0.2) -- (0.5,0.4);
\draw[->] (0.5,-0.4) -- (0.35,-0.2);

\end{tikzpicture}
\end{center}
Whenever there is a vertex of dimension greater than $1$, we write the representation completely. If $m > 1$, we denote by $i_k: K \to K^m$ and $\pi_k: K^m \to K$ the canonical inclusion and the canonical projection on the $k$-th coordinate, respectively.

The choice of the orientation $Q$ of $\Delta$ matters when deciding if a given representation of $\overline{Q}$ comes from a $\Lambda_Q$-module. Indeed, the signs appearing in the preprojective relations come from the orientation. Hence, we will choose a particular orientation to fix the signs when needed. For convenience, our choices will satisfy $\sigma Q = Q$ (see Section \ref{subsection:Calabi--Yau completion}).

\begin{example}\label{example:B3}
Suppose $(\Delta,\sigma) = (\mathsf{A}_5,\vee)$ is of type $\mathsf{B}_3$. Let $\Q$ be the Q-datum of Example \ref{example:first example of twisted AR quiver}. The indecomposable objects of $\C([\Q])$ are presented below. We place the object $M^{[\mathcal{Q}]}_{\alpha}$ at the vertex of $\Gamma_{\Q}$ corresponding to the positive root $\alpha$ via the isomorphism $\Gamma_{\Q} \cong \Upsilon_{[\Q]}$ of Theorem \ref{thm:isomorphism between twisted and combinatorial AR quivers}.
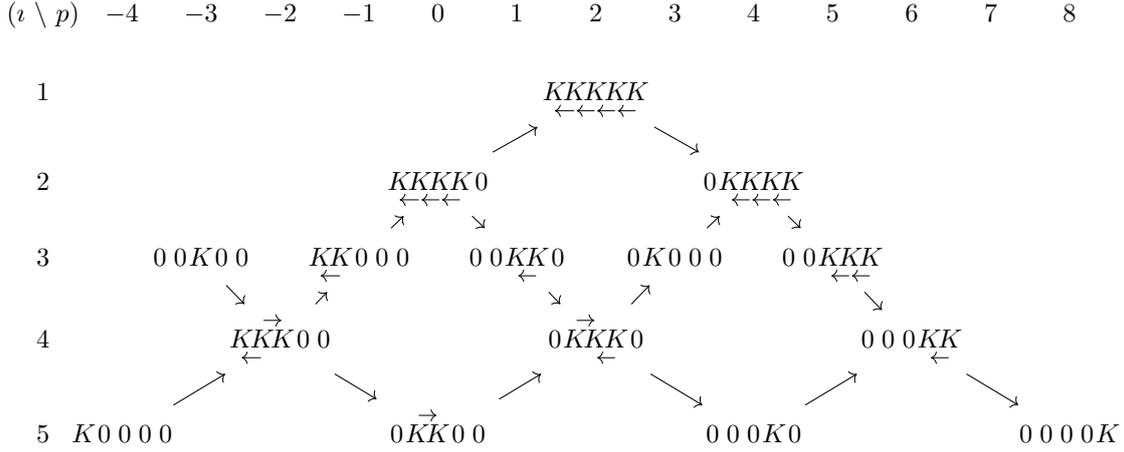
\begin{center}
	\begin{tikzpicture}
	\node[scale=0.9] (a) at (0,0){
	\begin{tikzcd}[column sep={3em,between origins},row sep={1em}]
		{(\im \ \backslash \ p)} & {-4} & {-3} & {-2} & {-1} & 0 & 1 & 2 & 3 & 4 & 5 & 6 & 7 & 8 \\
		1 &&&&&&& {\begin{tikzpicture}

			\node at (0,0) {$K$};
			\node at (0.3,0) {$K$};
			\node at (0.6,0) {$K$};
			\node at (0.9,0) {$K$};
			\node at (1.2,0) {$K$};
		
			\draw[->] (0.28,-0.15) -- (0,-0.15);
			\draw[->] (0.58,-0.15) -- (0.32,-0.15);
			\draw[->] (0.88,-0.15) -- (0.62,-0.15);
			\draw[->] (1.18,-0.15) -- (0.92,-0.15);
			
		\end{tikzpicture}} \\
		2 &&&&& {\begin{tikzpicture}

			\node at (0,0) {$K$};
			\node at (0.3,0) {$K$};
			\node at (0.6,0) {$K$};
			\node at (0.9,0) {$K$};
			\node at (1.2,0) {$0$};
		
			\draw[->] (0.28,-0.15) -- (0,-0.15);
			\draw[->] (0.58,-0.15) -- (0.32,-0.15);
			\draw[->] (0.88,-0.15) -- (0.62,-0.15);
			
		\end{tikzpicture}} &&&& {\begin{tikzpicture}

			\node at (0,0) {$0$};
			\node at (0.3,0) {$K$};
			\node at (0.6,0) {$K$};
			\node at (0.9,0) {$K$};
			\node at (1.2,0) {$K$};
		
			\draw[->] (0.58,-0.15) -- (0.32,-0.15);
			\draw[->] (0.88,-0.15) -- (0.62,-0.15);
			\draw[->] (1.18,-0.15) -- (0.92,-0.15);
			
		\end{tikzpicture}} \\[-0.5em]
		3 && {\begin{tikzpicture}

			\node at (0,0) {$0$};
			\node at (0.3,0) {$0$};
			\node at (0.6,0) {$K$};
			\node at (0.9,0) {$0$};
			\node at (1.2,0) {$0$};
			
		\end{tikzpicture}} && {\begin{tikzpicture}

			\node at (0,0) {$K$};
			\node at (0.3,0) {$K$};
			\node at (0.6,0) {$0$};
			\node at (0.9,0) {$0$};
			\node at (1.2,0) {$0$};
		
			\draw[->] (0.28,-0.15) -- (0,-0.15);
			
		\end{tikzpicture}} && {\begin{tikzpicture}

			\node at (0,0) {$0$};
			\node at (0.3,0) {$0$};
			\node at (0.6,0) {$K$};
			\node at (0.9,0) {$K$};
			\node at (1.2,0) {$0$};
		
			\draw[->] (0.88,-0.15) -- (0.62,-0.15);
			
		\end{tikzpicture}} && {\begin{tikzpicture}

			\node at (0,0) {$0$};
			\node at (0.3,0) {$K$};
			\node at (0.6,0) {$0$};
			\node at (0.9,0) {$0$};
			\node at (1.2,0) {$0$};
		
		\end{tikzpicture}} && {\begin{tikzpicture}

			\node at (0,0) {$0$};
			\node at (0.3,0) {$0$};
			\node at (0.6,0) {$K$};
			\node at (0.9,0) {$K$};
			\node at (1.2,0) {$K$};
		
			\draw[->] (0.88,-0.15) -- (0.62,-0.15);
			\draw[->] (1.18,-0.15) -- (0.92,-0.15);
			
		\end{tikzpicture}} \\[-0.5em]
		4 &&& {\begin{tikzpicture}

			\node at (0,0) {$K$};
			\node at (0.3,0) {$K$};
			\node at (0.6,0) {$K$};
			\node at (0.9,0) {$0$};
			\node at (1.2,0) {$0$};
		
			\draw[->] (0.28,-0.15) -- (0,-0.15);
			\draw[->] (0.32,0.4) -- (0.58,0.4);
			
		\end{tikzpicture}} &&&& {\begin{tikzpicture}

			\node at (0,0) {$0$};
			\node at (0.3,0) {$K$};
			\node at (0.6,0) {$K$};
			\node at (0.9,0) {$K$};
			\node at (1.2,0) {$0$};
		
			\draw[->] (0.32,0.4) -- (0.58,0.4);
			\draw[->] (0.88,-0.15) -- (0.62,-0.15);
			
		\end{tikzpicture}} &&&& {\begin{tikzpicture}

			\node at (0,0) {$0$};
			\node at (0.3,0) {$0$};
			\node at (0.6,0) {$0$};
			\node at (0.9,0) {$K$};
			\node at (1.2,0) {$K$};
		
			\draw[->] (1.18,-0.15) -- (0.92,-0.15);
			
		\end{tikzpicture}} \\
		5 & {\begin{tikzpicture}

			\node at (0,0) {$K$};
			\node at (0.3,0) {$0$};
			\node at (0.6,0) {$0$};
			\node at (0.9,0) {$0$};
			\node at (1.2,0) {$0$};
			
		\end{tikzpicture}} &&&& {\begin{tikzpicture}

			\node at (0,0) {$0$};
			\node at (0.3,0) {$K$};
			\node at (0.6,0) {$K$};
			\node at (0.9,0) {$0$};
			\node at (1.2,0) {$0$};
		
			\draw[->] (0.32,0.4) -- (0.58,0.4);
			
		\end{tikzpicture}} &&&& {\begin{tikzpicture}

			\node at (0,0) {$0$};
			\node at (0.3,0) {$0$};
			\node at (0.6,0) {$0$};
			\node at (0.9,0) {$K$};
			\node at (1.2,0) {$0$};
			
		\end{tikzpicture}} &&&& {\begin{tikzpicture}

			\node at (0,0) {$0$};
			\node at (0.3,0) {$0$};
			\node at (0.6,0) {$0$};
			\node at (0.9,0) {$0$};
			\node at (1.2,0) {$K$};
			
		\end{tikzpicture}}
		\arrow[from=2-8, to=3-10]
		\arrow[from=3-6, to=2-8]
		\arrow[from=3-6, to=4-7]
		\arrow[from=3-10, to=4-11]
		\arrow[from=4-3, to=5-4]
		\arrow[from=4-5, to=3-6]
		\arrow[from=4-7, to=5-8]
		\arrow[from=4-9, to=3-10]
		\arrow[from=4-11, to=5-12]
		\arrow[from=5-4, to=4-5]
		\arrow[from=5-4, to=6-6]
		\arrow[from=5-8, to=4-9]
		\arrow[from=5-8, to=6-10]
		\arrow[from=5-12, to=6-14]
		\arrow[from=6-2, to=5-4]
		\arrow[from=6-6, to=5-8]
		\arrow[from=6-10, to=5-12]
	\end{tikzcd}};
	\end{tikzpicture}
\end{center}
This picture agrees with Example \ref{example:twisted AR quiver with positive roots}. To compute the objects above, we used Proposition \ref{prop:computing the spherical twist}. See also Proposition \ref{prop:orientations that appear in Bn}. We remark that the arrows above can be seen as nonzero morphisms in $\C([\Q])$, as it will be explained in Section \ref{section:morphisms}.
\end{example}

\begin{example}
Suppose $(\Delta,\sigma) = (\mathsf{D}_5,\vee)$ is of type $\mathsf{C}_4$. Let $Q$ be the following orientation for $\Delta$:
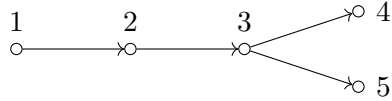
\begin{center}
\begin{tikzpicture}
		
	\node[circle, draw, inner sep=1.5pt] (D1) at (1.5,-0.5) {};
	\node[above] at (1.5,-0.4) {$1$};
	\node[circle, draw, inner sep=1.5pt] (D2) at (3,-0.5) {};
	\node[above] at (3,-0.4) {$2$};
	\node[circle, draw, inner sep=1.5pt] (D3) at (4.5,-0.5) {};
	\node[above] at (4.5,-0.4) {$3$};
	\node[circle, draw, inner sep=1.5pt] (D4) at (6,0) {};
	\node[right] at (6.1,0) {$4$};
	\node[circle, draw, inner sep=1.5pt] (D5) at (6,-1) {};
	\node[right] at (6.1,-1) {$5$};

	\draw[->] (D1) -- (D2);
	\draw[->] (D2) -- (D3);
	\draw[->] (D3) -- (D4);
	\draw[->] (D3) -- (D5);

\end{tikzpicture}
\end{center}
Let $\Q = (\Delta,\sigma,\xi)$ be the Q-datum defined by $\xi_k = k$ for $1 \leq k \leq 4$ and $\xi_5 = 2$. The indecomposable objects of $\C([\Q])$ are shown below.
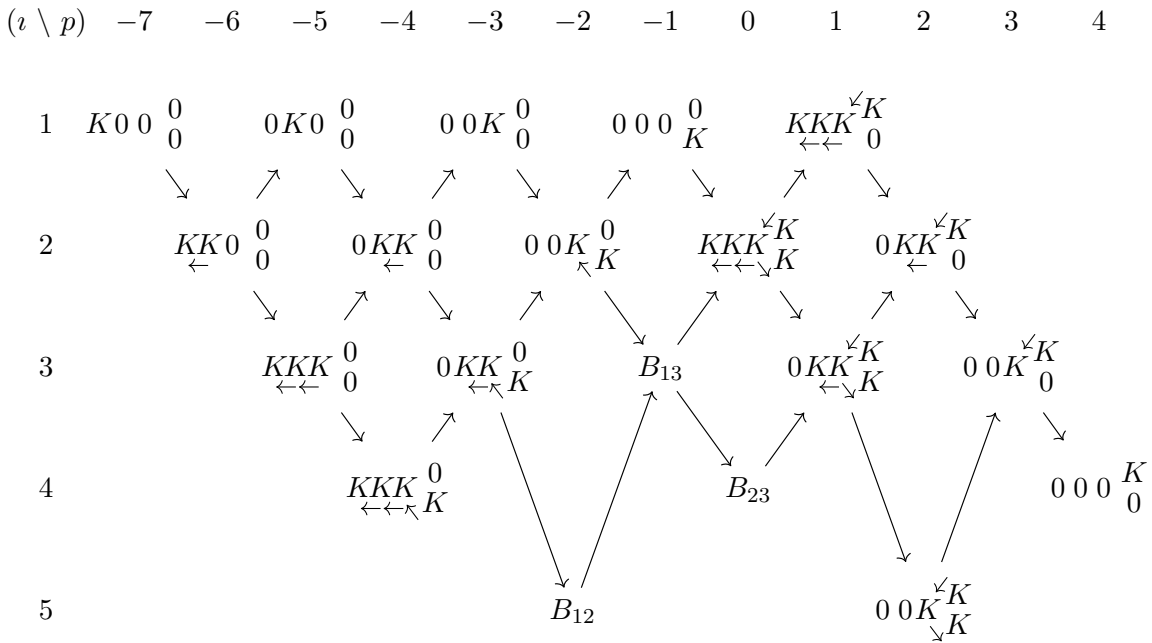
\[\begin{tikzcd}[column sep={3em,between origins},row sep={1em}]
	{(\im \ \backslash \ p)} & {-7} & {-6} & {-5} & {-4} & {-3} & {-2} & {-1} & 0 & 1 & 2 & 3 & 4 \\
	1 & {\begin{tikzpicture}

		\node at (0,0) {$K$};
		\node at (0.3,0) {$0$};
		\node at (0.6,0) {$0$};
		\node at (1,0.2) {$0$};
		\node at (1,-0.2) {$0$};
		
		\end{tikzpicture}} && {\begin{tikzpicture}

		\node at (0,0) {$0$};
		\node at (0.3,0) {$K$};
		\node at (0.6,0) {$0$};
		\node at (1,0.2) {$0$};
		\node at (1,-0.2) {$0$};
			
		\end{tikzpicture}} && {\begin{tikzpicture}

		\node at (0,0) {$0$};
		\node at (0.3,0) {$0$};
		\node at (0.6,0) {$K$};
		\node at (1,0.2) {$0$};
		\node at (1,-0.2) {$0$};
				
		\end{tikzpicture}} && {\begin{tikzpicture}

		\node at (0,0) {$0$};
		\node at (0.3,0) {$0$};
		\node at (0.6,0) {$0$};
		\node at (1,0.2) {$0$};
		\node at (1,-0.2) {$K$};
		
		\end{tikzpicture}} && {\begin{tikzpicture}

		\node at (0,0) {$K$};
		\node at (0.3,0) {$K$};
		\node at (0.6,0) {$K$};
		\node at (1,0.2) {$K$};
		\node at (1,-0.2) {$0$};
			
		\draw[->] (0.28,-0.15) -- (0.02,-0.15);
		\draw[->] (0.58,-0.15) -- (0.32,-0.15);
		\draw[->] (0.85,0.55) -- (0.7,0.35);
			
		\end{tikzpicture}} \\
	2 && {\begin{tikzpicture}

		\node at (0,0) {$K$};
		\node at (0.3,0) {$K$};
		\node at (0.6,0) {$0$};
		\node at (1,0.2) {$0$};
		\node at (1,-0.2) {$0$};
		
		\draw[->] (0.28,-0.15) -- (0.02,-0.15);
		
		\end{tikzpicture}} && {\begin{tikzpicture}

		\node at (0,0) {$0$};
		\node at (0.3,0) {$K$};
		\node at (0.6,0) {$K$};
		\node at (1,0.2) {$0$};
		\node at (1,-0.2) {$0$};
			
		\draw[->] (0.58,-0.15) -- (0.32,-0.15);
			
		\end{tikzpicture}} && {\begin{tikzpicture}

		\node at (0,0) {$0$};
		\node at (0.3,0) {$0$};
		\node at (0.6,0) {$K$};
		\node at (1,0.2) {$0$};
		\node at (1,-0.2) {$K$};
			
		\draw[->] (0.775,-0.3) -- (0.625,-0.1);
			
		\end{tikzpicture}} && {\begin{tikzpicture}

		\node at (0,0) {$K$};
		\node at (0.3,0) {$K$};
		\node at (0.6,0) {$K$};
		\node at (1,0.2) {$K$};
		\node at (1,-0.2) {$K$};
		
		\draw[->] (0.28,-0.15) -- (0.02,-0.15);
		\draw[->] (0.58,-0.15) -- (0.32,-0.15);
		\draw[->] (0.85,0.55) -- (0.7,0.35);
		\draw[->] (0.625,-0.1) -- (0.775,-0.3);
		
		\end{tikzpicture}} && {\begin{tikzpicture}

		\node at (0,0) {$0$};
		\node at (0.3,0) {$K$};
		\node at (0.6,0) {$K$};
		\node at (1,0.2) {$K$};
		\node at (1,-0.2) {$0$};
			
		\draw[->] (0.58,-0.15) -- (0.32,-0.15);
		\draw[->] (0.85,0.55) -- (0.7,0.35);
			
			\end{tikzpicture}} \\
	3 &&& {\begin{tikzpicture}

		\node at (0,0) {$K$};
		\node at (0.3,0) {$K$};
		\node at (0.6,0) {$K$};
		\node at (1,0.2) {$0$};
		\node at (1,-0.2) {$0$};
		
		\draw[->] (0.28,-0.15) -- (0.02,-0.15);
		\draw[->] (0.58,-0.15) -- (0.32,-0.15);
		
		\end{tikzpicture}} && {\begin{tikzpicture}

		\node at (0,0) {$0$};
		\node at (0.3,0) {$K$};
		\node at (0.6,0) {$K$};
		\node at (1,0.2) {$0$};
		\node at (1,-0.2) {$K$};
			
		\draw[->] (0.58,-0.15) -- (0.32,-0.15);
		\draw[->] (0.775,-0.3) -- (0.625,-0.1);
			
		\end{tikzpicture}} && B_{13} && {\begin{tikzpicture}

		\node at (0,0) {$0$};
		\node at (0.3,0) {$K$};
		\node at (0.6,0) {$K$};
		\node at (1,0.2) {$K$};
		\node at (1,-0.2) {$K$};
			
		\draw[->] (0.58,-0.15) -- (0.32,-0.15);
		\draw[->] (0.85,0.55) -- (0.7,0.35);
		\draw[->] (0.625,-0.1) -- (0.775,-0.3);
			
		\end{tikzpicture}} && {\begin{tikzpicture}

		\node at (0,0) {$0$};
		\node at (0.3,0) {$0$};
		\node at (0.6,0) {$K$};
		\node at (1,0.2) {$K$};
		\node at (1,-0.2) {$0$};
				
		\draw[->] (0.85,0.55) -- (0.7,0.35);
				
		\end{tikzpicture}} \\
	4 &&&& {\begin{tikzpicture}

		\node at (0,0) {$K$};
		\node at (0.3,0) {$K$};
		\node at (0.6,0) {$K$};
		\node at (1,0.2) {$0$};
		\node at (1,-0.2) {$K$};
		
		\draw[->] (0.28,-0.15) -- (0.02,-0.15);
		\draw[->] (0.58,-0.15) -- (0.32,-0.15);
		\draw[->] (0.775,-0.3) -- (0.625,-0.1);
		
		\end{tikzpicture}} &&&& B_{23} &&&& 	{\begin{tikzpicture}

		\node at (0,0) {$0$};
		\node at (0.3,0) {$0$};
		\node at (0.6,0) {$0$};
		\node at (1,0.2) {$K$};
		\node at (1,-0.2) {$0$};
		
		\end{tikzpicture}}\\
	5 &&&&&& B_{12} &&&& {\begin{tikzpicture}

		\node at (0,0) {$0$};
		\node at (0.3,0) {$0$};
		\node at (0.6,0) {$K$};
		\node at (1,0.2) {$K$};
		\node at (1,-0.2) {$K$};
		
		\draw[->] (0.85,0.55) -- (0.7,0.35);
		\draw[->] (0.625,-0.1) -- (0.775,-0.3);
		
		\end{tikzpicture}}
	\arrow[from=2-2, to=3-3]
	\arrow[from=2-4, to=3-5]
	\arrow[from=2-6, to=3-7]
	\arrow[from=2-8, to=3-9]
	\arrow[from=2-10, to=3-11]
	\arrow[from=3-3, to=2-4]
	\arrow[from=3-3, to=4-4]
	\arrow[from=3-5, to=2-6]
	\arrow[from=3-5, to=4-6]
	\arrow[from=3-7, to=2-8]
	\arrow[from=3-7, to=4-8]
	\arrow[from=3-9, to=2-10]
	\arrow[from=3-9, to=4-10]
	\arrow[from=3-11, to=4-12]
	\arrow[from=4-4, to=3-5]
	\arrow[from=4-4, to=5-5]
	\arrow[from=4-6, to=3-7]
	\arrow[from=4-6, to=6-7]
	\arrow[from=4-8, to=3-9]
	\arrow[from=4-8, to=5-9]
	\arrow[from=4-10, to=3-11]
	\arrow[from=4-10, to=6-11]
	\arrow[from=4-12, to=5-13]
	\arrow[from=5-5, to=4-6]
	\arrow[from=5-9, to=4-10]
	\arrow[from=6-7, to=4-8]
	\arrow[from=6-11, to=4-12]
\end{tikzcd}\]
The objects $B_{12}$, $B_{13}$ and $B_{23}$ are given by the following representations:
\begin{align*}
B_{12} &= \begin{tikzcd}[ampersand replacement=\&, baseline={([yshift=-0.4em]current bounding box.center)}]
	\&\&\& K \\
	K \& {K^2} \& {K^2} \\
	\&\&\& K
	\arrow["{i_2}", bend left=20, from=1-4, to=2-3]
	\arrow["{i_1}", bend left=20, from=2-1, to=2-2]
	\arrow["{\pi_2}", bend left=20, from=2-2, to=2-1]
	\arrow["{\begin{bmatrix}0 & 1\\0 & 0\end{bmatrix}}", bend left=20, from=2-2, to=2-3]
	\arrow["0", bend left=20, from=2-3, to=1-4]
	\arrow["\id", bend left=20, from=2-3, to=2-2]
	\arrow["{\pi_2}", bend left=20, from=2-3, to=3-4]
	\arrow["{i_1}", bend left=20, from=3-4, to=2-3]
\end{tikzcd}\\
B_{13} &= \begin{tikzcd}[ampersand replacement=\&, baseline={([yshift=-0.4em]current bounding box.center)}]
	\&\&\& K \\
	K \& K \& {K^2} \\
	\&\&\& K
	\arrow["{i_2}", bend left=20, from=1-4, to=2-3]
	\arrow["0", bend left=20, from=2-1, to=2-2]
	\arrow["\id", bend left=20, from=2-2, to=2-1]
	\arrow["{i_1}", bend left=20, from=2-2, to=2-3]
	\arrow["0", bend left=20, from=2-3, to=1-4]
	\arrow["{\pi_2}", bend left=20, from=2-3, to=2-2]
	\arrow["{\pi_2}", bend left=20, from=2-3, to=3-4]
	\arrow["{i_1}", bend left=20, from=3-4, to=2-3]
\end{tikzcd}\\
B_{23} &= \begin{tikzcd}[ampersand replacement=\&, baseline={([yshift=-0.4em]current bounding box.center)}]
	\&\&\& K \\
	0 \& K \& {K^2} \\
	\&\&\& K
	\arrow["{i_2}", bend left=20, from=1-4, to=2-3]
	\arrow[bend left=20, from=2-1, to=2-2]
	\arrow[bend left=20, from=2-2, to=2-1]
	\arrow["{i_1}", bend left=20, from=2-2, to=2-3]
	\arrow["0", bend left=20, from=2-3, to=1-4]
	\arrow["{\pi_2}", bend left=20, from=2-3, to=2-2]
	\arrow["{\pi_2}", bend left=20, from=2-3, to=3-4]
	\arrow["{i_1}", bend left=20, from=3-4, to=2-3]
\end{tikzcd}\\
\end{align*}
\end{example}

\subsection{The repetition category and tilting functors} Let $[\bm{i}]$ be a commutation class of reduced words for the longest element $w_0$. We define the \emph{repetition category} $\R([\bm{i}])$ of $[\bm{i}]$ as the full additive subcategory of $\pvd(\Pi_Q)$ generated by the objects of the form $\Sigma^kM$ for $k \in \Z$ and $M \in \C([\bm{i}])$. Denote by $\ind(\R([\bm{i}]))$ a set of representatives for the isomorphism classes of indecomposable objects of $\R([\bm{i}])$.

If $M \in \ind(\R([\bm{i}]))$ has cohomology concentrated in degree $k$, we set $c_{[\bm{i}]}(M) = ([M],-k) \in \widehat{\mathsf{R}}$, where $[M]$ denotes the class of $M$ in the Grothendieck group $K_0(\pvd(\Pi_Q))$. Here, we use the identification of $K_0(\pvd(\Pi_Q))$ with the root lattice of $\Delta$, and observe that $(-1)^{-k}[M]$ is indeed a positive root. We refer to $c_{[\bm{i}]}(M)$ as the \emph{coordinate of $M$ in $\widehat{\Upsilon}_{[\bm{i}]}$} and its residue as the \emph{residue of $M$}. We remark that the function $c_{[\bm{i}]}: \ind(\R([\bm{i}])) \longrightarrow \widehat{\mathsf{R}}$ is a bijection. Moreover, the suspension functor $\Sigma$ sends the indecomposable object whose coordinate is $(\alpha,k) \in \widehat{\mathsf{R}}$ to the indecomposable object whose coordinate is $(-\alpha,k+1)$.

\begin{prop}\label{prop:tilting for the repetitive category, general case}
Let $[\bm{i}]$ be a commutation class of reduced words for $w_0$. If $(i_1,\dots,i_k)$ is a source sequence for $[\bm{i}]$, then the equivalence
\[
T_{i_k}^{-1}T_{i_{k-1}}^{-1}\dotsb T_{i_1}^{-1}: \pvd(\Pi_{Q}) \longrightarrow \pvd(\Pi_{Q})
\]
restricts to an equivalence from $\R([\bm{i}])$ to $\R([\bm{j}])$, where $[\bm{j}] = r_{i_k}r_{i_{k-1}}\dotsb r_{i_1}[\bm{i}]$. Additionally, the following square is commutative:
\[\begin{tikzcd}[column sep = 6em]
	{\ind(\R([\bm{i}]))} & {\ind(\R([\bm{j}]))} \\
	{(\widehat{\Upsilon}_{[\bm{i}]})_0} & {(\widehat{\Upsilon}_{[\bm{j}]})_0}
	\arrow["{T_{i_k}^{-1}T_{i_{k-1}}^{-1}\dotsb T_{i_1}^{-1}}", from=1-1, to=1-2]
	\arrow["{c_{[\bm{i}]}}"', from=1-1, to=2-1]
	\arrow["{c_{[\bm{j}]}}", from=1-2, to=2-2]
	\arrow[from=2-1, to=2-2]
\end{tikzcd}\]
where the bottom map comes from the isomorphism $\widehat{\Upsilon}_{[\bm{i}]} \longrightarrow \widehat{\Upsilon}_{[\bm{j}]}$ induced by the source sequence.
\end{prop}

\begin{proof}
Without loss of generality, suppose that $k = 1$. The first part can be obtained by applying Proposition \ref{prop:equivalence under reflection, general case} and using that $T_{i_1}^{-1}S_{i_1} \cong \Sigma S_{i_1}$. The second part follows from the explicit description of the bottom map given at the end of Section \ref{subsection:combinatorial reflection functors} and the fact that $T_{i_1}^{-1}$ acts as the simple reflection $s_{i_1}$ on $K_0(\pvd(\Pi_Q))$.
\end{proof}

An equivalence $T: \R([\bm{i}]) \longrightarrow \R([\bm{j}])$ is a \emph{tilting functor} if it is a composition of equivalences as the ones appearing in Proposition \ref{prop:tilting for the repetitive category, general case} or their inverses. Note that a tilting functor is compatible with the coordinate functions in the sense of Proposition \ref{prop:tilting for the repetitive category, general case}.

\subsection{The (combinatorially) derived category} Let $[\bm{i}]$ be a commutation class of reduced words for the longest element $w_0$. We will now define an ideal $\I_{[\bm{i}]}$ of the repetition category $\R([\bm{i}])$. For indecomposable objects $M,N \in \R([\bm{i}])$, we set $\I_{[\bm{i}]}(M,N) = 0$ if there is a tilting functor $T: \R([\bm{i}]) \longrightarrow \R([\bm{j}])$ such that $T(M)$ and $T(N)$ are both concentrated in degree $0$, that is, $T(M)$ and $T(N)$ belong to $\C([\bm{j}])$. Otherwise, we define
\[
\I_{[\bm{i}]}(M,N) = \Hom_{\R([\bm{i}])}(M,N).
\]
We extend $\I_{[\bm{i}]}$ to all objects by additivity. Let us prove that $\I_{[\bm{i}]}$ is indeed an ideal.

\begin{lemma}\label{lemma:morphisms in the repetition category follow paths}
Suppose $\Delta \neq \mathsf{A}_1$. Let $[\bm{i}]$ be a commutation class of reduced words for $w_0$ and take $M,M' \in \ind(\R([\bm{i}]))$. If there is a nonzero morphism $M \longrightarrow M'$, then there is a path from $c_{[\bm{i}]}(M)$ to $c_{[\bm{i}]}(M')$ in $\widehat{\Upsilon}_{[\bm{i}]}$.
\end{lemma}

\begin{proof}
By Proposition \ref{prop:tilting for the repetitive category, general case}, after applying a suitable tilting functor, we may replace $[\bm{i}]$ by any other commutation class in the same $r$-cluster point. Hence, if we write $\bm{i} = (i_1,\dots,i_N)$, we can assume that $c_{[\bm{i}]}(M') = \widehat{\beta}^{\bm{i}}_1 = (\alpha_{i_1},0)$ and that $i_1$ is the unique source of $[\bm{i}]$. In this case, $\alpha_{i_1}$ is the unique sink of $\Upsilon_{[\bm{i}]}$. Since $M'$ is concentrated in degree $0$, the existence of a nonzero morphism $M \longrightarrow M'$ implies that $M$ is concentrated in a nonnegative degree, that is, $c_{[\bm{i}]}(M) = (\alpha,k)$ for some $k \leq 0$. Thus, by the segment property of $\widehat{\Upsilon}_{[\bm{i}]}$, there is a path from $c_{[\bm{i}]}(M)$ to $(\beta,0)$ for some $\beta \in \mathsf{R}^+$. Since there is a path from $\beta$ to $\alpha_{i_1}$ in $\Upsilon_{[\bm{i}]}$, there is a path from $(\beta,0)$ to $c_{[\bm{i}]}(M') = (\alpha_{i_1},0)$ in $\widehat{\Upsilon}_{[\bm{i}]}$, concluding the proof.
\end{proof}

\begin{prop}\label{prop:I is an ideal of the repetition category}
For every commutation class $[\bm{i}]$ of reduced words for $w_0$, $\I_{[\bm{i}]}$ is an ideal of the category $\R([\bm{i}])$. Moreover, it is preserved by the functor $\Sigma$, and a tilting functor $\R([\bm{i}]) \longrightarrow \R([\bm{j}])$ sends $\I_{[\bm{i}]}$ to $\I_{[\bm{j}]}$.
\end{prop}

\begin{proof}
We give a proof assuming $\Delta \neq \mathsf{A}_1$, the other case being easy. Let $\varphi: M_1 \longrightarrow M_2$ and $\psi: M_2 \longrightarrow M_3$ be morphisms between indecomposable objects of $\R([\bm{i}])$. If $\psi\varphi \not\in \I_{[\bm{i}]}(M_1,M_3)$, then we must have $\psi\varphi \neq 0$ and there must be a tilting functor $T: \R([\bm{i}]) \longrightarrow \R([\bm{j}])$ such that $T(M_1),T(M_3) \in \C([\bm{j}])$. In particular, $\varphi$ and $\psi$ are nonzero and, by Lemma \ref{lemma:morphisms in the repetition category follow paths}, there is a path in $\widehat{\Upsilon}_{[\bm{j}]}$ from $c_{[\bm{j}]}(T(M_1))$ to $c_{[\bm{j}]}(T(M_3))$ passing by $c_{[\bm{j}]}(T(M_2))$. Since $\Upsilon_{[\bm{j}]}$ is convex in $\widehat{\Upsilon}_{[\bm{j}]}$, we deduce that $T(M_2) \in \C([\bm{j}])$. Therefore, we conclude that $\varphi \not\in \I_{[\bm{i}]}(M_1,M_2) = 0$ and $\psi \not\in \I_{[\bm{i}]}(M_2,M_3) = 0$. This proves that $\I_{[\bm{i}]}$ is an ideal of $\R([\bm{i}])$. The second part of the proposition is clear from the definition of $\I_{[\bm{i}]}$ and Remark \ref{rmk:the longest twist acts as twisted shift}.
\end{proof}

For a commutation class $[\bm{i}]$ of reduced words for $w_0$, we define its \emph{combinatorially derived category} (or \emph{c-derived category} for short) as the quotient $\D([\bm{i}])$ of $\R([\bm{i}])$ by the ideal $\I_{[\bm{i}]}$. It is still an additive $K$-linear category and has the same objects as $\R([\bm{i}])$. One can check that $\D([\bm{i}])$ and $\R([\bm{i}])$ also have the same indecomposable objects, so we will alternatively write $\ind(\D([\bm{i}]))$ for $\ind(\R([\bm{i}]))$. For $M,N \in \ind(\D([\bm{i}]))$, we have
\[
\Hom_{\D([\bm{i}])}(M,N) \cong \Hom_{\pvd(\Pi_{Q})}(M,N)
\]
if there is a tilting functor $T: \R([\bm{i}]) \longrightarrow \R([\bm{j}])$ such that $T(M)$ and $T(N)$ belong to $\C([\bm{j}])$. Otherwise, the $\Hom$-space from $M$ to $N$ in $\D([\bm{i}])$ is zero.

By the definition of $\I_{[\bm{i}]}$, the composition of the inclusion $\C([\bm{i}]) \subset \R([\bm{i}])$ with the quotient functor to $\D([\bm{i}])$ is a fully faithful functor. In this way, we shall identify $\C([\bm{i}])$ as a full subcategory of $\D([\bm{i}])$. By Proposition \ref{prop:I is an ideal of the repetition category}, the functor $\Sigma$ descends to an autoequivalence of $\D([\bm{i}])$, and any tilting functor $\R([\bm{i}]) \longrightarrow \R([\bm{j}])$ induces an equivalence $\D([\bm{i}]) \longrightarrow \D([\bm{j}])$ which we will also call a tilting functor.

The following result justifies the chosen construction for the c-derived category.

\begin{prop}\label{prop:derived category for adapted word}
Let $\bm{i}$ be a reduced word for $w_0$ that is a source sequence for some orientation $Q$ of $\Delta$. We have an equivalence of $K$-linear categories
\[
\D^b(\modcat KQ) \longrightarrow \D([\bm{i}])
\]
that commutes with the suspension functors. Moreover, tilting functors on $\D([\bm{i}])$ naturally correspond to compositions of derived reflection functors on $\D^b(\modcat KQ)$.
\end{prop}

\begin{proof}
Since every indecomposable object of $\D^b(\modcat KQ)$ is a shift of some indecomposable $KQ$-module, the essential image of the restriction functor $\D^b(\modcat KQ) \longrightarrow \pvd(\Pi_{Q^{\circ}})$ is the repetition category $\R([\bm{i}])$ by Proposition \ref{prop:concentrated in zero for adapted word}. If $\D^b(\modcat KQ) \longrightarrow \R([\bm{i}])$ is the functor obtained by restricting the codomain, it commutes with the suspension functors and, by Remark \ref{rem:spherical twist gives reflection functor on the classical case}, it also has the last property stated above (if we replace $\D([\bm{i}])$ by $\R([\bm{i}])$). Hence, we only need to prove that its composition with the quotient functor $\R([\bm{i}]) \longrightarrow \D([\bm{i}])$ yields an equivalence.

Let $M$ and $N$ be two indecomposable objects of $\D^b(\modcat KQ)$. We have to show that the map
\begin{equation}\tag{$*$}\label{eq:fully faithful}
\Hom_{\D^b(\modcat KQ)}(M,N) \longrightarrow \Hom_{\R([\bm{i}])}(M,N)
\end{equation}
is a bijection if there is a tilting functor $T:\R([\bm{i}]) \longrightarrow \R([\bm{j}])$ such that $T(M)$ and $T(N)$ are concentrated in degree $0$, and that
\[
\Hom_{\D^b(\modcat KQ)}(M,N) = 0	
\]
otherwise. We argue in each case separately.

In the first case, Remark \ref{rem:spherical twist gives reflection functor on the classical case} gives the following diagram of functors (which is commutative up to a natural isomorphism):
\[\begin{tikzcd}
	{\D^b(\modcat KQ)} & {\R([\bm{i}])} \\
	{\D^b(\modcat KQ')} & {\R([\bm{j}])}
	\arrow[from=1-1, to=1-2]
	\arrow["T'"', from=1-1, to=2-1]
	\arrow["T", from=1-2, to=2-2]
	\arrow[from=2-1, to=2-2]
\end{tikzcd}\]
The vertical functor $T'$ on the left is the composition of the derived reflection functors (as the ones appearing in Remark \ref{rem:spherical twist gives reflection functor on the classical case}) corresponding to the sequence of sources/sinks that defines the tilting functor $T$. By Proposition \ref{prop:concentrated in zero for adapted word}, the horizontal functor on the bottom is fully faithful on the canonical heart of $\D^b(\modcat KQ')$. Since $T'(M)$ and $T'(N)$ belong to this heart and the vertical functors are equivalences, we conclude that the map (\ref{eq:fully faithful}) is indeed a bijection.

In the second case, we can find a composition of reflection functors $T'$ from $\D^b(\modcat KQ)$ to $\D^b(\modcat KQ')$ such that $T'(M)$ is an indecomposable projective $KQ'$-module. By hypothesis, $T'(N)$ is not concentrated in degree $0$, so there are no morphisms from $T'(M)$ to $T'(N)$ in $\D^b(\modcat KQ')$. Since $T'$ is an equivalence, there are no morphisms from $M$ to $N$ in $\D^b(\modcat KQ)$, as desired.
\end{proof}

\section{Irreducible morphisms and Gabriel quivers}\label{section:morphisms}

Let $[\bm{i}]$ be a commutation class of reduced words for $w \in \mathsf{W}$. By construction, the set of indecomposable objects of $\C([\bm{i}])$ is in bijection with $\mathsf{R}^+(w)$, which is the vertex set of the combinatorial AR quiver $\Upsilon_{[\bm{i}]}$. We will show that each arrow of $\Upsilon_{[\bm{i}]}$ gives rise to an essentially unique irreducible morphism in $\C([\bm{i}])$; hence, $\Upsilon_{[\bm{i}]}$ can be seen as a subquiver of the Gabriel quiver of $\C([\bm{i}])$. At the end of the section, we will consider the case $w = w_0$ and prove similar results for the categories $\R([\bm{i}])$ and $\D([\bm{i}])$.

If there is an arrow $\beta \longrightarrow \alpha$ in $\Upsilon_{[\bm{i}]}$ between two positive roots in $\mathsf{R}^+(w)$, then there is $\bm{j} \in [\bm{i}]$ such that $\alpha$ and $\beta$ are consecutive with respect to the order $<_{\bm{j}}$ (this follows from Theorem \ref{thm:compatible reading gives reduced word}). In particular, if $\alpha$ and $\beta$ have residues $i$ and $j$ in $\Delta_0$, then
\[
M^{[\bm{i}]}_{\alpha} \cong T(S_i) \quad \textrm{and} \quad M^{[\bm{i}]}_{\beta} \cong T(T_i(S_j)),
\]
where $T$ is a composition of spherical twists. Since $i \sim j$, $\Ext^n_{\Pi_Q}(S_i,S_j)$ vanishes for all $n \neq 1$ and $\Ext^1_{\Pi_Q}(S_i,S_j)$ is one-dimensional. Thus, the triangle (\ref{eq:triangle for spherical twist}) in the definition of the spherical twist $T_i(S_j)$ becomes
\[
\begin{tikzcd}
    {\Sigma^{-1}S_i} & {S_j} & {T_i(S_j)} & {S_i}
	\arrow[from=1-1, to=1-2]
	\arrow[from=1-2, to=1-3]
	\arrow[from=1-3, to=1-4]
\end{tikzcd}
\]
and $T_i(S_j)$ must be the unique nonsplit extension of $S_i$ by $S_j$. We deduce that the space $\Hom_{\pvd(\Pi_{Q})}(T_i(S_j), S_i)$ is one-dimensional, and so is $\Hom_{\pvd(\Pi_{Q})}(M^{[\bm{i}]}_{\beta}, M^{[\bm{i}]}_{\alpha})$ as $T$ is an equivalence of categories. In other words, given an arrow $\beta \longrightarrow \alpha$ in $\Upsilon_{[\bm{i}]}$, there is a unique nonzero map (up to multiplication by a nonzero scalar) from $M^{[\bm{i}]}_{\beta}$ to $M^{[\bm{i}]}_{\alpha}$ in $\C({[\bm{i}]})$.

This discussion motivates the definition of the following relation on $\ind([\bm{i}])$. For two indecomposable objects $M,N \in \ind([\bm{i}])$, consider the smallest transitive relation $\preceq_{\C([\bm{i}])}$ satisfying
\[
\Hom_{\C({[\bm{i}]})}(N,M) \neq 0 \implies M \preceq_{\C([\bm{i}])} N.
\]
We call it the \emph{$\Hom$-order of} $\C([\bm{i}])$. Our next goal is to prove that $\preceq_{\C([\bm{i}])}$ is indeed a partial order and that the map $\alpha \mapsto M^{[\bm{i}]}_{\alpha}$ induces an isomorphism between the posets $(\mathsf{R}^+(w), \preceq_{[\bm{i}]})$ and $(\ind([\bm{i}]), \preceq_{\C([\bm{i}])})$.

\begin{lemma}\label{lemma:not relatable in combinatorial AR quiver implies zero Hom}
Let $\alpha,\beta \in \mathsf{R}^+(w)$. If $\alpha \not\preceq_{[\bm{i}]} \beta$, then
\[
\Hom_{\C({[\bm{i}]})}(M^{[\bm{i}]}_{\beta}, M^{[\bm{i}]}_{\alpha}) = 0.
\]
\end{lemma}

\begin{proof}
By the definition of $\preceq_{[\bm{i}]}$, there is $\bm{j} = (j_1,\dots,j_t)$ in $[\bm{i}]$ such that $\beta <_{\bm{j}} \alpha$. Let $1 \leq k < l \leq t$ be such that $\beta = \beta^{\bm{j}}_k$ and $\alpha = \beta^{\bm{j}}_l$, so that we have $M^{[\bm{i}]}_{\beta} \cong M^{\bm{j}}_k$ and $M^{[\bm{i}]}_{\alpha} \cong M^{\bm{j}}_l$. Thus,
\[
\Hom_{\C([\bm{i}])}(M^{[\bm{i}]}_{\beta}, M^{[\bm{i}]}_{\alpha}) \cong \Hom_{\C([\bm{i}])}(M^{\bm{j}}_k, M^{\bm{j}}_l) = \Hom_{\pvd(\Pi_{Q})}(T_{j_1}\dotsb T_{j_{k-1}}(S_{j_k}), T_{j_1}\dotsb T_{j_{l-1}}(S_{j_l})). 
\]
Applying the functor $T_{j_k}^{-1}T_{j_{k-1}}^{-1}\dotsb T_{j_1}^{-1}$ and recalling that $k < l$, the $\Hom$-space above is isomorphic to
\[
\Hom_{\pvd(\Pi_{Q})}(T_{j_k}^{-1}(S_{j_k}), T_{j_{k+1}}\dotsb T_{j_{l-1}}(S_{j_l})).
\]
But $T_{j_k}^{-1}(S_{j_k}) \cong \Sigma S_{j_k}$ is concentrated in degree $-1$, while $T_{j_{k+1}}\dotsb T_{j_{l-1}}(S_{j_l})$ is concentrated in degree $0$ by Theorem \ref{thm:indecomposables are concentrated in zero}. We deduce that the $\Hom$-space above is zero.
\end{proof}

\begin{prop}\label{prop:categorical order agrees with combinatorial order}
For $\alpha, \beta \in \mathsf{R}^+(w)$, we have
\[
\alpha \preceq_{[\bm{i}]} \beta \iff M^{[\bm{i}]}_{\alpha} \preceq_{\C([\bm{i}])} M^{[\bm{i}]}_{\beta}.
\]
In particular, the relation $\preceq_{\C([\bm{i}])}$ is a partial order on the set $\ind([\bm{i}])$.
\end{prop}

\begin{proof}
If $\alpha \preceq_{[\bm{i}]} \beta$, then Theorem \ref{thm:combinatorial AR quiver realizes partial order} gives a path in $\Upsilon_{[\bm{i}]}$ from $\beta$ to $\alpha$. As argued before, each arrow on this path gives rise to a nonzero morphism in $\C([\bm{i}])$, so we have $M^{[\bm{i}]}_{\alpha} \preceq_{\C([\bm{i}])} M^{[\bm{i}]}_{\beta}$. Conversely, if $M^{[\bm{i}]}_{\alpha} \preceq_{\C([\bm{i}])} M^{[\bm{i}]}_{\beta}$, then there is a sequence $\alpha_1,\dots,\alpha_k \in \mathsf{R}^+(w)$ such that $\alpha_1 = \alpha$, $\alpha_k = \beta$ and
\[
\Hom_{\C([\bm{i}])}(M^{[\bm{i}]}_{\alpha_{l+1}},M^{[\bm{i}]}_{\alpha_l}) \neq 0
\]
for all $1 \leq l < k$. By Lemma \ref{lemma:not relatable in combinatorial AR quiver implies zero Hom}, we have $\alpha_l \preceq_{[\bm{i}]} \alpha_{l+1}$ for all $1 \leq l < k$, which gives $\alpha \preceq_{[\bm{i}]} \beta$ by transitivity.
\end{proof}

Together with Theorem \ref{thm:combinatorial AR quiver realizes partial order}, this result implies the following corollary.

\begin{cor}\label{cor:combinatorial AR quiver is Hasse quiver of categorical order}
The combinatorial AR quiver $\Upsilon_{[\bm{i}]}$ is isomorphic to the Hasse quiver of the $\Hom$-order of $\C([\bm{i}])$. 
\end{cor}

We now describe the relationship between $\Upsilon_{[\bm{i}]}$ and the Gabriel quiver of $\C([\bm{i}])$. Let us first recall some definitions. Let $\C$ be an additive $K$-linear Krull--Schmidt category and assume for simplicity that the radical quotient of the endomorphism algebra of any indecomposable object is isomorphic to $K$. We say that a morphism $\varphi: M \longrightarrow N$ in $\C$ is \emph{irreducible} if $\varphi$ is neither a section nor a retraction and, for any factorization $\varphi = \psi_2\psi_1$, we have that $\psi_1$ is a section or $\psi_2$ is a retraction. If $M$ and $N$ are indecomposable, then $\varphi$ is irreducible if and only if $\varphi \in \rad_{\C}(M,N) \setminus \rad_{\C}^2(M,N)$, where $\rad_{\C}$ is the radical ideal of $\C$ (see \cite[Section V.7]{AuslanderReitenSmalo}). The \emph{Gabriel quiver $\Gamma(\C)$ of} $\C$ has as a set of vertices a set of representatives for the isomorphism classes of indecomposable objects in $\C$. The number of arrows in $\Gamma(\C)$ from an indecomposable object $M$ to another indecomposable object $N$ is the dimension of the $K$-vector space $\rad_{\C}(M,N)/\rad_{\C}^{2}(M,N)$.

\begin{lemma}\label{lemma:arrows are irreducible}
Let $\alpha, \beta \in \mathsf{R}^+(w)$. If there is an arrow $\beta \longrightarrow \alpha$ in $\Upsilon_{[\bm{i}]}$, then any nonzero morphism $\varphi: M^{[\bm{i}]}_{\beta} \longrightarrow M^{[\bm{i}]}_{\alpha}$ is irreducible in $\C([\bm{i}])$.
\end{lemma}

\begin{proof}
Take a factorization $\varphi = \psi_2\psi_1$, where $\psi_1: M^{[\bm{i}]}_{\beta} \longrightarrow X$ and $\psi_2: X \longrightarrow M^{[\bm{i}]}_{\alpha}$ are morphisms in $\C([\bm{i}])$. Since $\varphi$ is nonzero, there is at least one indecomposable summand $Y$ of $X$ such that the corresponding components $\psi'_1: M^{[\bm{i}]}_{\beta} \longrightarrow Y$ and $\psi'_2: Y \longrightarrow M^{[\bm{i}]}_{\alpha}$ of $\psi_1$ and $\psi_2$ are both nonzero. We have $Y \cong M^{[\bm{i}]}_\gamma$ for some $\gamma \in \mathsf{R}^+(w)$, and the existence of these nonzero morphisms gives $\alpha \preceq_{[\bm{i}]} \gamma \preceq_{[\bm{i}]} \beta$ by Proposition \ref{prop:categorical order agrees with combinatorial order}. Since there is an arrow $\beta \longrightarrow \alpha$ in $\Upsilon_{[\bm{i}]}$, we must have $\gamma = \alpha$ or $\gamma = \beta$. This implies that either $\psi_1'$ or $\psi_2'$ is an isomorphism because the endomorphism algebra of any indecomposable object in $\C([\bm{i}])$ is one-dimensional. We deduce that $\psi_1$ is a section or $\psi_2$ is a retraction. Finally, since $M^{[\bm{i}]}_{\alpha}$ and $M^{[\bm{i}]}_{\beta}$ are nonisomorphic indecomposable objects, $\varphi$ is neither a section nor a retraction.
\end{proof}

For the next result, we will say that an arrow $\alpha: x \longrightarrow y$ in an acyclic quiver is \emph{superfluous} if there is a path from $x$ to $y$ with length strictly greater than one.

\begin{thm}\label{thm:combinatorial AR quiver is obtained from Gabriel quiver}
Let $[\bm{i}]$ be a commutation class of reduced words for $w \in \mathsf{W}$. The combinatorial AR quiver $\Upsilon_{[\bm{i}]}$ is isomorphic to the quiver obtained from the Gabriel quiver $\Gamma(\C([\bm{i}]))$ by removing all superfluous arrows.
\end{thm}

\begin{proof}
By Lemma \ref{lemma:arrows are irreducible}, $\Upsilon_{[\bm{i}]}$ can be seen as a subquiver of $\Gamma(\C([\bm{i}]))$. It contains all vertices. On the other hand, since $\Delta$ is an acyclic graph, $\Upsilon_{[\bm{i}]}$ does not have superfluous arrows by construction. Given an arrow $\varphi: M_{\alpha}^{[\bm{i}]} \longrightarrow M_{\beta}^{[\bm{i}]}$ of $\Gamma(\C([\bm{i}]))$ which is not superfluous, let us show that it comes from an arrow in $\Upsilon_{[\bm{i}]}$. First, we have $\alpha \neq \beta$ as any nontrivial endomorphism of an indecomposable object in $\C([\bm{i}])$ is invertible. By Proposition \ref{prop:categorical order agrees with combinatorial order}, we have $\beta \preceq_{[\bm{i}]} \alpha$ and, by Theorem \ref{thm:combinatorial AR quiver realizes partial order}, there is a path from $\alpha$ to $\beta$ in $\Upsilon_{[\bm{i}]}$. Any such path has to be of length one because, otherwise, we would get by Lemma \ref{lemma:arrows are irreducible} a path in $\Gamma(\C([\bm{i}]))$ of length greater than one from $M_{\alpha}^{[\bm{i}]}$ to $M_{\beta}^{[\bm{i}]}$, contradicting the hypothesis that $\varphi$ is not superfluous. In other words, there is an edge in $\Upsilon_{[\bm{i}]}$ from $\alpha$ to $\beta$. It corresponds to an edge $\psi: M_{\alpha}^{[\bm{i}]} \longrightarrow M_{\beta}^{[\bm{i}]}$ in $\Gamma(\C([\bm{i}]))$. Since $\Hom_{\C([\bm{i}])}(M_{\alpha}^{[\bm{i}]}, M_{\beta}^{[\bm{i}]})$ is one-dimensional, there is at most one arrow from $M_{\alpha}^{[\bm{i}]}$ to $M_{\beta}^{[\bm{i}]}$ in $\Gamma(\C([\bm{i}]))$; thus, $\varphi = \psi$ and the proof is completed.
\end{proof}

Suppose now $w = w_0$ is the longest element. Let $\varphi: (\alpha,k) \longrightarrow (\beta,l)$ be an arrow in $\widehat{\Upsilon}_{[\bm{i}]}$. By applying combinatorial reflection functors to $[\bm{i}]$, we can find another commutation class $[\bm{j}]$ and an isomorphism $\widehat{\Upsilon}_{[\bm{i}]} \longrightarrow \widehat{\Upsilon}_{[\bm{j}]}$ that sends $\varphi$ to an arrow in $\Upsilon_{[\bm{j}]}$. Using the previous discussion and Proposition \ref{prop:tilting for the repetitive category, general case}, we deduce that
\[
\dim_K\Hom_{\R([\bm{i}])}(M,N) = \dim_K\Hom_{\D([\bm{i}])}(M,N) = 1,
\]
where $M$ and $N$ are the indecomposable objects in $\ind(\R([\bm{i}]))$ with coordinates $(\alpha,k)$ and $(\beta,l)$, respectively. Therefore, as before, any arrow in $\widehat{\Upsilon}_{[\bm{i}]}$ gives rise to a unique nonzero map in $\R([\bm{i}])$ and in $\D([\bm{i}])$ (up to multiplication by a nonzero scalar) between the corresponding indecomposable objects.

We define the $\Hom$-order of $\R([\bm{i}])$ and $\D([\bm{i}])$ in the same way as we did it for $\C([\bm{i}])$. These are relations $\preceq_{\R([\bm{i}])}$ and $\preceq_{\D([\bm{i}])}$ on $\ind(\R(\bm{i}))$.

\begin{prop}
Suppose $\Delta \neq \mathsf{A}_1$. If $M,N \in \ind(\R([\bm{i}]))$, then
\[
N \preceq_{\R([\bm{i}])} M \iff \textrm{there is a path from } c_{[\bm{i}]}(M) \textrm{ to } c_{[\bm{i}]}(N) \textrm{ in } \widehat{\Upsilon}_{[\bm{i}]} \iff N \preceq_{\D([\bm{i}])} M.
\]
In particular, the relations $\preceq_{\R([\bm{i}])}$ and $\preceq_{\D([\bm{i}])}$ on the set $\ind(\R([\bm{i}]))$ coincide and are partial orders.
\end{prop}

\begin{proof}
The reasoning in the proof of Proposition \ref{prop:categorical order agrees with combinatorial order} works here after replacing Lemma \ref{lemma:not relatable in combinatorial AR quiver implies zero Hom} by Lemma \ref{lemma:morphisms in the repetition category follow paths}.
\end{proof}

After adapting the proofs of Lemma \ref{lemma:arrows are irreducible} and Theorem \ref{thm:combinatorial AR quiver is obtained from Gabriel quiver}, we obtain the following result.

\begin{thm}
Let $[\bm{i}]$ be a commutation class of reduced words for the longest element $w_0$. The combinatorial repetition quiver $\widehat{\Upsilon}_{[\bm{i}]}$ is isomorphic to the quiver obtained from the Gabriel quiver $\Gamma(\D([\bm{i}]))$ by removing all superfluous arrows. If $\Delta \neq \mathsf{A}_1$, the same result holds if we replace $\D([\bm{i}])$ by $\R([\bm{i}])$.
\end{thm}

\section{Meshes and distinguished triangles}\label{section:meshes and distinguished triangles}

We will now investigate the relationship between the triangulated structure of $\pvd(\Pi_Q)$ and the categories $\R([\bm{i}])$ and $\D([\bm{i}])$ for a commutation class $[\bm{i}]$ of reduced words for the longest element $w_0$. The main result of this section is Theorem \ref{thm:mesh gives rise to triangles}, which states that certain distinguished triangles with corners in $\R([\bm{i}])$ can be obtained by looking at the meshes of $\widehat{\Upsilon}_{[\bm{i}]}$. As a result, we give an alternative proof for the $\mathfrak{g}$-additive property of \cite{FujitaOh} (Theorem \ref{thm:g-additive property}) and generalize it to any commutation class.

As in Section \ref{subsection:meshes}, we assume that $\Delta \neq \mathsf{A}_1$.

\begin{lemma}\label{lemma:cone of arrows is in R}
Let $[\bm{i}]$ be a commutation class of reduced words for $w_0$. Suppose $\varphi: M' \longrightarrow M$ is an irreducible morphism of $\R([\bm{i}])$ corresponding to an arrow in $\widehat{\Upsilon}_{[\bm{i}]}$. Then, the cocone of $\varphi$ in $\pvd(\Pi_Q)$ is indecomposable and belongs to $\R([\bm{i}])$.
\end{lemma}

\begin{proof}
Let $\bm{i} = (i_1,\dots,i_N)$ be a reduced word in $[\bm{i}]$. After applying a tilting functor, we can suppose that $M \cong M^{\bm{i}}_1 = S_{i_1}$, $M' \cong M^{\bm{i}}_2 = T_{i_1}(S_{i_2})$ and $i_1 \sim i_2$. As argued in Section \ref{section:morphisms}, the distinguished triangle (\ref{eq:triangle for spherical twist}) in the definition of the spherical twist $T_{i_1}(S_{i_2})$ becomes
\[
\begin{tikzcd}
    {\Sigma^{-1}S_{i_1}} & {S_{i_2}} & {T_{i_1}(S_{i_2})} & {S_{i_1}}.
	\arrow[from=1-1, to=1-2]
	\arrow[from=1-2, to=1-3]
	\arrow[from=1-3, to=1-4]
\end{tikzcd}
\]
Up to multiplication by a nonzero scalar, the map $T_{i_1}(S_{i_2}) \longrightarrow S_{i_1}$ in the triangle above is $\varphi$. We conclude that the cocone of $\varphi$ in $\pvd(\Pi_Q)$ is isomorphic to $S_{i_2}$, which is indeed in $\R([\bm{i}])$ by Corollary \ref{cor:simples are in C([i])}.
\end{proof}

\begin{remark}\label{rmk:cone can belong to different commutations classes}
With the notation above, suppose we have $M \cong M^{\bm{i}}_k$ and $M' \cong M^{\bm{i}}_{k+1}$ for some $1 \leq k < N$. By keeping track of the tilting functor used to reduce the proof above, one can easily check that the cocone of any nonzero morphism $M' \longrightarrow M$ is isomorphic to $X = T_{i_1}T_{i_2}\dotsb T_{i_{k-1}}(S_{i_{k+1}})$. In particular, $X \in \R([\bm{j}])$ for any reduced word $\bm{j}$ which starts with the sequence $(i_1,\dots,i_{k-1})$ since $T_{i_1}T_{i_2}\dotsb T_{i_{k-1}}$ is a tilting functor from $\R(r_{i_{k-1}}\dotsb r_{i_1}[\bm{j}])$ to $\R([\bm{j}])$ and $S_{i_{k+1}} \in \R(r_{i_{k-1}}\dotsb r_{i_1}[\bm{j}])$.
\end{remark}

\begin{lemma}\label{lemma:triangle from braid move}
Let $\bm{i} = (i_1,\dots,i_N)$ be a reduced word for $w_0$ and suppose that there is $1 < k < N$ such that $i_{k-1} = i_{k+1} \sim i_k$. Then there is a nonsplit distinguished triangle
\[
\begin{tikzcd}
	{M^{\bm{i}}_{k+1}} & {M^{\bm{i}}_k} & {M^{\bm{i}}_{k-1}} & {\Sigma M^{\bm{i}}_{k+1}}
	\arrow[from=1-1, to=1-2]
	\arrow[from=1-2, to=1-3]
	\arrow[from=1-3, to=1-4]
\end{tikzcd}
\]
in $\pvd(\Pi_Q)$. 
\end{lemma}

\begin{proof}
By applying a tilting functor, we may assume $k = 2$. In this case, $M^{\bm{i}}_3 = T_{i_1}T_{i_2}(S_{i_1}) \cong S_{i_2}$, $M^{\bm{i}}_2 = T_{i_1}(S_{i_2})$, and $M^{\bm{i}}_1 = S_{i_1}$. Therefore, we obtain the triangle in the statement by rotating the triangle in the proof of Lemma \ref{lemma:cone of arrows is in R}. 
\end{proof}

\begin{lemma}\label{lemma:triangles descend to D if concentrated in degree zero}
Let $[\bm{i}]$ be a commutation class of reduced words for $w_0$. Suppose that there is a distinguished triangle
\[
\begin{tikzcd}
    {M} & {M'} & {M''} & {\Sigma M}
	\arrow["\varphi", from=1-1, to=1-2]
	\arrow["\varphi'", from=1-2, to=1-3]
	\arrow["\varphi''", from=1-3, to=1-4]
\end{tikzcd}
\]
in $\pvd(\Pi_Q)$ such that $M,M',M'' \in \ind([\bm{i}])$. If the three maps above are nonzero, then they do not belong to the ideal $\I_{[\bm{i}]}$ of $\R([\bm{i}])$.
\end{lemma}

\begin{proof}
By the definition of $\I_{[\bm{i}]}$ and our hypotheses, we have $\I_{[\bm{i}]}(M,M') = \I_{[\bm{i}]}(M',M'') = 0$, so $\varphi$ and $\varphi'$ do not belong to $\I_{[\bm{i}]}$. By Lemma \ref{lemma:morphisms in the repetition category follow paths}, there is a path in $\Upsilon_{[\bm{i}]}$ from $c_{[\bm{i}]}(M)$ to $c_{[\bm{i}]}(M'')$. Observe that these coordinates are distinct since no indecomposable object in $\R([\bm{i}])$ has a nonsplit self-extension. Consequently, there is a source sequence of $[\bm{i}]$ such that the induced tilting functor $T: \R([\bm{i}]) \longrightarrow \R([\bm{j}])$ satisfies $\Sigma^{-1}T(M''),T(M) \in \C([\bm{j}])$. We conclude that $\I_{[\bm{i}]}(M'',\Sigma M) = 0$, finishing the proof.
\end{proof}

Recall from Section \ref{subsection:meshes} the notion of a mesh and the related definitions. For $M \in \ind(\R([\bm{i}]))$, denote by $\M_{[\bm{i}]}(M)$ the mesh of $\widehat{\Upsilon}_{[\bm{i}]}$ at the vertex $x = c_{[\bm{i}]}(M)$. Let $s_{[\bm{i}]}(M) \in \ind(\R([\bm{i}]))$ be the indecomposable object whose coordinate is $s_{[\bm{i}]}(x)$. We also define the \emph{set of abutters $V_{[\bm{i}]}(M)$ of $M$} to be the set of all $X \in \ind(\R([\bm{i}]))$ whose coordinate is an abutter of $x$, that is, $c_{[\bm{i}]}(X) \in V_{[\bm{i}]}(x)$. An ordering $X_1,\dots,X_t$ of the elements in $V_{[\bm{i}]}(M)$ is \emph{anticompatible} if $(c_{[\bm{i}]}(X_t),\dots,c_{[\bm{i}]}(X_2),c_{[\bm{i}]}(X_1))$ is a compatible reading of $V_{[\bm{i}]}(x) \subset \widehat{\Upsilon}_{[\bm{i}]}$.

\begin{thm}\label{thm:mesh gives rise to triangles}
Let $[\bm{i}]$ be a commutation class of reduced words for $w_0$ and take an indecomposable object $M \in \R([\bm{i}])$. Let $X_1,\dots,X_t$ be an anticompatible ordering of the set of abutters $V_{[\bm{i}]}(M)$ of $M$. Then, there are indecomposable objects $Y_1 = s_{[\bm{i}]}(M),Y_2,\dots,Y_t,Y_{t+1} = \Sigma^{-1}M$ in $\R([\bm{i}])$ and distinguished triangles in $\pvd(\Pi_Q)$ of the form
\[
\begin{tikzcd}
    {Y_{k+1}} & {Y_k} & {X_k} & {\Sigma Y_{k+1}}
	\arrow[from=1-1, to=1-2]
	\arrow[from=1-2, to=1-3]
	\arrow[from=1-3, to=1-4]
\end{tikzcd}
\]
for all $1 \leq k \leq t$. Moreover, none of the morphisms in the triangles above belongs to the ideal $\I_{[\bm{i}]}$ of $\R([\bm{i}])$.
\end{thm}

\begin{proof}
Let $\M = \M_{[\bm{i}]}(M)$. By Lemma \ref{lemma:mesh is contained in combinatorial AR quiver}, we can apply a tilting functor and assume that $\M \subseteq \Upsilon_{[\bm{i}]}$. Moreover, since $\M$ is convex in $\widehat{\Upsilon}_{[\bm{i}]}$, we can also assume there is a representative $\bm{i} = (i_1,\dots,i_N)$ of $[\bm{i}]$ that starts with the residues $(i_1,\dots,i_k)$ of a compatible reading of the vertex set $\M_0$. By the hypothesis on the ordering of $V_{[\bm{i}]}(M)$, we can further assume that, if $c_{[\bm{i}]}(X_l) = \widehat{\beta}^{\bm{i}}_m$ and $c_{[\bm{i}]}(X_{l'}) = \widehat{\beta}^{\bm{i}}_{m'}$, then $l < l'$ implies $m > m'$. In particular, since $i_2 \sim i_1 = i_k \sim i_{k-1}$, we have $X_t \cong M^{\bm{i}}_2$ and $X_1 \cong M^{\bm{i}}_{k-1}$. Since $M^{\bm{i}}_k = s_{[\bm{i}]}(M) = Y_1$, we have an irreducible morphism $Y_1 \longrightarrow X_1$ corresponding to the arrow $\beta^{\bm{i}}_k \longrightarrow \beta^{\bm{i}}_{k-1}$ in $\Upsilon_{[\bm{i}]}$. By Lemma \ref{lemma:cone of arrows is in R}, it is part of a distinguished triangle
\[
\begin{tikzcd}
    {Y_2} & {Y_1} & {X_1} & {\Sigma Y_2}
	\arrow[from=1-1, to=1-2]
	\arrow[from=1-2, to=1-3]
	\arrow[from=1-3, to=1-4]
\end{tikzcd}
\]
where $Y_2 \in \R([\bm{i}])$ is indecomposable. By Lemma \ref{lemma:mesh gives a reduced word ready to apply braid move} (and assuming that $k > 3$), the sequence $(i_1,\dots,i_{k-1},i_k,i_{k-1})$ represents a reduced word, which we can extend to a reduced word $\bm{j}$ for $w_0$. By Lemma \ref{lemma:triangle from braid move}, we have $Y_2 \cong M^{\bm{j}}_{k+1}$. This implies that $Y_2$ is concentrated in degree $0$ and, by Lemma \ref{lemma:triangles descend to D if concentrated in degree zero}, the three morphisms in the triangle above do not belong to $\I_{[\bm{i}]}$. This argument completes the construction of the first triangle.

By applying a braid relation, we can transform $\bm{j}$ into a new reduced word $\bm{i'}$ that starts with the sequence $(i_1,\dots,i_{k-2},i_k,i_{k-1},i_k)$. We still have $M = S_{i_1} \in \C([\bm{i'}])$, and so, we can work with the mesh $\M' = \M_{[\bm{i'}]}(M)$. It is not hard to see that, as before, there is a reduced word in $[\bm{i'}]$ that starts with the residues of a compatible reading of $\M'_0$. Notice that $V_{[\bm{i'}]}(M) = V_{[\bm{i}]}(M) \setminus \{X_1\}$ and that $X_2,\dots,X_t$ is still an anticompatible ordering of $V_{[\bm{i'}]}(M)$. Additionally, we have
\[
s_{[\bm{i'}]}(M) = M^{\bm{i'}}_{k-1} = T_{i_1}\dotsb T_{i_{k-2}}(S_{i_k}) \cong T_{i_1}\dotsb T_{i_{k-2}}(T_{i_{k-1}}T_{i_k}(S_{i_{k-1}})) =  M^{\bm{j}}_{k+1} \cong Y_2.
\]
Therefore, we have essentially the same conditions as in the previous paragraph, and a similar argument gives an indecomposable object $Y_3 \in \C([\bm{i'}])$ and a distinguished triangle
\[
\begin{tikzcd}
    {Y_3} & {Y_2} & {X_2} & {\Sigma Y_3}.
	\arrow[from=1-1, to=1-2]
	\arrow[from=1-2, to=1-3]
	\arrow[from=1-3, to=1-4]
\end{tikzcd}
\]
Using Remark \ref{rmk:cone can belong to different commutations classes}, one can show that $Y_3$ also belongs to $\C([\bm{i}])$. Again, by Lemma \ref{lemma:triangles descend to D if concentrated in degree zero}, the three morphisms in this new triangle do not belong to $\I_{[\bm{i}]}$.

In this way, we can continue by finding new reduced words for $w_0$ for which the set of abutters of $M$ has one element less at each step. We obtain the first $t-1$ triangles from the statement along the way. At the end, we find a reduced word $\bm{i''}$ starting with the residues of a compatible reading of the mesh of $M$ such that $V_{[\bm{i''}]}(M) = \{X_t\}$ and $s_{[\bm{i''}]}(M) \cong Y_t$. In this case, $\bm{i''}$ has to start with the sequence $(i_1,i_2,i_1)$, otherwise $M$ would have more abutters. By Lemma \ref{lemma:triangle from braid move}, we get the distinguished triangle
\[
\begin{tikzcd}
    {Y_t} & {X_t} & {M} & {\Sigma Y_t},
	\arrow[from=1-1, to=1-2]
	\arrow[from=1-2, to=1-3]
	\arrow[from=1-3, to=1-4]
\end{tikzcd}
\]
which gives the final triangle of the statement after a rotation. Since the last three terms of this final triangle are concentrated in degree $0$, an analogous result to Lemma \ref{lemma:triangles descend to D if concentrated in degree zero} implies that its morphisms do not lie in $\I_{[\bm{i}]}$.
\end{proof}

\begin{example}
Let $\Q = (\mathsf{A}_5,\vee,\xi)$ be the Q-datum of type $\mathsf{B}_3$ from Example \ref{example:B3}. Let $M \in \C([\Q])$ be the indecomposable object with coordinate $(4,6)$. We highlight in the diagram below the mesh determined by $M$ and color in red its abutters.
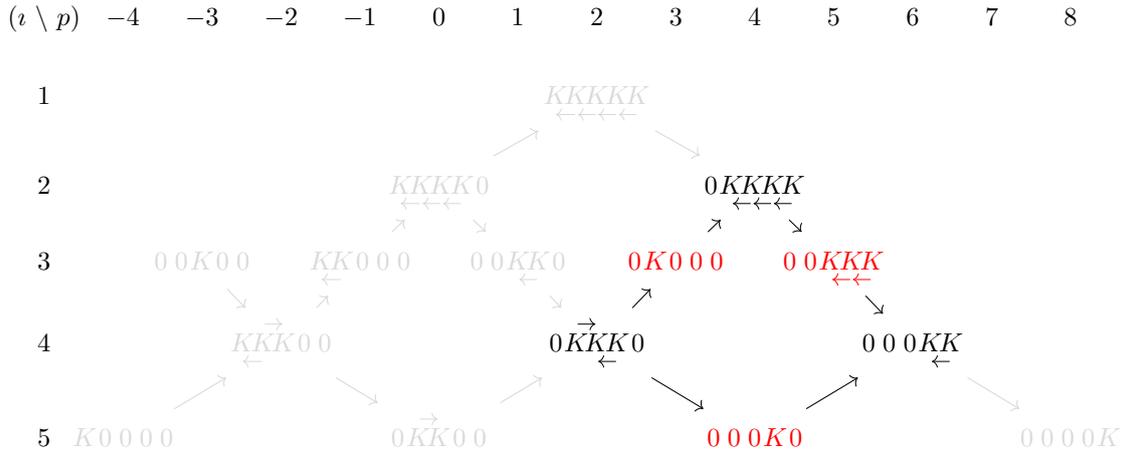
\begin{center}
	\begin{tikzpicture}
	\node[scale=0.9] (a) at (0,0){
	\begin{tikzcd}[column sep={3em,between origins},row sep={1em}]
		{(\im \ \backslash \ p)} & {-4} & {-3} & {-2} & {-1} & 0 & 1 & 2 & 3 & 4 & 5 & 6 & 7 & 8 \\
		1 &&&&&&& {\begin{tikzpicture}[color={rgb,255:red,217;green,217;blue,217}]

			\node at (0,0) {$K$};
			\node at (0.3,0) {$K$};
			\node at (0.6,0) {$K$};
			\node at (0.9,0) {$K$};
			\node at (1.2,0) {$K$};
		
			\draw[->] (0.28,-0.15) -- (0,-0.15);
			\draw[->] (0.58,-0.15) -- (0.32,-0.15);
			\draw[->] (0.88,-0.15) -- (0.62,-0.15);
			\draw[->] (1.18,-0.15) -- (0.92,-0.15);
			
		\end{tikzpicture}} \\
		2 &&&&& {\begin{tikzpicture}[color={rgb,255:red,217;green,217;blue,217}]

			\node at (0,0) {$K$};
			\node at (0.3,0) {$K$};
			\node at (0.6,0) {$K$};
			\node at (0.9,0) {$K$};
			\node at (1.2,0) {$0$};
		
			\draw[->] (0.28,-0.15) -- (0,-0.15);
			\draw[->] (0.58,-0.15) -- (0.32,-0.15);
			\draw[->] (0.88,-0.15) -- (0.62,-0.15);
			
		\end{tikzpicture}} &&&& {\begin{tikzpicture}

			\node at (0,0) {$0$};
			\node at (0.3,0) {$K$};
			\node at (0.6,0) {$K$};
			\node at (0.9,0) {$K$};
			\node at (1.2,0) {$K$};
		
			\draw[->] (0.58,-0.15) -- (0.32,-0.15);
			\draw[->] (0.88,-0.15) -- (0.62,-0.15);
			\draw[->] (1.18,-0.15) -- (0.92,-0.15);
			
		\end{tikzpicture}} \\[-0.5em]
		3 && {\begin{tikzpicture}[color={rgb,255:red,217;green,217;blue,217}]

			\node at (0,0) {$0$};
			\node at (0.3,0) {$0$};
			\node at (0.6,0) {$K$};
			\node at (0.9,0) {$0$};
			\node at (1.2,0) {$0$};
			
		\end{tikzpicture}} && {\begin{tikzpicture}[color={rgb,255:red,217;green,217;blue,217}]

			\node at (0,0) {$K$};
			\node at (0.3,0) {$K$};
			\node at (0.6,0) {$0$};
			\node at (0.9,0) {$0$};
			\node at (1.2,0) {$0$};
		
			\draw[->] (0.28,-0.15) -- (0,-0.15);
			
		\end{tikzpicture}} && {\begin{tikzpicture}[color={rgb,255:red,217;green,217;blue,217}]

			\node at (0,0) {$0$};
			\node at (0.3,0) {$0$};
			\node at (0.6,0) {$K$};
			\node at (0.9,0) {$K$};
			\node at (1.2,0) {$0$};
		
			\draw[->] (0.88,-0.15) -- (0.62,-0.15);
			
		\end{tikzpicture}} && {\begin{tikzpicture}[color={rgb,255:red,255;green,0;blue,0}]

			\node at (0,0) {$0$};
			\node at (0.3,0) {$K$};
			\node at (0.6,0) {$0$};
			\node at (0.9,0) {$0$};
			\node at (1.2,0) {$0$};
		
		\end{tikzpicture}} && {\begin{tikzpicture}[color={rgb,255:red,255;green,0;blue,0}]

			\node at (0,0) {$0$};
			\node at (0.3,0) {$0$};
			\node at (0.6,0) {$K$};
			\node at (0.9,0) {$K$};
			\node at (1.2,0) {$K$};
		
			\draw[->] (0.88,-0.15) -- (0.62,-0.15);
			\draw[->] (1.18,-0.15) -- (0.92,-0.15);
			
		\end{tikzpicture}} \\[-0.5em]
		4 &&& {\begin{tikzpicture}[color={rgb,255:red,217;green,217;blue,217}]

			\node at (0,0) {$K$};
			\node at (0.3,0) {$K$};
			\node at (0.6,0) {$K$};
			\node at (0.9,0) {$0$};
			\node at (1.2,0) {$0$};
		
			\draw[->] (0.28,-0.15) -- (0,-0.15);
			\draw[->] (0.32,0.4) -- (0.58,0.4);
			
		\end{tikzpicture}} &&&& {\begin{tikzpicture}

			\node at (0,0) {$0$};
			\node at (0.3,0) {$K$};
			\node at (0.6,0) {$K$};
			\node at (0.9,0) {$K$};
			\node at (1.2,0) {$0$};
		
			\draw[->] (0.32,0.4) -- (0.58,0.4);
			\draw[->] (0.88,-0.15) -- (0.62,-0.15);
			
		\end{tikzpicture}} &&&& {\begin{tikzpicture}

			\node at (0,0) {$0$};
			\node at (0.3,0) {$0$};
			\node at (0.6,0) {$0$};
			\node at (0.9,0) {$K$};
			\node at (1.2,0) {$K$};
		
			\draw[->] (1.18,-0.15) -- (0.92,-0.15);
			
		\end{tikzpicture}} \\
		5 & {\begin{tikzpicture}[color={rgb,255:red,217;green,217;blue,217}]

			\node at (0,0) {$K$};
			\node at (0.3,0) {$0$};
			\node at (0.6,0) {$0$};
			\node at (0.9,0) {$0$};
			\node at (1.2,0) {$0$};
			
		\end{tikzpicture}} &&&& {\begin{tikzpicture}[color={rgb,255:red,217;green,217;blue,217}]

			\node at (0,0) {$0$};
			\node at (0.3,0) {$K$};
			\node at (0.6,0) {$K$};
			\node at (0.9,0) {$0$};
			\node at (1.2,0) {$0$};
		
			\draw[->] (0.32,0.4) -- (0.58,0.4);
			
		\end{tikzpicture}} &&&& {\begin{tikzpicture}[color={rgb,255:red,255;green,0;blue,0}]

			\node at (0,0) {$0$};
			\node at (0.3,0) {$0$};
			\node at (0.6,0) {$0$};
			\node at (0.9,0) {$K$};
			\node at (1.2,0) {$0$};
			
		\end{tikzpicture}} &&&& {\begin{tikzpicture}[color={rgb,255:red,217;green,217;blue,217}]

			\node at (0,0) {$0$};
			\node at (0.3,0) {$0$};
			\node at (0.6,0) {$0$};
			\node at (0.9,0) {$0$};
			\node at (1.2,0) {$K$};
			
		\end{tikzpicture}}
		\arrow[color={rgb,255:red,217;green,217;blue,217}, from=2-8, to=3-10]
		\arrow[color={rgb,255:red,217;green,217;blue,217}, from=3-6, to=2-8]
		\arrow[color={rgb,255:red,217;green,217;blue,217}, from=3-6, to=4-7]
		\arrow[from=3-10, to=4-11]
		\arrow[color={rgb,255:red,217;green,217;blue,217}, from=4-3, to=5-4]
		\arrow[color={rgb,255:red,217;green,217;blue,217}, from=4-5, to=3-6]
		\arrow[color={rgb,255:red,217;green,217;blue,217}, from=4-7, to=5-8]
		\arrow[from=4-9, to=3-10]
		\arrow[from=4-11, to=5-12]
		\arrow[color={rgb,255:red,217;green,217;blue,217}, from=5-4, to=4-5]
		\arrow[color={rgb,255:red,217;green,217;blue,217}, from=5-4, to=6-6]
		\arrow[from=5-8, to=4-9]
		\arrow[from=5-8, to=6-10]
		\arrow[color={rgb,255:red,217;green,217;blue,217}, from=5-12, to=6-14]
		\arrow[color={rgb,255:red,217;green,217;blue,217}, from=6-2, to=5-4]
		\arrow[color={rgb,255:red,217;green,217;blue,217}, from=6-6, to=5-8]
		\arrow[from=6-10, to=5-12]
	\end{tikzcd}};
	\end{tikzpicture}
\end{center}
An example of anticompatible ordering for $V_{[\Q]}(M)$ is
\[
	{\begin{tikzpicture}[baseline={([yshift=-0.35em]current bounding box.center)}]

		\node at (0,0) {$0$};
		\node at (0.3,0) {$0$};
		\node at (0.6,0) {$0$};
		\node at (0.9,0) {$K$};
		\node at (1.2,0) {$0$};
		
	\end{tikzpicture}}, \quad {\begin{tikzpicture}[baseline={([yshift=-0.35em]current bounding box.center)}]
	
		\node at (0,0) {$0$};
		\node at (0.3,0) {$K$};
		\node at (0.6,0) {$0$};
		\node at (0.9,0) {$0$};
		\node at (1.2,0) {$0$};
	
	\end{tikzpicture}}, \quad {\begin{tikzpicture}[baseline={([yshift=-0.2em]current bounding box.center)}]
	
		\node at (0,0) {$0$};
		\node at (0.3,0) {$0$};
		\node at (0.6,0) {$K$};
		\node at (0.9,0) {$K$};
		\node at (1.2,0) {$K$};
	
		\draw[->] (0.88,-0.275) -- (0.62,-0.275);
		\draw[->] (1.18,-0.275) -- (0.92,-0.275);
	
	\end{tikzpicture}}.
\]
The distinguished triangles in $\pvd(\Pi_Q)$ given by Theorem \ref{thm:mesh gives rise to triangles} that correspond to this anticompatible ordering are of the following form:
\[\begin{tikzcd}
	{\begin{tikzpicture}

		\node at (0,0) {$0$};
		\node at (0.3,0) {$K$};
		\node at (0.6,0) {$K$};
		\node at (0.9,0) {$0$};
		\node at (1.2,0) {$0$};
	
		\draw[->] (0.32,0.4) -- (0.58,0.4);
		
	\end{tikzpicture}} & {\begin{tikzpicture}

		\node at (0,0) {$0$};
		\node at (0.3,0) {$K$};
		\node at (0.6,0) {$K$};
		\node at (0.9,0) {$K$};
		\node at (1.2,0) {$0$};
	
		\draw[->] (0.32,0.4) -- (0.58,0.4);
		\draw[->] (0.88,-0.15) -- (0.62,-0.15);
		
	\end{tikzpicture}} & {\begin{tikzpicture}

		\node at (0,0) {$0$};
		\node at (0.3,0) {$0$};
		\node at (0.6,0) {$0$};
		\node at (0.9,0) {$K$};
		\node at (1.2,0) {$0$};
		
	\end{tikzpicture}} & {\Sigma\!\!\begin{tikzpicture}

		\node at (0,0) {$0$};
		\node at (0.3,0) {$K$};
		\node at (0.6,0) {$K$};
		\node at (0.9,0) {$0$};
		\node at (1.2,0) {$0$};
	
		\draw[->] (0.32,0.4) -- (0.58,0.4);
		
	\end{tikzpicture}},
	\arrow[from=1-1, to=1-2]
	\arrow[from=1-2, to=1-3]
	\arrow[from=1-3, to=1-4]
\end{tikzcd}\]
\[\begin{tikzcd}
	{\begin{tikzpicture}

		\node at (0,0) {$0$};
		\node at (0.3,0) {$0$};
		\node at (0.6,0) {$K$};
		\node at (0.9,0) {$0$};
		\node at (1.2,0) {$0$};
		
	\end{tikzpicture}} & {\begin{tikzpicture}

		\node at (0,0) {$0$};
		\node at (0.3,0) {$K$};
		\node at (0.6,0) {$K$};
		\node at (0.9,0) {$0$};
		\node at (1.2,0) {$0$};
	
		\draw[->] (0.32,0.4) -- (0.58,0.4);
		
	\end{tikzpicture}} & {\begin{tikzpicture}

		\node at (0,0) {$0$};
		\node at (0.3,0) {$K$};
		\node at (0.6,0) {$0$};
		\node at (0.9,0) {$0$};
		\node at (1.2,0) {$0$};
		
	\end{tikzpicture}} & {\Sigma\!\!\begin{tikzpicture}

		\node at (0,0) {$0$};
		\node at (0.3,0) {$0$};
		\node at (0.6,0) {$K$};
		\node at (0.9,0) {$0$};
		\node at (1.2,0) {$0$};
		
	\end{tikzpicture}},
	\arrow[from=1-1, to=1-2]
	\arrow[from=1-2, to=1-3]
	\arrow[from=1-3, to=1-4]
\end{tikzcd}\]
\[\begin{tikzcd}
	{\Sigma^{-1}\!\!\begin{tikzpicture}

		\node at (0,0) {$0$};
		\node at (0.3,0) {$0$};
		\node at (0.6,0) {$0$};
		\node at (0.9,0) {$K$};
		\node at (1.2,0) {$K$};
	
		\draw[->] (1.18,-0.15) -- (0.92,-0.15);
		
	\end{tikzpicture}} & {\begin{tikzpicture}

		\node at (0,0) {$0$};
		\node at (0.3,0) {$0$};
		\node at (0.6,0) {$K$};
		\node at (0.9,0) {$0$};
		\node at (1.2,0) {$0$};
		
	\end{tikzpicture}} & {\begin{tikzpicture}

		\node at (0,0) {$0$};
		\node at (0.3,0) {$0$};
		\node at (0.6,0) {$K$};
		\node at (0.9,0) {$K$};
		\node at (1.2,0) {$K$};
	
		\draw[->] (0.88,-0.15) -- (0.62,-0.15);
		\draw[->] (1.18,-0.15) -- (0.92,-0.15);
		
	\end{tikzpicture}} & {\begin{tikzpicture}

		\node at (0,0) {$0$};
		\node at (0.3,0) {$0$};
		\node at (0.6,0) {$0$};
		\node at (0.9,0) {$K$};
		\node at (1.2,0) {$K$};
	
		\draw[->] (1.18,-0.15) -- (0.92,-0.15);
		
	\end{tikzpicture}}.
	\arrow[from=1-1, to=1-2]
	\arrow[from=1-2, to=1-3]
	\arrow[from=1-3, to=1-4]
\end{tikzcd}\]
\end{example}

\begin{remark}
When $\bm{i}$ is a source sequence for an orientation $Q$ of $\Delta$, we know from Propositions \ref{prop:concentrated in zero for adapted word} and \ref{prop:derived category for adapted word} that $\C([\bm{i}]) \cong \modcat KQ$ and $\D([\bm{i}]) \cong \D^b(\modcat KQ)$. Hence, $\C([\bm{i}])$ is an abelian category and $\D([\bm{i}])$ has a triangulated structure. The latter has AR triangles (in the sense of \cite{Happel}), which we can obtain by assembling the distinguished triangles of Theorem \ref{thm:mesh gives rise to triangles} into a single one.

One may wonder if we have similar results for general commutation classes. A quick inspection of examples shows that $\C([\bm{i}])$ is not abelian in general. However, one may still search for a nontrivial extriangulated structure (in the sense of \cite{NakaokaPalu}) on $\D([\bm{i}])$ for which $\C([\bm{i}])$ is an extension-closed subcategory. It would be desirable for this structure to take into account the triangles of Theorem \ref{thm:mesh gives rise to triangles}, which can be naturally seen as triangles in $\D([\bm{i}])$. Unfortunately, we were not able to find such a structure. A sign against its existence is given in Remark \ref{rmk:euler form does not respect triangles}.
\end{remark}

\begin{cor}\label{cor:generalized g-additive property}
Let $[\bm{i}]$ be a commutation class of reduced words for $w_0$. For $x \in (\widehat{\Upsilon}_{[\bm{i}]})_0 = \widehat{\mathsf{R}}$, we have
\[
\pi(x) + \pi(s_{[\bm{i}]}(x)) = \sum_{y \in V_{[\bm{i}]}(x)}\pi(y),
\]
where $\pi: \widehat{\mathsf{R}} \longrightarrow \mathsf{R}$ denotes the projection onto the first coordinate.
\end{cor}

\begin{proof}
Let $M \in \R([\bm{i}])$ be the indecomposable object of coordinate $x$. After choosing an anticompatible reading for $V_{[\bm{i}]}(M)$, we obtain the formula by summing the relations in $K_0(\pvd(\Pi_Q))$ induced by the distinguished triangles from Theorem \ref{thm:mesh gives rise to triangles}.	
\end{proof}

\begin{remark}\label{rem:implicit in the literature}
When $[\bm{i}]$ is a commutation class coming from a Q-datum for $\mathfrak{g}$, the formula above specializes to the $\mathfrak{g}$-additive property of \cite{FujitaOh} (see Theorem \ref{thm:g-additive property}). For general $[\bm{i}]$, this result was already implicit in the literature. Indeed, by considering the subword of $\bm{i}$ corresponding to a given mesh, one obtains from \cite[Theorem 4.25]{KashiwaraKimOhPark} a nonsplit exact sequence between certain modules over a quantum affine algebra, known as a generalized T-system. These modules must belong to the same block in the block decomposition of the category of finite-dimensional integrable modules over this algebra (see \cite{KKOPSimplyLacedRootSystems}), which yields the equality in Corollary \ref{cor:generalized g-additive property}. Similarly, one can use the recent paper \cite{KashiwaraKimOhParkIII} to extend the above result to words that are not necessarily reduced. 
\end{remark}

We end this section with an application to frieze patterns. We say that a function $f: (\widehat{\Upsilon}_{[\bm{i}]})_0 \to \Z$ is an \emph{additive frieze on} $\widehat{\Upsilon}_{[\bm{i}]}$ if it satisfies the relation
\[
f(x) + f(s_{[\bm{i}]}(x)) = \sum_{y \in V_{[\bm{i}]}(x)}f(y)
\]
for all $x \in (\widehat{\Upsilon}_{[\bm{i}]})_0$. When $[\bm{i}]$ is the commutation class associated with an orientation $Q$ of $\Delta$, this definition specializes to that of an additive function on the repetition quiver of $Q$ (in the sense of \cite{Gabriel}). Thus, our definition can be seen as an additive analog of Coxeter's frieze patterns \cite{Coxeter}. We refer the reader to \cite{MorierGenoud} for a recent survey on other frieze patterns appearing in representation theory.

For $i \in \Delta_0$, let $\rho_i: \mathsf{R} \to \Z$ be the function that returns the coefficient of $\alpha_i$ in a root $\alpha \in \mathsf{R}$ when written as a linear combination of the simple roots. By Corollary \ref{cor:generalized g-additive property}, the composition $\rho_i \circ \pi$ is an additive frieze on $\widehat{\Upsilon}_{[\bm{i}]}$. We have the following result.

\begin{prop}
Let $\bm{i} = (i_1,\dots,i_N)$ be a reduced word for $w_0$. The functions $\rho_i \circ \pi$ for $i \in \Delta_0$ form a $\Z$-basis of the space of additive friezes on $\widehat{\Upsilon}_{[\bm{i}]}$. Consequently, any additive frieze $f: (\widehat{\Upsilon}_{[\bm{i}]})_0 \to \Z$ satisfies $f(\widehat{\beta}^{\bm{i}}_{k+N}) = -f(\widehat{\beta}^{\bm{i}}_k)$ for $k \in \Z$ and is thus periodic.
\end{prop}

\begin{proof}
	Denote by $x_j$ the injective vertex of $\widehat{\Upsilon}_{[\bm{i}]}$ associated with $j \in \Delta_0$. Let $M$ be the square matrix indexed by $\Delta_0$ whose $(i,j)$-entry is $\rho_i(\pi(x_j))$. It is not hard to see that any additive frieze $f: (\widehat{\Upsilon}_{[\bm{i}]})_0 \to \Z$ is determined by the values $f(x_j)$ for $j \in \Delta_0$, which can be chosen to be arbitrary. Therefore, to prove the first statement, it suffices to show that $M$ is invertible over the integers. This is simple: if we order $\Delta_0$ so that $i < j$ whenever there is a path from $x_j$ to $x_i$ in $\widehat{\Upsilon}_{[\bm{i}]}$, then $M$ becomes upper unitriangular hence invertible. The second statement holds for the functions $\rho_i \circ \pi$ since $\pi(\widehat{\beta}^{\bm{i}}_{k+N}) = -\pi(\widehat{\beta}^{\bm{i}}_k)$ for all $k \in \Z$, so it must hold for all additive friezes.
\end{proof}

\section{Extension groups and the Euler form}\label{section:extension groups and the Euler form}

In this section, we define extension groups and use them to construct a sort of Euler form on the c-derived category of a commutation class. The main result is that the symmetrization of this Euler form agrees with the symmetric bilinear form of the root lattice (Theorem \ref{thm:symmetrization of Euler form is Cartan form}). Along the way, we also introduce the concept of projective and injective objects.

\subsection{Projective/injective objects and extension groups}\label{subsection:projective/injective objects} Let $[\bm{i}]$ be a commutation class of reduced words for the longest element $w_0$. The \emph{injective indecomposable object $I^{[\bm{i}]}_i$ associated with $i \in \Delta_0$} is the indecomposable object of $\C([\bm{i}])$ whose coordinate is the injective vertex associated with $i$ (see Sections \ref{section:commutation classes and combinatorial AR quivers} and \ref{subsection:combinatorial reflection functors}). Dually, the \emph{projective indecomposable object $P^{[\bm{i}]}_i$ associated with $i \in \Delta_0$} is the indecomposable object of $\C([\bm{i}])$ whose coordinate is the projective vertex associated with $i$. Notice that $I^{[\bm{i}]}_i$ has residue $i$, while $P^{[\bm{i}]}_i$ has residue $i^*$.

\begin{remark}
Let $\bm{i}$ be a reduced word for $w_0$ that is a source sequence for an orientation $Q$ of $\Delta$. By Proposition \ref{prop:concentrated in zero for adapted word}, $\C([\bm{i}])$ and $\modcat KQ$ are equivalent categories. In this case, the usual notion of projective/injective indecomposable objects in $\modcat KQ$ coincides with our definition. This follows by directly computing the projective/injective positive roots with respect to $[\bm{i}]$.
\end{remark}

By Theorem \ref{thm:indecomposables are concentrated in zero}, we can see any object of $\C([\bm{i}])$ as a module over the preprojective algebra $\Lambda_{Q}$. In particular, we can see it as a representation of the double quiver $\overline{Q}$ of $Q$. For $M \in \C([\bm{i}])$ and $i \in \Delta_0$, we denote by $M(i)$ the $K$-vector space at the vertex $i$ if we see $M$ as a representation of $\overline{Q}$. Recall from Section \ref{subsection:alternative descriptions of spherical twists} (when we defined the reflection functor $\Sigma_i$) that we also have a diagram
\[\begin{tikzcd}
	{\widetilde{M}(i)} & {M(i)} & {\widetilde{M}(i)}
	\arrow["{M_{\mathrm{in}(i)}}", from=1-1, to=1-2]
	\arrow["{M_{\mathrm{out}(i)}}", from=1-2, to=1-3]
\end{tikzcd}\]
describing the representation around the vertex $i$.

The next result follows from \cite[Proposition 2.4]{AmiotIyamaReitenTodorov}. We provide a different proof.

\begin{lemma}\label{lemma:morphism from projective to injective}
If $i \in \Delta_0$, then the space $\Hom_{\C([\bm{i}])}(P^{[\bm{i}]}_i,I^{[\bm{i}]}_i)$ is one-dimensional.
\end{lemma}

\begin{proof}
After applying a tilting functor not involving a spherical twist at the vertex $i$, we can assume that $i$ is a sink of $[\bm{i}]$ so that $P^{[\bm{i}]}_i \cong S_i$. If we view $M = I^{[\bm{i}]}_i$ as a representation of $\overline{Q}$, we need to show that $\dim_K\ker M_{\mathrm{out}(i)} = 1$ in order to prove the lemma. We claim that we have a stronger property: $M_{\mathrm{out}(i)} = 0$ and $\dim_KM(i) = 1$.

Since $M$ is injective, we can write $M \cong T_{i_1}\dotsb T_{i_k}(S_i)$ where the indices $i_1,\dots,i_k$ are different from $i$. By Theorem \ref{thm:indecomposables are concentrated in zero} and Proposition \ref{prop:computing the spherical twist}, we may replace each spherical twist $T_{i_l}$ above by the corresponding reflection functor $\Sigma_{i_l}$. As a representation, $S_i$ satisfies $\dim_KS_i(i) = 1$ and $(S_i)_{\mathrm{out}(i)} = 0$. Thus, it suffices to show that, if $N$ is a representation of $\Lambda_Q$ satisfying $\dim_KN(i) = 1$ and $N_{\mathrm{out}(i)} = 0$, then $\Sigma_j(N)$ has the same property whenever $j \neq i$. One can easily check this claim using the definition of $\Sigma_j$.
\end{proof}

We define the \emph{$k$-th extension group} ($k \in \Z$) between two objects $M, N \in \D([\bm{i}])$ as
\[
\Ext_{[\bm{i}]}^k(M,N) = \Hom_{\D([\bm{i}])}(M, \Sigma^kN).
\]
Observe that this vector space is zero if $M,N \in \C([\bm{i}])$ and $k < 0$.

\begin{prop}\label{prop:projective/injective if and only if Ext1 = 0}
Let $[\bm{i}]$ be a commutation class of reduced words for $w_0$ and take $M \in \ind([\bm{i}])$. Then $M$ is projective if and only if $\Ext_{[\bm{i}]}^k(M,M') = 0$ for all $M' \in \C([\bm{i}])$ and $k > 0$. Dually, $M$ is injective if and only if $\Ext_{[\bm{i}]}^k(M',M) = 0$ for all $M' \in \C([\bm{i}])$ and $k > 0$.
\end{prop}

\begin{proof}
We will prove the first statement. The injective case is similar.

Assume $M = P^{[\bm{i}]}_i$ is projective and let $I = I^{[\bm{i}]}_i$ be the corresponding injective indecomposable object. Suppose by contradiction that there exists $M' \in \C([\bm{i}])$ and $k > 0$ such that $\Ext_{[\bm{i}]}^k(M,M') \neq 0$. We can assume that $M'$ is indecomposable and that $\Delta \neq \mathsf{A}_1$. By Lemma \ref{lemma:morphisms in the repetition category follow paths}, there is a path in $\widehat{\Upsilon}_{[\bm{i}]}$ from $c_{[\bm{i}]}(M)$ to $c_{[\bm{i}]}(\Sigma^kM')$. Since $\Sigma^kM' \not\in \C([\bm{i}])$ and $\Upsilon_{[\bm{i}]}$ is convex in $\widehat{\Upsilon}_{[\bm{i}]}$, there is no path in $\widehat{\Upsilon}_{[\bm{i}]}$ from $c_{[\bm{i}]}(\Sigma^kM')$ to $c_{[\bm{i}]}(I)$. By our hypothesis and by the definition of the c-derived category, there is a tilting functor $T: \D([\bm{i}]) \longrightarrow \D([\bm{j}])$ such that $T(M), T(\Sigma^kM') \in \C([\bm{j}])$. Since there is a path from $c_{[\bm{i}]}(M)$ to $c_{[\bm{i}]}(\Sigma^kM')$, we can assume that $T(M)$ is still projective and that $T(M) \cong S_i$. By Lemma \ref{lemma:projective goes to projective iff injective goes to injective}, $T(I) \in \C([\bm{j}])$ and it is still injective of residue $i$. Since there is no path from the coordinate of $T(\Sigma^kM')$ to the coordinate of $T(I)$, we conclude from Theorem \ref{thm:combinatorial AR quiver realizes partial order} that $T(\Sigma^kM') \cong T_{i_1}\dotsb T_{i_{l-1}}(S_{i_l})$ where $(i_1,\dots,i_l)$ is a source sequence for $[\bm{j}]$ that does not contain $i$. But then $T(\Sigma^kM')(i) = 0$ by Proposition \ref{prop:computing the spherical twist} and there cannot exist a nonzero morphism from $T(M) \cong S_i$ to $T(\Sigma^kM')$, a contradiction.

Let us prove the converse. Take a representative $\bm{i} = (i_1,\dots,i_N)$ for $[\bm{i}]$ and let $1 \leq k \leq N$ be such that $M \cong M^{\bm{i}}_k$. If $M$ is not projective, there is $k < l \leq N$ such that $i_l = i_k$. Suppose that $l$ is the smallest such index. Since $(i_N^*,i_{N-1}^*,\dots,i_l^*)$ is a sink sequence for $[\bm{i}]$, it induces a tilting functor $T: \D([\bm{i}])\longrightarrow \D([\bm{j}])$, where $\bm{j} = (i_l^*,\dots,i_N^*,i_1,\dots,i_{l-1})$. Notice that $T(M)$ is now projective in $\D([\bm{j}])$ and the corresponding injective indecomposable object is $I = M^{\bm{j}}_1$. We deduce that $c_{[\bm{i}]}(T^{-1}(I)) = \widehat{\beta}^{\bm{i}}_{l-N}$, so $\Sigma^{-1}T^{-1}(I) \in \C([\bm{i}])$ since $-N < l-N \leq 0$. Finally, we have
\[
\Ext_{[\bm{i}]}^1(M,\Sigma^{-1}T^{-1}(I)) = \Hom_{\D([\bm{i}])}(M,T^{-1}(I)) \cong \Hom_{\D([\bm{j}])}(T(M),I),
\]
which is nonzero by Lemma \ref{lemma:morphism from projective to injective}.
\end{proof}

\begin{prop}\label{prop:derived category is hereditary}
Let $[\bm{i}]$ be a commutation class of reduced words for $w_0$. If $M,M' \in \C([\bm{i}])$, then $\Ext_{[\bm{i}]}^k(M,M') = 0$ for all $k > 1$.
\end{prop}

\begin{proof}
Take a representative $\bm{i} = (i_1,\dots,i_N)$ for $[\bm{i}]$ and let $1 \leq l \leq N$ be such that $M \cong M^{\bm{i}}_l$. Since $(i_N^*,i_{N-1}^*,\dots,i_{l+1}^*)$ is a sink sequence for $[\bm{i}]$, it induces a tilting functor $T: \D([\bm{i}])\longrightarrow \D([\bm{j}])$, where $\bm{j} = (i_{l+1}^*,\dots,i_N^*,i_1,\dots,i_l)$. Note that $T(M) \cong M^{\bm{j}}_N$ is now projective. On the other hand, if $M' = M^{\bm{i}}_{l'}$ for $1 \leq l' \leq N$, then $c_{[\bm{j}]}(T(M')) = \widehat{\beta}^{\bm{j}}_{l'+ N - l}$. Since $1 \leq l'+ N - l < 2N$, $T(M')$ is concentrated in degree $m$ for $m = 0$ or $m = 1$. Thus, if $k > 1$, we have $k - m > 0$, and so,
\[
\Ext_{[\bm{i}]}^k(M,M') \cong \Ext_{[\bm{j}]}^k(T(M),T(M')) = \Ext_{[\bm{j}]}^{k-m}(T(M), \Sigma^mT(M')) = 0
\]
by Proposition \ref{prop:projective/injective if and only if Ext1 = 0} as $T(M)$ is projective and $\Sigma^mT(M') \in \C([\bm{j}])$.
\end{proof}

\subsection{The Grothendieck group and the Euler form} Let $[\bm{i}]$ be a commutation class of reduced words for $w_0$. The split Grothendieck group $K_0^{\mathrm{sp}}(\D([\bm{i}]))$ of the additive category $\D([\bm{i}])$ is the free abelian group generated by isomorphism classes $[M]$ of objects $M$ in $\D([\bm{i}])$ modulo the relation $[M \oplus N] = [M] + [N]$. We define the \emph{Grothendieck group $K_0(\D([\bm{i}]))$ of $\D([\bm{i}])$} to be the quotient of $K_0^{\mathrm{sp}}(\D([\bm{i}]))$ where we impose the additional relation $[M]+[\Sigma M] = 0$. Note that $K_0(\D([\bm{i}]))$ is isomorphic to the split Grothendieck group of $\C([\bm{i}])$.

The \emph{Euler form of $\D([\bm{i}])$} is the pairing
\[
\langle -,-\rangle_{[\bm{i}]}: K_0(\D([\bm{i}])) \times K_0(\D([\bm{i}])) \longrightarrow \Z
\]
defined by
\[
\langle [M],[N]\rangle_{[\bm{i}]} = \sum_{k \in \Z}(-1)^k\dim_K\Ext_{[\bm{i}]}^k(M,N)
\]
for objects $M,N \in \D([\bm{i}])$. The sum above is finite by Proposition \ref{prop:derived category is hereditary}. Moreover, it is compatible with direct sums and the suspension functor, so we indeed get a well-defined pairing on $K_0(\D([\bm{i}]))$. For simplicity, we will denote $\langle [M],[N]\rangle_{[\bm{i}]}$ by $\langle M,N\rangle_{[\bm{i}]}$.

\begin{remark}\label{rmk:euler form does not respect triangles}
The Euler form defined above does not always respect the triangulated structure of $\pvd(\Pi_Q)$. For example, let $\Q = (\mathsf{A}_3,\vee,\xi)$ be the Q-datum of type $\mathsf{B}_2$ with height function given by $\xi_1 = 3$, $\xi_2 = 4$ and $\xi_3 = 5$. We can describe $\ind([\Q])$ with the following picture:
\[\begin{tikzcd}[column sep={5em,between origins},row sep={0.5em}]
	{(\im \ \backslash \ p)} & 0 & 1 & 2 & 3 & 4 & 5 \\
	1 &&&& {\begin{tikzpicture}

		\node at (0,0) {$K$};
		\node at (0.3,0) {$K$};
		\node at (0.6,0) {$K$};
	
		\draw[->] (0.28,-0.15) -- (0,-0.15);
		\draw[->] (0.6,-0.15) -- (0.32,-0.15);
		
	\end{tikzpicture}} \\
	2 & {\begin{tikzpicture}

		\node at (0,0) {$0$};
		\node at (0.3,0) {$K$};
		\node at (0.6,0) {$0$};
		
	\end{tikzpicture}} && {\begin{tikzpicture}

		\node at (0,0) {$K$};
		\node at (0.3,0) {$0$};
		\node at (0.6,0) {$0$};
		
	\end{tikzpicture}} && {\begin{tikzpicture}

		\node at (0,0) {$0$};
		\node at (0.3,0) {$K$};
		\node at (0.6,0) {$K$};
	
		\draw[->] (0.6,-0.15) -- (0.3,-0.15);
		
	\end{tikzpicture}} \\
	3 && {\begin{tikzpicture}

		\node at (0,0) {$K$};
		\node at (0.3,0) {$K$};
		\node at (0.6,0) {$0$};
	
		\draw[->] (0,0.4) -- (0.3,0.4);
		
	\end{tikzpicture}} &&&& {\begin{tikzpicture}

		\node at (0,0) {$0$};
		\node at (0.3,0) {$0$};
		\node at (0.6,0) {$K$};
		
	\end{tikzpicture}}
	\arrow[from=2-5, to=3-6]
	\arrow[from=3-2, to=4-3]
	\arrow[from=3-4, to=2-5]
	\arrow[from=3-6, to=4-7]
	\arrow[from=4-3, to=3-4]
\end{tikzcd}\]
If we consider the mesh determined by $I^{[\Q]}_2$, we obtain the distinguished triangle
\[\begin{tikzcd}
	{\begin{tikzpicture}

		\node at (0,0) {$K$};
		\node at (0.3,0) {$0$};
		\node at (0.6,0) {$0$};
		
	\end{tikzpicture}} & {\begin{tikzpicture}

		\node at (0,0) {$K$};
		\node at (0.3,0) {$K$};
		\node at (0.6,0) {$K$};
	
		\draw[->] (0.28,-0.15) -- (0,-0.15);
		\draw[->] (0.6,-0.15) -- (0.32,-0.15);
		
	\end{tikzpicture}} & {\begin{tikzpicture}

		\node at (0,0) {$0$};
		\node at (0.3,0) {$K$};
		\node at (0.6,0) {$K$};

		\draw[->] (0.6,-0.15) -- (0.32,-0.15);
		
	\end{tikzpicture}} & {\Sigma\!\!\begin{tikzpicture}

		\node at (0,0) {$K$};
		\node at (0.3,0) {$0$};
		\node at (0.6,0) {$0$};
		
	\end{tikzpicture}}
	\arrow[from=1-1, to=1-2]
	\arrow[from=1-2, to=1-3]
	\arrow[from=1-3, to=1-4]
\end{tikzcd}\]
in $\pvd(\Pi_Q)$ with corners in $\R([\Q])$. One can compute that
\[
\langle {\begin{tikzpicture}[baseline={([yshift=-0.35em]current bounding box.center)}]

	\node at (0,0) {$K$};
	\node at (0.3,0) {$0$};
	\node at (0.6,0) {$0$};
	
\end{tikzpicture}}, {\begin{tikzpicture}[baseline={([yshift=-0.35em]current bounding box.center)}]

	\node at (0,0) {$0$};
	\node at (0.3,0) {$K$};
	\node at (0.6,0) {$0$};

\end{tikzpicture}} \rangle_{[\Q]} = -1, \quad \langle {\begin{tikzpicture}[baseline={([yshift=-0.2em]current bounding box.center)}]

	\node at (0,0) {$K$};
	\node at (0.3,0) {$K$};
	\node at (0.6,0) {$K$};

	\draw[->] (0.28,-0.275) -- (0,-0.275);
	\draw[->] (0.6,-0.275) -- (0.32,-0.275);
	
\end{tikzpicture}}, {\begin{tikzpicture}[baseline={([yshift=-0.35em]current bounding box.center)}]

	\node at (0,0) {$0$};
	\node at (0.3,0) {$K$};
	\node at (0.6,0) {$0$};

\end{tikzpicture}} \rangle_{[\Q]} = 0 \quad \textrm{and} \quad \langle {\begin{tikzpicture}[baseline={([yshift=-0.2em]current bounding box.center)}]

	\node at (0,0) {$0$};
	\node at (0.3,0) {$K$};
	\node at (0.6,0) {$K$};

	\draw[->] (0.6,-0.275) -- (0.32,-0.275);
	
\end{tikzpicture}}, {\begin{tikzpicture}[baseline={([yshift=-0.35em]current bounding box.center)}]

	\node at (0,0) {$0$};
	\node at (0.3,0) {$K$};
	\node at (0.6,0) {$0$};

\end{tikzpicture}} \rangle_{[\Q]} = 0.
\]
Hence, the equality
\[
\langle {\begin{tikzpicture}[baseline={([yshift=-0.35em]current bounding box.center)}]

	\node at (0,0) {$K$};
	\node at (0.3,0) {$0$};
	\node at (0.6,0) {$0$};
	
\end{tikzpicture}}, {\begin{tikzpicture}[baseline={([yshift=-0.35em]current bounding box.center)}]

	\node at (0,0) {$0$};
	\node at (0.3,0) {$K$};
	\node at (0.6,0) {$0$};

\end{tikzpicture}} \rangle_{[\Q]} - \langle {\begin{tikzpicture}[baseline={([yshift=-0.2em]current bounding box.center)}]

	\node at (0,0) {$K$};
	\node at (0.3,0) {$K$};
	\node at (0.6,0) {$K$};

	\draw[->] (0.28,-0.275) -- (0,-0.275);
	\draw[->] (0.6,-0.275) -- (0.32,-0.275);
	
\end{tikzpicture}}, {\begin{tikzpicture}[baseline={([yshift=-0.35em]current bounding box.center)}]

	\node at (0,0) {$0$};
	\node at (0.3,0) {$K$};
	\node at (0.6,0) {$0$};

\end{tikzpicture}} \rangle_{[\Q]} + \langle {\begin{tikzpicture}[baseline={([yshift=-0.2em]current bounding box.center)}]

	\node at (0,0) {$0$};
	\node at (0.3,0) {$K$};
	\node at (0.6,0) {$K$};

	\draw[->] (0.6,-0.275) -- (0.32,-0.275);
	
\end{tikzpicture}}, {\begin{tikzpicture}[baseline={([yshift=-0.35em]current bounding box.center)}]

	\node at (0,0) {$0$};
	\node at (0.3,0) {$K$};
	\node at (0.6,0) {$0$};

\end{tikzpicture}} \rangle_{[\Q]} = 0
\]
is false. More strongly, one can check that the square matrix indexed by $\ind([\Q])$ whose $(M,N)$-entry is $\langle M,N\rangle_{[\Q]}$ is invertible. Therefore, we cannot impose more relations in the definition of $K_0(\D([\Q]))$ if we want to work with the Euler form defined above.
\end{remark}

When $Q$ is a Dynkin quiver of simply laced type $\Delta$, the Grothendieck group of $\D^b(\modcat KQ)$ can be naturally identified with the root lattice $\mathsf{Q}$ of $\Delta$. With this identification, the symmetrization of the Euler form coincides with the usual symmetric bilinear form $(-,-)$ on $\mathsf{Q}$. We generalize this result to our setting.

\begin{thm}\label{thm:symmetrization of Euler form is Cartan form}
Let $[\bm{i}]$ be a commutation class of reduced words for $w_0$. If $M$ and $N$ are objects in $\D([\bm{i}])$, then
\[
\langle M,N\rangle_{[\bm{i}]} + \langle N,M\rangle_{[\bm{i}]} = ([M]_{\pvd(\Pi_{Q})},[N]_{\pvd(\Pi_{Q})})
\]
where $[M]_{\pvd(\Pi_{Q})}$ and $[N]_{\pvd(\Pi_{Q})}$ denote the classes of $M$ and $N$ in $K_0(\pvd(\Pi_{Q}))$.
\end{thm}

\begin{proof}
By additivity, we may assume that $M$ and $N$ are indecomposable. Tilting functors preserve the formula above, so we can also suppose that $N = S_i$ for a source $i \in \Delta_0$ of $[\bm{i}]$. Finally, by replacing $M$ with a suitable shift, we may take $M \in \C([\bm{i}])$. 

Since $S_i$ is injective, we have
\[
\langle M,S_i\rangle_{[\bm{i}]} = \dim_K\Hom_{\C([\bm{i}])}(M,S_i)
\]
by Proposition \ref{prop:projective/injective if and only if Ext1 = 0}. If $M \cong S_i$, we thus obtain $\langle S_i,S_i\rangle_{[\bm{i}]} = 1$ and
\[
([S_i]_{\pvd(\Pi_Q)},[S_i]_{\pvd(\Pi_Q)}) = (\alpha_i,\alpha_i) = 2 = 2\langle S_i,S_i\rangle_{[\bm{i}]},
\]
as needed. Hence, we may suppose $M \not\cong S_i$. In this case, the tilting functor $T_i^{-1}: \D([\bm{i}]) \longrightarrow \D(r_i[\bm{i}])$ takes $M$ to an object of $\C(r_i[\bm{i}])$ by Proposition \ref{prop:equivalence under reflection, general case}. Since $S_i$ is projective with respect to $r_i[\bm{i}]$, we get
\[
\langle S_i,M\rangle_{[\bm{i}]} = \langle T_i^{-1}(S_i),T_i^{-1}(M)\rangle_{r_i[\bm{i}]} = -\langle S_i,T_i^{-1}(M)\rangle_{r_i[\bm{i}]} = -\dim_K\Hom_{\C(r_i[\bm{i}])}(S_i,T_i^{-1}(M)),
\]
where we used item (1) of Proposition \ref{prop:properties of spherical twists} and Proposition \ref{prop:projective/injective if and only if Ext1 = 0}. Therefore, by Lemma \ref{lemma:equality of defects after applying reflection} below,
\[
\langle M,S_i\rangle_{[\bm{i}]} + \langle S_i,M\rangle_{[\bm{i}]} = \dim_KM(i) - \dim_KT_i^{-1}(M)(i),
\]
which we can write as
\[
((1 - s_i)[M]_{\pvd(\Pi_Q)},\varpi_i) = ([M]_{\pvd(\Pi_Q)},(1-s_i)\varpi_i) = ([M]_{\pvd(\Pi_Q)},\alpha_i) = ([M]_{\pvd(\Pi_Q)},[S_i]_{\pvd(\Pi_Q)}),
\]
finishing the proof.
\end{proof}

\begin{lemma}\label{lemma:equality of defects after applying reflection}
Let $[\bm{i}]$ be a commutation class of reduced words for $w_0$ and let $i \in \Delta_0$ be a source of $[\bm{i}]$. If $M \in \ind(r_i[\bm{i}])$ and $M \not\cong S_i$, then we have an equality:
\[
\dim_KM(i) - \dim_K\Hom_{\C(r_i[\bm{i}])}(S_i,M) = \dim_KT_i(M)(i) - \dim_K\Hom_{\C([\bm{i}])}(T_i(M),S_i).
\]
\end{lemma}

Note that we indeed have $T_i(M) \in \ind([\bm{i}])$ by Proposition \ref{prop:equivalence under reflection, general case}.

\begin{proof}
Viewing the objects above as representations of $\overline{Q}$, observe that the dimension of the $\Hom$-space from $S_i$ to $M$ is precisely $\dim_K\ker M_{\mathrm{out}(i)}$. Therefore, the left-hand side in the statement equals $\dim_K\operatorname{im}M_{\mathrm{out}(i)}$. Since $T_i(M)$ is concentrated in degree $0$ by Proposition \ref{prop:equivalence under reflection, general case}, we know that $M_{\mathrm{in}(i)}$ is surjective and that $T_i(M) \cong \Sigma_i(M)$ by Proposition \ref{prop:computing the spherical twist}. Consequently, in the diagram
\[\begin{tikzcd}[column sep=5em]
	{\widetilde{T_i(M)}(i)} & {T_i(M)(i)} & {\widetilde{T_i(M)}(i)}
	\arrow["{(T_i(M))_{\mathrm{in}(i)}}", from=1-1, to=1-2]
	\arrow["{(T_i(M))_{\mathrm{out}(i)}}", from=1-2, to=1-3]
\end{tikzcd}\]
defining the representation associated with $T_i(M)$ around the vertex $i$, we have
\[
T_i(M)(i) = \ker M_{\mathrm{in}(i)}, \quad \widetilde{T_i(M)}(i) = \widetilde{M}(i) \quad \textrm{and} \quad (T_i(M))_{\mathrm{in}(i)} = M_{\mathrm{out}(i)}M_{\mathrm{in}(i)}.
\]
The dimension of the $\Hom$-space from $T_i(M)$ to $S_i$ equals the dimension of the cokernel of $(T_i(M))_{\mathrm{in}(i)}$; hence, the right-hand side in the statement is $\dim_K\operatorname{im}M_{\mathrm{out}(i)}M_{\mathrm{in}(i)}$. The lemma then follows because $M_{\mathrm{in}(i)}$ is a surjective map.
\end{proof}

\section{The case of Q-data}\label{section:the case of Q-data}

We now specialize to the case of a commutation class $[\Q]$ coming from a Q-datum $\Q$. In the first subsection, we show how to see the generalized twisted Coxeter element $\tau_{\Q}$ as an autoequivalence of $\R([\Q])$ and $\D([\Q])$. In the subsequent section, we prove that $\D([\Q])$ has a kind of partial Serre duality (Theorem \ref{thm:Serre duality}). Finally, we apply our results to reinterpret a formula for computing inverse quantum Cartan matrices due to \cite{FujitaOh}.

\begin{remark}
To simplify the notation, we will replace $[\Q]$ for $\Q$ in most of the symbols introduced before. For example, we shall write $\C([\Q])$, $M^{[\Q]}_{\alpha}$, and $P^{[\Q]}_{\im}$ as $\C(\Q)$, $M^{\Q}_{\alpha}$, and $P^{\Q}_{\im}$. We will also adopt the convention from Remark \ref{rmk:notation from FO} for notating vertices of $\Delta$.

Moreover, instead of seeing the coordinate function $c_{\Q}$ as a map from $\ind(\R(\Q))$ to $\widehat{\mathsf{R}}$, we will use the bijection $\phi_{\Q}$ from Section \ref{subsection:phiQ and g-additive} to view it as a function $c_{\Q}: \ind(\R(\Q)) \longrightarrow \widehat{\Delta}^{\sigma}_0$. In this case, we remark that the suspension functor sends the indecomposable object with coordinate $(\im,p) \in \widehat{\Delta}^{\sigma}_0$ to the indecomposable object with coordinate $(\im^*,p+rh^{\vee})$ (see \cite[Corollary 3.40]{FujitaOh}). We also extend the terminology from Section \ref{section:Q-data combinatorics} and define the \emph{height of} $M \in \ind(\R(\Q))$ as the height of $c_{\Q}(M)$.
\end{remark}

\subsection{A categorification of the generalized twisted Coxeter element}\label{subsection:tau categorified} Let $\Q = (\Delta,\sigma,\xi)$ be a Q-datum. Recall from Section \ref{subsection:twisted Coxeter elements} the definition of $X_{\Q}^{\circ}$ and $X_{\Q}'$. Choose compatible readings $(x_1,\dots,x_n)$ and $(y_1,\dots,y_m)$ for them and define
\[
\tau_{\Q}^{\circ} = T_{\im_1}T_{\im_2} \dotsb T_{\im_n} \sigma,
\]
and
\[
T[X_{\Q}'] = T_{\jm_1}T_{\jm_2} \dotsb T_{\jm_m},
\]
where $\im_k = \pi(x_k)$ and $\jm_l = \pi(y_l)$ (we remind that $\pi: \widehat{\Delta}^{\sigma}_0 \to \Delta_0$ is the projection onto the first coordinate). By the same reasoning as in \cite[Lemma 3.18]{FujitaOh}, $\tau_{\Q}^{\circ}$ and $T[X_{\Q}']$ are independent of the choice of compatible reading. We set $\tau_{\Q}: \pvd(\Pi_{Q}) \longrightarrow \pvd(\Pi_{Q})$ as the composition
\[
\tau_{\Q} = T[X_{\Q}']^{-1} \circ \tau_{\Q}^{\circ} \circ T[X_{\Q}'],
\]
which is an autoequivalence of $\pvd(\Pi_{Q})$.

\begin{remark}
Note that $\tau_{\Q}^{\circ}$ and $\tau_{\Q}$ can either denote an element of the coset $\mathsf{W}\sigma$ or an autoequivalence of $\pvd(\Pi_Q)$. The context will always allow us to resolve this ambiguity.
\end{remark}

\begin{lemma}\label{lemma:tau respects reflection}
If $\im \in \Delta_0$ is a source of $\Q$, then $T_{\im}^{-1}\tau_{\Q}T_{\im} \cong \tau_{s_{\im}\Q}$. Consequently, we have $T\tau_{\Q}T^{-1} \cong \tau_{\Q'}$ for any tilting functor $T: \R(\Q) \longrightarrow \R(\Q')$.
\end{lemma}

\begin{proof}
Since the spherical twists satisfy the same commutation relations as the reflections $s_{\jm}$ ($\jm \in \Delta_0$), the proof is essentially the same as the one in \cite[Proposition 3.34]{FujitaOh}.  The only care that we should take into consideration is that, in \cite{FujitaOh}, all appearances of $s_{\jm}^{-1}$ are replaced by $s_{\jm}$ (we cannot do this on the level of spherical twists), but this replacement is not necessary for the proof.
\end{proof}

\begin{prop}\label{prop:tau preserve the indecomposables associated with Q}
Let $\Q = (\Delta,\sigma,\xi)$ be a Q-datum. Then $\tau_{\Q}: \pvd(\Pi_Q) \longrightarrow \pvd(\Pi_Q)$ can be restricted to an autoequivalence of the repetition category $\R(\Q)$. It also preserves the ideal $\I_{\Q}$ and induces an autoequivalence of $\D(\Q)$. Moreover, it sends the indecomposable object with coordinate $(\im,p) \in \widehat{\Delta}^{\sigma}$ to the indecomposable object with coordinate $(\sigma(\im),p-2)$.
\end{prop}

\begin{proof}
Suppose first that $\sigma = \id$. In this case, $\tau_{\Q} = T_{\im_1}\dotsb T_{\im_n}$ where $(\im_1,\dots,\im_n)$ is a source sequence for $\Q$ where each vertex appears exactly once. Thus, $\tau_{\Q}$ restricts to a tilting functor from $\R(\Q')$ to $\R(\Q)$, where $\Q' = s_{\im_n}\dotsb s_{\im_1}\Q$. The height function defining $\Q'$ is a translation by $-2$ of the height function defining $\Q$. In particular, $\Q$ and $\Q'$ represent the same Dynkin quiver and we have $\R(\Q) = \R(\Q')$ (but the coordinate functions differ by $-2$). Therefore, all three statements above are proved.

From now on, suppose $\sigma \neq \id$. After applying a suitable tilting functor and using Lemma \ref{lemma:tau respects reflection}, we can assume that $\Q = \Q^{\circ}$. Let $(x_1,\dots,x_n)$ be a compatible reading of $X^{\circ}_{\Q}$ and denote $\im_k = \pi(x_k)$ for $1 \leq k \leq n$. By \cite[Proposition 3.31]{FujitaOh}, we have $\bm{i} \in [\Q]$ where
\[
\bm{i} = (\im_1, \dots, \im_n, \sigma(\im_1), \dots, \sigma(\im_n), \dots\dots,
\sigma^{rh^{\vee}/2-1}(\im_1), \ldots, \sigma^{rh^{\vee}/2-1}(\im_n)).
\]
By the definition of $[\Q]$, we deduce that $(\im_1,\dots,\im_n)$ is a source sequence for $\Q$. Denote $\Q^{\sigma} = s_{\im_n}\dotsb s_{\im_1}\Q$. Hence, we have $\bm{i}^{\sigma} \in [\Q^{\sigma}]$ for
\[
\bm{i}^{\sigma} = (\sigma(\im_1), \dots, \sigma(\im_n), \dots\dots,
\sigma^{rh^{\vee}/2-1}(\im_1), \ldots, \sigma^{rh^{\vee}/2-1}(\im_n),\im_1^*,\dots,\im_n^*).
\]
By \cite[Remark 3.1]{FujitaOh}, we have $\sigma^{rh^{\vee}/2}(\im_k) = \im_k^*$ for $1 \leq k \leq n$, so $\bm{i}^{\sigma}$ is obtained from $\bm{i}$ by applying $\sigma$ to its terms. We deduce from item (4) in Proposition \ref{prop:properties of spherical twists} that $\sigma$, now viewed as an autoequivalence of $\pvd(\Pi_Q)$, restricts to a functor $\R(\Q) \longrightarrow \R(\Q^{\sigma})$. On the other hand, by Proposition \ref{prop:tilting for the repetitive category, general case}, the composition $T_{\im_1}\dotsb T_{\im_n}$ restricts to a tilting functor $\R(\Q^{\sigma}) \longrightarrow \R(\Q)$. Since $\tau_{\Q} = T_{\im_1}\dotsb T_{\im_n}\sigma$, this proves the first statement. For the second statement, by Proposition \ref{prop:I is an ideal of the repetition category}, it suffices to prove that $\sigma: \R(\Q) \longrightarrow \R(\Q^{\sigma})$ sends $\I_{\Q}$ to $\I_{\Q^{\sigma}}$, which follows from item (4) in Proposition \ref{prop:properties of spherical twists} and the fact that $\sigma$ sends objects concentrated in degree $0$ to objects concentrated in degree $0$.

For the last part, since tilting functors preserve coordinates, we may assume that $\Q = \Q^{\circ}$, $p = \xi_{\im}$, and $\im$ starts the reduced word $\bm{i}$ from the previous paragraph. In this case, the indecomposable object of $\R(\Q)$ with coordinate $(\im,p)$ is $S_{\im}$, and the functor $\sigma$ sends it to $S_{\sigma(\im)}$. Since $S_{\sigma(\im)} \cong M^{\bm{i}^{\sigma}}_1$, this object is injective in $\C(\Q^{\sigma})$ and its coordinate in $\R(\Q^{\sigma})$ is $(\sigma(\im),\xi^{\sigma}_{\sigma(\im)})$, where $\xi^{\sigma}$ denotes the height function defining $\Q^{\sigma}$. Since $\Q = \Q^{\circ}$, one can easily check that $\xi^{\sigma}_{\sigma(\im)} = \xi_{\im} - 2 = p-2$. Since $\tau_{\Q}$ is the composition of $\sigma$ with a tilting functor, we conclude that the coordinate of $\tau_{\Q}(S_{\im})$ is also $(\sigma(\im),p-2)$, as desired.
\end{proof}

\begin{cor}\label{cor:power of tau sends projective to shift of injective}
Let $\Q = (\Delta, \sigma, \xi)$ be a Q-datum and take $M \in \ind(\Q)$ of residue $\im \in \Delta_0$. Then $\Sigma^k\tau_{\Q}^{d_{\overline{\im}}}(M) \in \C(\Q)$ for some $k \in \{0,1\}$. Moreover, we have $k = 1$ if and only if $M$ is projective and, in this case,
\[
\Sigma\tau_{\Q}^{d_{\overline{\im}}}(M) \cong I^{\Q}_{\im^*}.
\]
\end{cor}

\begin{proof}
Let $(\im,p) \in \widehat{\Delta}^{\sigma}$ be the coordinate of $M$. By Proposition \ref{prop:tau preserve the indecomposables associated with Q}, the coordinate of $\tau_{\Q}^{d_{\overline{\im}}}(M)$ is $(\im,p-2d_{\overline{\im}})$. We deduce that $\tau_{\Q}^{d_{\overline{\im}}}(M) \in \C(\Q)$ if $M$ is not projective. On the other hand, if $M$ is projective, then $M = P^{\Q}_{\im^*}$ and we have $p = \xi_{\im^*} - rh^{\vee} + 2d_{\overline{\im}}$ by Lemma \ref{lemma:height of projective vertex}. Thus, $\tau_{\Q}^{d_{\overline{\im}}}(M)$ has coordinate $(\im,\xi_{\im^*} - rh^{\vee})$, which is also the coordinate of $\Sigma^{-1}I^{\Q}_{\im^*}$.
\end{proof}

\begin{cor}\label{cor:using tau to find the repetition category}
Let $\Q = (\Delta, \sigma, \xi)$ be a Q-datum. If $M \in \ind(\R(\Q))$ has coordinates $(\im,p) \in \widehat{\Delta}^{\sigma}_0$, then
\[
M \cong \tau_{\Q}^{(\xi_{\im} - p)/2}(I^{\Q}_{\im}).
\]
\end{cor}

\begin{proof}
It is immediate since both sides have the same coordinate (see Proposition \ref{prop:tau preserve the indecomposables associated with Q}).
\end{proof}

\subsection{Partial Serre duality} Let $[\bm{i}]$ be a commutation class of reduced words for $w_0$. If $M \in \C([\bm{i}])$, recall that $M$ can be seen as a representation of the double quiver $\overline{Q}$ by Theorem \ref{thm:indecomposables are concentrated in zero} and we describe it around the vertex $i \in \Delta_0$ by a diagram
\[\begin{tikzcd}
	{\widetilde{M}(i)} & {M(i)} & {\widetilde{M}(i)}
	\arrow["{M_{\mathrm{in}(i)}}", from=1-1, to=1-2]
	\arrow["{M_{\mathrm{out}(i)}}", from=1-2, to=1-3]
\end{tikzcd}\]
(see Sections \ref{subsection:alternative descriptions of spherical twists} and \ref{subsection:projective/injective objects}). We say that $i$ is a \emph{source in $M$} if $M_{\mathrm{in}(i)} = 0$. Dually, $i$ is a \emph{sink in $M$} if $M_{\mathrm{out}(i)} = 0$.

\begin{lemma}\label{lemma:source of Q gives source of the representation}
Let $\Q = (\Delta, \sigma, \xi)$ be a Q-datum. Take $\im \in \Delta_0$ and let $M \in \ind(\Q)$ be an indecomposable object of residue $\jm \in \Delta_0$. If $\im$ is a source (resp. sink) of $\Q$, then $\im$ is a source (resp. sink) in $M$ if one of the following conditions hold:
\begin{enumerate}[(1)]
	\item $d_{\overline{\im}} = r$;
	\item $\Q$ is of type $\mathsf{B}_n$ and $\im = \jm = n$;
	\item $\Q$ is of type $\mathsf{C}_n$ and $\im + \jm \leq n$;
	\item $\Q$ is of type $\mathsf{F}_4$ and $\im = \jm = 6$.
\end{enumerate}
\end{lemma}

\begin{proof}
If we know the result for the case when $\im$ is a source, then we can deduce it for when $\im$ is a sink, and conversely. By Proposition \ref{prop:computing the spherical twist}, one statement follows from the other one after applying a reflection as in Proposition \ref{prop:equivalence under reflection, general case}. With this in mind, let us consider each of the cases above separately.

\begin{enumerate}[(1)]
	\item Assume $d_{\overline{\im}} = r$. For this case, we suppose that $\im$ is a source of $\Q$. We argue for each type separately, the simply laced case being immediate by Proposition \ref{prop:concentrated in zero for adapted word}.

	Suppose that $\Q$ is of type $\mathsf{B}_n$. By Proposition \ref{prop:orientations that appear in Bn}, there are two quivers $Q^{\leq n}$ and $Q^{\geq n}$ orienting $\Delta = \mathsf{A}_{2n-1}$ such that each indecomposable object of $\C(\Q)$ is a representation of either $Q^{\leq n}$ or $Q^{\geq n}$. Since $d_{\overline{\im}} = r$, $\im$ is not the central vertex of $\Delta$. Hence, if $\im$ is a source of $\Q$, it is not hard to see that it is also a source of $Q^{\leq n}$ and $Q^{\geq n}$, implying our claim.

	Now, suppose that $\Q$ is of type $\mathsf{C}_n$. Assume first that $\im$ is the vertex $n$ in $\Delta = \mathsf{D}_{n+1}$. By applying reflections, we may assume that $\Q$ is the Q-datum of Example \ref{example:description of the objects in type Cn}. A quick inspection of the list of objects given by this example shows that $\im$ is a source in any indecomposable object of $\C(\Q)$, as desired. If $\im \neq n$, since $d_{\overline{\im}} = r$, we must have $\im = n+1$. But $\sigma(n+1) = n$, therefore, by applying the induced functor $\sigma: \pvd(\Pi_Q) \longrightarrow \pvd(\Pi_Q)$, we are reduced to the previous case.

	Finally, we can check the lemma for types $\mathsf{F}_4$ and $\mathsf{G}_2$ using Examples \ref{example:F4} and \ref{example:G2}, respectively. Although these examples do not cover all possibilities for the source $\im$, we can obtain the remaining cases by applying $\sigma$ as we did in the previous paragraph.

	\item If $\Q$ is of type $\mathsf{B}_n$ and $M$ is of residue $n$, then Proposition \ref{prop:orientations that appear in Bn} says that we can view $M$ as a representation of both quivers $Q^{\leq n}$ and $Q^{\geq n}$. There is a unique edge of $\Delta$ where these orientations differ. Hence, the map in the representation that corresponds to such an edge must be zero. This fact implies without difficulty that if $n$ is a source (resp. sink) of $\Q$, then $n$ is a source (resp. sink) in any indecomposable object of $\C(\Q)$ that has residue $n$.
	
	\item Suppose that $\Q$ is of type $\mathsf{C}_n$ and that $\im$ is a sink of $\Q$. Assume $\im + \jm \leq n$, so that, in particular, we have $\im,\jm \leq n-1$ and $d_{\overline{\im}} = d_{\overline{\jm}} = 1$. If $\Q'$ denotes the Q-datum of Example \ref{example:description of the objects in type Cn}, then we may assume that $\Q = s_{\im-1}^{-1}s_{\im-2}^{-2}\dotsb s_1^{-1}\Q'$ after possibly applying some reflections and the automorphism $\sigma$ to $\Q$. Thus, by Proposition \ref{prop:tilting for the repetitive category, general case}, we can compute the objects of residue $\jm$ in $\ind(\Q)$ by applying the functor $T_{\im-1}T_{\im-2}\dotsb T_1$ to those in $\ind(\Q')$ and then a shift if necessary. By the description of the action of $\tau_{\Q'}$ in Example \ref{example:description of the objects in type Cn}, we see that the objects of residue $\jm$ in $\ind(\Q')$ are precisely
	
	\noindent\begin{minipage}{\linewidth}
	\[
	A^{1,0}_{\jm}, A^{1,1}_{\jm - 1}, B_{\jm-2,n-1}, B_{\jm-3,n-2}, \dots, B_{1,n-\jm+2}, A^{0,1}_{n-\jm+1}, A^{0,0}_{n-\jm,n-1}, A^{0,0}_{n-\jm-1,n-2}, \dots, A^{0,0}_{1,\jm},
	\]
	\end{minipage}
		
	\vspace{0.5em}
	\noindent where the objects of the form $B_{\im',\jm'}$ do not appear if $\jm \leq 2$ and the object $A^{1,1}_{\jm-1}$ does not appear if $\jm = 1$. Computing the image of the objects above under $T_{\im-1}T_{\im-2}\dotsb T_1$ using Proposition \ref{prop:computing the spherical twist}, it is not hard to see that $\im$ is a sink in any object of residue $\jm$ in $\ind(\Q)$. The crucial observation here is that the inequality $\im + \jm \leq n$ guarantees that we have $\jm' \geq \im + 2$ in all objects of the form $B_{\im',\jm'}$ above. In particular, when computing the spherical twists and checking whether $\im$ is a sink, we may ignore the vertices where $B_{\im',\jm'}$ has dimension $2$. The inequality also forces $\jm = 1$ when $\im = n-1$, so that $A^{1,1}_{\jm-1}$ is not taken into account in this case.

	\item Suppose $\Q$ is of type $\mathsf{F}_4$. If $6$ is a sink of $\Q$, then we may assume that $\Q$ is the Q-datum of Example \ref{example:F4} after applying reflections. One can then directly check that $6$ is a sink in any indecomposable object of residue $6$.
\end{enumerate}
\vspace{-1.5em}
\end{proof}

\begin{remark}\label{rmk:not always source or sink when expected}
The lemma is not true if we remove the conditions (1)--(4). Counterexamples can be easily found using the explicit descriptions in Appendix \ref{appendix:examples}. For instance, in Example \ref{example:G2}, the vertex $\im = 2$ is a sink of the given Q-datum, but it is not a sink in the indecomposable object with coordinate $(2,0)$. Note that $d_{\overline{\im}} = 1$ in this case.
\end{remark}

\begin{prop}\label{prop:projectives and injectives gives M(i)}
Let $\Q = (\Delta, \sigma, \xi)$ be a Q-datum and take $\im \in \Delta_0$. If $d_{\overline{\im}} = r$, then there are isomorphisms
\[\begin{tikzcd}
	{\Hom_{\C(\Q)}(P^{\Q}_{\im}, M)} & {M(\im)} & {D\Hom_{\C(\Q)}(M,I^{\Q}_{\im})}
	\arrow["\sim", from=1-1, to=1-2]
	\arrow["\sim"', from=1-3, to=1-2]
\end{tikzcd}\]
natural in the variable $M \in \C(\Q)$, where $D$ denotes the duality functor for $K$-vector spaces.
\end{prop}

\begin{proof}
We consider only the second natural isomorphism above. The proof for the first is similar.

Let $X \subset \ind(\Q)$ be the set of indecomposable objects $N$ satisfying $N \preceq_{\C(\Q)} I^{\Q}_{\im}$ and $I^{\Q}_{\im} \not\cong N$. Let $(\im_1,\dots,\im_l)$ be the residues of a compatible reading of the coordinates of the objects in $X$. Notice that $\im_k \neq \im$ for all $1 \leq k \leq l$ since $I_{\im}^{\Q}$ is injective. By Proposition \ref{prop:equivalence under reflection, general case}, the autoequivalence $T = T_{\im_l}^{-1}\dotsb T_{\im_2}^{-1}T_{\im_1}^{-1}$ of $\pvd(\Pi_{Q})$ restricts to a fully faithful functor from the full additive subcategory of $\C(\Q)$ generated by $\ind(\Q) \setminus X$ to $\C(\Q')$, where $\Q' = s_{\im_l}\dotsb s_{\im_2}s_{\im_1}\Q$. It is easy to see that $\im$ is a source for $\Q'$ and $T(I_{\im}^{\Q}) \cong S_{\im} \in \C(\Q')$.

We can assume that $M \in \ind(\Q)$ since both sides are additive on the variable $M$. If $M \not\in X$, then $T$ induces an isomorphism
\[\begin{tikzcd}
	{\Hom_{\C(\Q)}(M,I^{\Q}_{\im})} & {\Hom_{\C(\Q')}(T(M), S_{\im}).}
	\arrow["\sim", from=1-1, to=1-2]
\end{tikzcd}\]
By Lemma \ref{lemma:source of Q gives source of the representation}, $\im$ is a source in $T(M)$ so, if we regard $T(M)$ and $S_{\im}$ as representations of the double quiver $\overline{Q}$, there is an isomorphism
\[\begin{tikzcd}
	{\Hom_{\C(\Q')}(T(M), S_{\im})} & {\Hom_K(T(M)(\im), K) = D(T(M)(\im)).}
	\arrow["\sim", from=1-1, to=1-2]
\end{tikzcd}\]
Finally, since $\im_k \neq \im$ for all $k$, the spherical twists $T_{\im_k}^{-1}$ do not change $M$ at the vertex $\im$ (see Proposition \ref{prop:computing the spherical twist}), hence we have an isomorphism $D(T(M)(\im)) \xlongrightarrow[]{\sim} D(M(\im))$. Composing the inverse of these isomorphisms and applying the duality functor, we get an isomorphism from $D\Hom_{\C(\Q)}(M,I^{\Q}_{\im})$ to $M(\im)$.

On the other hand, if $M \in X$, then $\Hom_{\C(\Q)}(M,I^{\Q}_{\im}) = 0$ since $I^{\Q}_{\im} \not\preceq_{\C(\Q)} M$. There also exists $1 \leq k \leq l$ such that
\[
M \cong T_{\im_1}T_{\im_2}\dotsb T_{\im_{k-1}}(S_{\im_k}),
\]
which implies that $M(\im) = 0$ because $\im_{k'} \neq \im$ for all $1 \leq k' \leq l$. Consequently, $D\Hom_{\C(\Q)}(M,I^{\Q}_{\im})$ is trivially isomorphic to $M(\im)$.

We have thus defined the isomorphism in the statement. It remains to check that it is natural in the variable $M$. For this, it is enough to show that the square induced by a morphism $M \longrightarrow M'$ between indecomposable objects is commutative. One can do this on a case-by-case analysis by considering whether $M$ and $M'$ are in $X$ or not.
\end{proof}

To extend the previous result to the whole c-derived category $\D(\Q)$, we define $S_{\Q}: \D(\Q) \longrightarrow \D(\Q)$ to be the equivalence $S_{\Q} = \Sigma \circ \tau_{\Q}^r$. It behaves like a Serre functor:

\begin{thm}\label{thm:Serre duality}
Let $\Q = (\Delta, \sigma, \xi)$ be a Q-datum and take $M \in \ind(\D(\Q))$ of residue $\im \in \Delta_0$ satisfying $d_{\overline{\im}} = r$. There is an isomorphism
\[\begin{tikzcd}
	{\Hom_{\D(\Q)}(M, N)} & {D\Hom_{\D(\Q)}(N,S_{\Q}M)}
	\arrow["\sim", from=1-1, to=1-2]
\end{tikzcd}\]
natural in the variable $N \in \D(\Q)$, where $D$ denotes the duality functor for $K$-vector spaces.
\end{thm}

\begin{proof}
By applying a tilting functor and using Lemma \ref{lemma:tau respects reflection}, we may assume that $M$ is projective. In this case, by Corollary \ref{cor:power of tau sends projective to shift of injective}, $S_{\Q}M$ is the injective object corresponding to the same vertex. By Proposition \ref{prop:projective/injective if and only if Ext1 = 0}, the $\Hom$-spaces above are zero if $N$ is indecomposable and is not in $\C(\Q)$. Thus, we may assume that $N$ is in $\C(\Q)$, and then the theorem follows from Proposition \ref{prop:projectives and injectives gives M(i)}.
\end{proof}

\begin{remark}
If $N \in \ind(\D(\Q))$ has residue $\im \in \Delta_0$, then $S_{\Q}^{-1}N$ has residue $\im^*$. Thus, if $d_{\overline{\im}} = d_{\overline{\im^*}} = r$, we have isomorphisms
\[
\Hom_{\D(\Q)}(M,N) \cong \Hom_{\D(\Q)}(M,S_{\Q}(S_{\Q}^{-1}N)) \cong D\Hom_{\D(\Q)}(S_{\Q}^{-1}N,M) \cong D\Hom_{\D(\Q)}(N,S_{\Q}M),
\]
natural in the variable $M \in \D(\Q)$. This gives another form of Theorem \ref{thm:Serre duality}.
\end{remark}

\begin{remark}
In the nonsimply laced case, one may also consider the autoequivalence $S'_{\Q}: \D(\Q) \longrightarrow \D(\Q)$ given by $S'_{\Q} = \Sigma \circ \tau_{\Q}$. We can adapt the previous proofs to construct an isomorphism
\[\begin{tikzcd}
	{\Hom_{\D(\Q)}(M, N)} & {D\Hom_{\D(\Q)}(N,S'_{\Q}M)}
	\arrow["\sim", from=1-1, to=1-2]
\end{tikzcd}\]
for $M,N \in \ind(\Q)$ with residues $\im,\jm \in \Delta_0$ satisfying one of the conditions (2), (3), or (4) in Lemma \ref{lemma:source of Q gives source of the representation}. Notice that in these cases, we have $d_{\overline{\im}} = d_{\overline{\jm}} = 1$.
\end{remark}

For the next result, we identify $K_0(\pvd(\Pi_Q))$ with the root lattice $\mathsf{Q}$ of $\Delta$ and denote the class of $M \in \pvd(\Pi_Q)$ in $K_0(\pvd(\Pi_Q))$ as $[M]_{\pvd(\Pi_{Q})}$ (see Section \ref{subsection:categorification of the root system}).

\begin{prop}\label{prop:getting M(i) using Euler form}
Let $\Q = (\Delta, \sigma, \xi)$ be a Q-datum and take $M \in \ind(\D(\Q))$ of residue $\jm \in \Delta_0$. For $\im \in \Delta_0$, if we have $\max\{d_{\overline{\im}},d_{\overline{\jm}}\} = r$ or if one of the conditions (2), (3), or (4) in Lemma \ref{lemma:source of Q gives source of the representation} holds, then
\[
([M]_{\pvd(\Pi_{Q})},\varpi_{\im}) = \left\langle M,\bigoplus_{k = 0}^{\lceil d_{\overline{\jm}}/d_{\overline{\im}}\rceil - 1}\tau_{\Q}^k(I^{\Q}_{\im})\right\rangle_{\!\!\Q} = \left\langle\bigoplus_{k = 0}^{\lceil d_{\overline{\jm}}/d_{\overline{\im}}\rceil - 1}\tau_{\Q}^{-k}(P^{\Q}_{\im}), M\right\rangle_{\!\!\Q}.
\]
\end{prop}

\begin{proof}
By shifting $M$, we can suppose that $M \in \C(\Q)$. If $d_{\overline{\im}} = r$, the result follows from Propositions \ref{prop:projective/injective if and only if Ext1 = 0} and \ref{prop:projectives and injectives gives M(i)}. We can similarly prove the result if one of the conditions (2), (3), or (4) in Lemma \ref{lemma:source of Q gives source of the representation} holds by modifying the proof of the latter proposition. From now on, we additionally suppose that $d_{\overline{\im}} = 1$ and $d_{\overline{\jm}} = r > 1$. Once we prove the first equality, the second will follow from Theorem \ref{thm:Serre duality} and the fact that $P^{\Q}_{\im} \cong \Sigma^{-1}\tau_{\Q}^{-1}(I^{\Q}_{\im})$ (see Corollary \ref{cor:power of tau sends projective to shift of injective}).

We claim that $\Ext_{\Q}^l(M, \tau_{\Q}^k(I^{\Q}_{\im})) = 0$ for $0 \leq k < r$ and $l > 0$. By Proposition \ref{prop:derived category is hereditary}, we only need to check the case $l = 1$. By Theorem \ref{thm:Serre duality}, we have
\[
\Ext_{\Q}^1(M, \tau_{\Q}^k(I^{\Q}_{\im})) \cong \Hom_{\D(\Q)}(\Sigma^{-1}M, \tau_{\Q}^k(I^{\Q}_{\im})) \cong D\Hom_{\D(\Q)}(\tau_{\Q}^k(I^{\Q}_{\im}), \tau_{\Q}^r(M)).
\]
Hence, by Lemma \ref{lemma:morphisms in the repetition category follow paths}, it is enough to show that there is no path from the coordinate of $\tau_{\Q}^k(I^{\Q}_{\im})$ to the coordinate of $\tau_{\Q}^r(M)$ in $\widehat{\Delta}^{\sigma}$. If $p$ is the height of $M$, it is not hard to check that $p \leq \xi_{\im} + d_{\Delta}^{\sigma}(\im,\jm)$ by the definition of a height function and the fact that $d_{\overline{\im}} = 1$. Keeping in mind that the height of $\tau_{\Q}^k(I^{\Q}_{\im})$ is $\xi_{\im} - 2k$ and that the height of $\tau_{\Q}^r(M)$ is $p - 2r$, the inequality
\[
(p - 2r) - (\xi_{\im} - 2k) = (p - \xi_{\im}) + 2(k - r) < d_{\Delta}^{\sigma}(\im,\jm)
\]
ensures what we need. Therefore, the first equality of the statement simplifies to
\[
\dim_K M(\im) = \dim_K\Hom_{\C(\Q)}\left(M,\bigoplus_{k = 0}^{r-1}\tau_{\Q}^k(I^{\Q}_{\im})\right).
\]

Assume first that $M \not\preceq_{\C(\Q)} \tau_{\Q}^{r-1}(I^{\Q}_{\im})$. In this case, there is a sink sequence $(\im_1,\dots,\im_l)$ of $\Q$ such that the induced tilting functor $T: \D(\Q) \longrightarrow \D(\Q')$ satisfies $T(M) \cong P^{\Q'}_{\jm^*}$ and $T(\tau^k_{\Q}(I^{\Q}_{\im})) \in \C(\Q')$ for $0 \leq k < r$. Hence, if we denote $w = s_{\im_l}\dotsb s_{\im_2}s_{\im_1} \in \mathsf{W}$, we have
\begin{align*}
\dim_K\Hom_{\C(\Q)}\left(M,\bigoplus_{k = 0}^{r-1}\tau_{\Q}^k(I^{\Q}_{\im})\right) &= \dim_K\Hom_{\C(\Q')}\left(P^{\Q'}_{\jm^*},\bigoplus_{k = 0}^{r-1}T(\tau_{\Q}^k(I^{\Q}_{\im}))\right)\\
&= \left(\varpi_{\jm^*}, \sum_{k = 0}^{r-1}w\tau_{\Q}^k(\gamma^{\Q}_{\im})\right),
\end{align*}
where we used Proposition \ref{prop:projectives and injectives gives M(i)} for the last equality. Since $d_{\overline{\im}} = 1$, we have $\gamma^{\Q}_{\im} = (1-\tau_{\Q})\varpi_{\im}$, which gives
\[
\left(\varpi_{\jm^*}, \sum_{k = 0}^{r-1}w\tau_{\Q}^k(\gamma^{\Q}_{\im})\right) = \left(\varpi_{\jm^*}, w(1 - \tau_{\Q}^r)\varpi_{\im}\right) = \left((1-\tau_{\Q}^{-r})w^{-1}\varpi_{\jm^*}, \varpi_{\im}\right) = (w^{-1}(1-\tau_{\Q'}^{-r})\varpi_{\jm^*},\varpi_{\im}),
\]
where the last equality follows from Proposition \ref{prop:properties of Coxeter element}. By Lemma \ref{lemma:projective in terms of tau and the fundamental weight}, $(1-\tau_{\Q'}^{-r})\varpi_{\jm^*}$ represents the class of $P^{\Q'}_{\jm^*}$ in $K_0(\pvd(\Pi_{Q}))$, so the isomorphism $M \cong T^{-1}(P^{\Q'}_{\jm^*})$ tells us that the left term inside the last bilinear form above is the class of $M$ in $K_0(\pvd(\Pi_{Q}))$, as desired.

Suppose now that $M \preceq_{\C(\Q)} \tau_{\Q}(I^{\Q}_{\im})$. By our hypotheses, we have
\[
\Hom_{\C(\Q)}(M,\tau_{\Q}^k(I^{\Q}_{\im})) = 0
\]
for $1 \leq k < r$, so we need to prove that
\[
\dim_KM(\im) = \dim_K\Hom_{\C(\Q)}(M,I^{\Q}_{\im}).
\]
To do so, we can proceed exactly as in the proof of Proposition \ref{prop:projectives and injectives gives M(i)} by defining an analogous set $X \subset \ind(\Q)$ and a fully faithful functor $T$ from the full additive subcategory of $\C(\Q)$ generated by $\ind(\Q) \setminus X$ to some $\C(\Q')$ where $\im$ is a source of $\Q'$. The only part which we cannot immediately adapt is that $\im$ is a source in $T(M)$ when $M \not\in X$ (we cannot apply Lemma \ref{lemma:source of Q gives source of the representation}), but this follows from the additional hypothesis that $M \preceq_{\C(\Q)} \tau_{\Q}(I^{\Q}_{\im})$. Indeed, it implies that
\[
T(M) \cong T_{\im_1}\dots T_{\im_{k-1}}(S_{\im_k})
\]
for some source sequence $(\im_1,\dots,\im_k)$ for $\Q'$ where $\im$ appears at most once. Since $\im$ is a source of $\Q'$, we can assume $\im_1 = \im$, and then it is not hard to verify that $\im$ is a source in $T(M)$ using Proposition \ref{prop:computing the spherical twist}.

If $r = 2$, this finishes the proof. However, if $\Q$ is of type $\mathsf{G}_2$, it remains to consider the case when $M \preceq_{\C(\Q)} \tau_{\Q}^2(I^{\Q}_{\im})$ but $M \not\preceq_{\C(\Q)} \tau_{\Q}(I^{\Q}_{\im})$. In this case, $M$ has to be the unique indecomposable object with residue $\jm$ such that $d_{\overline{\jm}} = r$ and that satisfies $\tau_{\Q}(I^{\Q}_{\im}) \preceq_{\C(\Q)} M \preceq_{\C(\Q)} \tau_{\Q}^2(I^{\Q}_{\im})$. The formula of the statement can then be easily checked directly (see Example \ref{example:G2}).
\end{proof}

\subsection{Application: inverse quantum Cartan matrices}\label{subsection:applications}

Let $C = (c_{ij})_{i,j \in I}$ denote the Cartan matrix of the complex finite-dimensional simple Lie algebra $\mathfrak{g}$. Using the associated pair $(\Delta, \sigma)$, we have
\[
c_{ij} = \begin{cases}
2 &\textrm{if } i = j,\\
-\lceil d_j/d_i\rceil &\textrm{if } i \sim j,\\
0 &\textrm{otherwise,}
\end{cases}
\]
for $i,j \in I$. The \emph{quantum Cartan matrix of $\mathfrak{g}$} is a $\Z[q^{\pm1}]$-valued matrix $C(q) = (C_{ij}(q))_{i,j \in I}$, where $q$ is an indeterminate. It is defined by
\[
C_{ij}(q) = \begin{cases}
q^{d_i} + q^{-d_i} &\textrm{if } i = j,\\
\frac{q^{c_{ij}} - q^{-c_{ij}}}{q-q^{-1}} &\textrm{if } i \neq j,
\end{cases}
\]
for $i,j \in I$. Such a matrix is invertible if regarded as a matrix with entries in the field $\mathbb{Q}(\!(q)\!)$ of formal Laurent series over $\mathbb{Q}$. Let $\widetilde{C}(q) = (\widetilde{C}_{ij}(q))_{i,j \in I}$ be its inverse and denote by
\[
\widetilde{C}_{ij}(q) = \sum_{u \in \Z}\widetilde{c}_{ij}(u)q^u
\]
the formal Laurent expansion of the $(i,j)$-entry.

One of the main results in \cite{FujitaOh} is a combinatorial formula for the coefficients $\widetilde{c}_{ij}(u)$ above, generalizing a formula of Hernandez--Leclerc (cf. \cite{HernandezLeclerc}) to the nonsimply laced case.

\begin{thm}[{\cite[Theorem 4.8]{FujitaOh}}]\label{thm:Fujita-Oh formula}
For each $i,j \in I$ and $u \in \Z$, we have
\[
\widetilde{c}_{ij}(u) = \begin{cases}
(\varpi_{\im}, \tau_{\Q}^{(u+\xi_{\jm}-\xi_{\im}-d_i)/2}(\gamma^{\Q}_{\jm})) &\textrm{if } u \geq 0 \textrm{ and } u + \xi_{\jm} - \xi_{\im} - d_i \in 2\Z,\\
0 &\textrm{otherwise,}
\end{cases}
\]
where $\Q = (\Delta, \sigma, \xi)$ is a Q-datum for $\mathfrak{g}$ and $\im \in i$, $\jm \in j$.
\end{thm}

In \cite[Corollary 3.6]{Fujita}, Fujita reinterprets the formula above in the simply laced case using representations of a Dynkin quiver of the same type. Our next proposition generalizes his result and proof. Before we state it, we introduce some notation.

Following \cite[Section 3.8]{FujitaOh}, define
\[
\widehat{I} = \{(i,p) \in I \times \Z \mid \exists(\im,p) \in \widehat{\Delta}^{\sigma}_0,\overline{\im}=i\}.
\]
The map $f: \widehat{\Delta}^{\sigma}_0 \longrightarrow \widehat{I}$ given by projecting the first coordinate from $\Delta_0$ to $I$ is a bijection called the \emph{folding map}. For a Q-datum $\Q = (\Delta,\sigma,\xi)$, we have a bijection $H_{\Q}: \widehat{I} \longrightarrow \ind(\D(\Q))$ that sends $(i,p) \in \widehat{I}$ to the indecomposable object of $\D(\Q)$ with coordinate $f^{-1}(i,p)$.

\begin{prop}\label{prop:computing the inverse of the quantum cartan matrix}
Let $(i,p), (j,s) \in \widehat{I}$ be such that $p - s + d_i \geq 0$. Suppose one of the following conditions holds:
\begin{enumerate}[(1)]
	\item $\max\{d_i,d_j\} = r$;
	\item $\mathfrak{g}$ is of type $\mathsf{B}_n$ and $i = j = n$;
	\item $\mathfrak{g}$ is of type $\mathsf{C}_n$ and $i + j \leq n$;
	\item $\mathfrak{g}$ is of type $\mathsf{F}_4$ and $i = j = 4$.
\end{enumerate}
Then, for any Q-datum $\Q = (\Delta, \sigma, \xi)$ for $\mathfrak{g}$, we have
\[
\widetilde{c}_{ij}(p-s+d_i) = \left\langle H_{\Q}(j,s), \bigoplus_{k=0}^{\lceil d_j/d_i\rceil-1}\tau_{\Q}^k(H_{\Q}(i,p))\right\rangle_{\!\!\Q}.
\]
Moreover, by varying $p$ and $s$, the value $p-s+d_i$ covers all integers $u \in \Z$ such that $\widetilde{c}_{ij}(u) \neq 0$.
\end{prop}

\begin{proof}
Let $\im \in i$ and $\jm \in j$ be the residues of $H_{\Q}(i,p)$ and $H_{\Q}(j,s)$, respectively. Note that
\[
(p-s+d_i) + \xi_{\jm} - \xi_{\im} - d_i = (\xi_{\jm} - s) - (\xi_{\im} - p) \in 2\Z
\]
since $p$ is congruent to $\xi_{\im}$ modulo $2d_i$ and $s$ is congruent to $\xi_{\jm}$ modulo $2d_j$. By Theorem \ref{thm:Fujita-Oh formula}, we have
\[
\widetilde{c}_{ij}(p-s+d_i) = (\varpi_{\im}, \tau_{\Q}^{(\xi_{\jm} - s)/2 - (\xi_{\im} - p)/2}(\gamma^{\Q}_{\jm})).
\]
By Proposition \ref{prop:getting M(i) using Euler form}, this equals
\begin{align*}
\left\langle\tau_{\Q}^{(\xi_{\jm} - s)/2 - (\xi_{\im} - p)/2}(I^{\Q}_{\jm}), \bigoplus_{k = 0}^{\lceil d_j/d_i\rceil - 1}\tau_{\Q}^k(I^{\Q}_{\im})\right\rangle_{\!\!\Q} &= \left\langle\tau_{\Q}^{(\xi_{\jm} - s)/2}(I^{\Q}_{\jm}), \bigoplus_{k = 0}^{\lceil d_j/d_i\rceil - 1}\tau_{\Q}^k(\tau_{\Q}^{(\xi_{\im} - p)/2}(I^{\Q}_{\im}))\right\rangle_{\!\!\Q}\\
&= \left\langle H_{\Q}(j,s),\bigoplus_{k = 0}^{\lceil d_j/d_i\rceil - 1}\tau_{\Q}^k(H_{\Q}(i,p))\right\rangle_{\!\!\Q},
\end{align*}
where the last equality comes from Corollary \ref{cor:using tau to find the repetition category}. The last statement also follows from Theorem \ref{thm:Fujita-Oh formula} once we realize that $(i,p+2k) \in \widehat{I}$ for all $k \in \Z$ (see \cite[Section 3.8]{FujitaOh}).
\end{proof}

\begin{remark}
	By the formula above, the quasi-periodicity property of $\widetilde{c}_{ij}(u)$ described in \cite[Corollary 4.10(1)]{FujitaOh} arises from the behavior of the Euler form under shifts. Indeed, since $\Sigma H_{\Q}(j,s) = H_{\Q}(j^*,s+rh^{\vee})$ and $\langle \Sigma M,N \rangle_{\Q} = - \langle M,N \rangle_{\Q}$, we deduce that $\widetilde{c}_{ij}(u + rh^{\vee}) = -\widetilde{c}_{ij^*}(u)$ for $u \geq 0$. We also have the following interpretation for the positivity property described in \cite[Corollary 4.10(6)]{FujitaOh}. By applying a tilting functor that takes $H_{\Q}(i,p)$ to an injective object and writing the Euler form as an alternated sum, we see that there can be at most one term in this sum that does not vanish (the case where $d_i = 1$ and $d_j = r > 1$ was done in the proof of Proposition \ref{prop:getting M(i) using Euler form}). If we additionally suppose that $p-s+d_i \leq rh^{\vee}$, then such a term is nonnegative by \cite[Corollary 4.10(6)]{FujitaOh}. Thus, in case this term does not vanish, there must exist a tilting functor taking both objects inside the Euler form into some category $\C(\Q')$, otherwise we would need to apply an \emph{even} power of the shift functor to one of the objects for this to be true and we would get a contradiction to the inequality. We conclude that
	\[
	\widetilde{c}_{ij}(p-s+d_i) = \dim_K\Hom_{\D(\Q)}\left(H_{\Q}(j,s), \bigoplus_{k=0}^{\lceil d_j/d_i\rceil-1}\tau_{\Q}^k(H_{\Q}(i,p))\right).
	\]
	This equality is an analog of \cite[Proposition 3.8]{Fujita}.

	Finally, it is also interesting to remark that the direct sum appearing in the formula when $d_i = 1$ and $d_j = r > 1$ reflects the symmetry described in \cite[Lemma 4.4(2)]{FujitaOh}: in this case, if $u = p-s+1$ satisfies $0 \leq u \pm (r-1) \leq rh^{\vee}$, we have
	\begin{align*}
		\widetilde{c}_{ij}(u) &= \sum_{k=0}^{r-1}\langle H_{\Q}(j,s),\tau^k_{\Q}(H_{\Q}(i,p))\rangle_{\Q} = \sum_{k=0}^{r-1}\langle \tau^k_{\Q}(H_{\Q}(i,p)), S_{\Q}(H_{\Q}(j,s))
		\rangle_{\Q}\\
		&= \sum_{k=0}^{r-1}\langle H_{\Q}(i,p-2k), H_{\Q}(j,s-2r+rh^{\vee})\rangle_{\Q} = \sum_{k=0}^{r-1}\widetilde{c}_{ji}(rh^{\vee} - (u + (r-1) - 2k))\\
		&= \sum_{k=0}^{r-1}\widetilde{c}_{ji}(u + (r-1) - 2k),
	\end{align*}
	where we used Theorem \ref{thm:Serre duality} and \cite[Corollary 4.10(3)]{FujitaOh} in the second and fifth equalities, respectively.
\end{remark}

\begin{remark}
	The proposition above gives a categorical interpretation for all the entries of the inverse of the quantum Cartan matrix in simply laced types and in type $\mathsf{B}_n$. In the remaining types, there are values of $(i,p)$ and $(j,s)$ for which the formula does not give the desired result. This occurs due to the failure of Lemma \ref{lemma:source of Q gives source of the representation} when its hypotheses are removed (see Remark \ref{rmk:not always source or sink when expected}). However, it is interesting to point out a link with the conjectural formula \cite[Conjecture 6.7]{FujitaOh} for the denominator of the $R$-matrix between Kirillov--Reshetikhin modules over the quantum affine algebra (see also Conjecture 5.9 and Remark 5.10 in \cite{FujitaMurakami}). In \cite[Conjecture 6.15]{FujitaMurakami}, this conjectural formula is extended to all cases by adding an extra term $\Delta_{ij}(q)$, which is zero precisely when one of the conditions in Proposition \ref{prop:computing the inverse of the quantum cartan matrix} holds (see \cite[Proposition A.1]{FujitaMurakami}).
\end{remark}

In \cite{Hernandez}, the inverse of the quantum Cartan matrix appears in the definition of a skew-symmetric pairing $\mathscr{N}: \widehat{I} \times \widehat{I} \longrightarrow \Z$, an essential ingredient to construct a quantization of the Grothendieck ring of a quantum affine algebra. It is defined as
\[
\mathscr{N}(i,p;j,s) = \widetilde{c}_{ij}(p-s-d_i) - \widetilde{c}_{ij}(p-s+d_i) - \widetilde{c}_{ij}(s-p-d_i) + \widetilde{c}_{ij}(s-p+d_i).
\]
We can describe it using the c-derived category of a Q-datum.

\begin{prop}
For distinct $(i,p), (j,s) \in \widehat{I}$ with $s \geq p$, we have
\[
\mathscr{N}(i,p;j,s) = \langle H_{\Q}(i,p),H_{\Q}(j,s)\rangle_{\Q} + \langle H_{\Q}(j,s),H_{\Q}(i,p)\rangle_{\Q}
\]
for any Q-datum $\Q = (\Delta, \sigma, \xi)$ for $\mathfrak{g}$.
\end{prop}

\begin{proof}
The result follows from Theorem \ref{thm:symmetrization of Euler form is Cartan form} and \cite[Proposition 5.21]{FujitaOh}.
\end{proof}

\appendix

\section{Explicit descriptions in the case of Q-data}\label{appendix:examples}

In this appendix, we present some explicit descriptions of the indecomposable objects of $\C([\bm{i}])$ when $[\bm{i}]$ is a commutation class of reduced words $w_0$ coming from a Q-datum. They are important for the proof of Lemma \ref{lemma:source of Q gives source of the representation}. What follows depends mostly on Sections \ref{section:objects} and \ref{subsection:tau categorified}. In particular, we adopt the same conventions as at the end of Section \ref{section:indecomposable objects from commutation classes} when we deal with representations.

\begin{example}\label{example:description of the objects in type Bn}
Let us fix $(\Delta,\sigma) = (\mathsf{A}_{2n-1}, \vee)$ and let $\Q = (\Delta,\sigma,\xi)$ be a Q-datum of type $\mathsf{B}_n$. If we view $M \in \ind(\Q)$ as a representation of the double quiver $\overline{Q}$, then we have a vector space of dimension at most one on each vertex. Since $M$ satisfies the preprojective relations, one can easily deduce that $M$ is the restriction of a representation of some orientation $Q'$ of $\mathsf{A}_{2n-1}$ under the quotient map $K\overline{Q} \longrightarrow KQ'$. In this example, we will prove that there are two orientations $Q^{\leq n} = Q^{\leq n}(\Q)$ and $Q^{\geq n} = Q^{\geq n}(\Q)$ of $\mathsf{A}_{2n-1}$ such that any indecomposable object of $\C(\Q)$ comes from a representation of $Q^{\leq n}$ or $Q^{\geq n}$.

The definition of the orientations $Q^{\leq n}$ and $Q^{\geq n}$ depends on whether the height function $\xi$ satisfies $\xi_n = \xi_{n-1} + 1$ or $\xi_n = \xi_{n+1} + 1$, and is given as follows. Let $\im,\jm \in \Delta_0$ be such that $\im \sim \jm$.  If $\xi_n = \xi_{n-1} + 1$, then there is an arrow $\im \rightarrow \jm$ in $Q^{\leq n}$ if and only if $\xi_{\im} > \xi_{\jm}$, and $Q^{\geq n}$ can obtained from $Q^{\leq n}$ by reversing the direction of the arrow between the vertices $n-1$ and $n$. On the other hand, if $\xi_n = \xi_{n+1} + 1$, then there is an arrow $\im \rightarrow \jm$ in $Q^{\geq n}$ if and only if $\xi_{\im} > \xi_{\jm}$, and $Q^{\leq n}$ can obtained from $Q^{\geq n}$ by reversing the direction of the arrow between the vertices $n$ and $n+1$.

\begin{prop}\label{prop:orientations that appear in Bn}
Let $\Q = (\Delta,\sigma,\xi)$ be a Q-datum of type $\mathsf{B}_n$, and let $Q^{\leq n}$ and $Q^{\geq n}$ be the orientations of $\Delta = \mathsf{A}_{2n-1}$ defined above. Let $M \in \C(\Q)$ be an indecomposable object of residue $\im \in \Delta_0$ and view it as a representation of the double quiver $\overline{Q}$.
\begin{itemize}
	\item If $\im \leq n$, then $M$ is the restriction of an indecomposable representation of $Q^{\leq n}$.
	\item If $\im \geq n$, then $M$ is the restriction of an indecomposable representation of $Q^{\geq n}$.
\end{itemize}
\end{prop}

\begin{proof}
Without loss of generality, we assume that $\xi_n = \xi_{n-1} + 1$. Let us first show the following statement: if $\im \in \Delta_0 \setminus \{n\}$ is a source of $\Q$ and if $s_{\im}\Q$ satisfies the proposition, then so does $\Q$. By our hypothesis on $\xi$, we have $\im \neq n-1$. Note that $\im$ is also a source for the quivers $Q^{\leq n}(\Q)$ and $Q^{\geq n}(\Q)$, and we have $Q^{\leq n}(s_{\im}\Q) = s_{\im}Q^{\leq n}(\Q)$ and $Q^{\geq n}(s_{\im}\Q) = s_{\im}Q^{\geq n}(\Q)$. Thus, if we apply the spherical twist $T_{\im}$ to an indecomposable representation of $Q^{\leq n}(s_{\im}\Q)$ or $Q^{\geq n}(s_{\im}\Q)$ not isomorphic to $S_{\im}$ (viewed as an object of $\pvd(\Pi)$ in the canonical way), we get a representation of $Q^{\leq n}(\Q)$ or $Q^{\geq n}(\Q)$ by Remark \ref{rem:spherical twist gives reflection functor on the classical case}. Hence, by Proposition \ref{prop:equivalence under reflection, general case}, if $s_{\im}\Q$ satisfies the proposition above, so does $\Q$, as desired.

By the previous paragraph, after applying a source sequence, we may assume that $n$ is the unique source of $\Q$. Therefore, up to adding a constant, the height function $\xi$ is given as follows: 
\[
\xi_{\im} = \begin{cases}
	-d^{\sigma}_{\Delta}(\im,n) &\textrm{if } \im \leq n,\\
	-d^{\sigma}_{\Delta}(\im,n) - 2 &\textrm{if } \im > n.
\end{cases}
\]
The quivers $Q^{\leq n} = Q^{\leq n}(\Q)$ and $Q^{\geq n} = Q^{\geq n}(\Q)$ are
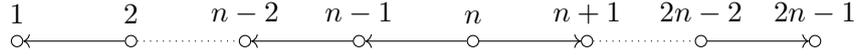
\begin{center}
\begin{tikzpicture}
	
	\node[circle, draw, inner sep=1.5pt] (A1) at (0,0) {};
	\node[above] at (0,0.1) {$1$};
	\node[circle, draw, inner sep=1.5pt] (A2) at (1.5,0) {};
	\node[above] at (1.5,0.1) {$2$};
	\node[circle, draw, inner sep=1.5pt] (Anm2) at (3,0) {};
	\node[above] at (3,0.1) {$n-2$};
	\node[circle, draw, inner sep=1.5pt] (Anm1) at (4.5,0) {};
	\node[above] at (4.5,0.1) {$n-1$};
	\node[circle, draw, inner sep=1.5pt] (An) at (6,0) {};
	\node[above] at (6,0.1) {$n$};
	\node[circle, draw, inner sep=1.5pt] (Anp1) at (7.5,0) {};
	\node[above] at (7.5,0.1) {$n+1$};
	\node[circle, draw, inner sep=1.5pt] (A2nm2) at (9,0) {};
	\node[above] at (9,0.1) {$2n-2$};
	\node[circle, draw, inner sep=1.5pt](A2nm1) at (10.5,0) {};
	\node[above] at (10.5,0.1) {$2n-1$};

	\draw[<-] (A1) -- (A2);
	\draw[dotted] (A2) -- (Anm2);
	\draw[<-] (Anm2) -- (Anm1);
	\draw[<-] (Anm1) -- (An);
	\draw[->] (An) -- (Anp1);
	\draw[dotted] (Anp1) -- (A2nm2);
	\draw[->] (A2nm2) -- (A2nm1);

\end{tikzpicture}
\end{center}
and
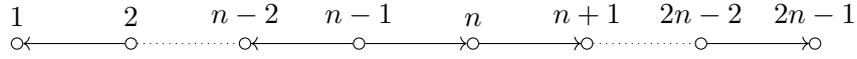
\begin{center}
\begin{tikzpicture}
	
	\node[circle, draw, inner sep=1.5pt] (A1) at (0,0) {};
	\node[above] at (0,0.1) {$1$};
	\node[circle, draw, inner sep=1.5pt] (A2) at (1.5,0) {};
	\node[above] at (1.5,0.1) {$2$};
	\node[circle, draw, inner sep=1.5pt] (Anm2) at (3,0) {};
	\node[above] at (3,0.1) {$n-2$};
	\node[circle, draw, inner sep=1.5pt] (Anm1) at (4.5,0) {};
	\node[above] at (4.5,0.1) {$n-1$};
	\node[circle, draw, inner sep=1.5pt] (An) at (6,0) {};
	\node[above] at (6,0.1) {$n$};
	\node[circle, draw, inner sep=1.5pt] (Anp1) at (7.5,0) {};
	\node[above] at (7.5,0.1) {$n+1$};
	\node[circle, draw, inner sep=1.5pt] (A2nm2) at (9,0) {};
	\node[above] at (9,0.1) {$2n-2$};
	\node[circle, draw, inner sep=1.5pt](A2nm1) at (10.5,0) {};
	\node[above] at (10.5,0.1) {$2n-1$};

	\draw[<-] (A1) -- (A2);
	\draw[dotted] (A2) -- (Anm2);
	\draw[<-] (Anm2) -- (Anm1);
	\draw[->] (Anm1) -- (An);
	\draw[->] (An) -- (Anp1);
	\draw[dotted] (Anp1) -- (A2nm2);
	\draw[->] (A2nm2) -- (A2nm1);

\end{tikzpicture}
\end{center}
respectively.

Note that $\Q^{\circ} = \Q$. Moreover, there is a unique compatible reading for $X_{\Q}^{\circ}$ and the corresponding sequence of residues is $(n,n-1,\dots,1)$. Hence, $\tau_{\Q} = T_nT_{n-1}\dotsb T_1\sigma$ and, by \cite[Proposition 3.31]{FujitaOh}, we have $\bm{i} \in [\Q]$ for
\[
\bm{i} = (n,n-1,\dots,1,\sigma(n),\sigma(n-1),\dots, \sigma(1),\dots\dots,\sigma^{h^{\vee}-1}(n),\sigma^{h^{\vee}-1}(n-1),\dots,\sigma^{h^{\vee}-1}(1)).
\]
By item (4) in Proposition \ref{prop:properties of spherical twists}, we deduce that $\tau_{\Q}(M^{\bm{i}}_k) \cong M^{\bm{i}}_{k+n}$ for $1 \leq k \leq N - n$. Observe that $M^{\bm{i}}_1,M^{\bm{i}}_2,\dots,M^{\bm{i}}_n$ have residue $n, n-1,\dots, 1$, respectively. In addition, by Proposition \ref{prop:tau preserve the indecomposables associated with Q}, if we repeatedly apply $\tau_{\Q}$ to an indecomposable object of $\C(\Q)$, the residue of the objects we get alternate between being less than $n$ and being greater than $n$. With these observations, we will have proved the proposition after we verify the following assertions.
\begin{enumerate}[(1)]
	\item The indecomposable objects $M^{\bm{i}}_1,\dots,M^{\bm{i}}_n$ are representations of $Q^{\leq n}$ and $M^{\bm{i}}_1$ is also a representation of $Q^{\geq n}$.
	\item Let $\star$ denote either $\leq$ or $\geq$. If $M \in \ind(\Q)$ is a representation of $Q^{\star n}$ and $\tau_{\Q}(M) \in \C(\Q)$, then $\tau_{\Q}(M)$ is a representation of $Q^{\mathrm{op}(\star)n}$, where $\mathrm{op}(\star)$ denotes the inequality sign opposite to $\star$.
\end{enumerate}

For $1 \leq \im \leq \jm \leq 2n-1$ and $\star \in \{\leq,\geq\}$, denote by $M^{\star}(\im,\jm)$ the indecomposable representation of $Q^{\star n}$ whose support are the vertices between $\im$ and $\jm$. We view it as an object of $\pvd(\Pi_Q)$. One can compute using Proposition \ref{prop:computing the spherical twist} that $M^{\bm{i}}_k \cong M^{\leq}(n-k+1,n)$ for $1 \leq k \leq n$ and $M^{\bm{i}}_1 \cong S_n \cong M^{\geq}(n,n)$, which proves (1). The same proposition allows us to calculate that
\[
\tau_{\Q}(M^{\star}(\im,\jm)) \cong \begin{cases}
M^{\mathrm{op}(\star)}(\sigma(\jm),\sigma(\im)) &\textrm{if } \jm < n-1,\\
M^{\mathrm{op}(\star)}(\sigma(\jm)-1,\sigma(\im)) &\textrm{if } \im < n \textrm{ and } n-1 \leq \jm < 2n-1,\\
M^{\mathrm{op}(\star)}(\sigma(\jm)-1,\sigma(\im)-1) &\textrm{if } \im \geq n \textrm{ and } n \leq \jm < 2n-1,\\
M^{\leq}(n+1,\sigma(\im)) &\textrm{if } \im < n, \jm=2n-1 \textrm{ and } \star = \ \geq,\\
T_n(M^{\leq}(n,\sigma(\im))) &\textrm{if } \im < n, \jm=2n-1 \textrm{ and } \star = \ \leq,\\
\Sigma^{-1}M^{\leq}(\sigma(\im),n) &\textrm{if } \im \geq n \textrm{ and } \jm = 2n-1.
\end{cases}
\]
In the fifth case above, note that $M^{\leq}(n,\sigma(\im))$ does not satisfy the condition of Proposition \ref{prop:computing the spherical twist} for the vertex $n$, hence $T_n(M^{\leq}(n,\sigma(\im)))$ is not concentrated in degree zero. We deduce that if $\tau_{\Q}(M^{\star}(\im,\jm))$ is concentrated in degree zero, then it is a representation of $Q^{\mathrm{op}(\star)n}$. This implies (2) and finishes the proof.
\end{proof}
\end{example}

\begin{example}\label{example:description of the objects in type Cn}
Suppose $(\Delta,\sigma) = (\mathsf{D}_{n+1},\vee)$ is of type $\mathsf{C}_n$.  Let $Q$ be the following orientation for $\Delta$:
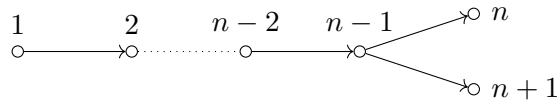
\begin{center}
\begin{tikzpicture}
	
	\node[circle, draw, inner sep=1.5pt] (D1) at (11.5,-0.5) {};
	\node[above] at (11.5,-0.4) {$1$};
	\node[circle, draw, inner sep=1.5pt] (D2) at (13,-0.5) {};
	\node[above] at (13,-0.4) {$2$};
	\node[circle, draw, inner sep=1.5pt] (Dnm2) at (14.5,-0.5) {};
	\node[above] at (14.5,-0.4) {$n-2$};
	\node[circle, draw, inner sep=1.5pt] (Dnm1) at (16,-0.5) {};
	\node[above] at (16,-0.4) {$n-1$};
	\node[circle, draw, inner sep=1.5pt] (Dn) at (17.5,0) {};
	\node[right] at (17.6,0) {$n$};
	\node[circle, draw, inner sep=1.5pt] (Dnp1) at (17.5,-1) {};
	\node[right] at (17.6,-1) {$n+1$};

	\draw[->] (D1) -- (D2);
	\draw[dotted] (D2) -- (Dnm2);
	\draw[->] (Dnm2) -- (Dnm1);
	\draw[->] (Dnm1) -- (Dn);
	\draw[->] (Dnm1) -- (Dnp1);
	
\end{tikzpicture}
\end{center}
Let $\Q = (\Delta,\sigma,\xi)$ be the Q-datum defined by $\xi_k = k$ for $1 \leq k \leq n$ and $\xi_{n+1} = n-2$. The indecomposable objects of $\C(\Q)$ can be listed in the following five families.
\begin{enumerate}[(1)]
	\item For $1 \leq \im \leq \jm \leq n-1$, define
	\[
	A^{0,0}_{\im,\jm} = \begin{tikzcd}[row sep=0.5em, column sep=1.5em, baseline={([yshift=-0.3em]current bounding box.center)}]
		&&&&&&&&&& 0 \\
		0 & \dotsb & 0 & K & K & \dotsb & K & 0 & \dotsb & 0 \\
		&&&&&&&&&& 0
		\arrow[bend left=20, from=1-11, to=2-10]
		\arrow[bend left=20, from=2-1, to=2-2]
		\arrow[bend left=20, from=2-2, to=2-1]
		\arrow[bend left=20, from=2-2, to=2-3]
		\arrow[bend left=20, from=2-3, to=2-2]
		\arrow[bend left=20, from=2-3, to=2-4]
		\arrow[bend left=20, from=2-4, to=2-3]
		\arrow["0", bend left=20, from=2-4, to=2-5]
		\arrow["\id", bend left=20, from=2-5, to=2-4]
		\arrow["0", bend left=20, from=2-5, to=2-6]
		\arrow["\id", bend left=20, from=2-6, to=2-5]
		\arrow["0", bend left=20, from=2-6, to=2-7]
		\arrow["\id", bend left=20, from=2-7, to=2-6]
		\arrow[bend left=20, from=2-7, to=2-8]
		\arrow[bend left=20, from=2-8, to=2-7]
		\arrow[bend left=20, from=2-8, to=2-9]
		\arrow[bend left=20, from=2-9, to=2-8]
		\arrow[bend left=20, from=2-9, to=2-10]
		\arrow[bend left=20, from=2-10, to=1-11]
		\arrow[bend left=20, from=2-10, to=2-9]
		\arrow[bend left=20, from=2-10, to=3-11]
		\arrow[bend left=20, from=3-11, to=2-10]
	\end{tikzcd}
	\]
	If we read from left to right, the first nonzero vector space appears at vertex $\im$ and the last at vertex $\jm$.
	
	\item For $1 \leq \im \leq n$, define
	\[
	A^{1,0}_{\im} = \begin{tikzcd}[row sep=0.5em, column sep=1.5em, baseline={([yshift=-0.3em]current bounding box.center)}]
		&&&&&&& K \\
	0 & \dotsb & 0 & K & K & \dotsb & K \\
	&&&&&&& 0
	\arrow["\id", bend left=20, from=1-8, to=2-7]
	\arrow[bend left=20, from=2-1, to=2-2]
	\arrow[bend left=20, from=2-2, to=2-1]
	\arrow[bend left=20, from=2-2, to=2-3]
	\arrow[bend left=20, from=2-3, to=2-2]
	\arrow[bend left=20, from=2-3, to=2-4]
	\arrow[bend left=20, from=2-4, to=2-3]
	\arrow["0", bend left=20, from=2-4, to=2-5]
	\arrow["\id", bend left=20, from=2-5, to=2-4]
	\arrow["0", bend left=20, from=2-5, to=2-6]
	\arrow["\id", bend left=20, from=2-6, to=2-5]
	\arrow["0", bend left=20, from=2-6, to=2-7]
	\arrow["0", bend left=20, from=2-7, to=1-8]
	\arrow["\id", bend left=20, from=2-7, to=2-6]
	\arrow[bend left=20, from=2-7, to=3-8]
	\arrow[bend left=20, from=3-8, to=2-7]
	\end{tikzcd}
	\]
	If we read from left to right, the first nonzero vector space appears at vertex $\im$.

	\item For $1 \leq \im \leq n$, define
	\[
	A^{0,1}_{\im} = \begin{tikzcd}[row sep=0.5em, column sep=1.5em, baseline={([yshift=-0.3em]current bounding box.center)}]
		&&&&&&& 0 \\
	0 & \dotsb & 0 & K & K & \dotsb & K \\
	&&&&&&& K
	\arrow[bend left=20, from=1-8, to=2-7]
	\arrow[bend left=20, from=2-1, to=2-2]
	\arrow[bend left=20, from=2-2, to=2-1]
	\arrow[bend left=20, from=2-2, to=2-3]
	\arrow[bend left=20, from=2-3, to=2-2]
	\arrow[bend left=20, from=2-3, to=2-4]
	\arrow[bend left=20, from=2-4, to=2-3]
	\arrow["0", bend left=20, from=2-4, to=2-5]
	\arrow["\id", bend left=20, from=2-5, to=2-4]
	\arrow["0", bend left=20, from=2-5, to=2-6]
	\arrow["\id", bend left=20, from=2-6, to=2-5]
	\arrow["0", bend left=20, from=2-6, to=2-7]
	\arrow[bend left=20, from=2-7, to=1-8]
	\arrow["\id", bend left=20, from=2-7, to=2-6]
	\arrow["0", bend left=20, from=2-7, to=3-8]
	\arrow["\id", bend left=20, from=3-8, to=2-7]
	\end{tikzcd}
	\]
	If we read from left to right, the first nonzero vector space appears at vertex $\im$ (or at vertex $n+1$ if $\im = n$).

	\item For $1 \leq \im \leq n-1$, define
	\[
	A^{1,1}_{\im} = \begin{tikzcd}[row sep=0.5em, column sep=1.5em, baseline={([yshift=-0.3em]current bounding box.center)}]
		&&&&&&& K \\
	0 & \dotsb & 0 & K & K & \dotsb & K \\
	&&&&&&& K
	\arrow["\id" pos=.2, bend left=20, from=1-8, to=2-7]
	\arrow[bend left=20, from=2-1, to=2-2]
	\arrow[bend left=20, from=2-2, to=2-1]
	\arrow[bend left=20, from=2-2, to=2-3]
	\arrow[bend left=20, from=2-3, to=2-2]
	\arrow[bend left=20, from=2-3, to=2-4]
	\arrow[bend left=20, from=2-4, to=2-3]
	\arrow["0", bend left=20, from=2-4, to=2-5]
	\arrow["\id", bend left=20, from=2-5, to=2-4]
	\arrow["0", bend left=20, from=2-5, to=2-6]
	\arrow["\id", bend left=20, from=2-6, to=2-5]
	\arrow["0", bend left=20, from=2-6, to=2-7]
	\arrow["0", bend left=20, from=2-7, to=1-8]
	\arrow["\id", bend left=20, from=2-7, to=2-6]
	\arrow["\id" pos=.9, bend left=20, from=2-7, to=3-8]
	\arrow["0", bend left=20, from=3-8, to=2-7]
	\end{tikzcd}
	\]
	If we read from left to right, the first nonzero vector space appears at vertex $\im$.

	\item For $1 \leq \im < \jm \leq n-1$, define
	\[
	B_{\im,\jm} = \begin{tikzcd}[row sep=0.5em, column sep=1.5em, baseline={([yshift=-0.7em]current bounding box.center)}, ampersand replacement=\&]
	\&\&\&\&\&\&\&\&\&\& K \\
	0 \& \dotsb \& 0 \& K \& \dotsb \& K \& {K^2} \& {K^2} \& \dotsb \& {K^2} \\
	\&\&\&\&\&\&\&\&\&\& K
	\arrow["{i_2}" pos=.2, bend left=20, from=1-11, to=2-10]
	\arrow[bend left=20, from=2-1, to=2-2]
	\arrow[bend left=20, from=2-2, to=2-1]
	\arrow[bend left=20, from=2-2, to=2-3]
	\arrow[bend left=20, from=2-3, to=2-2]
	\arrow[bend left=20, from=2-3, to=2-4]
	\arrow[bend left=20, from=2-4, to=2-3]
	\arrow["0", bend left=20, from=2-4, to=2-5]
	\arrow["\id", bend left=20, from=2-5, to=2-4]
	\arrow["0", bend left=20, from=2-5, to=2-6]
	\arrow["\id", bend left=20, from=2-6, to=2-5]
	\arrow["{i_1}" pos=0.6, bend left=20, from=2-6, to=2-7]
	\arrow["{\pi_2}" pos=0.4, bend left=20, from=2-7, to=2-6]
	\arrow["{\begin{bmatrix}0 & 1\\0 & 0\end{bmatrix}}", bend left=20, from=2-7, to=2-8]
	\arrow["\id", bend left=20, from=2-8, to=2-7]
	\arrow["{\begin{bmatrix}0 & 1\\0 & 0\end{bmatrix}}", bend left=20, from=2-8, to=2-9]
	\arrow["\id", bend left=20, from=2-9, to=2-8]
	\arrow["{\begin{bmatrix}0 & 1\\0 & 0\end{bmatrix}}", bend left=20, from=2-9, to=2-10]
	\arrow["0", bend left=20, from=2-10, to=1-11]
	\arrow["\id", bend left=20, from=2-10, to=2-9]
	\arrow["{\pi_2}" pos=0.9, bend left=20, from=2-10, to=3-11]
	\arrow["{i_1}", bend left=20, from=3-11, to=2-10]
	\end{tikzcd}
	\]
	If we read from left to right, the first nonzero vector space appears at vertex $\im$ and the first two-dimensional vector space at vertex $\jm$.
\end{enumerate}

An easy computation shows that $I^{\Q}_{\im} \cong A^{1,0}_{\im}$ for $1 \leq \im \leq n$ and $I^{\Q}_{n+1} \cong A^{1,1}_{n-1}$. By Corollary \ref{cor:using tau to find the repetition category}, we can compute any object in $\ind(\Q)$ by applying some power of $\tau_{\Q}$ to $I^{\Q}_{\im}$ for some $1 \leq \im \leq n+1$. Since the set of objects displayed above represents all positive roots of $\mathsf{R}^+$, we only need to show that it is closed (up to isomorphism and up to an application of the suspension functor) under the action of $\tau_{\Q}$. Indeed, knowing that $\tau_{\Q} = T_n\dotsb T_2T_1\sigma$, one can compute:
\[
\tau_{\Q}(A^{0,0}_{\im,\jm}) \cong \begin{cases}
A^{0,0}_{\im-1,\jm-1} &\textrm{if }\im \neq 1,\\
\Sigma^{-1}A^{1,0}_{\jm} &\textrm{if }\im=1,
\end{cases} \quad\quad
\tau_{\Q}(A^{1,0}_{\im}) \cong \begin{cases}
A^{1,1}_{\im-1} &\textrm{if }\im \neq 1,\\
A^{0,1}_n &\textrm{if }\im=1,
\end{cases}
\]
\[
\tau_{\Q}(A^{0,1}_\im) \cong \begin{cases}
A^{0,0}_{\im-1,n-1} &\textrm{if }\im \neq 1,\\
\Sigma^{-1}A^{1,0}_n &\textrm{if }\im=1,
\end{cases} \quad\quad
\tau_{\Q}(A^{1,1}_{\im}) \cong \begin{cases}
B_{\im-1,n-1} &\textrm{if }\im \neq 1,\\
A^{0,1}_{n-1} &\textrm{if }\im=1,
\end{cases}
\]
\[
\tau_{\Q}(B_{\im,\jm}) \cong \begin{cases}
B_{\im-1,\jm-1} &\textrm{if }\im \neq 1,\\
A^{0,1}_{\jm-1} &\textrm{if }\im=1.
\end{cases}
\]
By Corollary \ref{cor:power of tau sends projective to shift of injective}, we also deduce that $P_{\im}^{\Q} \cong A^{0,0}_{1,\im}$ for $1 \leq \im \leq n-1$, $P_n^{\Q} \cong B_{12}$, and $P_{n+1}^{\Q} \cong A^{0,1}_1$.
\end{example}

\begin{example}\label{example:F4} Suppose $(\Delta,\sigma) = (\mathsf{E}_6,\vee)$ is of type $\mathsf{F}_4$. Let $Q$ be the following orientation for $\Delta$:
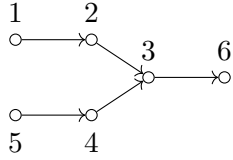
\begin{center}
\begin{tikzpicture}
		
	\node[circle, draw, inner sep=1.5pt] (E1) at (2,-4) {};
		\node[above] at (2,-3.9) {$1$};
		\node[circle, draw, inner sep=1.5pt] (E2) at (3,-4) {};
		\node[above] at (3,-3.9) {$2$};
		\node[circle, draw, inner sep=1.5pt] (E3) at (3.75,-4.5) {};
		\node[above] at (3.75,-4.4) {$3$};
		\node[circle, draw, inner sep=1.5pt] (E4) at (3,-5) {};
		\node[below] at (3,-5.1) {$4$};
		\node[circle, draw, inner sep=1.5pt] (E5) at (2,-5) {};
		\node[below] at (2,-5.1) {$5$};
		\node[circle, draw, inner sep=1.5pt] (E6) at (4.75,-4.5) {};
		\node[above] at (4.75,-4.4) {$6$};
	
		\draw[->] (E1) -- (E2);
		\draw[->] (E2) -- (E3);
		\draw[<-] (E3) -- (E4);
		\draw[<-] (E4) -- (E5);
		\draw[->] (E3) -- (E6);
\end{tikzpicture}
\end{center}
Let $\Q = (\Delta,\sigma,\xi)$ be the Q-datum defined by $\xi_1 = 7$, $\xi_2 = 9$, $\xi_3 = 8$, $\xi_4 = 7$, $\xi_5 = 9$, and $\xi_6 = 7$. The indecomposable objects of $\C(\Q)$ (with their coordinates in $\Gamma_{\mathcal{Q}}$) are depicted in Figure \ref{figure:F4}. The unspecified objects are as follows:
\[
M_1 = \begin{tikzcd}[row sep=1.5em, column sep=1.5em, baseline={([yshift=-0.3em]current bounding box.center)}, ampersand replacement=\&]
	\&\& K \\
	K \& {K^2} \& {K^2} \& K \& K
	\arrow["{i_1}" pos=0.53, bend left=20, from=1-3, to=2-3]
	\arrow["0" pos=0.6, bend left=20, from=2-1, to=2-2]
	\arrow["{\pi_1}" pos = 0.4, bend left=20, from=2-2, to=2-1]
	\arrow["\id", bend left=20, from=2-2, to=2-3]
	\arrow["{\pi_2}" pos=0.47, bend left=20, from=2-3, to=1-3]
	\arrow["0", bend left=20, from=2-3, to=2-2]
	\arrow["{\pi_2}"' pos=0.4, bend right=20, from=2-3, to=2-4]
	\arrow["{i_1}"' pos=0.51, bend right=20, from=2-4, to=2-3]
	\arrow["0"', bend right=20, from=2-4, to=2-5]
	\arrow["\id"', bend right=20, from=2-5, to=2-4]
\end{tikzcd} \quad\quad M_2 = \begin{tikzcd}[row sep=1.5em, column sep=1.5em, baseline={([yshift=-0.3em]current bounding box.center)}, ampersand replacement=\&]
	\&\& K \\
	0 \& {K} \& {K^2} \& K \& K
	\arrow["{i_1}" pos=0.53, bend left=20, from=1-3, to=2-3]
	\arrow[bend left=20, from=2-1, to=2-2]
	\arrow[bend left=20, from=2-2, to=2-1]
	\arrow["i_2" pos=0.6, bend left=20, from=2-2, to=2-3]
	\arrow["{\pi_2}" pos=0.47, bend left=20, from=2-3, to=1-3]
	\arrow["0" pos=0.4, bend left=20, from=2-3, to=2-2]
	\arrow["{\pi_2}"' pos=0.4, bend right=20, from=2-3, to=2-4]
	\arrow["{i_1}"' pos=0.51, bend right=20, from=2-4, to=2-3]
	\arrow["0"', bend right=20, from=2-4, to=2-5]
	\arrow["\id"', bend right=20, from=2-5, to=2-4]
\end{tikzcd}
\]
\[
M_3 = \begin{tikzcd}[row sep=1.5em, column sep=1.5em, baseline={([yshift=-0.3em]current bounding box.center)}, ampersand replacement=\&]
	\&\& K \\
	K \& {K^2} \& {K^2} \& {K^2} \& K
	\arrow["0" pos=0.53, bend left=20, from=1-3, to=2-3]
	\arrow["0" pos=0.6, bend left=20, from=2-1, to=2-2]
	\arrow["{\pi_1}" pos = 0.4, bend left=20, from=2-2, to=2-1]
	\arrow["\id", bend left=20, from=2-2, to=2-3]
	\arrow["{\pi_1 + \pi_2}" pos=0.47, bend left=20, from=2-3, to=1-3]
	\arrow["0", bend left=20, from=2-3, to=2-2]
	\arrow["{i_1\pi_2}"', bend right=20, from=2-3, to=2-4]
	\arrow["{i_1\pi_2}"', bend right=20, from=2-4, to=2-3]
	\arrow["0"' pos=0.4, bend right=20, from=2-4, to=2-5]
	\arrow["{i_2}"' pos=0.53, bend right=20, from=2-5, to=2-4]
\end{tikzcd} \quad\quad M_4 = \begin{tikzcd}[row sep=1.5em, column sep=1.5em, baseline={([yshift=-0.3em]current bounding box.center)}, ampersand replacement=\&]
	\&\& K \\
	K \& {K^2} \& {K^2} \& K \& 0
	\arrow["0" pos=0.53, bend left=20, from=1-3, to=2-3]
	\arrow["0" pos=0.6, bend left=20, from=2-1, to=2-2]
	\arrow["{\pi_1}" pos=0.4, bend left=20, from=2-2, to=2-1]
	\arrow["\id", bend left=20, from=2-2, to=2-3]
	\arrow["{\pi_1 + \pi_2}" pos=0.47, bend left=20, from=2-3, to=1-3]
	\arrow["0", bend left=20, from=2-3, to=2-2]
	\arrow["{\pi_2}"' pos=0.4, bend right=20, from=2-3, to=2-4]
	\arrow["0"' pos=0.51, bend right=20, from=2-4, to=2-3]
	\arrow[bend right=20, from=2-4, to=2-5]
	\arrow[bend right=20, from=2-5, to=2-4]
\end{tikzcd}
\]
\[
M_5 = \begin{tikzcd}[row sep=1.5em, column sep=1.5em, baseline={([yshift=-0.3em]current bounding box.center)}, ampersand replacement=\&]
	\&\&\& K \\
	K \& {K^2} \&\& {K^3} \&\& {K^2} \& K
	\arrow["{-i_1}" pos=0.53, bend left=20, from=1-4, to=2-4]
	\arrow["0" pos=0.6, bend left=20, from=2-1, to=2-2]
	\arrow["{\pi_1}" pos = 0.4, bend left=20, from=2-2, to=2-1]
	\arrow["{i_2\pi_1 + i_3\pi_2}", bend left=20, from=2-2, to=2-4]
	\arrow["{\pi_3}" pos=0.47, bend left=20, from=2-4, to=1-4]
	\arrow["0", bend left=20, from=2-4, to=2-2]
	\arrow["{-i_1\pi_2 + (i_1-i_2)\pi_3}"', bend right=20, from=2-4, to=2-6]
	\arrow["{i_1\pi_2}"', bend right=20, from=2-6, to=2-4]
	\arrow["0"' pos=0.4, bend right=20, from=2-6, to=2-7]
	\arrow["{i_2}"' pos=0.53, bend right=20, from=2-7, to=2-6]
\end{tikzcd} \quad M_6 = \begin{tikzcd}[row sep=1.5em, column sep=1.5em, baseline={([yshift=-0.3em]current bounding box.center)}, ampersand replacement=\&]
	\&\& K \\
	0 \& {K} \& {K^2} \& {K^2} \& K
	\arrow["{-i_1}" pos=0.53, bend left=20, from=1-3, to=2-3]
	\arrow[bend left=20, from=2-1, to=2-2]
	\arrow[bend left=20, from=2-2, to=2-1]
	\arrow["{i_2}" pos=0.6, bend left=20, from=2-2, to=2-3]
	\arrow["{\pi_2}" pos=0.47, bend left=20, from=2-3, to=1-3]
	\arrow["0" pos=0.4, bend left=20, from=2-3, to=2-2]
	\arrow["{(i_1-i_2)\pi_2}"', bend right=20, from=2-3, to=2-4]
	\arrow["{i_1\pi_2}"', bend right=20, from=2-4, to=2-3]
	\arrow["0"' pos=0.4, bend right=20, from=2-4, to=2-5]
	\arrow["{i_2}"' pos=0.6, bend right=20, from=2-5, to=2-4]
\end{tikzcd}
\]
\[
M_7 = \begin{tikzcd}[row sep=1.75em, column sep=2em, baseline={([yshift=-0.3em]current bounding box.center)}, ampersand replacement=\&]
	\&\&\& {K^2} \\
	K \& {K^2} \&\& {K^3} \&\& {K^2} \& K
	\arrow["{-i_1\pi_2}" pos=0.4, bend left=20, from=1-4, to=2-4]
	\arrow["0" pos=0.6, bend left=20, from=2-1, to=2-2]
	\arrow["{\pi_1}" pos = 0.4, bend left=20, from=2-2, to=2-1]
	\arrow["{i_2\pi_1 + i_3\pi_2}", bend left=20, from=2-2, to=2-4]
	\arrow["{i_1\pi_2 + i_2\pi_3}" pos=0.6, bend left=20, from=2-4, to=1-4]
	\arrow["0", bend left=20, from=2-4, to=2-2]
	\arrow["{i_1\pi_3}"', bend right=20, from=2-4, to=2-6]
	\arrow["{ \ \ \ -i_1\pi_1 + (i_1-i_2)\pi_2}"', bend right=20, from=2-6, to=2-4]
	\arrow["0"' pos=0.4, bend right=20, from=2-6, to=2-7]
	\arrow["{i_2}"' pos=0.53, bend right=20, from=2-7, to=2-6]
\end{tikzcd}
\]
\[
M_8 = \begin{tikzcd}[row sep=1.5em, column sep=1.5em, baseline={([yshift=-0.3em]current bounding box.center)}, ampersand replacement=\&]
	\&\& K \\
	K \& {K} \& {K^2} \& {K} \& K
	\arrow["0" pos=0.53, bend left=20, from=1-3, to=2-3]
	\arrow["0", bend left=20, from=2-1, to=2-2]
	\arrow["{\id}", bend left=20, from=2-2, to=2-1]
	\arrow["{i_2}" pos=0.6, bend left=20, from=2-2, to=2-3]
	\arrow["{\pi_2}" pos=0.47, bend left=20, from=2-3, to=1-3]
	\arrow["0" pos=0.4, bend left=20, from=2-3, to=2-2]
	\arrow["0"' pos=0.4, bend right=20, from=2-3, to=2-4]
	\arrow["{i_1-i_2}"' pos=0.53, bend right=20, from=2-4, to=2-3]
	\arrow["0"', bend right=20, from=2-4, to=2-5]
	\arrow["\id"', bend right=20, from=2-5, to=2-4]
\end{tikzcd} \quad\quad M_9 = \begin{tikzcd}[row sep=1.5em, column sep=1.5em, baseline={([yshift=-0.3em]current bounding box.center)}, ampersand replacement=\&]
	\&\& K \\
	K \& {K} \& {K^2} \& {K^2} \& K
	\arrow["0" pos=0.53, bend left=20, from=1-3, to=2-3]
	\arrow["0", bend left=20, from=2-1, to=2-2]
	\arrow["{\id}", bend left=20, from=2-2, to=2-1]
	\arrow["i_2" pos=0.6, bend left=20, from=2-2, to=2-3]
	\arrow["{\pi_1 - \pi_2}" pos=0.47, bend left=20, from=2-3, to=1-3]
	\arrow["0" pos=0.4, bend left=20, from=2-3, to=2-2]
	\arrow["{i_1\pi_2}"', bend right=20, from=2-3, to=2-4]
	\arrow["{i_1\pi_2}"', bend right=20, from=2-4, to=2-3]
	\arrow["0"' pos=0.4, bend right=20, from=2-4, to=2-5]
	\arrow["{i_2}"' pos=0.53, bend right=20, from=2-5, to=2-4]
\end{tikzcd}
\]
\[
M_{10} = \begin{tikzcd}[row sep=1.5em, column sep=1.5em, baseline={([yshift=-0.3em]current bounding box.center)}, ampersand replacement=\&]
	\&\& K \\
	0 \& {K} \& {K^2} \& {K} \& 0
	\arrow["{i_1}" pos=0.53, bend left=20, from=1-3, to=2-3]
	\arrow[bend left=20, from=2-1, to=2-2]
	\arrow[bend left=20, from=2-2, to=2-1]
	\arrow["{i_2}" pos=0.6, bend left=20, from=2-2, to=2-3]
	\arrow["{\pi_2}" pos=0.47, bend left=20, from=2-3, to=1-3]
	\arrow["0" pos=0.4, bend left=20, from=2-3, to=2-2]
	\arrow["{\pi_2}"' pos=0.4, bend right=20, from=2-3, to=2-4]
	\arrow["{i_1}"' pos=0.53, bend right=20, from=2-4, to=2-3]
	\arrow[bend right=20, from=2-4, to=2-5]
	\arrow[bend right=20, from=2-5, to=2-4]
\end{tikzcd} \quad\quad M_{11} = \begin{tikzcd}[row sep=1.5em, column sep=1.5em, baseline={([yshift=-0.3em]current bounding box.center)}, ampersand replacement=\&]
	\&\& K \\
	K \& {K} \& {K^2} \& K \& 0
	\arrow["{i_1}" pos=0.53, bend left=20, from=1-3, to=2-3]
	\arrow["0", bend left=20, from=2-1, to=2-2]
	\arrow["{\id}", bend left=20, from=2-2, to=2-1]
	\arrow["{i_2}" pos=0.6, bend left=20, from=2-2, to=2-3]
	\arrow["{\pi_2}" pos=0.47, bend left=20, from=2-3, to=1-3]
	\arrow["0" pos=0.4, bend left=20, from=2-3, to=2-2]
	\arrow["{\pi_2}"' pos=0.4, bend right=20, from=2-3, to=2-4]
	\arrow["{i_1}"' pos=0.51, bend right=20, from=2-4, to=2-3]
	\arrow[bend right=20, from=2-4, to=2-5]
	\arrow[bend right=20, from=2-5, to=2-4]
\end{tikzcd}
\]

\begin{sidewaysfigure}
\[\begin{tikzcd}[column sep={3em,between origins},row sep={1em},ampersand replacement=\&]
	{(\im \ \backslash \ p)} \& {-9} \& {-8} \& {-7} \& {-6} \& {-5} \& {-4} \& {-3} \& {-2} \& {-1} \& 0 \& 1 \& 2 \& 3 \& 4 \& 5 \& 6 \& 7 \& 8 \& 9 \\
	1 \&\&\&\&\& {
		\begin{tikzpicture}

			\node at (-0.3,0) {$K$};
			\node at (-0.6,0) {$K$};
			\node at (-1,0.2) {$0$};
			\node at (-1,-0.2) {$K$};
			\node at (-1.3, 0.2) {$0$};
			\node at (-1.3, -0.2) {$K$};
	
			\draw[->] (-0.58,-0.15) -- (-0.32,-0.15);
			%\draw[->] (-0.80,0.55) -- (-0.65,0.35);
			%\draw[->] (-1.28,0.6) -- (-1.02, 0.6);
			\draw[->] (-0.825,-0.3) -- (-0.675,-0.1);
			\draw[->] (-1.28,-0.35) -- (-1.02, -0.35);
			
		\end{tikzpicture}
	} \&\&\&\& {
		\begin{tikzpicture}

			\node at (-0.3,0) {$0$};
			\node at (-0.6,0) {$K$};
			\node at (-1,0.2) {$K$};
			\node at (-1,-0.2) {$K$};
			\node at (-1.3, 0.2) {$K$};
			\node at (-1.3, -0.2) {$0$};
	
			%\draw[->] (-0.58,-0.15) -- (-0.32,-0.15);
			\draw[->] (-0.80,0.55) -- (-0.65,0.35);
			\draw[<-] (-1.28,0.6) -- (-1.02, 0.6);
			\draw[<-] (-0.825,-0.3) -- (-0.675,-0.1);
			%\draw[->] (-1.28,-0.35) -- (-1.02, -0.35);
			
		\end{tikzpicture}
	} \&\&\&\& {M_2} \&\&\&\& {
		\begin{tikzpicture}

			\node at (-0.3,0) {$0$};
			\node at (-0.6,0) {$0$};
			\node at (-1,0.2) {$K$};
			\node at (-1,-0.2) {$0$};
			\node at (-1.3, 0.2) {$K$};
			\node at (-1.3, -0.2) {$0$};
	
			%\draw[->] (-0.58,-0.15) -- (-0.32,-0.15);
			%\draw[->] (-0.80,0.55) -- (-0.65,0.35);
			\draw[<-] (-1.28,0.6) -- (-1.02, 0.6);
			%\draw[<-] (-0.825,-0.3) -- (-0.675,-0.1);
			%\draw[->] (-1.28,-0.35) -- (-1.02, -0.35);
			
		\end{tikzpicture}
	} \\
	2 \&\&\& {
		\begin{tikzpicture}

			\node at (-0.3,0) {$K$};
			\node at (-0.6,0) {$K$};
			\node at (-1,0.2) {$0$};
			\node at (-1,-0.2) {$K$};
			\node at (-1.3, 0.2) {$0$};
			\node at (-1.3, -0.2) {$0$};
	
			\draw[->] (-0.58,-0.15) -- (-0.32,-0.15);
			%\draw[->] (-0.80,0.55) -- (-0.65,0.35);
			%\draw[<-] (-1.28,0.6) -- (-1.02, 0.6);
			\draw[->] (-0.825,-0.3) -- (-0.675,-0.1);
			%\draw[->] (-1.28,-0.35) -- (-1.02, -0.35);
			
		\end{tikzpicture}
	} \&\&\&\& {M_9} \&\&\&\& {M_5} \&\&\&\& {M_1} \&\&\&\& {
		\begin{tikzpicture}

			\node at (-0.3,0) {$0$};
			\node at (-0.6,0) {$0$};
			\node at (-1,0.2) {$K$};
			\node at (-1,-0.2) {$0$};
			\node at (-1.3, 0.2) {$0$};
			\node at (-1.3, -0.2) {$0$};
	
			%\draw[->] (-0.58,-0.15) -- (-0.32,-0.15);
			%\draw[->] (-0.80,0.55) -- (-0.65,0.35);
			%\draw[<-] (-1.28,0.6) -- (-1.02, 0.6);
			%\draw[<-] (-0.825,-0.3) -- (-0.675,-0.1);
			%\draw[->] (-1.28,-0.35) -- (-1.02, -0.35);
			
		\end{tikzpicture}
	} \\
	3 \&\& {
		\begin{tikzpicture}

			\node at (-0.3,0) {$K$};
			\node at (-0.6,0) {$K$};
			\node at (-1,0.2) {$0$};
			\node at (-1,-0.2) {$0$};
			\node at (-1.3, 0.2) {$0$};
			\node at (-1.3, -0.2) {$0$};
	
			\draw[->] (-0.58,-0.15) -- (-0.32,-0.15);
			%\draw[->] (-0.80,0.55) -- (-0.65,0.35);
			%\draw[<-] (-1.28,0.6) -- (-1.02, 0.6);
			%\draw[<-] (-0.825,-0.3) -- (-0.675,-0.1);
			%\draw[->] (-1.28,-0.35) -- (-1.02, -0.35);
			
		\end{tikzpicture}
	} \&\& {
		\begin{tikzpicture}

			\node at (-0.3,0) {$0$};
			\node at (-0.6,0) {$K$};
			\node at (-1,0.2) {$0$};
			\node at (-1,-0.2) {$K$};
			\node at (-1.3, 0.2) {$0$};
			\node at (-1.3, -0.2) {$0$};
	
			%\draw[->] (-0.58,-0.15) -- (-0.32,-0.15);
			%\draw[->] (-0.80,0.55) -- (-0.65,0.35);
			%\draw[<-] (-1.28,0.6) -- (-1.02, 0.6);
			\draw[->] (-0.825,-0.3) -- (-0.675,-0.1);
			%\draw[->] (-1.28,-0.35) -- (-1.02, -0.35);
			
		\end{tikzpicture}
	} \&\& {
		\begin{tikzpicture}

			\node at (-0.3,0) {$K$};
			\node at (-0.6,0) {$K$};
			\node at (-1,0.2) {$K$};
			\node at (-1,-0.2) {$K$};
			\node at (-1.3, 0.2) {$K$};
			\node at (-1.3, -0.2) {$0$};
	
			\draw[->] (-0.58,-0.15) -- (-0.32,-0.15);
			\draw[->] (-0.80,0.55) -- (-0.65,0.35);
			\draw[<-] (-1.28,0.6) -- (-1.02, 0.6);
			\draw[<-] (-0.825,-0.3) -- (-0.675,-0.1);
			%\draw[->] (-1.28,-0.35) -- (-1.02, -0.35);
			
		\end{tikzpicture}
	} \&\& {M_8} \&\& {M_6} \&\& {M_4} \&\& {
		\begin{tikzpicture}

			\node at (-0.3,0) {$0$};
			\node at (-0.6,0) {$K$};
			\node at (-1,0.2) {$K$};
			\node at (-1,-0.2) {$K$};
			\node at (-1.3, 0.2) {$K$};
			\node at (-1.3, -0.2) {$K$};
	
			%\draw[->] (-0.58,-0.15) -- (-0.32,-0.15);
			\draw[->] (-0.80,0.55) -- (-0.65,0.35);
			\draw[<-] (-1.28,0.6) -- (-1.02, 0.6);
			\draw[->] (-0.825,-0.3) -- (-0.675,-0.1);
			\draw[->] (-1.28,-0.35) -- (-1.02, -0.35);
			
		\end{tikzpicture}
	} \&\& {
		\begin{tikzpicture}

			\node at (-0.3,0) {$K$};
			\node at (-0.6,0) {$K$};
			\node at (-1,0.2) {$K$};
			\node at (-1,-0.2) {$K$};
			\node at (-1.3, 0.2) {$0$};
			\node at (-1.3, -0.2) {$K$};
	
			\draw[->] (-0.58,-0.15) -- (-0.32,-0.15);
			\draw[->] (-0.80,0.55) -- (-0.65,0.35);
			%\draw[<-] (-1.28,0.6) -- (-1.02, 0.6);
			\draw[<-] (-0.825,-0.3) -- (-0.675,-0.1);
			\draw[->] (-1.28,-0.35) -- (-1.02, -0.35);
			
		\end{tikzpicture}
	} \&\& {
		\begin{tikzpicture}

			\node at (-0.3,0) {$0$};
			\node at (-0.6,0) {$K$};
			\node at (-1,0.2) {$K$};
			\node at (-1,-0.2) {$0$};
			\node at (-1.3, 0.2) {$0$};
			\node at (-1.3, -0.2) {$0$};
	
			%\draw[->] (-0.58,-0.15) -- (-0.32,-0.15);
			\draw[->] (-0.80,0.55) -- (-0.65,0.35);
			%\draw[<-] (-1.28,0.6) -- (-1.02, 0.6);
			%\draw[<-] (-0.825,-0.3) -- (-0.675,-0.1);
			%\draw[->] (-1.28,-0.35) -- (-1.02, -0.35);
			
		\end{tikzpicture}
	} \\
	4 \&\&\&\&\& {M_{11}} \&\&\&\& {M_7} \&\&\&\& {M_3} \&\&\&\& {
		\begin{tikzpicture}

			\node at (-0.3,0) {$0$};
			\node at (-0.6,0) {$K$};
			\node at (-1,0.2) {$K$};
			\node at (-1,-0.2) {$K$};
			\node at (-1.3, 0.2) {$0$};
			\node at (-1.3, -0.2) {$K$};
	
			%\draw[->] (-0.58,-0.15) -- (-0.32,-0.15);
			\draw[->] (-0.80,0.55) -- (-0.65,0.35);
			%\draw[<-] (-1.28,0.6) -- (-1.02, 0.6);
			\draw[<-] (-0.825,-0.3) -- (-0.675,-0.1);
			\draw[->] (-1.28,-0.35) -- (-1.02, -0.35);
			
		\end{tikzpicture}
	} \\
	5 \&\&\& {
		\begin{tikzpicture}

			\node at (-0.3,0) {$0$};
			\node at (-0.6,0) {$0$};
			\node at (-1,0.2) {$0$};
			\node at (-1,-0.2) {$0$};
			\node at (-1.3, 0.2) {$K$};
			\node at (-1.3, -0.2) {$0$};
	
			%\draw[->] (-0.58,-0.15) -- (-0.32,-0.15);
			%\draw[->] (-0.80,0.55) -- (-0.65,0.35);
			%\draw[<-] (-1.28,0.6) -- (-1.02, 0.6);
			%\draw[<-] (-0.825,-0.3) -- (-0.675,-0.1);
			%\draw[->] (-1.28,-0.35) -- (-1.02, -0.35);
			
		\end{tikzpicture}
	} \&\&\&\& {M_{10}} \&\&\&\& {
		\begin{tikzpicture}

			\node at (-0.3,0) {$K$};
			\node at (-0.6,0) {$K$};
			\node at (-1,0.2) {$K$};
			\node at (-1,-0.2) {$K$};
			\node at (-1.3, 0.2) {$K$};
			\node at (-1.3, -0.2) {$K$};
	
			\draw[->] (-0.58,-0.15) -- (-0.32,-0.15);
			\draw[->] (-0.80,0.55) -- (-0.65,0.35);
			\draw[<-] (-1.28,0.6) -- (-1.02, 0.6);
			\draw[->] (-0.825,-0.3) -- (-0.675,-0.1);
			\draw[->] (-1.28,-0.35) -- (-1.02, -0.35);
			
		\end{tikzpicture}
	} \&\&\&\& {
		\begin{tikzpicture}

			\node at (-0.3,0) {$0$};
			\node at (-0.6,0) {$K$};
			\node at (-1,0.2) {$K$};
			\node at (-1,-0.2) {$K$};
			\node at (-1.3, 0.2) {$0$};
			\node at (-1.3, -0.2) {$0$};
	
			%\draw[->] (-0.58,-0.15) -- (-0.32,-0.15);
			\draw[->] (-0.80,0.55) -- (-0.65,0.35);
			%\draw[<-] (-1.28,0.6) -- (-1.02, 0.6);
			\draw[<-] (-0.825,-0.3) -- (-0.675,-0.1);
			%\draw[->] (-1.28,-0.35) -- (-1.02, -0.35);
			
		\end{tikzpicture}
	} \&\&\&\& {
		\begin{tikzpicture}

			\node at (-0.3,0) {$0$};
			\node at (-0.6,0) {$0$};
			\node at (-1,0.2) {$0$};
			\node at (-1,-0.2) {$0$};
			\node at (-1.3, 0.2) {$0$};
			\node at (-1.3, -0.2) {$K$};
	
			%\draw[->] (-0.58,-0.15) -- (-0.32,-0.15);
			%\draw[->] (-0.80,0.55) -- (-0.65,0.35);
			%\draw[<-] (-1.28,0.6) -- (-1.02, 0.6);
			%\draw[<-] (-0.825,-0.3) -- (-0.675,-0.1);
			%\draw[->] (-1.28,-0.35) -- (-1.02, -0.35);
			
		\end{tikzpicture}
	} \\
	6 \& {
		\begin{tikzpicture}

			\node at (-0.3,0) {$K$};
			\node at (-0.6,0) {$0$};
			\node at (-1,0.2) {$0$};
			\node at (-1,-0.2) {$0$};
			\node at (-1.3, 0.2) {$0$};
			\node at (-1.3, -0.2) {$0$};
	
			%\draw[->] (-0.58,-0.15) -- (-0.32,-0.15);
			%\draw[->] (-0.80,0.55) -- (-0.65,0.35);
			%\draw[<-] (-1.28,0.6) -- (-1.02, 0.6);
			%\draw[<-] (-0.825,-0.3) -- (-0.675,-0.1);
			%\draw[->] (-1.28,-0.35) -- (-1.02, -0.35);
			
		\end{tikzpicture}
	} \&\& {
		\begin{tikzpicture}

			\node at (-0.3,0) {$0$};
			\node at (-0.6,0) {$K$};
			\node at (-1,0.2) {$0$};
			\node at (-1,-0.2) {$0$};
			\node at (-1.3, 0.2) {$0$};
			\node at (-1.3, -0.2) {$0$};
	
			%\draw[->] (-0.58,-0.15) -- (-0.32,-0.15);
			%\draw[->] (-0.80,0.55) -- (-0.65,0.35);
			%\draw[<-] (-1.28,0.6) -- (-1.02, 0.6);
			%\draw[<-] (-0.825,-0.3) -- (-0.675,-0.1);
			%\draw[->] (-1.28,-0.35) -- (-1.02, -0.35);
			
		\end{tikzpicture}
	} \&\& {
		\begin{tikzpicture}

			\node at (-0.3,0) {$0$};
			\node at (-0.6,0) {$0$};
			\node at (-1,0.2) {$0$};
			\node at (-1,-0.2) {$K$};
			\node at (-1.3, 0.2) {$0$};
			\node at (-1.3, -0.2) {$0$};
	
			%\draw[->] (-0.58,-0.15) -- (-0.32,-0.15);
			%\draw[->] (-0.80,0.55) -- (-0.65,0.35);
			%\draw[<-] (-1.28,0.6) -- (-1.02, 0.6);
			%\draw[<-] (-0.825,-0.3) -- (-0.675,-0.1);
			%\draw[->] (-1.28,-0.35) -- (-1.02, -0.35);
			
		\end{tikzpicture}
	} \&\& {
		\begin{tikzpicture}

			\node at (-0.3,0) {$K$};
			\node at (-0.6,0) {$K$};
			\node at (-1,0.2) {$K$};
			\node at (-1,-0.2) {$0$};
			\node at (-1.3, 0.2) {$K$};
			\node at (-1.3, -0.2) {$0$};
	
			\draw[->] (-0.58,-0.15) -- (-0.32,-0.15);
			\draw[->] (-0.80,0.55) -- (-0.65,0.35);
			\draw[<-] (-1.28,0.6) -- (-1.02, 0.6);
			%\draw[<-] (-0.825,-0.3) -- (-0.675,-0.1);
			%\draw[->] (-1.28,-0.35) -- (-1.02, -0.35);
			
		\end{tikzpicture}
	} \&\& {
		\begin{tikzpicture}

			\node at (-0.3,0) {$0$};
			\node at (-0.6,0) {$K$};
			\node at (-1,0.2) {$0$};
			\node at (-1,-0.2) {$K$};
			\node at (-1.3, 0.2) {$0$};
			\node at (-1.3, -0.2) {$K$};
	
			%\draw[->] (-0.58,-0.15) -- (-0.32,-0.15);
			%\draw[->] (-0.80,0.55) -- (-0.65,0.35);
			%\draw[<-] (-1.28,0.6) -- (-1.02, 0.6);
			\draw[->] (-0.825,-0.3) -- (-0.675,-0.1);
			\draw[->] (-1.28,-0.35) -- (-1.02, -0.35);
			
		\end{tikzpicture}
	} \&\& {
		\begin{tikzpicture}

			\node at (-0.3,0) {$K$};
			\node at (-0.6,0) {$K$};
			\node at (-1,0.2) {$K$};
			\node at (-1,-0.2) {$K$};
			\node at (-1.3, 0.2) {$0$};
			\node at (-1.3, -0.2) {$0$};
	
			\draw[->] (-0.58,-0.15) -- (-0.32,-0.15);
			\draw[->] (-0.80,0.55) -- (-0.65,0.35);
			%\draw[<-] (-1.28,0.6) -- (-1.02, 0.6);
			\draw[<-] (-0.825,-0.3) -- (-0.675,-0.1);
			%\draw[->] (-1.28,-0.35) -- (-1.02, -0.35);
			
		\end{tikzpicture}
	} \&\& {
		\begin{tikzpicture}

			\node at (-0.3,0) {$0$};
			\node at (-0.6,0) {$K$};
			\node at (-1,0.2) {$K$};
			\node at (-1,-0.2) {$0$};
			\node at (-1.3, 0.2) {$K$};
			\node at (-1.3, -0.2) {$0$};
	
			%\draw[->] (-0.58,-0.15) -- (-0.32,-0.15);
			\draw[->] (-0.80,0.55) -- (-0.65,0.35);
			\draw[<-] (-1.28,0.6) -- (-1.02, 0.6);
			%\draw[<-] (-0.825,-0.3) -- (-0.675,-0.1);
			%\draw[->] (-1.28,-0.35) -- (-1.02, -0.35);
			
		\end{tikzpicture}
	} \&\& {
		\begin{tikzpicture}

			\node at (-0.3,0) {$0$};
			\node at (-0.6,0) {$0$};
			\node at (-1,0.2) {$0$};
			\node at (-1,-0.2) {$K$};
			\node at (-1.3, 0.2) {$0$};
			\node at (-1.3, -0.2) {$K$};
	
			%\draw[->] (-0.58,-0.15) -- (-0.32,-0.15);
			%\draw[->] (-0.80,0.55) -- (-0.65,0.35);
			%\draw[<-] (-1.28,0.6) -- (-1.02, 0.6);
			%\draw[<-] (-0.825,-0.3) -- (-0.675,-0.1);
			\draw[->] (-1.28,-0.35) -- (-1.02, -0.35);
			
		\end{tikzpicture}
	} \&\& {
		\begin{tikzpicture}

			\node at (-0.3,0) {$K$};
			\node at (-0.6,0) {$K$};
			\node at (-1,0.2) {$K$};
			\node at (-1,-0.2) {$0$};
			\node at (-1.3, 0.2) {$0$};
			\node at (-1.3, -0.2) {$0$};
	
			\draw[->] (-0.58,-0.15) -- (-0.32,-0.15);
			\draw[->] (-0.80,0.55) -- (-0.65,0.35);
			%\draw[<-] (-1.28,0.6) -- (-1.02, 0.6);
			%\draw[<-] (-0.825,-0.3) -- (-0.675,-0.1);
			%\draw[->] (-1.28,-0.35) -- (-1.02, -0.35);
			
		\end{tikzpicture}
	}
	\arrow[from=2-6, to=3-8]
	\arrow[from=2-10, to=3-12]
	\arrow[from=2-14, to=3-16]
	\arrow[from=2-18, to=3-20]
	\arrow[from=3-4, to=2-6]
	\arrow[from=3-4, to=4-5]
	\arrow[from=3-8, to=2-10]
	\arrow[from=3-8, to=4-9]
	\arrow[from=3-12, to=2-14]
	\arrow[from=3-12, to=4-13]
	\arrow[from=3-16, to=2-18]
	\arrow[from=3-16, to=4-17]
	\arrow[from=4-3, to=3-4]
	\arrow[bend right = 6, from=4-3, to=7-4, end anchor={[shift={(-9pt,0pt)}]}]
	\arrow[from=4-5, to=5-6]
	\arrow[bend right = 6, from=4-5, to=7-6, end anchor={[shift={(-9pt,0pt)}]}]
	\arrow[from=4-7, to=3-8]
	\arrow[bend right = 6, from=4-7, to=7-8, end anchor={[shift={(-9pt,0pt)}]}]
	\arrow[from=4-9, to=5-10]
	\arrow[bend right = 6, from=4-9, to=7-10, end anchor={[shift={(-9pt,0pt)}]}]
	\arrow[from=4-11, to=3-12]
	\arrow[bend right = 6, from=4-11, to=7-12, end anchor={[shift={(-9pt,0pt)}]}]
	\arrow[from=4-13, to=5-14]
	\arrow[bend right = 6, from=4-13, to=7-14, end anchor={[shift={(-9pt,0pt)}]}]
	\arrow[from=4-15, to=3-16]
	\arrow[bend right = 6, from=4-15, to=7-16, end anchor={[shift={(-9pt,0pt)}]}]
	\arrow[from=4-17, to=5-18]
	\arrow[bend right = 6, from=4-17, to=7-18, end anchor={[shift={(-9pt,0pt)}]}]
	\arrow[from=4-19, to=3-20]
	\arrow[from=5-6, to=4-7]
	\arrow[from=5-6, to=6-8]
	\arrow[from=5-10, to=4-11]
	\arrow[from=5-10, to=6-12]
	\arrow[from=5-14, to=4-15]
	\arrow[from=5-14, to=6-16]
	\arrow[from=5-18, to=4-19]
	\arrow[from=5-18, to=6-20]
	\arrow[from=6-4, to=5-6]
	\arrow[from=6-8, to=5-10]
	\arrow[from=6-12, to=5-14]
	\arrow[from=6-16, to=5-18]
	\arrow[bend right = 6, from=7-2, to=4-3, start anchor={[shift={(9pt,0pt)}]}]
	\arrow[bend right = 6, from=7-4, to=4-5, start anchor={[shift={(9pt,0pt)}]}]
	\arrow[bend right = 6, from=7-6, to=4-7, start anchor={[shift={(9pt,0pt)}]}]
	\arrow[bend right = 6, from=7-8, to=4-9, start anchor={[shift={(9pt,0pt)}]}]
	\arrow[bend right = 6, from=7-10, to=4-11, start anchor={[shift={(9pt,0pt)}]}]
	\arrow[bend right = 6, from=7-12, to=4-13, start anchor={[shift={(9pt,0pt)}]}]
	\arrow[bend right = 6, from=7-14, to=4-15, start anchor={[shift={(9pt,0pt)}]}]
	\arrow[bend right = 6, from=7-16, to=4-17, start anchor={[shift={(9pt,0pt)}]}]
	\arrow[bend right = 6, from=7-18, to=4-19, start anchor={[shift={(9pt,0pt)}]}]
\end{tikzcd}\]

\caption{Indecomposable objects of $\C(\Q)$ for a Q-datum $\Q$ of type $\mathsf{F}_4$ (see Example \ref{example:F4} for more details).}
\label{figure:F4}
\end{sidewaysfigure}
\end{example}

\begin{example}\label{example:G2}
Suppose $(\Delta,\sigma) = (\mathsf{D}_4,\widetilde{\vee})$ is of type $\mathsf{G}_2$. Let $Q$ be the following orientation for $\Delta$:
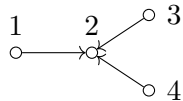
\begin{center}
	\begin{tikzpicture}
			
		\node[circle, draw, inner sep=1.5pt] (D1) at (3.5,-0.5) {};
		\node[above] at (3.5,-0.4) {$1$};
		\node[circle, draw, inner sep=1.5pt] (D2) at (4.5,-0.5) {};
		\node[above] at (4.5,-0.4) {$2$};
		\node[circle, draw, inner sep=1.5pt] (D3) at (5.25,0) {};
		\node[right] at (5.35,0) {$3$};
		\node[circle, draw, inner sep=1.5pt] (D4) at (5.25,-1) {};
		\node[right] at (5.35,-1) {$4$};
	
		\draw[->] (D1) -- (D2);
		\draw[->] (D3) -- (D2);
		\draw[->] (D4) -- (D2);
	
	\end{tikzpicture}
	\end{center}
Let $\Q = (\Delta,\sigma,\xi)$ be the Q-datum defined by $\xi_1 = 3$, $\xi_2 = 4$, $\xi_3 = 1$, and $\xi_4 = 5$. The indecomposable objects of $\C(\Q)$ are shown below.
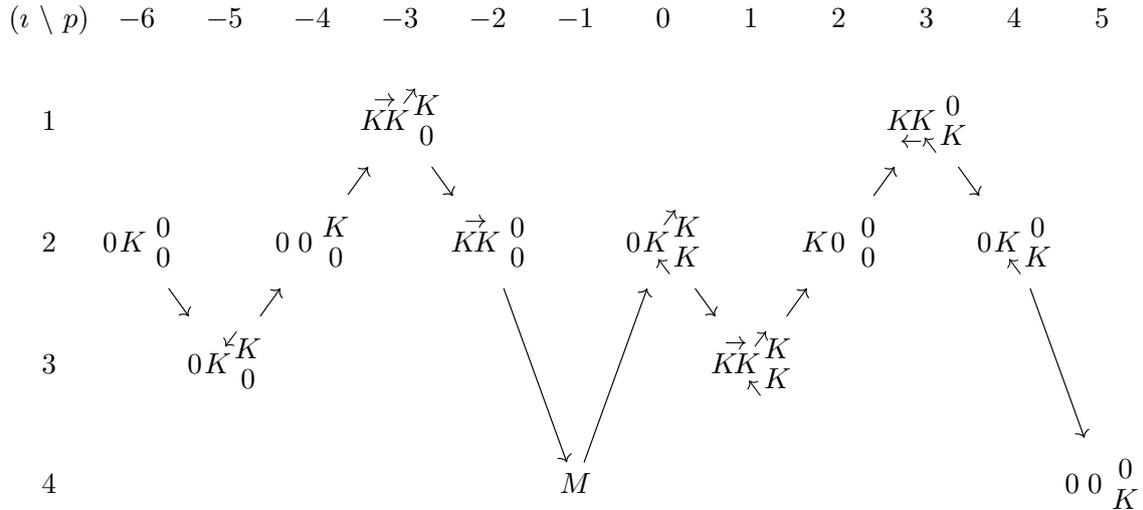
\[\begin{tikzcd}[column sep={3em,between origins},row sep={1em}]
	{(\im \ \backslash \ p)} & {-6} & {-5} & {-4} & {-3} & {-2} & {-1} & 0 & 1 & 2 & 3 & 4 & 5 \\
	1 &&&& {\begin{tikzpicture}

		\node at (0,0) {$K$};
		\node at (0.3,0) {$K$};
		\node at (0.7,0.2) {$K$};
		\node at (0.7,-0.2) {$0$};
		
		\draw[->] (0.02,0.4) -- (0.28,0.4);
		\draw[->] (0.4,0.35) -- (0.55,0.55);
		
		\end{tikzpicture}} &&&&&& {\begin{tikzpicture}

		\node at (0,0) {$K$};
		\node at (0.3,0) {$K$};
		\node at (0.7,0.2) {$0$};
		\node at (0.7,-0.2) {$K$};
			
		\draw[->] (0.28,-0.15) -- (0.02,-0.15);
		\draw[->] (0.475,-0.3) -- (0.325,-0.1);
			
		\end{tikzpicture}} \\
	2 & {\begin{tikzpicture}

		\node at (0,0) {$0$};
		\node at (0.3,0) {$K$};
		\node at (0.7,0.2) {$0$};
		\node at (0.7,-0.2) {$0$};
		
		\end{tikzpicture}} && {\begin{tikzpicture}

		\node at (0,0) {$0$};
		\node at (0.3,0) {$0$};
		\node at (0.7,0.2) {$K$};
		\node at (0.7,-0.2) {$0$};
			
		\end{tikzpicture}} && {\begin{tikzpicture}

		\node at (0,0) {$K$};
		\node at (0.3,0) {$K$};
		\node at (0.7,0.2) {$0$};
		\node at (0.7,-0.2) {$0$};
		
		\draw[->] (0.02,0.4) -- (0.28,0.4);
		
		\end{tikzpicture}} && {\begin{tikzpicture}

		\node at (0,0) {$0$};
		\node at (0.3,0) {$K$};
		\node at (0.7,0.2) {$K$};
		\node at (0.7,-0.2) {$K$};
		
		\draw[->] (0.4,0.35) -- (0.55,0.55);
		\draw[->] (0.475,-0.3) -- (0.325,-0.1);
		
		\end{tikzpicture}} && {\begin{tikzpicture}

		\node at (0,0) {$K$};
		\node at (0.3,0) {$0$};
		\node at (0.7,0.2) {$0$};
		\node at (0.7,-0.2) {$0$};
		
		\end{tikzpicture}} && {\begin{tikzpicture}

		\node at (0,0) {$0$};
		\node at (0.3,0) {$K$};
		\node at (0.7,0.2) {$0$};
		\node at (0.7,-0.2) {$K$};
		
		\draw[->] (0.475,-0.3) -- (0.325,-0.1);
		
		\end{tikzpicture}} \\
	3 && {\begin{tikzpicture}

		\node at (0,0) {$0$};
		\node at (0.3,0) {$K$};
		\node at (0.7,0.2) {$K$};
		\node at (0.7,-0.2) {$0$};
		
		\draw[->] (0.55,0.55) -- (0.4,0.35);
		
		\end{tikzpicture}} &&&&&& {\begin{tikzpicture}

		\node at (0,0) {$K$};
		\node at (0.3,0) {$K$};
		\node at (0.7,0.2) {$K$};
		\node at (0.7,-0.2) {$K$};
		
		\draw[->] (0.02,0.4) -- (0.28,0.4);
		\draw[->] (0.4,0.35) -- (0.55,0.55);
		\draw[->] (0.475,-0.3) -- (0.325,-0.1);
		
		\end{tikzpicture}} \\
	4 &&&&&& M &&&&&& {\begin{tikzpicture}

		\node at (0,0) {$0$};
		\node at (0.3,0) {$0$};
		\node at (0.7,0.2) {$0$};
		\node at (0.7,-0.2) {$K$};
		
		\end{tikzpicture}}
	\arrow[from=2-5, to=3-6]
	\arrow[from=2-11, to=3-12]
	\arrow[from=3-2, to=4-3]
	\arrow[from=3-4, to=2-5]
	\arrow[from=3-6, to=5-7]
	\arrow[from=3-8, to=4-9]
	\arrow[from=3-10, to=2-11]
	\arrow[from=3-12, to=5-13]
	\arrow[from=4-3, to=3-4]
	\arrow[from=4-9, to=3-10]
	\arrow[from=5-7, to=3-8]
\end{tikzcd}\]
Here, the object $M$ is the following representation:
\[\begin{tikzcd}
	&& K \\
	K & {K^2} \\
	&& K
	\arrow["{i_1}", bend left=20, from=1-3, to=2-2]
	\arrow["{-i_1}", bend left=20, from=2-1, to=2-2]
	\arrow["{\pi_2}", bend left=20, from=2-2, to=1-3]
	\arrow["{\pi_2}", bend left=20, from=2-2, to=2-1]
	\arrow["0", bend left=20, from=2-2, to=3-3]
	\arrow["{i_2}", bend left=20, from=3-3, to=2-2]
\end{tikzcd}\]
\end{example}

\sloppy
\printbibliography

\end{document}